\documentclass[11pt, reqno]{amsart}

\usepackage{graphicx,color,amsmath,amssymb,mathrsfs}
\usepackage{amsfonts,amsthm,epsfig,epstopdf,esint}
\usepackage{hyperref, mathtools, soul}
\usepackage[mathscr]{euscript}
\usepackage[T1]{fontenc}
\usepackage[utf8]{inputenc}
\usepackage{mathtools} 
\usepackage{amsthm}
\usepackage{verbatim}

\usepackage{latexsym, amssymb, amsmath,esint}

\usepackage{amsfonts, graphicx}
\usepackage{graphicx,color}

\newtheorem{thm}{Theorem}[section]
\newtheorem{theorem}[thm]{Theorem}
\newtheorem{lemma}[thm]{Lemma}
\newtheorem{proposition}[thm]{Proposition}
\newtheorem{corollary}[thm]{Corollary}
\newtheorem{remark}[thm]{Remark}
\newtheorem{definition}[thm]{Definition}
\newtheorem{example}[thm]{Example}

\newtheorem{conjecture}{Conjecture}[section]

\usepackage[utf8]{inputenc}
\usepackage{calc}
\usepackage{accents}

\newcommand{\R}{{\mathbb{R}}}
\newcommand{\e}{{\varepsilon}}
\newcommand{\g}{{\gamma}}

\newcommand{\pa}{{\partial}}

\newcommand{\vphi}{{\varphi}}
\newcommand{\W}{{\mathcal{W}}}

\newcommand{\uB}{{\underline{B}}}
\newcommand{\PP}{{P}}

\newcommand{\Ric}{{\rm{Ric}}}
\newcommand{\Rm}{{\rm{Rm}}}
\newcommand{\vol}{{\rm{vol}}}

\newcommand{\Om}{\Omega}
\newcommand{\om}{\omega}
\newcommand{\na}{{\nabla}}

\newcommand{\ETA}{{{\lambda}}}
\newcommand{\BETA}{{\theta}}

\newcommand{\B}{{\mathcal{B}}}

\definecolor{grey}{rgb}{.7,.7,.7}

\numberwithin{equation}{section}
\title[$d_p$-convergence with entropy and scalar lower bounds]{$d_p$ convergence and $\epsilon$-regularity theorems for entropy and scalar curvature lower bounds}

 \author{Man-Chun Lee}
\address{Department of Mathematics, Northwestern University, 2033 Sheridan Road, Evanston, IL 60208}
\address{Mathematics Institute, Zeeman Building,
University of Warwick, Coventry CV4 7AL}
\email{mclee@math.northwestern.edu, Man.C.Lee@warwick.ac.uk}

\author{Aaron Naber}
\address{Department of Mathematics, Northwestern University, 2033 Sheridan Road, Evanston, IL 60208}
\email{anaber@math.northwestern.edu}

\author{Robin Neumayer}
\address{Department of Mathematics, Northwestern University, 2033 Sheridan Road, Evanston, IL 60208}
\email{neumayer@math.northwestern.edu}

\begin{document}
\maketitle
\begin{abstract}
Consider a sequence of Riemannian manifolds $(M^n_i,g_i)$ whose scalar curvatures and entropies are bounded from below by small constants $R_i,\mu_i \geq -\epsilon_i$.  The goal of this paper is to understand notions of convergence and  the structure of limits for such spaces.  As a first issue, even in the seemingly rigid case $\epsilon_i\to 0$, we will construct examples showing that from the Gromov-Hausdorff or Intrinsic Flat points of view,  such a sequence may converge wildly, in particular to metric spaces with varying dimensions and topologies and at best a Finsler-type structure.  On the other hand, we will see that these classical notions of convergence are the incorrect ones to consider.  Indeed, even a metric space is the wrong underlying category to be working on. \\ 

Instead, we will introduce a weaker notion of convergence called  $d_p$ convergence that is valid for a class of rectifiable Riemannian  spaces.  These rectifiable spaces will have a well-behaved topology, measure theory, and analysis.  This includes the existence of gradients of functions and absolutely continuous curves, though potentially there will be no reasonably associated distance function.  Under this $d_p$ notion of closeness, a space with almost nonnegative scalar curvature and small entropy bounds must in fact always be close to Euclidean space, and this will constitute our $\epsilon$-regularity theorem.  In particular, any sequence $(M^n_i,g_i)$ with lower scalar curvature and entropies tending to zero must $d_p$ converge to Euclidean space.  \\

More generally, we have a compactness theorem saying that sequences of Riemannian manifolds $(M^n_i,g_i)$ with small lower scalar curvature and entropy bounds $R_i,\mu_i \geq -\epsilon$ must $d_p$ converge to such a  rectifiable  Riemannian space $X$.  In the context of the examples from the first paragraph, it may be that the distance functions of $M_i$ are degenerating, even though in a well-defined sense the analysis cannot be.  Applications for manifolds with small scalar and entropy lower bounds include an $L^\infty$-Sobolev embedding and apriori $L^p$ scalar curvature bounds for $p<1$.

\end{abstract}

\tableofcontents
\section{Introduction}

It is a well known theme that understanding the structure of a manifold $(M^n,g)$ under restrictions on curvature is essentially equivalent to understanding the structure of singular limits $M^n_i\to X$.  During the early days of studying manifolds with bounded curvature operator, it was sufficient to restrict the study to manifold limits $X$ under $C^{k,\alpha}$ convergence \cite{CheegerThesis, CheegerFiniteness}.  When the analysis of spaces with lower and bounded Ricci curvature began, it became necessary to expand this point of view to general metric space limits $X$, and to discuss convergence in the Gromov-Hausdorff sense \cite{Gromov2007}.  This allowed for the necessary formation of singularities in possible limit spaces.  It also became quite important at this stage to distinguish between {\it collapsed} and {\it noncollapsed} limits, where noncollapsing of the sequence $M^n_i$ can be understood as the existence of a uniform lower bound on the volumes of balls.  A key result in this context, and indeed the beginning point for the regularity theory, is an $\epsilon$-regularity theorem.  This says that if the volume of a unit ball is close to that of the Euclidean ball, then that ball must be close both topologically and geometrically to a Euclidean ball.\\

Our goal in this paper is to study manifolds and sequences $M^n_i$ under lower bounds on scalar curvature.  The correct replacement for {\it noncollapsing} in this context is a lower bound on the entropy $\mu$ of the manifold, or almost equivalently one could ask for bounds on the $L^1$-Sobolev constant, though we will see there are unnatural aspects to that assumption.  Our goal is then to prove and understand the corresponding $\epsilon$-regularity in this context: a statement which should say that if the scalar and entropy lower bounds are small, then a ball should be close to a Euclidean ball.\\

It is already understood from \cite{SormaniSurvey} that there is an immediate problem when dealing with lower scalar curvature bounds, namely the notion Gromov-Hausdorff closeness cannot be the correct one.  The examples in \cite{SormaniSurvey} mimic those from minimal surface theory, and show that small volume tentacles may appear when only a lower scalar curvature bound is assumed.  One possible fix for issues like this is the Intrinsic Flat distance \cite{sormani2011intrinsic}.  We will see in this paper that the problem is actually much worse.  We will build examples, see Theorems \ref{thm: counterex}--\ref{thm: fixed delta} and Section~\ref{examples-entropy}, which show that even under small lower bounds on scalar curvature and entropy, the Gromov-Hausdorff and Intrinsic Flat limits may be completely wild.  Wild here can include jumps in topology, dimension, and the formation of Finsler or worse types of geometries.  Fundamentally, the issue at hand is that distance functions simply do not behave well under lower scalar curvature and entropy bounds, and therefore any notion of convergence which is based on the distance function must correspondingly fail.  From the correct perspective, this should not be surprising as the distance function is closely related to the $W^{1,\infty}$ behavior of functions, and it may simply be too much to ask that this remains uniformly controlled in such a sequence.  Indeed, it is now well understood from the study of RCD spaces \cite{Ambrosio_CalcHeat,Bakry-Emery,Sturm_ricci,LottVillani} that (said correctly) $W^{1,\infty}$-control on the analysis is essentially equivalent to lower bounds on {\it Ricci} curvature, and therefore one should almost expect distance functions to break down in the context of only scalar curvature bounds. \\

In order to solve this problem we will introduce a new notion of convergence, $d_p$ convergence,  in this paper.  The effect of this will be to take the required $W^{1,\infty}$-control needed for convergence of distance functions, and reduce it to a required $W^{1,p}$-control for this weaker notion of convergence.   The notion of $d_p$ convergence is based on associating to a manifold, or more generally a rectifiable space, a natural family of distance functions $d_p$.  As we will see, $d_p$ understands and controls the behavior of the Sobolev space $W^{1,p}$, with $d_\infty=d$ becoming the standard distance function.  Let us begin with a definition:

\begin{definition}[$d_p$ distance on manifolds]
	Given a Riemannian manifold $(M^n,g)$ and a real number $p \in (n, \infty]$, we define the $d_{p,g}$ distance between any $x,y \in M$  by
\begin{align}
d_{p,g}(x,y) = d_p(x,y) = \sup\left\{ |f(x)-f(y)| \ : \ \int_M |\na f|^p \,d\vol_g \leq 1 \right\}.	
\end{align}
\end{definition}
\begin{remark}{\rm
It is interesting to check that this is an actual distance function, and for $p\to \infty$ this converges to the usual geodesic distance.
} 
\end{remark}

We will discuss this more precisely in Section~\ref{subsec: dp}, but let us observe that $d_p$ does not need an underlying metric structure in order to be defined.  A rectifiable structure, which gives the ability to differentiate functions and integrate them, will be sufficient.  In particular, the functions $d_p$ are well-defined on rectifiable Riemannian  spaces $(X,g)$.  These are precisely defined in Definition \ref{def: rect RM}, but roughly are topological measure spaces with a compatible rectifiable structure and Riemannian metric on the rectifiable charts. \\ 

For such a rectifiable space $X$, it may be that  $d_p(x,y)=0$ for $p$ sufficiently large, and thus $d_p$ only defines a weak distance function.  This will be possible even for limits of manifolds under small lower scalar curvature and entropy bounds $R,\mu\geq -\epsilon$; see Example \ref{example: delta fixed}.  We will say $X$ is $d_p$-complete when it defines an honest distance function whose topology is that of $X$, see Section \ref{subsec: dp} for a larger account of the subtle points which arise.  A consequence of our main theorems is that for $p\leq p(n,\epsilon)$ limits will be $d_p$ complete, and indeed $d_p$ is actually very well behaved. In particular, such limits $X$ will be doubling spaces up to scale $1$ with respect to the $d_p$ distances. 
\\

Now that we have the $d_p$ distance defined and the correct category of spaces to consider it on, namely rectifiable spaces $X$ with a Riemannian structure, let us consider their convergence.  As is usual let us begin with the compact case:

\begin{definition}[$d_p$ convergence]
A sequence $\{(X_i, g_i)\}$ of compact rectifiable Riemannian  spaces, in particular a sequence of compact Riemannian manifolds, converges to a compact rectifiable Riemannian space $(X,g)$  in the $d_p$ sense if 
\begin{align}
d_{mGH}\left( (X_i, d_{p,g_i}, d\vol_{g_i} ) , (X, d_{p,g}, d\vol_g) \right) \to 0\, .	
\end{align}
Here $d_{mGH}$ denotes the measured Gromov Hausdorff distance between metric spaces. 
\end{definition}
\begin{remark}\rm{
One could easily replace mGH with Intrinsic Flat distance in the above and the results of this paper are the same, the main key for us is the weakening of the usual distance with $p=\infty$ to $p<\infty$.
}
\end{remark}
\begin{remark}\rm{
Even if a sequence $\{(M_i,g_i)\}$ of Riemannian manifolds has a (geodesic) Gromov-Hausdorff limit $(Y,d)$, the spaces $X$ and $Y$ need not even be topologically equivalent.
} 
\end{remark}
\begin{remark}{\rm
Let us briefly mention that pointed convergence for noncompact spaces is defined in a similar spirit	 as in the Gromov-Hausdorff case, however there is a subtle point due to the behavior of $d_p$ at large distances.  See Definition~\ref{def: dp convergence} for precision.
}
\end{remark}

Throughout the paper we will let $\B_{p,g}(x, r)$ denote the ball of radius $r$ with respect to $d_p$. That is,
\begin{equation}\label{eqn: def of pball}
	\B_{p,g}(x, r) = \{ y \in M : d_{p}(x,y) <r\}.
\end{equation}


\vspace{.4cm}

\subsection{Main $\epsilon$-Regularity Theorem}

Let us now move toward our first main result of the paper.  We begin by recalling that the Perelman $\W$-functional, introduced in \cite{perelman2002entropy}, is defined for a function $f \in C^\infty(M)$ and real number $\tau>0$ by

\begin{equation}\label{eqn: W functional}
\W(g,f,\tau) = \frac{1}{(4\pi \tau)^{n/2}}\int_M \left\{ \tau(|\na f|^2 + R) +f-n\right\} {e^{-f}} \, d \vol_g  \,.
\end{equation}
The Perelman entropy $\mu(g,\tau)$, which can be viewed as the optimal constant in a log-Sobolev inequality at scale $\tau^{1/2}$, is given by 
\begin{equation}\label{eqn: perelman entropy def}
	\mu(g,\tau) = \inf \left\{ \W(g,f,\tau) : \frac{1}{(4\pi \tau)^{n/2}}\int_M {e^{-f}} \, d\vol_g= 1,{e^{-f/2}\in W^{1,2}(M)}\right\}\,.
\end{equation}
 Finally, Perelman's $\nu$-functional is given by 
\begin{equation*}
	\nu(g,\tau) = \inf\{ \mu(g,\tau'): \tau'\in (0, \tau)\}\, ,
\end{equation*}
and just guarantees that we are measuring the entropy at all scales below some point.  See Section~\ref{subsec: entropy} for more background.  The Perelman entropy $\mu(g, \tau)$ of a complete well-behaved Riemannian manifold $(M,g)$ is nonpositive for all $\tau >0$. Moreover, if the entropy is equal to zero for some $\tau>0$, then $(M,g)$ is isometric to Euclidean space. This rigidity statement is the basis of our  first main result, which is perturbative in nature.

\begin{theorem}[$\epsilon$-Regularity Theorem]\label{thm: main thm, Lp def new}
Let $(M^n,g)$ be a complete Riemannian manifold with bounded curvature
and fix $\e>0$ and $p\geq n+1$.  There exists $\delta = \delta(n,\e,p)$ such that if
	\begin{align}\label{eqn: th 1.2 hp}
		R\geq -\delta, &\qquad \qquad \nu(g, 2) \geq -\delta\, ,
	\end{align}
then for all $x\in M$, we have
 \begin{equation}\label{eqn: dp statement intro}
	d_{GH}\left( (\B_{p,g}(x, 1), d_{p,g}) ,\,  (\B_{p,g_{euc}}(0, 1), d_{p,g_{euc}})\right) \leq \e 
\end{equation}
and for any $0<r\leq 1$
\begin{equation}
(1-\e) |\B_{p,g_{euc}}(0, r) | \leq \vol_{g} (\B_{p,g}(x_0, r)) \leq (1+\e) |\B_{p,g_{euc}}(0, r) |.
\end{equation}
 Here $|\cdot|$ denotes the Euclidean volume. In particular, the measure $d\vol_g$ on the metric measure space $(M,d_{p,g},d\vol_g)$ is a doubling measure for all scales $r\leq 1$.
\end{theorem}
\begin{remark}{\rm
The assumption of (nonuniformly) bounded curvature is simply to control degeneration at infinity of $M$, a local version of these statements would drop this condition.
}
\end{remark}
\begin{remark}[$L^1$ Sobolev constant]
	{\rm 
	We may replace the entropy lower bound  in Theorems~\ref{thm: main thm, Lp def new} by a rigid bound on the $L^1$-Sobolev constant. Namely, we may replace the assumption $\nu(g,2) \geq -\delta$ in \eqref{eqn: th 1.2 hp} with the assumption that for all compactly supported $f:B_{g}(x, 1)\to \R$ with $x\in M$ we have
	\begin{align}
	\left( \int_M |f|^{\frac{n}{n-1}}\right)^{\frac{n-1}{n}}\leq (1+\delta)c_n\int_M |\nabla f|\, ,
	\end{align}
where $c_n$ is the sharp Sobolev constant on Euclidean space.  However, we avoid focusing on this because, as we will see, metric balls are badly behaved objects, and thus any condition which used a metric ball may be more restrictive than it appears.  The $\mu$-entropy intrinsically understands the correct $d_p$ distance, and thus $\mu(g,1)$ becomes a condition on the unit $d_p$-scale, as opposed to the $d=d_\infty$ scale.

	}
\end{remark}

 \begin{remark}[Scaling]\label{rmk: hyp under rescaling}
 	{\rm
 For any Riemannian manifold $(M,g)$, the rescaled metric $\tilde{g} = r^{-2}g$ satisfies   $\mathcal{B}_{p, \tilde{g}}(x_0,\rho) = \mathcal{B}_{g,p}(x_0,\rho r^{1-n/p})$, $R_{\tilde{g}} = r^{-2} R_g$,  $\nu(g,2 r^2 ) = \nu(\tilde{g}, 2)$. 
 If $(M,g)$ is closed or is well behaved at infinity (see \cite{ZhangLS} for instance), then
 $ 	\lim_{\tau \to 0} \mu(g,\tau) =0.$
In particular, for any such Riemannian manifold $(M,g)$, the hypotheses of Theorem~\ref{claim1} hold at some scale.
}

 \end{remark}

\vspace{.2cm}

The proof of Theorem \ref{thm: main thm, Lp def new} depends on the following, which guarantees the existence of $W^{1,p}$ charts on an $\epsilon$-regularity ball.  Further, one is able to get that for large but finite $p$ it is possible to control the $W^{1,p}$ energies of limiting functions; this connects to  the perspective discussed above of $d_p$ convergence  as  a type of convergence of Sobolev spaces.\\

\begin{theorem}[$L^p$ estimates for the metric coefficients]\label{thm: RF}\label{claim1}
Let $(M^n,g)$ be a complete Riemannian manifold with bounded curvature.
Fix, $\e>0$, $\kappa>1$ and $p \in [\kappa,\infty)$. There exists $\delta=\delta(n,p,\kappa, \e)>0$ such that if
	\begin{align}
		R\geq -\delta, &\qquad \qquad \nu(g, 2) \geq -\delta\, ,
	\end{align}
then for any $x \in M$ there exist an open set $\Omega\subset M$ containing $x$ and a smooth diffeomorphism 
	 $\psi: \Omega \to B(0,1) \subset \R^n$ with $\psi(x)=0$ 	  satisfying 
\begin{align}\label{eqn: W1p ests}
  \fint_{B(0,1)} |(\psi^{-1})^*g -g_{euc} |^{p} \, dy \leq \e, \qquad
 \fint_{\Omega} |\psi^* g_{euc} - g|^{p} \, d\vol_g \leq \e\, .
 \end{align}	
Furthermore, for any $f \in W^{1,p}(B(0,1))$, we have 
	\begin{align}
\label{eqn: Lp norms small error}	 	
(1-\e)\| \psi^* f\|_{L^{p/\kappa}(\Om)} &\leq \| f\|_{L^p(B(0,1))} \leq (1+\e) \| \psi^* f\|_{L^{\kappa p}(\Om)},\\
\label{eqn: gradient small error}
(1-\e)\| \na \psi^* f\|_{L^{p/\kappa}(\Om)} &\leq \|\nabla f\|_{L^p(B(0,1))} \leq (1+\e) \|\na \psi^* f\|_{L^{\kappa p}(\Om)} 	. 	
	 \end{align}
\end{theorem}
The notation  $\fint_\Omega u\, d\vol_{g}$ is used to denote ${\vol_{g}(\Omega)}^{-1} \int_\Omega u\, d\vol_{g}.$  In \eqref{eqn: W1p ests}, the notation $|\cdot |$ indicates the tensor norm with respect to $g_{euc}$ and $g$ respectively. 
\medskip

\vspace{.2cm}
\subsection{Examples and Counter-Examples}

We have been explaining from the beginning what can fail as one converges with sequences of spaces with lower scalar curvature and entropy bounds.  In particular, we have discussed how the distance function itself is almost entirely uncontrollable.  Let us now make this precise, and in the process see that the $d_p$ convergence in Theorem~\ref{thm: main thm, Lp def new} cannot be replaced with  Gromov-Hausdorff convergence or Intrinsic Flat convergence:

\begin{theorem}[Counterexample to Gromov-Hausdorff Convergence]\label{thm: counterex} Fix $n\geq 4$ and $\e>0$. 
There exists a sequence of metrics $(\mathbb{T}^n,g_i)$ on the $n$-dimensional torus with $\delta_i\to 0$ such that
	\begin{align}\label{eqn: hp counter ex}
		R_{g_{i}}\geq -\delta_i,  \qquad \qquad \nu(g_{i}, 2) \geq -\e\, ,
	\end{align} 
and such that $(\mathbb{T}^n,g_{i})$ converges in the Gromov-Hausdorff topology to a point and in the Intrinsic Flat topology to the zero current as $i \to \infty$. On the other hand, the sequence $(\mathbb{T}^n, g_i)$ converges to a flat torus $(\mathbb{T}^n, g_{flat})$ in the $d_p$ sense for all finite $p \in [n+1,\infty).$
\end{theorem}
The example of Theorem~\ref{thm: counterex} is given in Example~\ref{example-torus-3}.
Preserving the lower scalar curvature in the above example is not too challenging, but showing that the entropies are well behaved takes quite a bit more work.  Philosophically, this example is similar to situations studied very recently by Allen-Sormani \cite{AllenSormani}, without the lower scalar curvature and entropy requirements, see also \cite{AllenSormani2}.

In fact, in Section~\ref{examples-entropy}, we construct a variety of other compact and noncompact examples of sequences satisfying \eqref{eqn: hp counter ex} such that the Gromov-Hausdorff and Intrinsic Flat limits are not locally Euclidean. For example, we have the following.
\begin{theorem}\label{thm: lines on rn} Fix $n\geq 4$. 
	There exist a sequence of metrics  $(\R^n, g_i)$ satisfying 
	\begin{equation}
		R_{g_i} \geq - 1/i, \qquad \qquad \nu(g_i, 2) \geq -1/i.
	\end{equation}
	that converge in the $d_p$-sense to flat Euclidean space for any $p  \in [n+1, \infty)$, but whose pointed Gromov-Hausdorff limit is $(\R^n, \ell^\infty)$, i.e. the taxicab  metric on Euclidean space. 
\end{theorem}

The example of Theorem~\ref{thm: lines on rn} is given  in Example~\ref{example: lines on rn}.
Theorems~\ref{thm: counterex} and \ref{thm: lines on rn} demonstrate that one cannot replace $d_p$ closeness with Gromov-Hausdorff or Intrinsic Flat closeness in Theorem~\ref{thm: main thm, Lp def new}.
Furthermore, the following theorem shows that the $p$ for which we establish $d_p$ convergence in Theorem~\ref{thm: main thm, Lp def new}  cannot be taken arbitrarily large for fixed  $\delta.$
\begin{theorem}\label{thm: fixed delta}
	Fix $n\geq 4$ and $\delta>0$.  There exists a sequence $(\R^n, g_i)$ that satisfies 
	\begin{equation}
				R_{g_i} \geq - \delta, \qquad \qquad \nu(g_i, 2) \geq -\delta
	\end{equation}
and a singular metric $g_\infty$ on $\R^n$ such that $(\R^n, g_i)$ converges in $d_p$ to $(\R^n, g_\infty)$, as a rectifiable Riemannian space, for all $p \in [n+1, p_0)$ for some $p_0= p_0(\delta)$, but  does not $d_p$-converge to $(\R^n, g_\infty)$ for $p \geq p_0$. 
\end{theorem}
The example of Theorem~\ref{thm: fixed delta} is given in Example~\ref{example: delta fixed}.

\vspace{.2cm}
\subsection{Structure of Limit Spaces}

The next main result of the paper is the following compactness result and structure theorem for limit spaces.  We show that under almost nonnegative scalar curvature and entropy, we indeed have a rectifiable Riemannian  limit $X$.  

\begin{theorem}[Structure of limit spaces]\label{global-thm} 
Let $\{(M_i,g_i,x_i)\}$ be a sequence of complete pointed Riemannian manifolds with bounded curvature and let $p\geq n+1$. Then there exists $\delta=\delta(n,p)>0$ such that if
\begin{equation}
R_{g_i}\geq -\delta, \qquad\qquad \nu(g_i,2)\geq -\delta\, ,
\end{equation}
then there exists a pointed rectifiable Riemannian space $(X, g, x)$, with  $X$ topologically but not necessarily metrically a smooth manifold, such that the following holds.
\begin{enumerate}
	\item We have $d_p((M_i, g_i, x_i),  (X, g, x))\to 0$ in the pointed sense of Definition~\ref{def: dp convergence}.
	\item  The space $(X, g, x)$ is  $W^{1,p}$-rectifiably complete and $d_p$-rectifiably complete, in the sense of Definitions~\ref{def: rect completeness} and \ref{def: dp RC}  respectively.
\end{enumerate}
\end{theorem}

The first part of the above theorem just tells us that there exists a rectifiable space $X$ to which the $M_i$ converge.  As we have emphasized, it may be that $X$ does not have a well behaved metric structure and this convergence may not be in the Gromov-Hausdorff or Intrinsic Flat sense.  The second condition in the above theorem touches on some subtle points we have avoided in the introduction, and essentially tells us that $X$ is a well-behaved rectifiable space which behaves the way one might feel it should in a reasonable scenario.  In particular, the gradient of a function is indeed the coordinate gradient that one would compute in rectifiable charts, and the metric $d_p$ generates the topology of $X$.  Note for $q>>p$ this may fail for $d_q$.




\vspace{.2cm}
\subsection{Further Results under Lower Scalar Curvature and Entropy}

Finally, let us conclude by discussing some applications of the results to the underlying structure of spaces with lower scalar curvature and entropy bounds.  To begin, we obtain on such spaces an apriori $L^q$ bound for the scalar curvature, for $q<1$:\

\begin{theorem}[$L^q$ scalar curvature estimates]\label{thm: integral scalar bound}
Let $(M^n,g)$ be a closed Riemannian manifold and let $\epsilon>0$ and $q\in (0,1)$ be fixed.  There exists $\delta=\delta(n,q,\e)>0$ such that if
	\begin{align}
		R\geq -\delta, 	\qquad	\qquad	\nu(g, 2) \geq -\delta \, ,
	\end{align}
then we have 
	 \begin{equation}
	 	\fint_{M} |R|^q\,d\vol_{g}\leq \e\, .
	 \end{equation}
\end{theorem}

Motivated by the above we conjecture the following:

\begin{conjecture}
Let $(M^n,g)$ be a closed Riemannian manifold with $R,\nu(g, 2)\geq -A$, then there exists $B(n,A)>0$ such that 
\begin{equation}
	 	\fint_{M} |R|\,d\vol_{g}\leq B\, .
	 \end{equation}	
\end{conjecture}
\vspace{.1cm}

In our last main result we prove that Riemannian manifolds satisfying a uniform lower bound on entropy and scalar curvature satisfy a Morrey-Sobolev embedding with a uniform constant.

\begin{theorem}[$L^\infty$ Sobolev Embedding]\label{Sob-embedding}
Let $(M^n,g)$ be a complete Riemannian manifold with bounded curvature and let $p\geq n+1$ and $q>n$.    There exists $\delta= \delta(n,p,q)>0$ and $C_{n,q}>0$ such that if 
		\begin{align}
		\label{eqn: scalar lower bound}
		R\geq -\delta, 	\qquad	 \qquad \nu(g, 2) \geq -\delta \, ,
	\end{align}
then for all $f \in W^{1,q}(M)$, we have 
	\begin{equation}\label{So-em global-intro}
\|f \|_{L^\infty(M)}\leq  C_{n,q} \left(\|\nabla f\|_{L^{q}(M)}+ \|f\|_{L^{q}(M)}\right).
\end{equation}
More locally, for all  $x_0 \in M$ and $f \in W_0^{1,q}(\B_{p,g}(x_0, 1))$, we have 
	\begin{equation}\label{So-em-intro}
		\| f\|_{L^\infty(\B_{p,g}(x_0,1))} \leq  C_{n,q} \| \na f\|_{L^q(\B_{p,g}(x_0,1))}\, .
	\end{equation}
In terms of the $d_p$ distance we can upgrade this to a H\"older embedding: 
there exists $\alpha= \alpha(n,q) \in (0,1)$	\begin{equation}\label{So-em-2-intro}
	|f(x)-f(y)|\leq C_{n,q, p} d_p(x,y)^{\alpha}\|\nabla f\|_{L^q(\B_{p,g}(x_0,1))}\, 
	\end{equation}
for all $x,y\in \B_{p,g}(x_0,1)$.
\end{theorem}

\begin{remark}
	{\rm 
	The examples of Section~\ref{examples-entropy} demonstrate that the H\"{o}lder embedding of \eqref{So-em-intro} cannot hold with the geodesic distance in place of the $d_p$-distance.
	}
\end{remark}

\vspace{.1cm}
The following theorem provides a type of stability for a theorem of Schoen-Yau \cite{Schoen_1979}  and Gromov-Lawson \cite{GromovLawson} which states that a metric of nonnegative scalar curvature on a torus must be flat. Stability for this rigidity theorem statement was conjectured by Gromov in \cite{Gromov14}, with a more concrete formulation of the conjecture given by Sormani in \cite{SormaniSurvey}. Progress toward this conjecture has been made in various cases. The first developments were due to Gromov \cite{Gromov14}, also established by Bamler using Ricci flow \cite{Bamler16}, which showed that  if a sequence of metrics $g_i$ on a torus that converge in $C^0$ to a $C^2$ metric $g$ and have $R_{g_i} \geq -1/i$, then $g$ is the flat metric. Using (regularizing) Ricci flow, Burkhardt-Guim \cite{BG19} extended this result to limiting metrics that are only $C^0$, and also proved a generalization of the rigidity result to  $C^0$ metrics with non-negative scalar curvature in a weak sense.  Further progress toward this conjecture, in the form stated by  Sormani in \cite{SormaniSurvey}, has been made in the setting of warped product metrics \cite{WarpedTorus} by Allen, Hernandez-Vazquez, Parise, Payne, and Wang, graphical tori \cite{GraphTorus} by Cabrera Pacheco, Ketterer, and Perales, and metrics that are conformal to the flat metric \cite{AllenConformal} by Allen. We note that the hypotheses in Theorem~\ref{thm: flat torus} differ from those in the conjecture of Sormani; most notably our assumption of an entropy lower bound takes the place of the lower  bound on the minA quantity there, and here stability is with respect to the $d_p$ distance rather than the Intrinsic Flat distance.

\begin{theorem}\label{thm: flat torus}
	Fix $n\geq 2$ and $ p \geq n+1$. There exists $\delta = \delta(n,p)$ and $V_0 = V_0(n,p)$ such that the following holds. For any $V\geq V_0$, let $(M_i, g_i)$ be a sequence of compact Riemannian manifolds,  diffeomorphic to tori,  with $\vol_{g_i}(M_i) \leq V$  and satisfying 
	\begin{equation}
		\nu( g_i, 2) \geq -\delta , \qquad R_i \geq -1/i
	\end{equation}
	 Then $(M_i, g_i)$ converges in the $d_p$ sense to a flat torus with $\nu(g , 2)  \geq -\delta $.
\end{theorem}

\subsection{Acknowledgements}
 M.-C. Lee was partially supported by the National Science Foundation under Grant No. DMS-1709894 and EPSRC grant number EP/T019824/1. R. Neumayer was partially supported by the National Science Foundation under Grant No. DMS-1901427, as well as Grant No. DMS-1502632 ``RTG: Analysis on manifolds'' at Northwestern University and Grant No. DMS-1638352 at the Institute for Advanced Study.  Aaron Naber was partially supported by the National Science Foundation Grant No. DMS-1809011. We thank Brian Allen, Dan Lee, and Christina Sormani for helpful comments.


\bigskip
\section{Rectifiable Riemannian spaces and convergence}\label{sec: Lp}

\subsection{Rectifiable Riemannian spaces}
We introduce precisely  the notion of rectifiable Riemannian spaces, which are the objects which arise as limits in the $d_p$ sense of Riemannian manifolds with uniform lower bounds on scalar curvature and on the entropy.

Let $X$ be a Hausdorff topological space equipped with a measure $m$ on $X$. We will refer to $(X,m)$ as a  topological measure space.
 We first define the notion of rectifiability of a topological measure space. Since these spaces are not equipped with metrics, the rectifiable structure is provided via an atlas of charts with bi-Lipschitz transition maps.

\begin{definition}[Rectifiable atlas] Let $(X,m)$ be a topological measure space, and consider a collection of charts $\{(U_a, \phi_a)\}_{a \in \mathcal{I}}$, where $U_a \subseteq \R^n$,  $\phi_a :U_a \to X$  is one-to-one and continuous with continuous inverse on its image, and $ \mathcal{I}$ is  a countable index set.  For each $a,b\in  \mathcal{I}$ let us denote $U_{a,b}\equiv U_a\cap \phi_a^{-1}(\phi_b(U_b)) \subseteq \R^n$.  We say that $\{(U_a, \phi_a)\}_{a \in  \mathcal{I}}$ is a rectifiable atlas for  $(X,m)$  if:
	\begin{enumerate}
		\item For each $a,b \in  \mathcal{I}$ such that $U_{a,b}$	is nonempty,  every point in $U_{a,b}$  has Lebesgue density one.
		\item For each $a, b \in  \mathcal{I}$ such that $U_{a,b}$ is nonempty, the transition map 
		\[
		\phi_{ba}\equiv \phi^{-1}_b\circ \phi_a:U_{a,b}\to U_{b,a}
		\]
		 is  bi-Lipschitz.
		\item We have 
		\[
		m\left(X\setminus \bigcup_{a\in  \mathcal{I}} \phi_a(U_a)\right)=0 .
		\] 
		\item For each $U_a$ we have that $(\phi_{a}^{-1})_* m$ is absolutely continuous to the Lebesgue measure. 
	\end{enumerate}
\end{definition}

\smallskip

Given a topological measure space $(X,m)$  equipped with a rectifiable atlas $\{(U_a, \phi_a)\}$, we may define a Riemannian structure on $X$ by defining a (possibly degenerate) Riemannian metric in the charts $U_a$. Naturally, we must ask that this metric is suitably compatible with the rectifiable atlas and the measure. We call the resulting space a rectifiable Riemannian  space.
\begin{definition}[Rectifiable Riemannian space] \label{def: rect RM}
 Let $(X,m)$ be a  topological measure space. We say that  $(X, m)$ has a rectifiable Riemannian structure if there is a rectifiable atlas $\{(U_a, \phi_a)\}_{a\in  \mathcal{I}} $ on  $(X,m)$  and collection of matrix-valued functions $g_a : U_a \to \R^{n\times n}$  satisfying
\begin{enumerate}
\item For each $x \in U_a,$ $g_a(x)$ a positive definite symmetric matrix such that 
\[\sup_{x \in U_a} \| g_a\| + \| g_a^{-1}\| \leq C_a.
\]
and $g_a$ is continuous on $U_a$.
	\item For each nonempty $U_{ab}$, we have 
\[
g_b = \phi_{ba}^*g_a.
\]
	\item The measure $\phi_a^* m$  on $U_a$ is given by 
	\[
	 \phi_a^* m = \sqrt{\det g_a} \, dx.
	 \] 
\end{enumerate}
We say that $\{g_a\}_{a\in  \mathcal{I}}$ is the coordinate expression of a rectifiable Riemannian metric $g$ on $X$ and call $(X,g)$ a rectifiable Riemannian space.
\end{definition}

\begin{remark}{\rm
For the condition (1), the continuity assumption is only made for convenience. Indeed since $U_a$ is not necessarily an open set in $\mathbb{R}^n$, by the Lebesgue differentiation theorem one can always remove an arbitrarily small set so that the continuity holds under the $L^\infty$ condition.
}
\end{remark}

\medskip

\subsubsection{Examples of rectifiable Riemannian spaces}

One can imagine a variety of ways in which a rectifiable Riemannian space can degenerate.  We will first work our way through some basic examples which explore this.  This will give some first intuition on what kind of structure is needed to avoid this.  Future sections will explore examples which might arise as limits, which will tell us when these degeneracies can and cannot be avoided.\\

\begin{example}\label{ex: smooth m}
	{\rm
	Any smooth Riemannian manifold $(M,g)$ is a rectifiable Riemannian space.
	}
\end{example}
\vspace{.1cm}

With regard to Example~\ref{ex: smooth m}, observe that even for a smooth Riemannian manifold $(M,g)$ a given rectifiable atlas may only cover $M$ up to a set of measure zero:
\begin{example}\label{ex: rn with bad charts}
	{\rm
	Let $X= \R^n $ with $g_{euc}$ the Euclidean metric, and let $m = dx$ be the Lebesgue measure. Consider the rectifiable atlas $\{ (U_1, \phi_1), (U_2, \phi_2)\}$ where $U_1 = \{ (x_1,\dots, x_n ) : x_1 > 0\}$ and $U_2 = \{ (x_1,\dots, x_n) : x_1 < 0\}$ are complementary open half spaces and $\phi_i$ is the identity chart restricted to $U_i$ for $i=1,2$. Then $(\R^n, g_{euc})$ is a rectifiable Riemannian space with respect to this rectifiable atlas.
	}
\end{example}
\vspace{.1cm}
\begin{example}\label{ex: strat}
	{\rm
	Any stratified Riemannian manifold $X$ is a rectifiable Riemannian space. 
	}
\end{example}
\vspace{.1cm}
\begin{example}\label{ex: strat with lines}
{\rm
As a concrete case of Example~\ref{ex: strat}, let $X \subset \R^2$ be a countable union of lines $\{\ell_i\}_{i\in \mathbb{N}}$ passing through the origin and let $m$ be defined by $m|_{\ell_i}= \mathcal{H}^1|_{\ell_i}.$ Define $g|_{\ell_i} = g_{\R^2}|_{\ell_i}$ 
and let $\{( \R\setminus \{0\} , \phi_i)\}_{i \in \mathbb{N}}$ be the rectifiable atlas with $\phi_i: \R\setminus \{ 0\}  \to \ell_i \setminus\{0\}$ defined via the obvious isometric embedding.  
Then $(X,g)$ is a one dimensional rectifiable Riemannian space.
}
\end{example} 
\vspace{.1cm}
\begin{example}
	{\rm
	Let $(X^n,d,m)$ be a metric measure space which is also an $RCD(N,K)$ space, that is, a metric measure space with lower bounds on the Ricci curvature in the generalized sense.  Then it follows from \cite{MondinoNaber} that $X$ is a rectifiable space. 
	}
\end{example}
\vspace{.1cm}

Due to the flexibility, we can also allow the metric tensor to be mildly singular.  Let us consider some examples of this.

\begin{example}[Degenerate metric on $\R^n$]\label{example: prelim top change}
	{\rm 
	Let $X = \R^n $ and consider the metric defined by 
	$g=\sum_{i=1}^n f_i(x)^2 (dx^i)^2$, where each $f_i$ is a smooth non-negative function on $\mathbb{R}^n$ such that the set $\Sigma = \cup_{i=1}^n\{ x: f_i(x)= 0\}$ has  Lebesgue measure zero.  Further let $m = \sqrt{\det g} \,dx$ be the induced measure.  Consider the the rectifiable atlas on the topological measure space $(\R^n, m)$ given by $\{ (U_a , \phi_a)\}_{a \in \mathbb{N}}$ where 
$U_a = \cap_{i=1}^n\{x \in \R^n: a^{-1} \leq f_i \leq a\}$ and $\phi_a $ is the identity chart on $\R^n$ restricted to $U_a$. Then, with respect to this rectifiable atlas, $(\R^n,g)$ is a rectifiable Riemannian space.
}
\end{example}

An important feature of Example~\ref{example: prelim top change} is that, while the geodesic distance gives rise to a metric space structure $(X,d)$, the metric space may not even be topologically equivalent to $\R^n$, as seen in the following example.

\begin{example}\label{example: power growth}
	{\rm
	 Consider  the rectifiable Riemannian space $(\R^2, g)$ where $g=dx^2+|x|^2 dy^2$, which is a special case of Example~\ref{example: prelim top change} above. Let $d_g$ be the distance function with respect to $g$, i.e. $d_g(x,y) = \inf_{\g} \int_0^1 |\dot{\g}(t)|dt$, where the infimum is taken among all curves $\g$ 
	with $\g(0)=x, \g(1)=y$. Then we see that
		$d_g(p_1,p_2) = 0$  for all $p_1, p_2 \in \ell$
	where $\ell=\{(x,y):x=0\}$.	In particular, the metric space $(\mathbb{R}^2,d_g)$ collapses $\ell$ to a point and is not topologically equivalent to $\R^2$. We will examine the $d_p$ distance for this example in Section~\ref{subsec: dp}.
	}
\end{example}

In Section~\ref{examples-entropy}, we will construct rectifiable Riemannian metrics that are qualitatively similar to Examples~\ref{example: prelim top change} and \ref{example: power growth} that arise as limits of smooth Riemannian manifolds with uniform lower bounds on scalar curvature and entropy.

\medskip
\subsubsection{$W^{1,p}$ spaces on rectifiable Riemannian spaces}\label{sec: w1p structure}

We would like to use the rectifiable structure of a space in order to do analysis.  In order to do this, we need to make sense of $W^{1,p}$ functions in our context, which means being able to take gradients of functions and look at their norms.  Ideally, we would want to use the rectifiable charts in order to do this in coordinates.  Realistically, one has to be quite careful about this.  A function might be perfectly differentiable in every coordinate chart, but not really be a $W^{1,p}$ function as its gradient may have a distributional component, as we see in the following example.\\

\begin{example}\label{ex: rn bad gradient}{\rm
	Consider $(\R^n, g_{euc})$ with the rectifiable atlas $\{(U_1, \phi_1), (U_2, \phi_2)\}$ comprising two open half spaces as in Example~\ref{ex: rn with bad charts}. The function $f: \R^n \to \R$ defined by $f(x) = 0 $ if $x_1<0$ and $f(x) = 1 $ if $x_1\geq 0$ clearly does not have gradient in $L^p(\R^n)$, since its distributional gradient is a singular measure supported on $\{x_1=0\}$. However, letting $f_a = \phi_a^*f$ for $a=1,2$, we have $g^{ij} \pa_i f_a\pa_j f_a \equiv 0$ for all $x \in U_a$ for $a =1,2$.
		}
\end{example}

In order to deal with this issue, we will follow a classical approach from metric measure spaces (see for instance \cite[Sections 5-7]{Hiwash03}) to build the Sobolev space theory by considering the behavior of functions along curves.  The key is that these ideas adapt themselves very well to this context, as in the end even when no apriori distance function is available, the notion of an absolutely continuous curve and the behavior of a function along it is available and can be studied.\\

Let us begin by discussing  the notion of an absolutely continuous curve on a rectifiable Riemannian space. In the setting of a smooth manifold or a metric measure space,  the theory of Sobolev spaces can be built up by considering the behavior of functions along rectifiable curves. In these settings, a rectifiable curve is defined as one with finite length, where length is defined via approximation by polygonal curves. This definition is independent of parametrization, and every rectifiable curve in a smooth Riemannian manifold or a metric space admits an absolutely continuous (in fact, Lipschitz) parametrization, namely the arc length parametrization.  In practice, it is this absolutely continuous parametrization that is used in the Sobolev space theory.

Moving to the context of rectifiable Riemannian spaces, there are two major factors that must be taken into account when determining the appropriate class of curves along which to study the behavior of functions. First,  curves must be appropriately compatible with the rectifiable atlas on the space in order to avoid the difficulty illustrated in Example~\ref{ex: rn bad gradient}. Second,  the absence of a distance function prohibits us from speaking about the length of a polygonal curve, and thus of considering rectifiable curves in the sense described above. Instead, we must restrict our attention to absolutely continuous parametrizations of curves. To be more specific about these two considerations, the definition of an  absolutely continuous curve in a rectifiable Riemannian space is given in  Definition~\ref{def: rect curve} below.

Let $(X,g)$ be a rectifiable Riemannian space with rectifiable atlas $\{(U_a,\phi_a)\}_{a\in  \mathcal{I}}$, and denote the singular part of $X$ by $X^s=X\setminus \cup_a U_a$.  Note this may or may not correspond to topological singularities of the space.  

\begin{definition}[Absolutely Continuous Curves]\label{def: rect curve}
Let  $\g:[\alpha,\beta]\rightarrow X$ be a  continuous curve and define $I_{a}=\gamma^*(\phi_{a}(U_{a}))$ for each $a \in \mathcal{I}.$ We say that $\gamma$ is absolutely continuous if the following properties hold.
\begin{enumerate}
\item[(a)]
 {$\gamma^*(X^s)\subset [\alpha, \beta]$ is a countable set};
\item[(b)]  
For every $\e>0$, there exists $\delta>0$ such that  if $\{(s_i,t_i)\}_{i=1}^\infty$ is a collection of disjoint intervals in $[\alpha,\beta]$ such that for each $i$, we have  $s_i, t_i \in I_{a_i}$ for some $a_i\in \mathcal{I}$ and  $\sum_{i=1}^\infty |s_i-t_i|<\delta$, then 
\begin{equation}\label{eqn: sum to eps}
	\sum_{i=1}^\infty  |\gamma_{a_i}(s_i)-\gamma_{a_i}(t_i)|_{g_a(\gamma(s_i))}<\e.
\end{equation}
\end{enumerate}
\end{definition}
Some further discussion is in order about Definition~\ref{def: rect curve}. 
 As we saw in Example~\ref{ex: rn bad gradient}, part (a) of the definition is necessary to guarantee that the behavior of an absolutely continuous curve $\gamma $ can be entirely reflected in the charts of its rectifiable atlas, since the rectifiable charts may only cover $(X,g)$ up to a set of measure zero. Together with the assumption that $\g$ is continuous, part (a) ensures that there is no contribution to the singular part of the distributional derivative of $\g$ on the set $\gamma^*(X^s)$, and in particular eliminates the issue illustrated in Example~\ref{ex: rn bad gradient}:
  \begin{example}\label{eqn: rn bad curve}
 	{\rm 
 	Observe that in Example~\ref{ex: rn bad gradient}, the curve $\gamma(t) = (t, 0, \dots , 0)$ is not an absolutely continuous curve in the sense of Definition~\ref{def: rect curve} because it violates condition (a).
 	}
 \end{example}

 Part (b) of Definition~\ref{def: rect curve} is a replacement of the classical notion of a curve with finite length, as one typically takes the supremum over lengths of polygonal approximations to the curve.  As we noted above, in the context of rectifiable Riemannian spaces this notion is unsuitable as one does not have a notion of a distance function with which to measure the length. Instead, we will see that condition (b) guarantees that the curve is absolutely continuous in each chart in a suitably uniform sense. 
 
 Note that Definition~\ref{def: rect curve} is parametrization dependent, as it requires an absolutely continuous parametrization. This is not restrictive, as on a smooth Riemannian manifold with a classical atlas of charts, every rectifiable curve in the classical sense admits a reparametrization, namely its arc length parametrization, that is an absolutely continuous curve in the sense of Definition~\ref{def: rect curve}.  
%

  The following lemma provides some basic consequences of the Definition~\ref{def: rect curve} that further clarify this notion of absolutely continuous curve and how it fits into the classical notion on smooth spaces.
 \begin{lemma}\label{lem: basic prop rect curves}
Let $\gamma: [\alpha, \beta]\to X$ be an absolutely continuous curve in the sense of Definition~\ref{def: rect curve} above. Then the following properties hold.
\begin{enumerate}
	\item For each $a \in  \mathcal{I}$, the function $\gamma_a=\phi_a^{-1}\circ \gamma: I_a \to U_a$ is differentiable for a.e. $s \in I_a$. Here we again let $I_{a}=\gamma^*(\phi_{a}(U_{a}))\subset [\alpha, \beta]$.
	\item  For all $\e>0$, there exists $\delta>0$ such that  if $S\subset [\alpha ,\beta]$ with $|S|<\delta$, then $\int_S |\dot\gamma|_g <\e$.
	\item If $(X,g)$ is a smooth Riemannian manifold, then the length of $\gamma$ is given by $L(\gamma) = \int_\alpha^\beta |\dot{\gamma}|_g \,dt.$
\end{enumerate}
 \end{lemma}
 \begin{remark}
 	{\rm 
 	 In Lemma~\ref{lem: basic prop rect curves} (2) and (3) and in the sequel, we  let $|\dot\gamma|_g \equiv \sqrt{g(\dot\gamma,\dot\gamma)}$, which is well-defined for  a.e. $t \in [\alpha,\beta]$ by Lemma~\ref{lem: basic prop rect curves}(1) and via the rectifiable atlas  $\{(U_a,\phi_a)\}_{a\in  \mathcal{I}}$. 	
 }
 \end{remark}
  \begin{remark}
 	{\rm 
Having in mind Lemma~\ref{lem: basic prop rect curves}(3),  we define the length of an absolutely continuous curve $\gamma: [\alpha,\beta]\to X$ in a rectifiable Riemannian space by $L(\gamma) = \int_\alpha^\beta |\dot \gamma|_g $. One can check that this notion is independent of Lipschitz reparametrizations.
 }
 \end{remark}

 \begin{proof}[Proof of Lemma~\ref{lem: basic prop rect curves}]
 We first prove (1). 
 	Fix $a \in  \mathcal{I}$. 
 	Because $\phi_a^{-1}$ is continuous, we see that $\gamma_a : I_a \to U_a$ satisfies the absolute continuity property \eqref{eqn: sum to eps} for endpoints $s_i, t_i \in I_a$, thus $\gamma_a$ extends to a continuous function $\bar{\gamma}_a : \bar{I}_a \to \bar{U}_a$. Moreover, the complement $[\alpha,\beta]\setminus \bar{I}_a$ of  $\bar{I}_a$ is relatively open in $[\alpha,\beta]$ and therefore comprises a countable union of disjoint relatively open intervals $ (\alpha_j, \beta_j)$ in $[\alpha,\beta]$. We may therefore extend $\bar{\gamma}_a$ to a continuous curve $ \tilde{\gamma}_a : [\alpha, \beta] \to \R^n$ by letting $\tilde{\gamma}_a$ interpolate linearly between $\bar{\gamma}_a(\alpha_j )$ and $\bar{\gamma}_a(\beta_j)$ for each $j \in \mathbb{N}$. 

To prove (1), we will show that the curve $\tilde{\gamma}_a : [\alpha, \beta] \to \R^n$ is a uniformly continuous curve on $\R^n$. More specifically, we claim that for any $\e>0$, there exists  $\delta_0>0$ such that for any disjoint collection of intervals $\{ [s_i, t_i]\}_{i=1}^\infty$ such that $\sum_{i=1}^\infty|s_i - t_i|<\delta_0$,
we have 
  \begin{equation}\label{eqn: ac finite}
  \sum_{i=1}^\infty |\tilde\gamma(s_i) - \tilde\gamma(t_i) |_{euc} \leq \e.
  \end{equation}
  It will follow immediately from this absolute continuity of $\tilde{\gamma}$ that $\tilde{\gamma}$ is differentiable for a.e. $t \in [\alpha, \beta]$ and so in particular that $\gamma$ is differentiable for a.e. $t \in I_a$,  proving (1).
  
Fix $\e>0$, let $\delta_0>0$ be a fixed number to be determined within the proof, and consider a disjoint collection of intervals $\{ [s_i, t_i]\}_{i=1}^\infty$ in $[\alpha,\beta]$ such that $\sum_{i=1}^\infty|s_i - t_i|<\delta_0$. Up to refining the collection of intervals (which can only increase \eqref{eqn: ac finite}), we may assume for each $i\in \mathbb{N}$ that  $s_i,t_i$ either both lie in  $\bar{I_a}$ or both lie in $\overline{[\alpha,\beta] \setminus \bar{I}_a} = \bigcup_{j=1}^{\infty} [\alpha_j, \beta_j]$. So, subdividing the index set $\mathbb{N}$ by letting
\begin{equation}
	\begin{split}
\mathcal{J}_0 &= \left\{ i \in\mathbb{N} : s_i, t_i \in \bar I_a\right\}, \\
 \mathcal{J}_1 &= \bigg\{ i \in \mathbb{N}  : s_i, t_i \in\bigcup_{j=1}^{\infty} [\alpha_j, \beta_j]\bigg\},
	\end{split}
\end{equation} 
we will establish \eqref{eqn: ac finite} by showing that 
\begin{align}
\label{26a}	\sum_{i \in \mathcal{J}_0} |\tilde\gamma(s_i) - \tilde\gamma(t_i) |_{euc} & \leq \e/2,\\
\label{26b} \sum_{i \in \mathcal{J}_1} |\tilde\gamma(s_i) - \tilde\gamma(t_i) |_{euc} & \leq \e/2.
\end{align}
Recall that there exists $C_a>0$ such that 
\begin{equation}
	\label{eqn: comparable}
	C_a^{-1} g_a \leq  g_{euc} \leq C_a g_a
\end{equation}
in $U_a$. So, \eqref{26a} follows directly from this fact and \eqref{eqn: sum to eps}, provided we take $\delta_0 \leq \delta_1$, where $\delta_1>0$ is a number small enough so that \eqref{eqn: sum to eps} holds with $\e/4C_a$ in place of $\e$. In order to establish \eqref{26b}, we further subdivide the index set $\mathcal{J}_1$ in the following way. 
Since $\sum_{j=1}^{\infty}|\beta_j- \alpha_j| \leq |\beta-\alpha|,$ there exists $J_0$ such that $\sum_{j= J_0}^\infty |\beta_j - \alpha_j| \leq \delta_1$. Up to further refinement of our collection of intervals, we may assume that for each $i \in \mathcal{J}_1$, both endpoints $s_i, t_i$ lie in $\bigcup_{j=1}^{J_0-1} [\alpha_j, \beta_j]$ or both endpoints  lie in $\bigcup_{j=J_0}^{\infty} [\alpha_j, \beta_j]$. So, we let 
\begin{equation}
\begin{split}
	\mathcal{J}_{2} & =\bigg\{ i \in \mathbb{N}  : s_i, t_i \in\bigcup_{j=1}^{J_0-1} [\alpha_j, \beta_j]\bigg\} \subset \mathcal{J}_1, \\
	 \mathcal{J}_{3} & =\bigg\{ i \in \mathbb{N}  : s_i, t_i \in\bigcup_{j=J_0}^{\infty} [\alpha_j, \beta_j]\bigg\}\subset \mathcal{J}_1,
\end{split}
\end{equation}
so that $\mathcal{J}_{2} \cup \mathcal{J}_{3} =\mathcal{J}_1$. We establish \eqref{26b} by bounding the sums over $ \mathcal{J}_{2}$ and $\mathcal{J}_{3}$ separately, starting with $\mathcal{J}_{3}$. By the triangle inequality and the piece-wise linear way that $\tilde{\gamma}$ was defined on $[\alpha,\beta]\setminus \bar{I}_a,$ we see that 
 \begin{equation}\label{eqn: ac small open a}
\sum_{i \in \mathcal{J}_{3}} |\tilde\gamma(s_i) - \tilde\gamma(t_i) |_{euc}  \leq  \sum_{j= J_0}^\infty|\tilde\gamma(\alpha_j)-\tilde\gamma(\beta_j)|_{euc} .
 \end{equation}
 Then, since $\alpha_j, \beta_j \in \bar{I}_a$, we may apply \eqref{eqn: comparable} and  \eqref{eqn: sum to eps} to find that 
 \begin{equation}\label{eqn: ac small open b}
 	\sum_{j= J_0}^\infty|\tilde\gamma(\alpha_j)-\tilde\gamma(\beta_j)|_{euc}\leq \e/4
 \end{equation}
 provided $\delta_0 \leq \delta_1$, with $\delta_1$ as above. Together \eqref{eqn: ac small open a} and \eqref{eqn: ac small open b} show that 
 \begin{equation}\label{eqn: j1b sum}
 	\sum_{i \in \mathcal{J}_{3}} |\tilde\gamma(s_i) - \tilde\gamma(t_i) |_{euc} \leq \e/4.
 \end{equation}
Now, to bound the analogous summation over $\mathcal{J}_{2},$ notice that the linear segment of $\tilde{\gamma}$ defined on $[\alpha_j,\beta_j]$ is absolutely continuous for each $j=1,\dots, J_0$. Since there are finitely many such intervals, we may find $\delta_2>0$ such that if $\delta_0 \leq \delta_2,$ then 
 \begin{equation}\label{eqn: j1a sum}
 	\sum_{i \in \mathcal{J}_{2}} |\tilde\gamma(s_i) - \tilde\gamma(t_i) |_{euc} \leq \e/4.
 \end{equation}
Finally, choose $\delta_0 \leq \min\{\delta_1, \delta_2\}$. Then together \eqref{eqn: j1a sum} and \eqref{eqn: j1b sum} prove \eqref{26b}. This shows that $\tilde{\gamma}$ is an absolutely continuous curve on $\R^n$ and thereby establishes (1).

Before moving to the proof of (2), observe that 
as a consequence of (1), we may define
$|\dot \gamma_a|_{g_a} = \sqrt{g_a(\dot\gamma_a, \dot \gamma_a)}$ for a.e. $t \in I_a, $ and we have
\begin{equation}\label{eqn: difference  quotient a}
	\lim_{s \to t} \frac{|\tilde{\gamma}_a(t) - \tilde{\gamma}_a(s)|_{g_a(t)}}{|s-t|} = |\dot \gamma_a(t)|_{g_a(t)} 
\end{equation} 
for each such $t$. 
It is easily checked that this definition is independent of $a$ and thus we may define $|\dot \gamma|_g$ for a.e. $t \in [\alpha, \beta]$.

 Now we establish (2). The main idea will be to use \eqref{eqn: difference  quotient a} to relate the integral of  $|\dot  \gamma|_g$  to the absolute continuity assumption \eqref{eqn: sum to eps}. Fix  $\e>0$ and let  $\delta>0$ be a fixed number to be specified within the proof. Fix  a measurable set $S\subset [\alpha, \beta]$ with $|S|<\delta$, and  let $\{S_a\}_{a\in \mathcal{I}}$ be a collection of pairwise disjoint subsets of $S$ with $S_a \subset I_a$ for each $a \in \mathcal{I}$ such that $|S\setminus \cup_{a \in \mathcal{I}} S_a | = 0$,  $\tilde{\gamma}_a$ is differentiable for every $t \in S_a$, and every point of $S_a$ is a density point of $I_a$. 
 
 In order to make use of \eqref{eqn: difference  quotient a}, we must show that the limit in \eqref{eqn: difference  quotient a} is uniform on a large subset of each $S_a$. More specifically,  we claim that for any $\eta>0$, there exist $r_\eta>0$ and a measurable set $S_a^\eta \subset S_a$ with $|S_a^\eta| \geq (1-\eta) |S_a|$ such that if $ t \in S_a^\eta$, then 
 \begin{equation}\label{eqn: difference quotient}
 	|\dot\gamma_a(t)|_{g_a(t)} - \eta \leq \frac{|\tilde\gamma_a(t) - \tilde\gamma_a(s)|_{g_a(t)}}{|t-s|} \leq 	|\dot\gamma_a(t)|_{g_a(t) } + \eta
 \end{equation}
for all $s \in S_a$ with $|t-s|<r_\eta.$ 
 Indeed, let 
 \begin{equation}
 	R(s,t) = \left|  \frac{|\tilde\gamma_a(t) - \tilde\gamma_a(s)|_{g_a(t)}}{|t-s|} -|\dot\gamma_a(t)|_{g_a(t)}\right|.
 \end{equation}
 For each fixed $t \in S_a$, we have $\lim_{s\to t} R(s,t) = 0$. 
 So, if  we define the sequence of functions $\rho_k (t) = \sup\{ |R(s,t)| : s \in S_a ,\, |s-t| < 1/k\}$, we see that $\rho_k(t) \to 0$ pointwise as $k \to \infty$. Applying Egorov's theorem, for any $\eta>0$, there exists a measurable set $S_a^\eta$ such that $|S_a \setminus S_a^\eta| \leq \eta |S_a|$ and $\rho_k(t) \to 0$ uniformly on $S_a^\eta.$ In particular, letting $r_\eta = 1/k$ for $k$ chosen sufficiently large, we see that $|R(s,t)| \leq \eta$ for all $t \in S_a^\eta$. This establishes the claim.

 Next, for each $a \in  \mathcal{I}$, we aim to estimate  the integral of $|\dot \gamma|$ over $S_a^\eta.$ Consider a collection of disjoint intervals $\{I_{a,i}= [s_i,t_i]\}_{i=1}^\infty$ covering $S_a^\eta$ with endpoints in $S_a$ such that $|I_{a,i}|<r_\eta$ for each $i$ and such that $ \sum_i |I_{a,i}| \leq (1+\eta) |S_a^\eta|$.  
For each $i$, let $\hat{t}_i \in S_{a}^\eta \cap I_{a,i}$ be a point  such that 
\begin{equation}
	(1+\eta ) |\dot \gamma_a(\hat{t}_i)|_{g_a(\hat{t}_i)} \geq \sup\{   |\dot \gamma_a(t)|_{g_a(t)}  \,: \, t \in S_{a}^\eta \cap I_{a,i}\}.
\end{equation}
At least one of the endpoints of $I_{a,i}$ has distance at least $|I_{a,i}|/2$ from $\hat{t}_i$; without loss of generality suppose it is $s_i$. So, thanks to  \eqref{eqn: difference quotient}, we find that for each $i$,
\begin{equation}
\begin{split}
	\int_{S_a^\eta \cap I_{a,i}}  |\dot \gamma_a(t)|_{g_a(t)}\,dt 
	 & \leq (1+\eta)| I_{a,i}|\,   |\dot \gamma_a(\hat{t}_i)|_{g_a(\hat{t}_i)}\\
	& \leq (1+\eta)|I_{a,i} | \,  \frac{|\tilde\gamma_a(\hat{t}_i) - \tilde\gamma_a(s_i)|_{g_a(\hat{t}_i)}}{|s_i - \hat{t}_i|} + (1+\eta)\eta |I_{a,i}|\\
	& \leq 2(1+\eta) |\tilde\gamma_a(\hat{t}_i) -\tilde \gamma_a(s_i)|_{g_a(\hat{t}_i)}+(1+\eta)\eta |I_{a,i}|\\
	& =2(1+\eta) |\gamma_a(\hat{t}_i) - \gamma_a(s_i)|_{g_a(\hat{t}_i)}+(1+\eta)\eta |I_{a,i}|.
\end{split}	
\end{equation}
The  final equality follows from the definition of  $\tilde\gamma$ and the fact that $\hat{t}_i, s_i \in S_a \subset I_a$.
Consequently, summing up over $a$ and $i$, we find that 
 \begin{equation}\label{27}
 \begin{split}
 	\sum_{a\in \mathcal{I}} \int_{S_{a}^\eta} |\dot \gamma|
 	& \leq   \sum_{a\in \mathcal{I}} \sum_{i=1}^\infty \int_{S_{a}^\eta \cap I_{a,i}}  |\dot \gamma| \\
 &	\leq   \sum_{a\in \mathcal{I}} \sum_{i=1}^\infty  \big(2(1+\eta) |\gamma_a(\hat{t}_i) - \gamma_a(s_i)|+(1+\eta)\eta |I_{a,i}|\big)\\
 	& \leq 2(1+\eta) \e + (1+\eta)^2\eta \delta,
 	 \end{split}	
 \end{equation}
 where in the final inequality we have applied \eqref{eqn: sum to eps},  again using that $\hat{t}_i, s_i \in S_a \subset I_a$.
  
Finally, we send $\eta \to 0$. The right-hand side of \eqref{27} tends to $2\e$, while, making use of the dominated convergence theorem, we see that the left-hand side converges to $\int_S |\dot\gamma|$. We therefore see that (2) holds. We omit the proof of (3) since it is standard.
 \end{proof} 


Following the classical approach in metric measure spaces, we next want to use our absolutely continuous curves to define the notion of a $p$-weak upper gradient of a function. We will see that most of the Sobolev theory is built up in an identical fashion to the metric measure space setting.

 To begin, we need a notion of a collection of curves that have $p$-measure zero,  an  idea first introduced by Ahlfors and Beurling in \cite{AB50}  and further developed by Fuglede in \cite{F57} in the Euclidean and Riemannian settings. To this end, let $\mathfrak{M}$ denote the collection of all absolutely continuous curves on $(X,g)$.
For $1\leq p <\infty,$ we say that a family of curves $\Gamma\subset \mathfrak{M}$ has $\text{Mod}_p(\Gamma)=0$ if there exists a nonnegative Borel measurable function $f \in L^p(X)$ such that $\int_\g f = +\infty$ for every $\g \in \Gamma$.
Here and in the sequel, we use the notation $\int_\g f$ to mean
\begin{equation}
\int_\g f  := \int_\alpha^\beta f(\gamma(t))\,|\dot \gamma|_g\,dt.
\end{equation} A property is said to hold for $p$-a.e absolutely continuous curve if it holds for every curve in $\mathfrak{M}\setminus\Gamma$ where $\text{Mod}_p(\Gamma)=0$. 
For the corresponding definition of families of curves with $\text{Mod}_p(\Gamma)=0$ in the metric measure space context, see \cite[Definition 5.1]{Hiwash03} and the equivalent formulation of the definition given in  \cite[Theorem 5.5]{Hiwash03}.

It follows directly from the definition that for any nonnegative  Borel measurable function $f \in L^p(X)$, then $\int_\gamma f <\infty$ for $p$-a.e. absolutely continuous curve. In a similar vein, the following lemma shows that convergent sequences in $L^p(X)$ converge along $p$-a.e. curve:
\begin{lemma}[c.f. Theorem 5.7 of \cite{Hiwash03}]\label{lem: 5.7 Hiwash} Fix $1\leq p < \infty$. 
	Let $u_k : X \to \R\cup \{\pm \infty\} $ be a sequence of Borel measurable functions converging in $L^p(X)$ to a Borel measurable function $u : X \to \R\cup \{\pm \infty\}.$ Up to a subsequence, 
	\begin{equation}
		\int_a^b |u_k -u| |\dot{\g}|_g \,dt \to 0
	\end{equation}
 as $k \to \infty$ for all $\g \in \mathfrak{M} \setminus \Gamma,$ where $\text{Mod}_{p}(\Gamma) = 0$.
\end{lemma}
Lemma~\ref{lem: 5.7 Hiwash} was shown in the context of metric measure spaces in \cite[Theorem 5.7]{Hiwash03}. In our setting, the proof carries over without modification.
\\

Having in hand the notions of absolutely continuous curves and families $\Gamma$ of absolutely continuous curves with $\text{Mod}_p(\Gamma)=0$, we are now in a position to define upper gradients and $p$-weak upper gradients of functions $u: X \to \R$. The notion of weak upper gradient was first introduced by Heinonen and Koskela in \cite{HK98}, and the definition we give here is analogous to \cite[Definition 6.1]{Hiwash03}.
{
\begin{definition}[Upper gradients and $p$-weak upper gradients] Let $u:X\rightarrow \mathbb{R}$  and $G:X\rightarrow [0,\infty]$ be Borel measurable functions.
We say that $G$ is an upper gradient for $u$ if 
\begin{equation}\label{eqn: upper gradient condition}
|u(\gamma(a) ) - u(\gamma(b))| \leq \int_\g G
\end{equation}
for every absolutely continuous curve $\g:[a,b] \to X$. For $1 \leq p <\infty$, we say that $G$ is a $p$-weak upper gradient  for $u$ 
if the upper gradient condition \eqref{eqn: upper gradient condition} holds for $p$-a.e. absolutely continuous curve $\gamma:[a,b]\to X $.
\end{definition}
}
The following example shows that this is a natural notion of gradient.
\begin{example}
	{\rm
	Consider Euclidean space as a rectifiable Riemannian space with the rectifiable atlas comprising only the identity chart. For a smooth function $u : \R^n \to \R$, the classical gradient $|\na u|$ is an upper gradient for $u$. 
	}
\end{example}
Furthermore, we note that the potential issues highlighted by Example~\ref{ex: rn bad gradient} are eliminated with respect to this definition.
\begin{example}
	{\rm
	Consider the rectifiable Riemannian space and the function $f$ defined in Example~\ref{ex: rn bad gradient}. We see clearly that $G=0$ is not a $p$-weak upper gradient for $f$, since the upper gradient condition ~\ref{eqn: upper gradient condition} fails for any curve that crosses the hyperplane $\{x_1=0\}$. In fact, considering the family $\Gamma$ of absolutely continuous curves of the form $\gamma(t) = (0,x') + t e_1$ for $t \in (-\e,\e)$, we easily see that 
	$f$ has no $p$-weak upper gradient in $L^p(X)$. 
	}
\end{example}
\medskip

Now, let $\tilde{W}^{1,p}(X)$ be the collection of all  Borel measurable functions $u:X\to \R$ such that $u$ is $L^p$ integrable and   $u$ has a $p$-weak upper gradient  $G \in L^p(X)$. The Sobolev space $W^{1,p}(X)$ on a rectifiable Riemannian space is defined in the following way,  following the definition first introduced by Shanmugalingam in \cite{Shanmu}  in  the context of metric measure spaces and presented in Definition 7.1 of \cite{Hiwash03}. We note that a closely related  definition of Sobolev spaces on a metric measure space was given by Cheeger in \cite{CheegerMMS99}.

\begin{definition}\label{def: w1p} For any $u \in \tilde{W}^{1,p}(X)$,
we define
\begin{equation}
\|u \|_{W^{1,p}(X)}=\|u\|_{L^p(X)} +  \inf_G  \|G\|_{L^p(X)},
\end{equation}
where the infimum is taken over all $p$-weak upper gradients $G$ of $u$.
 We define the space $W^{1,p}(X) = 
\tilde{W}^{1,p}(X) / \sim,$ where  $u\sim v$ for $u,v \in \tilde{W}^{1,p}(X)$ if $\| u-v\|_{W^{1,p}(X)}=0$.
\end{definition}
\begin{remark}
	{\rm 
One subtlety of Definition~\ref{def: w1p}, which is also present in the analogous metric measure space setting, it that it is possible to modify a function $u \in \tilde{W}^{1,p}(X)$ on a set of $m$ measure zero to obtain a function $\tilde{u}$ that is not in $\tilde{W}^{1,p}(X)$. For instance, on Euclidean space, consider the functions $u \equiv 0$ in $W^{1,p}(\R^n)$ and $\tilde{u}= \chi_E$, where $E$ is the set of all rational points in $\R^n$. Then $\tilde{u}$ has no $p$-weak upper gradient in $L^p$.  This subtlety explains why some of the statements in Proposition~\ref{prop: basic prop w1p}  below require the choice of a certain representative for an $L^p$ function.  However, if $u, \tilde{u} \in \tilde{W}^{1,p}(X)$ and $u=\tilde{u}$ $m$-a.e., then $u$ and $\tilde{u}$ define the same element in $W^{1,p}(X)$. 
	}
\end{remark}

From this point, we can establish a number of basic properties of the space $W^{1,p}(X)$ showing that this space possesses many of the important features of Sobolev spaces in smooth settings, which we collect in the following proposition. The analogous properties are established in the metric measure space setting in \cite[Section 7]{Hiwash03}. In fact, the proofs there can be carried over almost verbatim,   with only the modification being the distinction between the use of absolutely continuous curves in our setting as opposed to rectifiable curves in the setting of metric measure spaces. For this reason, we omit the proofs and instead point the reader to the corresponding statements in \cite[Section 7]{Hiwash03}.  Properties  (1)-(3), (5)  were originally proven by Shanmugalingam \cite{Shanmu}, and property (4) was established by Cheeger in \cite{CheegerMMS99}  for $p>1$   and Haj\l  asz \cite{Hiwash03} for $p=1.$ 

\begin{proposition}[Basic properties of the Sobolev space $W^{1,p}(X)$] 
\label{prop: basic prop w1p}
Let $(X, g)$ be a rectifiable Riemannian space and fix $1 \leq p <\infty$. Then the following properties hold.
\begin{enumerate}
	\item (Closedness, c.f. \cite[Lemma 7.8]{Hiwash03})
Suppose $\{u_i\}_{i=1}^\infty,\{G_i\}_{i=1}^\infty$ are sequences in $L^p(X)$ such that $u_i$ and $G_i$ converge weakly to $u\in L^p(X)$ and $G\in L^p(X)$ respectively. If $G_i$ is a $p$-weak upper gradient of $u_i$ for each $i\in\mathbb{N}$, then there is a representative of $u$ in $L^p$ such that $G$ is a $p$-weak upper gradient of $u$.
\item(Lower semi-continuity, c.f. \cite[Cor. 7.10]{Hiwash03}) Let $p\in (1,\infty)$ and let $u_i \in W^{1,p}$ be a bounded sequence converging weakly in $L^p(X)$ to $u$. Then there is a representative of $u$ such that $u \in W^{1,p}(X)$ and 
\begin{equation}
	\| u \|_{W^{1,p}(X)} \leq \liminf_{i \to \infty} \| u_i\|_{W^{1,p}(X)}.
\end{equation}
\item (Banach space, c.f  \cite[Theorem 7.12]{Hiwash03}) The space $W^{1,p}(X)$ is a Banach space.
\item (Minimal $p$-weak upper gradient, c.f. \cite[Theorem 7.16]{Hiwash03}) There exists a minimal $p$-weak upper gradient $G_u\in L^p(X)$ in the sense that $G_u \leq G$ $m$-a.e. for every $p$-weak upper gradient $G \in L^p(X)$.
\item (Smooth spaces, c.f.   \cite[Theorem 7.13, Corollary 7.15]{Hiwash03})
Suppose $(X,g)$ is a smooth Riemannian manifold, then $W^{1,p}(X)$ coincides with the standard Sobolev space of $X$. Moreover, the norm of gradient vector $|\nabla u|_g$ is the least $p$-weak upper gradient for $u\in W^{1,p}(X)$. 
\end{enumerate}
	
\end{proposition}

Thanks to Proposition~\ref{prop: basic prop w1p}, for any $u \in W^{1,p}(X)$ we may write
\begin{equation}
	\| u\|_{W^{1,p}(X)} = \| u\|_{L^p(X)} + \| G_u\|_{L^p(X)},
\end{equation}
where $G_u$ is the least $p$-weak upper gradient of $u$. 
\\

Without imposing any additional structure, the space $W^{1,p}$ on a rectifiable Riemannian space may be trivial.  Moreover, as we saw in Example~\ref{ex: rn bad gradient}, the usual coordinate expression for the norm of the gradient may not be meaningful. For this reason, we introduce the notion of rectifiable Riemannian spaces that are {\it $W^{1,p}$-rectifiably complete}, that is, spaces for which  {the space $W^{1,p}$ is sufficiently large and} the minimal $p$-weak upper gradient coincides with the derivative in charts almost everywhere. 
\begin{definition}[$W^{1,p}$-rectifiable completeness]\label{def: rect completeness} Fix $p>n$.
We say that $(X,g)$ is $W^{1,p}$-rectifiably complete  if the following holds:
\begin{enumerate}
\item[(a)] $W^{1,p}(X)$ is dense in $L^p(X)$;
\item[(b)] For all $u\in W^{1,p}(X)$ and $a\in I$, the function
 $u_a=\phi_a^* u:U_a\rightarrow \mathbb{R}$ is differentiable a.e. and  
 \[
G_u(\phi_a(x))= |\nabla u|_g\equiv \sqrt{g_a^{-1}(\partial u_a(x),\partial u_a(x))}
 \]
   for $\phi^*_a m$-a.e. $ x\in  U_a$.  Here, $\partial u_a$ denotes the Euclidean gradient of $u_a$.
\end{enumerate}
\end{definition}

\vspace{.1cm}
\begin{example}
	{\rm 
	A smooth Riemannian manifold is $W^{1,p}$-rectifiably complete for any $p \in (n,\infty)$. 
	}
\end{example}
\vspace{.1cm}
\begin{example}
	{\rm
It is easy to check that the rectifiable Riemannian space of Example~\ref{ex: strat with lines} is $W^{1,p}$-rectifiably complete for any $p \in (n,\infty)$.
	}
\end{example}
\vspace{.1cm}
\begin{example}
	{\rm
For $\alpha >0$, consider the rectifiable Riemannian space $(\R^2, g_\alpha)$, where $g_\alpha=dx^2+|x|^{2\alpha}dy^2$.
This is a generalization of Example~\ref{example: power growth} and a special case of Example~\ref{example: prelim top change}.  
	Fix $ p >2$. There exists $\alpha = \alpha(p) \in (0, 1/2p)$ such that   $(\R^2, g_\alpha)$ is rectifiably complete; the proof of this fact is a special case of the proof of Proposition~\ref{prop: rc} in Section~\ref{sec: RC}.}
\end{example}
\vspace{.1cm}
\begin{example}{\rm
Thanks to Theorem~\ref{global-thm}, the $d_p$ limits of sequences of smooth Riemannian manifolds satisfying uniform lower bounds on scalar curvature and entropy are $W^{1,p}$-rectifiably complete for suitably chosen $p$. See Section~\ref{s:global_convergence} for further discussion and the proof of this fact.
}
\end{example}

\medskip

\subsection{The $d_p$ distance}\label{subsec: dp}

In view of Example~\ref{example: prelim top change}, we see that the geodesic distance does not reflect the underlying structure of  a rectifiable Riemannian space. This is mainly due to the degeneracy of the metric.  More seriously, we will see in Section \ref{examples-entropy} that these types of examples can arise as limits of manifolds with lower scalar curvature and entropy bounds.  At its heart, this occurs because the distance function requires $W^{1,\infty}$-control, which will be too much to ask for.  We introduce and discuss the notion of the $d_p$-distance here, which depends only on $W^{1,p}$-control of our space.  In the context of lower scalar curvature and entropy bounds, this will be obtainable for at least some $n<p<\infty$.

\begin{definition}[$d_p$ distance] \label{def: dp} Let $(X,g)$ be a rectifiable Riemannian space. Given $p \in (1,\infty)$ and $x, y \in X,$ 
the $d_p$ distance $d_{p,g,X }(x,y)$ between $x$ and $y$ is defined to be
\[
d_{p,g,X}(x,y)=\sup\left\{|f(x)-f(y)|: \int_X |\nabla f|^p\, dm \leq 1,\,  f\in W^{1,p}_{loc}(X)\cap C^0_{loc}(X)\right\}
\]
\end{definition}
When there is no ambiguity, we will frequently write $d_{p}$ or $d_{p,g}$ or $d_{p, X}$ in place of $d_{p, g, X}$.  On Euclidean space, we will often use the short-hand $d_{p,euc} = d_{p, g_{euc}, \R^n}$.  The definition of $d_p$ makes sense for any $p \in (1,\infty)$, but is only interesting for $p>n$. For instance, on Euclidean space,  $d_{p}(x,y)=+\infty$ for all $x \neq y$ whenever $p\leq n$.


\subsubsection{Examples}
We consider two examples of the $d_p$ distance on some rectifiable Riemannian spaces. To begin with, we study the behavior of the $d_p$ distance on Euclidean space.

\begin{example}[The $d_p$ distance on Euclidean space]
	{\rm

On Euclidean space, for $p>n$ we directly compute that $d_p(x,y) = S|x-y|^{1-n/p},$ where
\begin{equation}
	S= S_{n,p}= \sup\{ |f(x)-f(0)| : x \in B(0,1), \int_{\R^n} |\na f|^p \,dx \leq1 \}
\end{equation}
 is a normalizing constant.  Note $S_{n,p}\to 1$ as $p\to \infty$.  Thus, 
\begin{equation}\label{eqn: Euclidean ball and p ball}
	\B_{p,g_{euc}}(0, Sr^{1-n/p}) = B(0,r)
\end{equation}
for any $r>0$. 
{By taking the test function that is equal to $|x|\om_n^{-n/p}$ in $B(0,1)$ and is the constant $\om_n^{-n/p}$ on $\R^n\setminus B(0,1)$, we see that $S = S_{n,p} \geq \om^{-n/p}$, 
and thus $\B_{p,g_{euc}}(0,r) \subseteq B(0, \om_n^{n/(p-n)} r^{p/(p-n)}).$
 In particular, if $p>n$ then 
\begin{equation}\label{eqn: uniform ball containment} 
	\B_{p,g_{euc}}(0,r) \subseteq B(0, C_n r^{p/(p-n)})
\end{equation} 
for all $r>0$, where $C_n$ depends only on the dimension.}	
	}
\end{example}
\vspace{2mm}

\begin{example}[Hyperbolic space]\label{ex: hyperbolic}
{\rm
Given $n\geq 2$ and $n<p<\infty$, hyperbolic space $(\mathbb{H}^n, g_{hyp})$ has finite bounded diameter with respect to $d_p$. More specifically, there is a constant $C=C(n,p)$ such that $d_{p}(x,y) \leq C$ for all $x,y \in \mathbb{H}^n$. Indeed, this follows from the Morrey-Sobolev inequality on hyperbolic space established  in \cite{nguyen2018} (see also \cite{MugTal98} for the two dimensional case), which states that there exists $C= C(n,p)>0$ (which is, in fact, explicit and sharp) such that for any $f \in W^{1,p}(\mathbb{H}^n)$, we have 
\begin{equation}\label{eqn: hyp sob}
	\sup_{x \in \mathbb{H}^n} |f(x) | \leq C \| \na f\|_{L^p(\mathbb{H}^n)}.
\end{equation} 
In the definition of $d_p$, we do not require $f$ to be globally integrable. However, the proof of \eqref{eqn: hyp sob} is based on the Polya-Szeg\"{o} principle and the symmetric decreasing rearrangement of $f$, so it is easy to check that the same proof implies that, for any $R>0$ and $f \in W^{1,p}(B_{g_{hyp}}(R))$, 
\begin{equation}
	\sup_{x,y \in B_{g_{hyp}}(R)}  |f(x) - f(y)|  \leq C  \| \na f\|_{L^p(B_{g_{hyp}}(R))},
\end{equation}
where $C$ is the same constant as in \eqref{eqn: hyp sob} and in particular is independent of $R$. Thus, passing $R\to \infty,$ we arrive at the inequality 
\begin{equation}
	\sup_{x,y \in \mathbb{H}^n} |f(x) -f(y) | \leq C \| \na f\|_{L^p(\mathbb{H}^n)}.
\end{equation}
Consequently, $d_p(x,y) \leq C$ for all $x,y \in \mathbb{H}^n$.
}	
\end{example}
\vspace{2mm}

On any smooth closed Riemannian manifold $(M,g)$, $d_{p,g}$ defines a distance metric on $X$ as long as $p>n$. On the other hand, $d_p$ may only define a pseudometric if the metric is degenerate, as we see in the following example.
\begin{example}\label{example: dp}
	{\rm
Fix $\alpha>0$ and consider the rectifiable Riemannian space $(\R^2, g_\alpha)$, where $g_\alpha=dx^2+|x|^{2\alpha}dy^2$.
This is a generalization of Example~\ref{example: power growth} and a special case of Example~\ref{example: prelim top change}. For $p$ such that $\alpha p \geq 1,$ we have  $d_{p,g_\alpha}(x,y)=0$ for all $x,y \in \{0\}\times \mathbb{R}$. 
	}
\end{example}
\vspace{2mm}

\begin{remark}
	{\rm The definition of the  $d_p$ distance makes sense more generally for any space equipped with a $W^{1,p}$ structure. For instance, one may define the $d_p$ distance on a metric measure space $(X,d,m)$.  For reasonable spaces $X$, for instance those with lower Ricci bounds, one can show $d_p\to d$ as $p\to\infty$.
	}
\end{remark}
\vspace{2mm}

\subsubsection{Properties of the $d_p$ distance}
Let us discuss some basic facts about the $d_p$ distance.

First, as a consistency check, the following shows that two Riemannian manifolds which are $d_p$-isometric must in fact be isometric as Riemannian manifolds.
\begin{proposition}\label{prop: dp isometric} Fix $n\geq 2$ and $p>n$.
	Let $(M,g)$ and $(N,h)$ be compact Riemannian $n$-manifolds and suppose that $(M,d_{p,g})$ and $(N, d_{p,h})$ are isometric as metric spaces. Then $(M,g)$ and $(N,h)$ are isometric as Riemannian manifolds.
\end{proposition}
\begin{proof}

Since manifolds are $d_p$-complete this immediately tells us that we have a homeomorphism $\phi:M\rightarrow N$ such that $d_{p,g}(x,y)=d_{p,h}(\phi(x),\phi(y))$ for all $x,y\in M$. By the Myers-Steenrod theorem, it suffices to show that $d_g$ and $d_h$ are locally isometric. We first show that $d_{p,g}$ is locally a function of $d_g$.

\textit{Claim 1.} For any $x\in M$, we have 
\begin{equation}
\lim_{y\rightarrow x}\frac{d_{p,g}(x,y)}{d_g(x,y)^{1-n/p}}=S.
\end{equation}
\begin{proof}
[Proof of Claim 1.]
Suppose that on the contrary, we can find a sequence $\{y_i\}\subset M$ with $d_i=d_g(x,y_i)\rightarrow 0$ and $\e_0>0$ such that 
\begin{equation}\label{dp-expansion}
\left|\frac{d_{p,g}(x,y_i)}{d_i^{1-n/p}}-S \right|\geq \e_0>0.
\end{equation}
Consider the rescaled metric $g_i=d_i^{-2}g$ so that $d_{g_i}(x,y_i)=1$ for all $i\in \mathbb{N}$ and $d_{p,g_i}(x,y_i)=d_i^{n/p-1}d_{p,g}(x,y_i)$. Clearly, we have $(M,g_i,x)\rightarrow (\mathbb{R}^n,g_{euc},0)$ in the $C^\infty$-Cheeger-Gromov sense. By the proof of Theorem~\ref{thm: main thm, Lp def new}, it is easy to see that $d_{p,g_i}(x,y_i)\rightarrow S$ which contradicts \eqref{dp-expansion}.
\end{proof}

In particular, together from the claim  and the $d_p$ isometry, we conclude that for any $x\in M$, 
\begin{equation}\label{asymptotic-isometry}
\lim_{y\rightarrow x}\frac{d_h(\phi(x),\phi(y))}{d_g(x,y)}=1.
\end{equation}

Now fix $x\in M$ and $\phi(x)\in N$. Without loss of generality, we may assume that the injectivity radius at $x$ and $\phi(x)$ are at least $1$ by rescaling. By symmetry, it suffices to prove the following claim.
 
\textit{Claim 2.} 
For all $\e>0$, we have $d_h(\phi(x),\phi(y)) < (1+\e)d_g(x,y)$ for all $y\in B_g(x,1)$.
\begin{proof}
[Proof of Claim 2]
Let $s= \sup\{ r>0 : d_h(\phi(x),\phi(y))<  (1+\e)d_g(x,y) \text{ for all } y \in B_{g}(x,r)\}$. Thanks to \eqref{asymptotic-isometry}, we see that $s >0$. 
To establish Claim 2, we will show that $s\geq 1.$ 

Assume by way of contradiction that $s<1$. So, by definition of $s$, there exists $v\in T_xM$ with $|v|=1$ such that
\begin{enumerate}
\item[(a)] $d_h(\tilde \gamma(t),\tilde\gamma(0))<(1+\e) d_g(\gamma(t),\gamma(0))$ for all $t<s$;
\item[(b)]$d_h(\tilde \gamma(s),\tilde\gamma(0))=(1+\e) d_g(\gamma(s),\gamma(0))$
\end{enumerate}
where $\gamma(t)=\exp_x(tv)$ and $\tilde\gamma(t)=\phi(\gamma(t))$ for $t\in [0,1]$.
 Therefore, we have 
\begin{equation}
\begin{split}
(1+\e)s&=(1+\e) d_g(\gamma(s),\gamma(0))\\
&=d_h(\tilde\gamma(s),\tilde \gamma(0))\\
&\leq d_h(\tilde\gamma(s-\delta),\tilde \gamma(s))+d_h(\tilde\gamma(s-\delta),\tilde \gamma(0))\\
&\leq \delta\cdot  \frac{d_h(\tilde\gamma(s-\delta),\tilde\gamma(s))}{d_g(\gamma(s-\delta),\gamma(s))}+ (1+\e)(s-\delta).
\end{split}
\end{equation}
By \eqref{asymptotic-isometry}, if we choose $\delta$ sufficiently small, we have 
\begin{equation}
(1+\e)s\leq \delta(1+\e/2)+(1+\e)(s-\delta)
\end{equation}
which is impossible.
\end{proof}
By repeating the argument on $\phi^{-1}$ and letting $\e\rightarrow 0$, we have local geodesic distance isometry around $x$ and $\phi(x)$. Now, Myers-Steenrod Theorem implies that $\phi$ is differentiable at $x$ and satisfies $(\phi^*h)(x)=g(x)$. Since $x$ is arbitrary, this completes the proof.
\end{proof}
\vspace{2mm}
\begin{remark}[Scaling]\label{rmk: scaling prelim}
{\rm Given a $W^{1,p}$-complete rectifiable Riemannian space $(X,g)$,  let $\tilde{g}= \rho^{-2}g.$ Then  for any $x,y \in M$, we have 
\begin{equation}
	d_{p,\tilde{g}}(x,y)  =\rho^{n/p-1} d_{p,g}(x,y).
\end{equation}
}	
\end{remark}

\vspace{.1cm}

Next, in Example~\ref{example: dp}, we saw an example of a  degenerate metric on a rectifiable Riemannian space for which  $d_p$ only defined a pseudometric. It would therefore be sensible to formalize knowing when this does and does not happen: 

\begin{definition}[$d_p$-rectifiable completeness]\label{def: dp RC}
Given a rectifiable Riemannian space $(X,g)$, we say that  $(X,g)$ is $d_p$-rectifiably complete if $d_p$ defines a metric on $X$ and the topology induced by $d_p$ coincides with the topology of $X$.
\end{definition}

\begin{remark}\label{rmk: tautological sob}
	{\rm
From the definition of $d_p$, we see that for all $x \in X$,  we obtain the following local Sobolev inequality: 
\begin{equation}
	\sup_{y \in \B_p(x,R)} |f(x) -f(y)| \leq  R \| \na f\|_{L^p(X)} ,
\end{equation}
In particular, if $(X,g)$ is a compact rectifiable Riemannian space that is $d_p$-rectifiably complete, then it satisfies the Sobolev embedding $W^{1,p}(X) \hookrightarrow L^\infty(X)$.
}
\end{remark}


\begin{remark}[$d_p$ as the $(W^{1,p})^*$ norm]\label{rmk: dual space}
{\rm

Fix $p>n$ and let $(X,g)$ be a rectifiable Riemannian space satisfying the Sobolev embedding $W^{1,p}(X) \hookrightarrow L^\infty(X)$. For instance, one can consider any compact $d_p$-rectifiably complete $(X,g)$ by Remark~\ref{rmk: tautological sob} above.  Then for any $x \in X$, the Dirac delta $\delta_x$ is an element of the dual space $(W^{1,p}(X))^*$, and 
\begin{equation}
d_{p}(x,y) \equiv \| \delta_x - \delta_y \|_{(W^{1,p}(X))^*}.
\end{equation}
In particular, if we reinterpret the usual distance function $d(x,y)= \| \delta_x - \delta_y \|_{(W^{1,\infty}(X))^*}$, then we see the $d_p$ distance function been obtained by weakening the function space norm we use the measure the distance between the distributions $\delta_x,\delta_y$.
}	
\end{remark}

\vspace{2mm}

\subsubsection{$d_p$ convergence}

Our primary interest in the $d_p$ distance is to give rise to a notion of convergence which captures some Sobolev control.  We begin with $d_p$ convergence of compact sequences. 
\begin{definition}[$d_p$ convergence]\label{def: dp convergence} 
Let $(X,  g)$ and $(Y,h )$ be compact $d_p$ complete rectifiable Riemannian spaces.
Given $\e >0$, we say that 
\begin{equation}
	d_p((X,g) , (Y,h)) \leq \e 
\end{equation}
if there exist collections of points $\{x_i \}_{i=1}^N \subset X$ and $\{ y_i \}_{i=1}^N \subset Y$ such that each collection is $\e$-dense with respect to $d_p$ and 
\begin{equation}
	| d_{p, g,X}(x_i, x_j ) - d_{p,h,Y}(y_i, y_j ) | \leq \e
\end{equation} 
and further
\begin{equation}\label{eqn: dp conv vol}
1-\e \leq 	\frac{\vol_g(\B_p(x_i, r))}{\vol_h(\B_p(y_i, r))} \leq 1+\e
\end{equation}
for all $r \in [\e, 1]$.
\end{definition}
\begin{remark}
{\rm It is more standard to replace \eqref{eqn: dp conv vol} with something like
\begin{align}
	1-\e \leq 	\frac{\int_{\B_p(x_i, r)}1-r^{-1}d_p(x_i,z)}{\int_{\B_p(y_i, r)}1-r^{-1}d_p(y_i,z)} \leq 1+\e\, ,
\end{align}
which is strictly weaker and does not rule out the possibility that the measures are concentrating.  In our context we can work with the stronger condition, so we leave it like as in \eqref{eqn: dp conv vol}.
}
\end{remark}

In other words, two compact spaces are $\e$ close in the $d_p$ sense if their $d_p$ metric spaces are $\e$ Gromov-Hausdorff close and the volumes of balls above scale $\e$ are close. 
\begin{remark}
	{\rm In \eqref{eqn: dp conv vol}, we require volumes of  balls to be $\e$-close up to scale $1$. Up to scaling, we may replace  $1$ with any other  fixed number. 
	}
\end{remark}

Recall that a sequence of pointed proper metric spaces $(X_i, d_i, x_i)$ is said to converge to a pointed proper metric $(X,d, x)$   in the pointed Gromov-Hausdorff topology if $(\overline{B}_{d_i}(x_i ,R) , d_i) \to (\overline{B}_d(x, R),d)$ in the Gromov-Hausdorff topology for every $R>0$. In view of Example~\ref{ex: hyperbolic}, we clearly cannot adopt a direct analogue of this definition when defining pointed $d_p$ convergence. Indeed, we have seen that the hyperbolic space equipped with the $d_p$ metric is {\it not} a proper metric space for $p>n$, since sufficiently large balls have noncompact closure. For this reason, we cannot define pointed $d_p$ convergence by asking for $d_p$ convergence on $d_p$ balls of increasingly large radius. Instead, we make use of $d_p$ completeness to construct an exhaustion that plays the role of balls of large radius. 

More specifically, let $(X, g, x)$ be a $d_p$-complete pointed rectifiable Riemannian space. By $d_p$ completeness, for any $y \in X$, there is some radius $r\leq 1$ sufficiently small such that $\B_p(y,4r)\Subset X$ has compact closure. 
Roughly speaking, we define $Cov(x,N)$ to be the set of points that are linked to $x$ by a sequence of $N$ precompact balls of radius at most $1$. More concretely, define $Cov(x,N)$ be to the collection of points $y$  such that there is $\{(z_i,r_i)\}_{i=1}^N$ satisfying 
\begin{enumerate}
\item $r_i\leq 1$; 
\item $x,y\in \cup_{i=1}^N \B_p(z_i,r_i)$;
\item $\B_p(z_i,r_i)\cap \B_p(z_{i+1},r_{i+1})\neq \emptyset$ for all $i=1,...,N-1$;
\item $\B_p(z_i,4r_i)$ is pre-compact.
\end{enumerate}
\hfill\\
Note that $Cov(x,N)$ is an open set,  and  that by the triangle inequality, we always have the containment $Cov(x,N) \subseteq  \B_p(x,2N).$ 
To get an intuitive idea for how the sets $Cov(x,N)$  behave,  if we define  the analogue of $Cov(x,N)$ with respect to the geodesic distance instead of the $d_p$  distance,  then on any Riemannian manifold (or more  generally, on any  proper  length space),  this set is simply a geodesic ball of radius $2N$.



The main advantage of working with the sets $Cov(x,N)$ instead of $p$-balls of increasing radius is highlighted  in the following two initial lemmas, which show the sense in which $\{Cov(x,N)\}_{N\in\mathbb{N}}$ provides an exhaustion of $X$.  The first lemma shows that any $y \in X$  is contained in   $Cov(x,N)$   for some  $N\in\mathbb{N}$.



\begin{lemma} Let $(X,g,x)$  be a  $d_p$-complete  rectifiable Riemannian  space. 
	For any compact connected set $\Omega\subset X$ containing $x$, there exists $N \in \mathbb{N}$ such that $\Omega\subset Cov(x,N).$
\end{lemma}
\begin{proof}
By compactness of $\Omega$, we can find $1>r>0$ sufficiently small such that for all $z\in \Omega$, $\B_p(z,4r)\Subset X$. Then by compactness of $\Omega$, we can find $N$ such that $\Omega$ is covered by $\{\B_p(z_i,r) \}_{i=1}^N$ for $z_i\in \Omega$. And hence $Cov(x,N)$ contains $\Omega$.
\end{proof}

The second lemma is more subtle, and shows that $Cov(x,N)$ has compact closure for any $N \in \mathbb{N}$.
\begin{lemma} Let $(X,g,x)$  be a  $d_p$-complete  rectifiable Riemannian  space. 
  For any $N \in \mathbb{N}$  the set $Cov(x,N)$ has compact closure.
\end{lemma}
\begin{proof}
We  argue by induction.
For $N=1$, 
let $r_0 \leq 1$ be the  supremum over  radii $r \leq 1$ such that for some $y \in X$ we have $x \in \B_p(y,r ) $ and  $\B_p(y, 4r) \Subset X$. So, we may find some $y_0 \in X$ and $r_{y_0} \in (9r_0/10, r_0)$ such that $x \in \B_p(y_0,r_{y_0} ) $ and  $\B_p(y_0, 4r_{y_0}) \Subset X$ . We claim that 
\begin{equation}\label{eqn: contains}
	Cov(x, 1) \subset \B_p(y_0, 4r_{y_0}),
\end{equation}
the latter of which has compact closure by assumption. Indeed, for any $z \in Cov(x, 1)$, we have $z \in \B_p(y, r)$ for some $y \in X$ and for some $r \leq  r_0$ with $\B_p(y, 4r)\Subset X$.
Then by repeatedly applying the triangle inequality, 
 \begin{equation}
 z \in \B_p(y, r_0) \subset  \B_p(x, 2r_0) \subset \B_p(y_0, 3 r_0 ) \subset  \B_p (y_0 , 4r_{y_0}).
 \end{equation}
This establishes \eqref{eqn: contains}.

Now, suppose we have shown that $Cov(x,N)$ is pre-compact.  We claim that $Cov(x,N+1)\Subset X$. It suffices to show that for all sequences in $Cov(x,N+1)$, there is a subsequence that converges with respect to the underlying topology on $X$.  

Let $x_i$ be a sequence in $Cov(x,N+1)$.
For each $i$, we can find $\{(z_{i,a},r_{i,a})\}_{a=1}^{N+1}$ such that 
\begin{enumerate}
\item $r_{i,a}\leq 1$ for all $a\leq N+1$;
\item $x,x_i\in \cup_{a=1}^{N+1}\B_p(z_{i,a},r_{i,a})$;
\item $\B_p(z_{i,a},r_{i,a})\cap \B_p(z_{i,a+1},r_{i,a+1})\neq \emptyset$ for $a=1,...,N$;
\item $\B_p(z_{i,a},4r_{i,a})\Subset X$. 
\end{enumerate}

We may assume $x_i\in \B_p(z_{i,N+1},r_{i,N+1})$ and $x\in \cup_{a=1}^N \B_p(z_{i,a},r_{i,a})$ for all $i$ large enough. Otherwise, $x_i\in Cov(x,N)$ which we already know  to have compact closure and thus $x_i$ has convergent subsequence with respect  to the underlying topology. Let $z_i\in \B_p(z_{i,N+1},r_{i,N+1})\cap \B_p(z_{i,N},r_{i,N})\subseteq Cov(x,N)$. By compactness, we can assume $z_i\rightarrow z_\infty\in \overline{Cov(x,N)}$ which is compactly contained in $X$. Therefore for all $i,k$,
\begin{equation}
\begin{split}
d_p(x_i,z_{k})&\leq d_p(x_i,z_i)+d_p(z_i,z_k)\\
&\leq 2r_{i,N+1}+d_p(z_i,z_k).
\end{split}
\end{equation}

By compactness, we may assume $r_{i,N+1}\rightarrow r_\infty\in [0,1]$ as $i \to \infty$.  

{\bf Case 1.} If $r_\infty>0$, then we can find $K$ such that for all $i\geq K$,
\begin{equation}\begin{split}
d_p(z_i,z_K)<\frac12 r_{K,N+1},\quad \text{and}\quad 
r_{i,N+1}<\frac32 r_{K,N+1}\, .
\end{split}
\end{equation}
Therefore for $i\geq K,$ $x_i\in \B_p(z_K,3r_{K,N+1})\subset \B_p(z_{K,N+1},4r_{K,N+1})\Subset X$ which is compactly contained in $X$. 
\hfill\\

{\bf Case 2.} If $r_\infty=0$, then 
\begin{equation}\label{d_p-Cov}
\begin{split} 
d_p(x_i,z_\infty)&\leq d_p(x_i,z_i)+d_p(z_i,z_\infty)\\
&\leq 2r_{i,N+1}+d_p(z_i,z_\infty)\rightarrow 0.
\end{split}
\end{equation} 
By $d_p$ completeness, there is $r_{z_\infty}>0$ such that $\B_p (z_\infty,r_{z_\infty})\Subset X$. By \eqref{d_p-Cov}, $x_i\in \B_p (z_\infty,r_{z_\infty})$ for $i$ sufficiently large. This completes the proof. 
\end{proof}

With these two lemmas in hand, we can now define pointed $d_p$ convergence.
\begin{definition}\label{def: pointed dp} Let $(X_i, g_i, x_i)$ and $(X,g,x)$ be $d_p$-complete 
 pointed rectifiable Riemannian spaces. We say that
 \begin{equation}
 	(X_i, g_i, x_i) \to (X,g,x)
 \end{equation} 
 in the pointed $d_p$ sense if the following holds. 
  For all $N \in \mathbb{N}$, there exists $N'\geq  N$ and compact sets $\Omega\subset X$ and $\Omega_i \subset X_i$ such that 
 \begin{enumerate}
 	\item $Cov(x,N)  \subset \Omega \subset  Cov(x,N')$,
 	\item $Cov_{i}(x_i, N) \subset \Omega_i \subset Cov_{i}(x_i, N')$ for $i$ sufficiently large,
 	\item $d_p( (\Omega_i, g_i) , (\Omega, g)) \to 0.$
 \end{enumerate}
\end{definition}
\begin{remark}
	{\rm 
	Observe that in part (3) of Definition~\ref{def: pointed dp} above, the $d_p$ convergence on the compact sets $\Omega_i, \Omega$ corresponds to the relative $d_p$ distances $d_{p, g_i, \Omega_i}$ and $d_{p, g,\Omega}.$  This is necessary as $d_p$ is not a local object.
	}
\end{remark}


\vspace{.4cm}

\section{Further preliminaries}\label{sec: preliminaries} 

In this section, we introduce further preliminaries that will be needed in the paper.
 
 \medskip
\subsection{Ricci flows}\label{subsec: Ricci Flows}
Let us cover some of the basics of the Ricci flow in this subsection.  A Ricci flow $(M,g(t))_{t\in (0,T)}$ is a family of smooth metrics $g(t)$ on a smooth manifold $M^n$ satisfying the evolution equation 
\begin{equation}
	\pa_t g(t) = -2\Ric_{g(t)}.
\end{equation}
If $(M^n,g)$ is a complete Riemannian manifold with bounded curvature:
\begin{equation}
\sup_{x \in M} |\Rm_{g}|(x) <+\infty\, ,
\end{equation}
 then Shi in \cite{ShiExistence} established the short time existence of the Ricci flow for such a metric, following the existence theory on closed Riemannian manifolds due to Hamilton \cite{HamiltonExistence} and the trick by DeTurck \cite{DeTurck}.
 	
A sequence  $\{(M_i, g_i(t), x_i)_{t\in (0,T)}\}$ of complete pointed solutions of the Ricci flow satisfying
\begin{align} \label{eqn: scale invariant curvature bounds}
 		|\Rm_{g(t)}|\leq C/t,\qquad
 		\text{inj}(M,g(t)) \geq c_0 t^{1/2},
\end{align}
 	is compact in the $C^\infty$ Cheeger-Gromov topology. That is, up to a subsequence, $\{(M_i, g_i(t), x_i)_{t\in(0,T)}\}$ converges smoothly to a pointed complete solution of the Ricci flow $(M_\infty, g_\infty(t), x_\infty)_{t \in (0,T)}$, which also satisfies \eqref{eqn: scale invariant curvature bounds}.  This compactness theorem was originally observed in Hamilton \cite{HamiltonCptness}; see also \cite[Theorem 6.35]{CLNbook}.

Along the Ricci flow, the scalar curvature $R= R_{g(t)}(x)$ evolves by
 \begin{equation}\label{eqn: evolution of scalar}
 	\pa_t R = \Delta R + 2|\Ric|^2.
 \end{equation}
The equation is coupled to the Ricci flow  in the sense that $\Delta=\Delta_{g(t)}.$ 
 Because it is a super-solution of the heat equation, lower bounds on the scalar curvature are preserved under the Ricci flow. In other words, if $R_{g(0)} \geq -\delta $ for all $x \in M$, then 
 \begin{equation}\label{eqn: scalar lower bound flow2}
 R_{g(t)} \geq -\delta\, ,
 \end{equation} for all $x \in M$ and $t\in (0,T)$. This monotonicity provides one-sided control on the expansion of volumes under the Ricci flow, since the volume form evolves by
 \begin{equation}
 \pa_t d\vol_{g(t)}= -R_{g(t)}d\vol_{g(t)}.	
 \end{equation}
 As such, a flow satisfying \eqref{eqn: scalar lower bound flow2} has $
 	d\vol_{g(t)}\leq \exp\{\delta(t-s)\}d\vol_{g(s)}$ for all $s \leq t.$ So, provided $\delta\leq 1/2,$  a Taylor expansion shows that for all $ 0\leq s \leq t\leq \min\{1,T\}$, 
 	\begin{equation}
 	\label{eqn: volume forms}
 	d\vol_{g(t)}\leq \left\{1+2\delta(t-s)\right\}d\vol_{g(s)}. 
 \end{equation} 
\medskip   

\subsection{Heat flows coupled to Ricci flow}\label{sec: Kernels}
Given a complete bounded curvature Ricci flow $(M,g(t))$, $t \in [0,T)$, we consider the heat operator $\pa_t -\Delta$ coupled to Ricci flow, along with its formal adjoint the conjugate heat operator $-\pa_t -\Delta +R$. The Cauchy problem for the conjugate heat equation is well-posed backward in time.
 The two operators are conjugate in the sense that if $u,v$ are smooth functions, with suitable decay at infinity if $M$ is non-compact, then
 \begin{equation}\label{eqn: conjugate}
 \begin{split}
  \int_M uv \,&d\vol_{g(T)}-\int_M uv\, d\vol_{g(0)} = \int_0^T \int_M \{v(\pa_t-\Delta)u	-u(-\pa_t -\Delta +R)v\}\,d\vol_{g(t)}\,dt.	
 \end{split}	
 \end{equation}
For $0\leq s<t\leq T$, we let $K(x,t; y,s)$ denote the heat kernel with singularity at $(y,s)$, i.e. the solution to 
\begin{equation}
	\begin{cases}\label{eqn: heat kernel heat equation}
		(\pa_t-\Delta_x ) K(x,t;y,s) = 0 & \text{ for } s \in(0,t)\\
		\underset{t\searrow s}{\lim} K(\,\cdot\, ,t;y,s) = \delta_y\, .
	\end{cases}
\end{equation}
As a function of $(y,s)$, $K(x,t;y,s)$ is the kernel for the  conjugate heat equation with singularity at $(x,t)$, that is,
\begin{equation}
	\begin{cases}
		(\pa_s +\Delta_y- R) K(x,t;y,s) =0& \text{ for } s \in(0,t)\\
		\underset{s\nearrow t}{\lim} K(x,t;\ \cdot \ ,s) = \delta_x\,.
	\end{cases}
\end{equation}
Suppose $\vphi$ satisfies $(\pa_t +\Delta -R)\vphi = 0$ on $M\times [0,T]$, then 
\begin{equation}
	\label{eqn: L1 norm preserved}
	\int_M \vphi(x,t) \, d\vol_{g(t)}(x) = \int_M \vphi(x,T) \, d\vol_{g(T)}(x)
\end{equation}
for all $t \in [0,T]$. In particular, for any $s\in [0,t)$ we have
\begin{align}\label{eqn: int 1 for heat kernel}
	\int_M K(x,t ; y, s) \, d\vol_{g(s)}(y) &=1.\\
\shortintertext{
If \eqref{eqn: scalar lower bound flow2}  holds, then \eqref{eqn: volume forms} and  \eqref{eqn: heat kernel heat equation} imply that for $s\leq t \leq \min\{1,T\}$}
\label{eqn: int 2 for heat kernel} 
\int_M  K(x,t ; y, s) \, d\vol_{g(t)}(x) &\leq 1+ 2\delta(t-s)\,.	
\end{align}
From the evolution \eqref{eqn: evolution of scalar}, we have the following representation formula for the scalar curvature  in terms of the heat kernel:
\begin{equation}
\begin{split}
R_{g(t)}(x) &= \int_M K(x,t;y,0)R_{g(0)}(y)\, d\vol_{g(0)}(y)\\
&  +2 \int_0^t\int_M K(x,t;y,s)|\Ric_{g(s)}(y)|^2  d\vol_{g(s)}(y)\,ds.	
\end{split}	
\end{equation}

 We refer the reader to \cite[Chapter 26.1]{Chow3} for these basic properties on kernels for heat-type equations coupled to Ricci flow.
\\

\subsection{The $\mathcal{W}$ functional and Perelman entropy}\label{subsec: entropy}
Let $(M,g(t))_{t \in (0,T]}$ be a complete Ricci flow with bounded curvature.
The $\mathcal{W}$ functional defined in \eqref{eqn: W functional} is monotone along the Ricci flow in the following sense.  Set $\tau(t) = T-t$ and let 
 \begin{equation}
 	 u(x,t) = (4\pi \tau)^{-n/2}e^{-f(x,t)}\, ,
 \end{equation}
be a solution of the conjugate heat equation along the flow on $[0,T]$. 
Provided $u$ decays suitably at infinity (e.g. if $u(T)$ is compactly supported) if $M$ is non-compact, then 
 \begin{equation}\label{eqn: W monotone}
 	\mathcal{W}(g(s), f(s),\tau(s))\leq \mathcal{W}(g(t), f(t),\tau(t))
 \end{equation}
for $s\leq t.$ This monotonicity was shown in \cite{perelman2002entropy} for a closed manifolds and in \cite[Theorem 7.1(i),(ii)]{CTY11} for complete manifolds with bounded curvature. 
 By taking a compactly supported minimizing sequence for $\mu(g(t), T)$,  we see from \eqref{eqn: W monotone} that 
the Perelman entropy $\mu(\tau)$ defined in \eqref{eqn: perelman entropy def} is also monotone along the Ricci flow in the sense that 
\begin{equation}
	\mu(g(s),\tau(s))\leq \mu(g(t),\tau(t))\, ,
\end{equation}
for $s\leq t$.
Correspondingly, if $(M,g(t))_{t \in (0,2)}$ is a Ricci flow and if $\nu(g(0), 2) \geq -\delta$, then $\nu(g(t), 1) \geq -\delta$ for all $t \in (0,1].$ 

Suppose infimum in $\mu(g(t), \tau(t))$ is achieved. This is the case, for instance, if $M$ is closed; see \cite{ZhangLS} for necessary and sufficient conditions in the complete non-compact case with bounded geometry. Then the entropy $\mu(g(t), \tau(t))$ is constant in $t$ only if the Ricci flow is a gradient shrinking soliton that becomes singular at time $T$. This means that, if $f$ is a function achieving the infimum in $\mu(g(t), \tau(t))$ and  $\vphi(t)$ is the diffeomorphism generated by $\na f(t)$, then
	\begin{equation}\label{eqn: gradient shrinking soliton}
	g(t) = (T-t)\vphi(t)^*g(0).
	\end{equation}
On Euclidean space $(\R^n, g_{euc})$, the entropy $\mu(g_{euc}, \tau)$ is equal to zero for all $\tau>0$, and the infimum in  $\mu(g_{euc}, \tau)$ is achieved by $f(x)=|x|^2/4\tau.$
The following lemma asserts that Euclidean space is the only complete Riemannian manifold with bounded curvature with $\mu(g,\tau)=0$; c.f. \cite[Section 3.1]{perelman2002entropy}.
\begin{lemma}\label{lem: rigidity}
	Let $(M,g(t))_{t\in (0,T)}$ be a complete Ricci flow with bounded curvature. Then $\mu(g(0),\tau)\leq 0$ for all $\tau\in(0, T)$, with equality if and only if $M$ is noncompact and the flow is isometric to Euclidean space. Furthermore, fix  $t\in (0,T)$ and $x\in M$. Set $\tau(s) = t-s$, and let $f(y,s)$ be defined by  
	\begin{equation}\label{eqn: f kernel}
		K(x,t;y,s) = (4\pi\tau)^{-n/2} e^{-f(y,s)} .
	\end{equation}
If
	\begin{equation}
		\W(g(s),f(s),\tau(s)) =0\, ,
	\end{equation}
then the  flow is isometric to the constant Euclidean flow.
\end{lemma}
\begin{proof}
	Fix  $t\in (0,T)$ and $x\in M$. Set $\tau(s) = t-s$, and let $f(y,s)$ be defined by \eqref{eqn:  f kernel}.
	Then $\lim_{s\to t} \W(g(s), f(s), \tau(s) ) =0$. Hence, by monotonicity and the definition of $\mu(g,\bar\tau)$ as an infimum,
	\begin{equation}\label{eqn: nonpositive}
	\mu(g, t) \leq \W(g(0), f(0), t) \leq 0\,.
	\end{equation}
	This concludes the proof of the first claim. Furthermore, suppose that $\mu(g(0), t) = 0$ for some $t\in(0,T)$.  By \eqref{eqn: nonpositive}, $f$ achieves the infimum in $\mu(g(0), t)$, where again we let $f$ be defined by \eqref{eqn: f kernel}
	So, $g(t)$ is a gradient shrinking soliton given by \eqref{eqn: gradient shrinking soliton}.
In particular, 
\begin{equation}\label{eqn: soliton curvature}
	|\text{Rm}_{g(s)}| =   \frac{1}{t-s} |\text{Rm}_{g(0)}|.
\end{equation}
The flow exists and has bounded curvature for $s\in (0,T)$, with $T>t$. Thus 
\eqref{eqn: soliton curvature} implies that $|\text{Rm}_{g(0)}|= 0$. This (in particular, that $|\Ric_{g(0)}|=0$) together with  \eqref{eqn: gradient shrinking soliton} implies that  $(M,g(t))$ is a metric cone, c.f. \cite{WangLocalEnt}. However, a flat manifold which is also a metric cone can only be the Euclidean space.
We conclude 
$(M,g) = (\R^n, g_{euc})$.
\end{proof}

Finally, let us recall Perelman's no local collapsing theorem, which  ensures that small balls are non-collapsed along the flow if the entropy is bounded below. More specifically,  fix $x \in M$ and $t \in (0,T)$. Then for $ r\in (0,t^{1/2})$, if $R_{g(t)} \leq r^{-2}$ on $B_{g(t)}(x,r)$, then
  \begin{equation}\label{eqn: no local collapsing}
  \vol_{g(t)}(B_{g(t)}(x,r)) \geq \kappa r^n,	
  \end{equation}
  where $\kappa$ depends only on $n$ and $\nu(g(0), 2T).$

 \medskip
\subsection{Uniform existence time and scale invariant estimates}

A crucial point throughout the paper is that, under the lower bounds on the entropy assumed in our main theorems, the Ricci flow starting from $(M,g)$ exists for a uniform time and enjoys small scale invariant estimates on the curvature tensor. The following theorem establishes this fact, and  essentially follows from an epsilon regularity theorem of Hein and Naber in \cite{HeinNaber}. For the sake of completeness, and because the assumptions made here are slightly different than the ones there, we include the proof.  
\begin{theorem}\label{prop: uniform existence time and small curvature estimates}
	Fix $n\geq 2$ and $\ETA>0$. There exists $\delta=\delta(n,\ETA)>0$ such that the following holds. Let $(M,g)$ be a complete Riemannian $n$-manifold with bounded curvature satisfying 
		\begin{align}
		\label{eqn: entropy lower bounds2}
		\nu(g,2) \geq -\delta \,.
	\end{align}
	 Then the Ricci flow $(M, g(t))$ with $g(0)=g$ exists for $t \in (0, 1]$. Furthermore,
for all $x \in M$ and $t \in (0,1]$, we have the scale-invariant estimates
 \begin{equation}\label{eqn: tiny curvature bounds 0}
 	|\Rm_{g(t)}| \leq \ETA/t.
 \end{equation}
\end{theorem}
\begin{remark}\label{rmk: derivs}
\rm{ Thanks to Shi's derivative estimates, \eqref{eqn: tiny curvature bounds 0} implies that for all $k\in \mathbb{N}$, we have
\begin{equation}
	|\na^k \Rm|\leq C(n,k,\ETA)/t^{1+k/2}.
\end{equation}
}
\end{remark}
The key point of the proof of Theorem~\ref{prop: uniform existence time and small curvature estimates} is the continuity property contained in the following lemma.

\begin{lemma}\label{lem: limit space is Euclidean} Fix $n\geq 2 $, $C>0$.
	Let $\{(M_i, g_i(t), x_i)_{t \in (-1,1]}\}$ be a sequence of complete Ricci flows such that each time slice has bounded curvature and
	\begin{equation}\label{eqn: uniform curvature bound 2}
	|\Rm|\leq C \text{ on } M \times (-1,0].	
	\end{equation} 
	Suppose that 
	\begin{equation}\label{eqn: nu to zero}
	\nu(g_i(-1), 2) \geq -\delta_i
	\end{equation}
	where $\delta_i \to 0$. Then, up to a subsequence, $\{(M_i ,g_i(t) ,x_i)_{t \in (-1,0]}\}$ converges smoothly to the constant Euclidean flow $(\R^n, g_{euc}, 0^n)_{t \in (-1,0]}$.
\end{lemma}	
\begin{proof}
Perelman's no local collapsing \eqref{eqn: no local collapsing}
together with the curvature bounds \eqref{eqn: uniform curvature bound 2} provide a lower bound on the injectivity radius for each time-slice; see for instance \cite[Chapter 10, Lemma 51]{petersenBook}. So, by Hamilton's compactness theorem, the flows converge smoothly to a smooth limiting Ricci flow $(M_\infty, g_\infty(t), x_\infty)_{t\in[0,1]}$ that satisfies \eqref{eqn: uniform curvature bound 2}.

For each $i\in\mathbb{N}$, let $f_i(y,s) = f_i(x_i,1; y,s)$ be the function defined by 
\begin{equation}
	K_i(x_i, 1; y, s) = (4\pi(1-s))^{-n/2} \exp\{ -f_i(y,s) /4(1-s)\}.
\end{equation}
Here $K_i$ is the heat kernel on $(M_i,g_i,x_i)$ defined in \eqref{eqn: heat kernel heat equation}.
From \eqref{eqn: nu to zero}, we see that
\begin{equation}\label{eqn: w below}
	-\delta_i \leq \W(g_i(s), f_i(s), 1-s) \leq 0\,.
\end{equation}

The uniform curvature bounds \eqref{eqn: uniform curvature bound 2} ensure that the heat kernels are uniformly Gaussian. That is, setting $\tau =1-s$ for $s\in[0,1]$, we have
	\begin{equation}
		\frac{c}{\tau^{n/2}}\exp\Big\{ \frac{-d_{g_i(1)}(x_i,y)^2}{4\tau}\Big\} \leq K_i(x_i,1;y,s) \leq 
		\frac{C}{\tau^{n/2}}\exp\big\{ \frac{-d_{g_i(1)}(x_i,y)^2}{4\tau}\big\}.
	\end{equation}
Since $K_i(x_i,1;y,s)$ is also a solution of the conjugate heat equation \eqref{eqn: conjugate}, we determine that for $s \in (0,1)$, the sequence $\{f_i\}$ converges smoothly on compact sets to a function $f_\infty :M \times (0,1) \to \R$ that satisfies the minimization constraint in the definition of $\mu(g(0),\tau)$.  Thus from \eqref{eqn: w below}, we have
\begin{equation}
	\W(g_\infty(s), f_\infty(s), \tau(s)) = 0.
\end{equation}
We conclude by applying Lemma~\ref{lem: rigidity}.
\end{proof}

Now, Theorem~\ref{prop: uniform existence time and small curvature estimates} follows from Lemma~\ref{lem: limit space is Euclidean} by a standard contradiction argument, which we show below. Before giving the proof, it will be convenient to introduce the following notation. Let $(M,g(t))_{t\in(0,T)}$ be a smooth Ricci flow such that each time-slice is complete with bounded curvature.
Given $x\in M$ and $t\in(0,T)$,  we define the regularity scale $r_{|\Rm|}(x,t)$ to be
\begin{equation}
r_{|\Rm|}(x, t) = \sup \{ r>0 : \sup_{P(x,t,r)} |\Rm| \leq r^{-2} \} .
\end{equation}
Here $P(x,t,r) = B_{g(t)} (x,r) \times (t-r^2, t]$ is a parabolic cylinder.

\begin{proof}[Proof of Theorem~\ref{prop: uniform existence time and small curvature estimates}] 
In order to prove the theorem, it suffices to prove the following claim:

\noindent{\it Claim:}	Fix $n \in \mathbb{N}$ and $\ETA>0$. There exists $\delta=\delta(n,\ETA)$ such that if $(M,g(t))_{t\in(0,T]}$ satisfying $\nu(g(0), 2T)\geq -\delta $, then
	\begin{equation}
	r_{|\Rm|}(x,T)^2 \geq T/\ETA.
	\end{equation}

Before proving the claim, let us see how it implies the theorem.
 Let $T_{\max}>0$ be the maximal existence time for the Ricci flow with $g(0)=g$, and let $T_0 = \min\{1, T_{\max}\}$. For any $x\in M$ and $t \in(0 ,T_0)$,  we apply the claim to find that \eqref{eqn: tiny curvature bounds 0} holds at $(x,t)$. 
Finally, recall that if $T_{\max}<+\infty$, 
then $\sup_M |\Rm| \to +\infty$ 
as $t\to T_{\max}$. So, we conclude that $T_{\max}>T_0$, and thus $T_0 =1.$ 
This concludes the proof of the theorem.

Let us now prove the claim.
	Up to rescaling the flow and translating in time, it is equivalent to show that if $(M,g(t))_{t \in (-1,0]}$ satisfies $\nu(g(-1), 2)\geq -\delta$, then $r_{|\Rm|}(x,0)^2 \geq 1/\ETA$. Suppose for the sake of contradiction that the theorem fails. Then we may find a sequence of flows $\{(M_i, g_i(t))_{t \in [-1,0]}\}$
	with $\nu(g_i(-1),2) \geq -\delta_i$ for a sequence $\delta_i \to 0$, but 
	 $\inf\{ r_{|\Rm|}(x,0)^2 : x\in M_i \} < 1/\ETA.$ 
Let $x_i \in M_i$ be chosen so that 
\begin{equation}
 r_{|\Rm|}(x_i,0)^2 \leq 2 \inf\{ r_{|\Rm|}(x,0)^2 : x\in M_i \}
 \end{equation}
 and set $\rho_i^2 =\ETA  r_{|\Rm|}(x_i,0)^2/2 <1.$ Then the rescaled flow
 $
 \tilde{g}_i(t) = \rho_i^{-2} g_i(\rho_i^2 t)
 $
 is defined on $[-\rho_i^{-2}, 0]$ and satisfies $\nu(\tilde{g}_i(-\rho_i^{-2}), 2\rho_i^{-2}) \geq -\delta_i$.  In particular, thanks to the monotonicity of the Perelman entropy, \eqref{eqn: nu to zero} holds with $\tilde{g}_i$ replacing $g_i$.
  Furthermore, with respect to the rescaled metric, 
 \begin{equation}\label{eqn: reg scale}
 	r_{|\Rm|}(x_i,0)^2 =2/\ETA
 \end{equation}
 and $r_{|\Rm|}(y,0)^2 \geq 1/\ETA$ for all $y \in M_i$. This latter fact implies, in particular, that $ |\Rm_{\tilde{g}_i(t)}| \leq \ETA$
for all $(y,t) \in M_i \times [-1,0]$. 
Applying Lemma~\ref{lem: limit space is Euclidean}, we see that the flows converge smoothly to the constant Euclidean flow. On the other hand, the regularity scale \eqref{eqn: reg scale} passes to the limit, and we reach a contradiction.
	\end{proof}

\subsection{Applications of  Theorem~\ref{prop: uniform existence time and small curvature estimates}}

Finally, we have two direct applications of Theorem~\ref{prop: uniform existence time and small curvature estimates}. See also \cite{Wang2020} for applications of pseudolocality for a localized entropy. These applications could in fact be understood without most of the structure of Theorem \ref{global-thm}. The first is a finiteness theorem and is closely related to the pseudolocality finiteness in \cite{KleinerLott} and the $\epsilon$-regularity of \cite{HeinNaber}:
\begin{theorem}[Finiteness theorem]\label{thm: finiteness theorem}
	Fix $n\geq 2.$ There exists $\delta_0=\delta_0(n)>0$ such that for any $\delta \leq \delta_0$, and for any positive $C_0,\tau_0, V$,
 the space  $\mathcal{M}_{\delta, C_0,\tau_0, V}$ of compact Riemannian manifolds satisfying 
\begin{align*}
R\geq -C_0, \qquad \nu(g,\tau_0)\geq -\delta, \qquad \vol_g(M) \leq V
\end{align*}
contains finitely many diffeomorphism types.
\end{theorem}
\begin{proof}The proof is analogous to \cite[Theorem 37.1]{KleinerLott} which was proved by Perelman \cite[Remark 10.5]{perelman2002entropy} using Perelman's pseudolocality. In our case, we replace the use of Perelman's pseudolocality by Proposition~\ref{prop: uniform existence time and small curvature estimates} under assumption on entropy and the almost-monotonicity of the volume \eqref{eqn: volume forms}.
\end{proof}

We also show that a compact Riemannian manifold with entropy and scalar curvature lower bounds admits a metric of nonnegative scalar curvature.
\begin{theorem}
	Fix $n\geq 2.$ There exists $\delta_0 = \delta_0(n)>0$ such that  for any positive $\tau_0$ and $V$, there exists $\e = \e (\delta_0, \tau_0, V)$ such that if $(M,g)$ is a 
	\begin{equation}
		R \geq -\e, \qquad \nu(g, \tau_0) \geq \delta_0, \qquad \vol_g(M) \leq V.
	\end{equation}
	Then $M$ admits a metric of nonnegative scalar curvature.
\end{theorem}
\begin{proof}
Suppose by way of contradiction that the claim is false, so we may find a sequence of compact Riemannian manifolds $(M_i, g_i)$	such that $\nu(g_i, \tau_0) \geq -\delta_0$, $\vol_{g_i}(M_i) \leq V$, and $R_{g_i} \geq -1/i$  and $M_i$ does not admit a metric of nonnegative scalar curvature. Up to rescaling each metric, we may assume that without loss of generality $\tau_0  =1$. Applying Theorem~\ref{prop: uniform existence time and small curvature estimates}, we obtain a sequence of Ricci flows $(M_i, g_i(t))_{t \in (0,1]}$, such that the sequence $(M_i, g_i(1))$ has uniformly bounded geometry, $g_i(1)$-diameter uniformly bounded above, and satisfies $R_{g_i(1)} \geq -1/i.$ So, up to a subsequence, $(M_i, g_i(1))$ converges in the Cheeger-Gromov sense to a compact Riemannanian manifold $(M,g)$ with $R_{g} \geq 0$ and $\vol_g(M) \leq V$. In particular, for $i$ sufficiently large, $M_i$ is diffeomorphic to $M$. This contradicts the assumption that each $M_i$ does not admit a metric of nonnegative scalar curvature.
\end{proof}

\medskip

\subsection{Basic Ricci flow estimates}

In the final subsection of this preliminaries section, we give two further basic estimates that will be needed in the paper. 

First, the following lemma is a consequence of the proof of Theorem~\ref{prop: uniform existence time and small curvature estimates} which shows that the $g(1)$ ball of radius $16$ is smoothly close to the Euclidean ball of radius $16,$ provided that the entropy is chosen to be sufficiently small; see also \cite[Theorem 1.2]{Wang2020}. This observation will allow us to compare the metrics $g=g(0)$ and $g_{euc}$ by way of comparing $g(0)$ and $g(1)$, and will be used repeatedly throughout the paper.  

\begin{lemma}\label{rmk: good charts under regularity scale}
	{\rm
	Given any fixed $\e>0$, we may choose $\delta>0$ and $\ETA>0$ sufficiently small such that if \eqref{eqn: entropy lower bounds2} and \eqref{eqn: tiny curvature bounds 0} hold, then for any $x\in M$ and $t\in(0,1]$, we may find a
	diffeomorphism $\phi:B_{g(t)}(x_0,16t^{1/2})\to \Omega\subset \R^n$, with inverse $\phi = \psi^{-1}$ such that $\psi(x_0)=0$ and
\begin{equation}
		(1-\e) g_{euc}  \leq   \phi^*g(t)\leq (1+\e)g_{euc} 
\end{equation}
for all $x\in \Omega$. 
	In particular, 
	\begin{equation}\label{eqn: under reg scale euclidean volumes}
		(1-\e)\om_n r^n \leq \vol_{g(t)}(B_{g(t)}(x,r))\leq(1+\e)\om_n r^n
	\end{equation}
	for any $r\in (0,16t^{1/2}).$
	}
\end{lemma}

The second fact contained in this subsection is the  following elementary lemma using the scale invariant estimates \eqref{eqn: tiny curvature bounds 0} to bound the evolution of the metric from one dyadic scale to the next. We will make use of this lemma in Section~\ref{sec: decomposition}.

 \begin{lemma}\label{lem: bad bound} Fix $n\geq 2$ 
 and $\beta\in(0,1/4)$. There exists $\ETA=\ETA(n,\beta)>0$ such that the following holds. Let $(M,g(t))_{t \in (0,1]}$ be a Ricci flow satisfying \eqref{eqn: tiny curvature bounds 0}.
For any $x \in M$  and $0<s_1 \leq s_2 \leq 1$, we have 		
\begin{equation}\label{eqn: bad bound general}
		\Big(\frac{s_1}{s_2}\Big)^\beta g(s_1) \leq g(s_2) \leq \Big(\frac{s_1}{s_2}\Big)^{-\beta} g(s_1).
	\end{equation}
Consequently, for any $r >0$ we have 
	\begin{equation}\label{eqn: bad containment of balls}
		B_{g(s_1)} \big(x, rs_1^{1/2}\big) \subseteq B_{g(s_2)}\big(x,rs_2^{1/2} \big) \qquad \forall s_1 \leq s_2.
	\end{equation}
Furthermore, there is a universal constant $\ETA_0$ such that if $\ETA\leq \ETA_0$ then for $x \in B_{g(1)}(p, 2)$, we have
	\begin{equation}\label{eqn: basic containment}
		B_{g(t)}\big(x, 4t^{1/2}\big)\subseteq  B_{g(1)}(p, 4) \qquad \forall t \leq 1/2^5.
	\end{equation}

\end{lemma}
\begin{proof}
We first show the \eqref{eqn: bad bound general}. Fix $v\in TM$ and consider the function $g(t)(v,v)$. By the scale invariant bounds \eqref{eqn: tiny curvature bounds 0}, we have for all $t>0$,
\begin{equation}
\begin{split}
 -\frac{C_n\lambda}{t} g(v,v)\leq \partial_t g(v,v)&= -2Ric(v,v)\leq  \frac{C_n\lambda}{t} g(v,v).
\end{split}
\end{equation}
for some dimensional constant $C_n$. This can be seen by taking normal coordinate at each point such that $Ric$ is diagonal while $g$ is a identity matrix. Hence, by integrating the function $\log g(v,v)$, we see that for any $0<s\leq t\leq 1$, we have
	\begin{equation}
	\begin{split}
	\left( \frac{s}{t}\right)^{C_n\lambda }g(s)\leq g(t)\leq \left( \frac{t}{s}\right)^{C_n\lambda }g(s).
	\end{split}
	\end{equation}
	 This showed \eqref{eqn: bad bound general} by choosing $\lambda$ sufficiently small.	 


From \eqref{eqn: bad bound general} we directly deduce that for all $r>0$ and $s_1 \leq s_2$ we have
	\begin{equation}\label{eqn: immediate containment}
		B_{g(s_1)} \big(x, rs_1^{1/2}\big) \subseteq B_{g(s_2)}\big(x,rs_1^{1/2}(s_2/s_1)^\beta\big).
	\end{equation}
Then \eqref{eqn: bad containment of balls} follows from \eqref{eqn: immediate containment} because $\beta<1/2$. 
 To see \eqref{eqn: basic containment}, we take $r=4$ and $s_2=1$ in \eqref{eqn: immediate containment} to find
	\begin{equation}
	B_{g(t)}(x, 4t^{1/2})\subseteq B_{g(1)}(x, 4t^{1/2-\beta}) \subseteq B_{g(1)}(x, 2),
	\end{equation}
	where the second containment holds for all $t\leq 1/2^5$ because $\beta<1/4$.
The triangle inequality then ensures that \eqref{eqn: basic containment} holds.
\end{proof}

\bigskip
\section{Integral estimates for Ricci and scalar curvature}\label{sec: Ricci integral estimates}
This section has two main goals. The first is to prove an integral estimate for the Ricci curvature under the hypotheses of Theorem~\ref{thm: main thm, Lp def new}. This key estimate is used in Section~\ref{sec: decomposition} to prove the the Decomposition Theorem~\ref{thm: decomposition theorem}.
The second goal of this section is to prove the integral bounds of scalar curvature in Theorem~\ref{thm: integral scalar bound}, whose proof goes along similar lines to that of Theorem~\ref{thm: integral bounds for Ricci curvature}.

\begin{theorem}[Integral Ricci estimate]\label{thm: integral bounds for Ricci curvature}
	Fix $n\geq 2$, $\e>0$, and $\BETA \in [0,1/2)$. There exists $\delta=\delta(n,\BETA,\e)>0$ such that the following holds. Suppose $(M,g(t))_{t \in (0,1]}$ is a Ricci flow satisfying
	 \begin{align}
\label{eqn: scalar lower bound flow4} 	R_{g(0) } & \geq -\delta \\
 \label{eqn: entropy lower bound flow4}	\nu(g(0),2) & \geq -\delta.
 \end{align}
 Then for any $x \in M$ and any $t \in (0,1]$,
 \begin{align}\label{eqn: main ricci estimate}
		\int_0^t \left(\frac{s}{t}\right)^{-\BETA} \fint_{B_{g(t)}(x,4t^{1/2})} |\Ric_{g(s)}| \, d\vol_{g(s)} \, ds 
		& \leq   \e^2\, .
			\end{align}
\end{theorem}

This section is organized in the following way. In Section~\ref{subsec: heat kernel lower bounds}, we prove an almost-Gaussian lower bound for the conjugate heat kernel  and for a cutoff function evolving by the conjugate heat equation. A major tool is a heat kernel estimate due to Zhang (see Proposition~\ref{prop: Zhang} below). In Sections~\ref{subsec: proof of integral ricci estimate} and \ref{sec: scalar integral bound}, we establish Theorem~\ref{thm: integral bounds for Ricci curvature} and Theorem~\ref{thm: integral scalar bound} respectively by integrating the evolution equation for the scalar curvature  \eqref{eqn: evolution of scalar} against suitably chosen functions to which we apply the estimates of Section~\ref{subsec: heat kernel lower bounds}.

By Theorem~\ref{prop: uniform existence time and small curvature estimates}, the hypotheses of Theorem~\ref{thm: integral bounds for Ricci curvature} imply that the scale invariant curvature bounds 
 \begin{equation}\label{eqn: tiny curvature bounds4}
 	|\Rm_{g(t)}| \leq \ETA/t
 \end{equation}
 hold for all $x \in M$ and $t \in (0,1]$ with $\ETA$ as small as desired by choosing $\delta$ sufficiently small.

 \subsection{Heat kernel lower bounds and evolving cutoff function}\label{subsec: heat kernel lower bounds}
In Proposition~\ref{prop: heat kernel lower bounds} below, we establish lower bounds on the heat kernel, and in  Proposition~\ref{prop: cutoff} we prove lower bounds for a cutoff function that evolves by the conjugate heat equation. Both lower bounds become degenerate for small times. Nonetheless, the degeneration occurs in a sufficiently controlled way for our application  in the proofs of Theorems~\ref{thm: integral bounds for Ricci curvature} and \ref{thm: integral scalar bound}.

\begin{proposition}\label{prop: heat kernel lower bounds} Fix $n\geq 2$ and $\ETA>0$. There exist $\delta=\delta(n,\ETA)>0$ and $C=C(n)>0$ such that the following holds. Suppose that $(M,g(t))_{t \in [0,1]}$ is a complete Ricci flow with bounded curvature satisfying \eqref{eqn: scalar lower bound flow4}, \eqref{eqn: entropy lower bound flow4}, and \eqref{eqn: tiny curvature bounds4}. Then for any $0<s<t<1$, letting $\tau = t-s$, we have
	 \begin{equation}\label{eqn: heat kernel lower bound}
	 	K(x,t;y,s) \geq \left(\frac{s}{t}\right)^{\ETA}{C}\tau^{-n/2} \exp\left\{- 4d_{g(t)}(x,y)^2/\tau\right\}\,.
	 \end{equation}
\end{proposition}
\begin{remark}
	{\rm The hypothesis \eqref{eqn: scalar lower bound flow4} is not actually necessary in Proposition~\ref{prop: heat kernel lower bounds}, but for convenience we assume it to allow for a simpler proof; in particular we can directly call upon Zhang's heat kernel lower bound Proposition~\ref{prop: Zhang} below.
	}
\end{remark}
Before proving Proposition~\ref{prop: heat kernel lower bounds}, let us state its main consequence. Fix $t \in(0,1]$ and $r^2 \geq t$. Consider a smooth function $\vphi: M\times \{t\}$ such that 
\begin{equation}
	\label{eqn: cutoff}
	\vphi(y) = \begin{cases}
	1 & \text{ in } B_{g(t)}(x_0, 8r)\\
	0 & \text{ in } M\setminus B_{g(t)}(x_0,16r)\,.
\end{cases}
\end{equation}
For $s\in[0,t)$, let $\vphi(y,s)$ be the solution of the conjugate heat equation with terminal data given by $\vphi(y,t)$:
\begin{equation}\label{eqn: evolution of cutoff}
	\begin{cases}
		(\pa_s +\Delta - R) \vphi(y,s) = 0 & \text{ in } M\times [0,t)\\
		\vphi(y,t) = \vphi(y).
	\end{cases}
\end{equation}

Proposition~\ref{prop: cutoff} below shows that $\vphi(y,s)$ behaves sufficiently like a cutoff function for all $s\in (0,1]$ so that we can derive useful estimates.

\begin{proposition}\label{prop: cutoff}
	Fix $n\geq 2$. There exist $\delta=\delta(n)>0$ and $C=C(n)>0$ such that the following holds. Suppose that $(M,g(t))_{t \in [0,1]}$ is a Ricci flow for satisfying \eqref{eqn: scalar lower bound flow4}, \eqref{eqn: entropy lower bound flow4}, and \eqref{eqn: tiny curvature bounds4} for $\ETA\leq 1$.
	Then we have 
	\begin{equation}\label{eqn: cutoff lower bound}
		\vphi(y,s) \geq C\left(\frac{s}{t}\right)^{\ETA}
	\end{equation}
	for all $(y,s) \in B_{g(t)}(x_0,4r) \times (0,t)$.

\end{proposition}

We now proceed to the proof of Proposition~\ref{prop: heat kernel lower bounds}, which follows from the following lower heat kernel bound due to Zhang in \cite{Zhang12}.
\begin{proposition}[Zhang]\label{prop: Zhang}
	 Fix $n\geq 3$ and let $(M,g(t))_{t \in [0,1]}$ be a Ricci flow for satisfying \eqref{eqn: scalar lower bound flow4} and \eqref{eqn: entropy lower bound flow4}. Then for any $0\leq s< t\leq 1,$ we have 
		 \begin{equation}\label{eqn: zhang estimate}
		 \begin{split}
		K(x,t;y,s)\geq  \frac{c}{\tau^{n/2}} \exp\left\{ \frac{-4}{\tau}d_{g(t)}(x,y)^2 - \frac{1}{\sqrt{\tau} }\int_0^\tau \sqrt{s'} R(y,t-s') \,ds' - 2 \delta\right\}.	
		 \end{split}
		 \end{equation}
	  Here $c= c(n)$ and we let $\tau = t-s.$
\end{proposition}
	\begin{proof}[Proof of Proposition~\ref{prop: heat kernel lower bounds}]
Provided we choose $\delta\leq 1$, the estimate \eqref{eqn: zhang estimate} implies that 
	\begin{equation}
		K(x,t;y,s) \geq c\tau^{-n/2} \exp\left\{- 4d_{g(t)}(x,y)^2 /\tau \right\} \times  F(y,s)
	\end{equation}
	where $c$ is a universal constant and
	\begin{equation}\label{eqn: F for heat kernel bound-0}
	F(y,s) = 
		\exp\left\{-\frac{1}{\sqrt{\tau} }\int_0^\tau \sqrt{s'} R(y,t-s') \,ds'\right\}\,.
	\end{equation}
	We claim that 
	\begin{equation}\label{eqn: F for heat kernel bound}
	\log F(s,y)\geq \log \left(\frac{s}{t}\right)^{\ETA}.	
	\end{equation}
Exponentiating \eqref{eqn: F for heat kernel bound} will conclude the proof of the proposition.
	To this end, we change variables and then bound the scalar curvature using the scale invariant curvature bound \eqref{eqn: tiny curvature bounds4}, finding that
\begin{equation}\begin{split}
	-\log F(s,y)&=  \frac{1}{\tau^{1/2}} \int_0^\tau \sqrt{s'} R(y,t-s') \,ds'\\
	   &\leq  \frac{\ETA}{\tau^{1/2}}\int_0^\tau \sqrt{s'}(t-s')^{-1} \,ds'\\
	   & \leq \ETA \int_0^\tau (t-s')^{-1} \, ds' = \ETA \int_s^t \rho^{-1}\,d\rho = \ETA \log \frac{t}{s}.
\end{split}\end{equation}
Negating this expression establishes \eqref{eqn: F for heat kernel bound} and thus concludes the proof.
	\end{proof}	
We now prove Proposition~\ref{prop: cutoff}.
\begin{proof}[Proof of Proposition~\ref{prop: cutoff}] 
 Expressing the solution with respect to the conjugate heat kernel, we have  
 \begin{equation}
 	\label{eqn: evolve cutoff}
 	\vphi(y,s)= \int_M \vphi(x) K(x,t; y,s) \, d\vol_{g(t)}(x).
 \end{equation} 
 Fix any $y\in B_{g(t)}(p,4r)$. Having chosen $r^2\ge t$, we note that $B_{g(t)}(y,t^{1/2})\subset  B_{g(t)}(p,8r)$, and in particular $\vphi(x,t)=1$ in this set.
Using this observation, followed by Proposition~\ref{prop: heat kernel lower bounds}, we find that 
\begin{align}
	\left(\frac{s}{t}\right)^{-\ETA}\vphi(y,s)  & \geq \left(\frac{s}{t}\right)^{-\ETA}\int_{B_{g(t)}(y,t^{1/2})} K(x,t;y,s)\,d\vol_{g(t)}(x)\\
\label{eqn: integral lower bound}	&\geq \int_{B_{g(t)}(y,t^{1/2})} \frac{C}{\tau^{n/2}} \exp\left\{- \frac{4}{\tau}d_{g(t)}(x,y)^2\right\} \, d\vol_{g(t)}(x).
	\end{align}
	Since $\tau =t-s\leq t,$ we see from Lemma~\ref{rmk: good charts under regularity scale} that the right-hand side is bounded below by a universal constant, namely, by
	\begin{equation}
	\int_{B(0,\tau^{1/2})} \tau^{-n/2} \exp\{-4|x|^2/\tau\}\,dx = \int_{B(0,1)} \exp\{-4|x|^2\}\,dx,
\end{equation}
so long as $\ETA \leq \ETA_0.$
 This completes the proof.
\end{proof}

\subsection{Proof of Theorem~\ref{thm: integral bounds for Ricci curvature}} \label{subsec: proof of integral ricci estimate}
Before proving Theorem~\ref{thm: integral bounds for Ricci curvature}, let us make the following observation.
\begin{lemma}\label{lem: basic observation}
Fix $n\geq 2,$ $\delta>0$, and $\ETA>0$. Let $(M,g(t))_{t \in [0,1]}$ be a Ricci flow satisfying \eqref{eqn: scalar lower bound flow4} and \eqref{eqn: tiny curvature bounds4}. For any $t\in (0,1]$, let $\vphi:M\times \{t\} \to \R$ be a nonnegative smooth function, and if $M$ is  non-compact then assume $\vphi$ has compact support. Let $\vphi(y,s)$ be the evolution of $\vphi$ by the conjugate heat equation for $s \in (0,t)$. Then 
\begin{equation}\label{eqn: ricci bound a}
\begin{split}
	2\int_0^t \int_M  |\Ric_{g(s)}&(y)|^2\vphi(y,s)\, d\vol_{g(s)}(y)\, ds\leq \Big(\frac{\ETA}{t}+\delta\Big) \int_M  \varphi(y,t) d \vol_{g(t)}(y).
\end{split}	
\end{equation}
\end{lemma}
\begin{proof}
	We multiply $\vphi(y,s)$ by the evolution equation for the scalar curvature \eqref{eqn: evolution of scalar} and integrate in space and time to obtain the following. After an integration by parts, we find that 
	
\begin{equation}
\begin{split}
\label{eqn: rhs}		2\int_0^t \int_M & |\Ric_{g(s)}(y)|^2 \vphi(y,s) \, d\vol_{g(s)}(y) \, ds \\
&= \int_0^t \int_M (\pa_s -\Delta)R_{g(s)}(y) \vphi(y,s) \,  d\vol_{g(s)}(y) \, ds \\
	& =  \int_0^t \int_M  R_{g(s)}(y)\, (\pa_s +\Delta -R_{g(s)})\vphi(y,s) \,  d\vol_{g(s)} (y)\, ds \\
	& \ \ \ + \int_M R_{g(t)}\varphi(y,t) \, d\vol_{g(t)}(y) - \int_M R_{g(0)} \vphi(y,0) \vol_{g(0)}(y) \\
&=\int_M R_{g(t)}\varphi(y,t) \, d\vol_{g(t)}(y)- \int_M R_{g(0)} \vphi(y,0) \vol_{g(0)}(y) .
\end{split}
\end{equation}
Note that this integration by parts is justified because, for each fixed time-slice,  $\varphi$ and $|\na \varphi|$ decay exponentially with respect to $d_{g(t)}(x,\cdot);$ see \cite[Chapter 26.1]{Chow3}. 
We wish to bound the right-hand side of the last equation on \eqref{eqn: rhs} from above. By the maximum principle the function $\vphi(y,s)$ is nonnegative for all $y,s$. Hence, making use first of the lower bound on scalar curvature \eqref{eqn: scalar lower bound flow4} and then of the conservation of the $L^1$ norm under the conjugate heat equation \eqref{eqn: L1 norm preserved}, we have 
\begin{align}
 - \int_M R \vphi \vol_{g(0)} &  \leq \delta \int_M \vphi \vol_{g(0)} =\delta \int_M \vphi \vol_{g(t)}	.
\end{align}	

Pairing this with \eqref{eqn: rhs} and applying the scale-invariant curvature estimates \eqref{eqn: tiny curvature bounds4} to bound the scalar curvature in the $t$-time slice, we find that
\begin{equation}
\begin{split}
2\int_0^t \int_M  |\Ric_{g(s)}(y)|^2\vphi(y,s)\, d\vol_{g(s)}(y) \, ds &\leq \int_M (R_{g(t)}(y)+\delta) \varphi(y,t) d \vol_{g(t)}(y)\\
&\leq \Big(\frac{\ETA}{t}+\delta\Big) \int_M  \varphi(y,t) d \vol_{g(t)}(y)\,.
\end{split}
\end{equation}
This concludes the proof of the lemma.
\end{proof}

Finally, we prove Theorem~\ref{thm: integral bounds for Ricci curvature}.
\begin{proof}[Proof of Theorem~\ref{thm: integral bounds for Ricci curvature}]
Up to rescaling the flow, we may assume that $t=1.$
Together Lemma~\ref{lem: basic observation} and Proposition~\ref{prop: cutoff} (with $r=t^{1/2}=1$)  imply that
	 \begin{equation}\label{eqn: intermediate bound}
	 \begin{split}
		\int_0^1 s^{\ETA} \int_{B_{g(1)}(x,4)} |\Ric_{g(s)}|^2 \, d\vol_{g(s)}\,ds &\leq C(\ETA + \delta)\vol_{g(1)}(B_{g(1)}(x,16)) \leq  C\left(\ETA + \delta\right), 	
	 \end{split}	
	 \end{equation}
	 where the second inequality comes from \eqref{eqn: under reg scale euclidean volumes}.
	  Noting further that by \eqref{eqn: volume forms} and \eqref{eqn: under reg scale euclidean volumes}, we have 
 \begin{equation}\label{eqn: intermed volume constant}
 \inf_{0<s<1} \vol_{g(s)}(B_{g(1)}(x,4r)) \geq 
 	(1-2\delta)\vol_{g(1)}(B_{g(1)}(x,4r)) \geq c.
 \end{equation}
Hence,	
	 \begin{equation}\label{eqn: intermediate bound 2}
	 \begin{split}
		\int_0^1 s^{\ETA} &\fint_{B_{g(1)}(x,4)}|\Ric_{g(s)}|^2 \, d\vol_{g(s)}\,ds \leq C(\ETA +\delta)
	 \end{split}	
	 \end{equation}


Now, fix $\ETA>0$ sufficiently small so that $\ETA\leq 1/2-\BETA$. In this way, if we set $\BETA_0:=\BETA + \ETA/2$, we ensure that 
\begin{equation}\label{eqn: numerology}
1-2\BETA_0 \geq 1/2 - \BETA.
\end{equation}
 Choose $\delta$ sufficiently small so that \eqref{eqn: tiny curvature bounds4} holds for this choice of $\ETA.$
Then by H\"{o}lder's inequality, for any $\Omega \subseteq M$, we have
\begin{equation}
\begin{split}
	\label{eqn: holder 1}
	\int_0^1 s^{-\BETA} \fint_\Omega |\Ric_{g(s)}| \, d\vol_{g(s)} \, ds
	& \leq \Big(\int_0^1 s^{-2\BETA_0 } \,ds\Big)^{1/2}\Big(\int_0^1 s^{\ETA} \fint_\Omega |\Ric|^2 \, d\vol_{g(s)}\,ds \Big)^{1/2}\\
	 &= (1- 2\BETA_0)^{-1/2}\Big(\int_0^1 s^{\ETA} \fint_\Omega |\Ric|^2 \, d\vol_{g(s)}\,ds \Big)^{1/2}.
	 \end{split}
\end{equation}
	The constant $(1- 2\BETA_0)^{-1/2}$ is bounded above by $(1/2-\BETA)^{-1/2}$ thanks to \eqref{eqn: numerology}. 
	By choosing $\ETA$ and $\delta$ such that $C(\ETA + \delta)^{1/2}\leq \e^2$, together with \eqref{eqn: holder 1} and \eqref{eqn: intermediate bound 2}
	 conclude the proof. 
\end{proof}

\subsection{Integral bounds for the scalar curvature}\label{sec: scalar integral bound}
We now prove Theorem~\ref{thm: integral scalar bound}, which we restate below as Theorem~\ref{thm: integral scalar bound restated} below. The proof is similar to that of Theorem~\ref{thm: integral bounds for Ricci curvature}.
\begin{theorem}[$L^q$ scalar curvature estimates]\label{thm: integral scalar bound restated}
	Fix $n\geq 2$, $q\in(0,1)$, and $\e>0$.
There exists $\delta=\delta(n,q,\e)>0$ such that the following holds. 
Let $(M,g)$ be a closed Riemannian $n$-manifold 
 such that 
	\begin{align}
		R\geq -\delta, \qquad \qquad 
		\nu(g, 2) \geq -\delta \,.
	\end{align}
	 Then we have
	 \begin{equation}
	 	\fint_{M} |R|^q\,d\vol_{g}\leq \e.
	 \end{equation}
\end{theorem}
\begin{proof}[Proof of Theorem~\ref{thm: integral scalar bound}] Let $R_+$ and $R_-$ denote the positive and negative parts of $R$  respectively. Since $ \fint R_-^q \,d\vol_{g} \leq \delta^q$, we choose $\delta^q \leq \e/2$ and it suffices to show that $\fint R_+^q \,d\vol_{g} \leq \e/2$. By Theorem~\ref{prop: uniform existence time and small curvature estimates}, for any fixed $\ETA>0$, we may choose $\delta$ small enough so that the Ricci flow $(M,g(t))$ with $g(0)=0$ exists for $t\in(0,1]$ and enjoys the scale invariant curvature bounds \eqref{eqn: tiny curvature bounds4} for all $x\in M$ and $t\in (0,1]$.
	Consider the nonnegative function
	\begin{equation}
	f(x,t) = R_{g(t)}(x) + 2\delta.
	\end{equation}
	Note that $f\geq R_+$, so it suffices to show that $\fint_M f^q \,d\vol_{g(0)}\leq \e/2.$
	For any $q\in (0,1]$, we see that $f^q$ is a super-solution of the heat equation coupled to Ricci flow. Indeed, noting that $q(q-1)<1$ and recalling \eqref{eqn: evolution of scalar}, we compute that 
\begin{equation}
\begin{split}
(\pa_t -\Delta)f^q & = qf^{q-1}(\pa_t -\Delta)f -q(q-1) f^{q-2}|\na f|^2\\
	& \geq qf^{q-1}(\pa_t -\Delta)R\geq 0.
\end{split}	
\end{equation}
So, applying \eqref{eqn: conjugate}  with $u=f^q$ and $v=1,$ we find
\begin{align}\nonumber
\int_M f^q\,d\vol_{g(0)} & = \int_M f^q\,d\vol_{g(1)} -\int_0^1\int_M\{ (\pa_t -\Delta)f^q -Rf^q\}\,d\vol_{g(t)}\,dt\\
\label{eqn: scalar int bound intermediate}
& \leq 	 \int_M f^q\,d\vol_{g(1)} +\int_0^1\int_M Rf^q\,d\vol_{g(t)}\,dt.
\end{align}

We bound each of the terms on the right-hand side of \eqref{eqn: scalar int bound intermediate} separately. For the first term, using the scale-invariant curvature bounds \eqref{eqn: tiny curvature bounds4} and \eqref{eqn: volume forms}, we see that 
\begin{equation}\label{2}
\begin{split}
\int_M f^q\,d\vol_{g(1)} &\leq (\ETA+ 2\delta)^q \vol_{g(1)}(M)\\
& \leq 2(\ETA+2\delta )^q\vol_{g(0)}(M).	
\end{split}
\end{equation}
As for the second term on the right-hand side of \eqref{eqn: scalar int bound intermediate}, we note that $Rf^q \leq f^{q+1} \leq 2^{q+1}(|R|^{q+1}+(2\delta)^{q+1} )$. So, again making use of \eqref{eqn: volume forms}, we find
\begin{equation}\label{3}
	\int_0^1 \int_M Rf^q \, d\vol_{g(t)}\,dt \leq C\delta^q\vol_{g(0)}(M) + C\int_0^1 \int_M |R|^{q+1} \, d\vol_{g(t)}\,dt,
\end{equation}
where $C$ is a constant depending on $q$.
We bound second term on the right-hand side of \eqref{3} using the same argument as in the proof of Theorem~\ref{thm: integral bounds for Ricci curvature}. More specifically, let $\vphi: M\times (0,1) \to \R$ be the solution to the conjugate heat equation with terminal data $\vphi(x,1)=1 $ on $M \times \{1\}$. By (the proof of) Proposition~\ref{prop: cutoff}, we see that $\vphi(y,s) \geq cs^\ETA$ for all $y\in M$  and $s\in (0,1]$, where $c=c(n)$. Thus, applying Lemma~\ref{lem: basic observation} to this choice of $\vphi$, we find that
\begin{equation}\label{eqn: ric squared bound}
\begin{split}
	  \int_0^1 t^\ETA \int_M |R|^2 \, d\vol_{g(t)}\,dt 
\leq	C \int_0^1\int_M |R|^2 \vphi(y,t) \, d\vol_{g(t)}\,dt
	& \leq C(\ETA +\delta) \vol_{g(1)}(M)\\
	& \leq C (\ETA+\delta) \vol_{g(0)}(M).
\end{split}	
\end{equation}
 We choose $\delta$ sufficiently small so that $\ETA<(1-q)/(1+q)$, and set $\BETA=\ETA(1+q)/(1-q)<1$. So, H\"{o}lder's inequality (with $p=2/(1+q)$ and $p'=2/(1-q)$) together with \eqref{eqn: ric squared bound} allows us to deduce that
\begin{equation}\begin{split}
	\int_0^1 \int_M |R|^{q+1}\,d\vol_{g(t)}\,dt
	& 
	\leq \Big(2\vol_{g(0)}(M)\int_0^1 s^{-\BETA}\,ds\Big)^{(1-q)/2} \Big(\int_0^1 s^\ETA\int_M  |R|^2\,d\vol_{g(s)}\,ds\Big)^{(1+q)/2}\\
	&\leq C(q)(\ETA+\delta)^{(1+q)/2} \vol_{g(0)}(M).
\end{split}\end{equation}
 So, pairing this estimate with \eqref{eqn: scalar int bound intermediate}, \eqref{2} and \eqref{3} we find that
\begin{align}\label{eqn:  final scalar int bound}
	\fint_M f^q \,d\vol_{g(0)}& \leq C(\delta + \ETA)^q + C\delta^q +C(\ETA+\delta)^{(1+q)/2},
\end{align}
where $C$ depends on $q$ and $n$.
Choose $\delta$ sufficently small so that the right-hand side of \eqref{eqn:  final scalar int bound} is bounded above by $\e/2$. Recalling that $f\geq R_+$, this concludes the proof.
\end{proof}

\section{Decomposition theorem} \label{sec: decomposition}
The main goal of this section is to establish the Decomposition Theorem~\ref{thm: decomposition theorem} below.  The end purpose of this decomposition is to allow us to gain $W^{1,p}$-control on our initial manifold for large but finite $p<\infty$.  Thus, Theorem~\ref{thm: decomposition theorem} will be an essential tool used to prove Theorem~\ref{claim1}. 
The integral estimate for Ricci curvature established in Theorem~\ref{thm: integral bounds for Ricci curvature} is the key estimate in the proof.

Before stating the Decomposition Theorem precisely, let us give an informal description of its contents. Given a complete Ricci flow satisfying $-\delta$ lower bounds on the scalar curvature and the  entropy,  each ball $B_{g(1)}(x_0, 2)$ can be decomposed as a countable union of  ``good sets'' $\mathcal{G}^k$ and a ``bad set'' $\mathcal{A}$. The bad set has measure zero and on the $k$th good set, the metrics $g(0)$ and $g(1)$ are equivalent up to an error of size $(1+\e)^k$. Furthermore, the volumes of the $\mathcal{G}^k$ decay geometrically, and the complement of the first $k$ good sets satisfies a geometrically decaying content bound.

In fact, if we restrict the time to compare $g(0)$ and $g(t)$ for $t$ small, then we can obtain the same kind of decomposition with smaller error.

 \begin{theorem}[Decomposition Theorem]\label{thm: decomposition theorem}
 For each $\e>0$ there exists $\delta= \delta(n,\e)>0$, such that the following holds. 
 Let $(M,g(t))_{t \in (0,1]}$ be a complete Ricci flow with bounded curvature satisfying  
  \begin{align}
\label{eqn: scalar lower bound flow5} 	R_{g(0)} & \geq -\delta \\
 \label{eqn: entropy lower bound flow5}	\nu(g(0),2) & \geq -\delta.
 \end{align}
Fix $x_0 \in M$ and $\eta \leq \e$. There exists $\hat{t}= \hat{t}(n,\eta) \in (0,1]$ such that every ball $B_{g(1)}(x_0,2)$ can be decomposed into good sets $\mathcal{G}^k$ and a bad set $\mathcal{A}$ in the following way:
 \begin{equation}\label{eqn: decomposition}
 	B_{g(1)}(x_0,2) = \bigcup_{k=1}^\infty \mathcal{G}^{k} \cup \mathcal{A}
 \end{equation} 	
 where
 \begin{enumerate}
 	\item\label{item: decomp A vol}  $\vol_{g(0)}(\mathcal{A}) = 0.$
 	\item\label{item: decomp G bounds} For all $x \in \mathcal{G}^{k}$ and for all $s,t\in (0,\hat t]$, the metrics satisfy 
 	\begin{equation}
 		{(1-\eta)(1-\e)^{k-1} g(s) \leq g( t) \leq (1-\eta)(1+\e)^{k-1} g(s)}
	\end{equation}  
 	\item\label{item: decomp volume bounds G} For each $k\geq 2$, we have 
 {	$\vol_{g(0)}(\mathcal{G}^{k}) \leq (1+\e)^{k}\eta \e^{k-2}.$}
 	\item\label{item: decomp Ak bounds}  
 	For each $k\in \mathbb{N}$, let $\mathcal{A}^k = B_{g(1)}(x_0, 2) \, \Big\backslash \,  \bigcup_{\ell=1}^k \mathcal{G}^\ell$ be the complement of the first $k$ good sets. There is a countable collection $\mathcal{C}^k$ and a mapping $y\mapsto t_y$ for $y \in \mathcal{C}^k$ such that 
 	\begin{equation}\label{eqn: content bound 1}
 	 	\mathcal{A}^k \subseteq \bigcup_{y\in\mathcal{C}^k	} B_{g(t_y)}(y, 12t_y^{1/2}),
 	\end{equation}
with $\sum_{y\in\mathcal{C}^k} t_y^{n/2}\leq \eta \e^{k-1}.$
 \end{enumerate}
 When $\eta = \e$, then we may take $\hat{t} =1$.
 \end{theorem} 
\begin{remark}{\rm
The effect of $\eta>0$	in the above theorem is that for small $t$, one can force the bad sets comparing $g(0)$ to $g(t)$ to have decreasingly small volume.  In this way one gets that $g(t)$ is converging to $g(0)$ in various norms.
}
\end{remark}


Throughout this section we use the notation 
\begin{equation}\label{eqn: scale invariant balls}
	\uB_t (x):= B_{g(t)}(x,4t^{1/2})
\end{equation}
to denote the scale invariant balls of radius $4$ and we let
\begin{equation}\label{eqn: ub}
	\uB := \uB_1(x_0).
\end{equation}
In this notation, Theorem~\ref{thm: integral bounds for Ricci curvature} states that for any $\BETA \in (0,1/2)$ and $\e>0$, we may find $\delta=\delta(n,\BETA,\e)>0$ such that, under the hypotheses of Theorem~\ref{thm: decomposition theorem}, we have
 \begin{align}\label{eqn: new conclusion}
		\int_0^t \left(\frac{s}{t}\right)^{-\BETA} \fint_{\uB_t(x)} |\Ric_{g(s)}| \, d\vol_{g(s)} \, ds 
		& \leq   \e^2.
			\end{align}

\medskip

\subsection{Preliminary results}

 As in the previous sections we have by Theorem~\ref{prop: uniform existence time and small curvature estimates} that for any $\ETA>0$, we may choose $\delta$ sufficiently small in Theorem~\ref{thm: decomposition theorem} so that
 \begin{equation}\label{eqn: tiny curvature bounds5}
 	|\Rm_{g(t)}| \leq \ETA/t\, ,
 \end{equation}
 holds for all $x \in M$ and $t \in (0,1]$. The norm of the Ricci curvature $|\Ric_{g(t)}|$ evolves along the Ricci flow by 
	\begin{equation}\label{eqn: evolution of Ricci}
		(\pa_t -\Delta)|\Ric_{g(t)}| \leq c_n |\Rm_{g(t)}| |\Ric_{g(t)}|\, ;
	\end{equation}
see \cite[Lemma 6.38]{ChowKnopfBOOK}. For a Ricci flow satisfying \eqref{eqn: tiny curvature bounds5}, the evolution \eqref{eqn: evolution of Ricci} becomes 
\begin{equation}
	\Big(\pa_t -\Delta -\frac{c_n\ETA}{t}\Big)|\Ric_{g(t)}| \leq 0\, .
\end{equation}
That is to say, when $t$ is uniformly bounded away from zero, the norm of the Ricci curvature evolves as a sub-solution of a heat-type equation with smooth bounded potential. Note that \eqref{eqn: tiny curvature bounds5} provides uniform lower bounds for the Ricci tensor when $t$ is bounded away from zero, a necessary ingredient for establishing parabolic regularity estimates. 
So, after rescaling the metric, a standard Moser iteration argument (along with a trick of Li and Schoen
\cite{LiSchoen} to pass from the $L^2$ norm to the $L^1$ norm) leads to the following pointwise estimates for the norm of the Ricci curvature; see \cite[Theorem 25.2]{Chow3} for a proof. 
\begin{proposition}\label{prop: pointwise estimates} Fix $n\geq 2$. There exist  constants $C=C(n)$ and $\lambda_0(n)$ such that if $(M,g(t))_{t \in (0,1]}$ is a Ricci flow satisfying \eqref{eqn: tiny curvature bounds5} with $\ETA \leq \ETA_0(n)$, then for any $t \in (0,1]$ we have
	\begin{align}\label{eqn: pointwise estimates}
	|\Ric_{g(s)}(y)| \leq C\fint_{t/4}^t \fint_{\uB_t(x)} |\Ric_{g(s)}| \, d\vol_s(y) \, ds 
	\end{align}
for all $(y,s) \in B_{g(t)}\big(x,3t^{1/2}\big) \times (t/2, t)$.
\end{proposition}

In the proof of Theorem~\ref{thm: decomposition theorem}, we will need the following Vitali-type lemma. The difference from a usual Vitali cover is that the balls are not taken with respect to a fixed metric, but rather the covering comprises geodesic balls with respect different time slices $g(t)$ along a Ricci flow. At various points in the proof, we will call upon the elementary containments of balls established in Lemma~\ref{lem: bad bound}.
\begin{lemma}[Vitali-type lemma]\label{lem: vitali}
	Given $n\geq 2$, there exists $\ETA_0(n)$ such that the following holds. Let $(M,g(t))_{t \in (0,1]}$ be a Ricci flow satisfying \eqref{eqn: tiny curvature bounds5} with $\ETA \leq \ETA_0(n)$. For any $x_0 \in M$ and $t_0 \in (0,1]$, consider a set $\mathcal{A} \subseteq B_{g(t_0)}(x_0, 2t_0^{1/2})$ and a mapping $y\mapsto t_y\in(0,t_0/200]$ defined for all $y \in \mathcal{A}$. There exists a countable collection $\mathcal{C}\subseteq \mathcal{A}$ such that
	\begin{enumerate}
		\item\label{item: Vitali disjoint} The balls $\uB_{t_y}(y)$ 
	are pairwise disjoint for all $y \in \mathcal{C}$.
	\item\label{item: Vitali cover}  The collection $\{  B_{g(36 t_y)}(x, 12t_y^{1/2})\}_{y \in \mathcal{C}}$ is a covering of $\mathcal{A}$. 
	\item\label{item: Vitali containment} For each $y \in \mathcal{C}$,  $\uB_{36t_y}(y)\subseteq \uB_{t_0}(x_0)$ and $\uB_{t_y}(y)\subseteq \uB_{t_0}(x_0)$.
	\end{enumerate} 
	\end{lemma}
\begin{definition}\label{def: covering pair}
	We call a pair $(\mathcal{C}, y\mapsto t_y)$ satisfying \eqref{item: Vitali disjoint}-\eqref{item: Vitali containment} a {\it covering pair} of $\mathcal{A}$ in $\uB_{t_0}(x_0)$.
\end{definition}
	\begin{proof}[Proof of Lemma~\ref{lem: vitali}] Up to rescaling the flow, we may assume that $t_0 = 1.$
	The inductive construction of the cover is similar to a standard Vitali covering argument. For each $k\in \mathbb{N}$, let 
	\begin{equation}
		F_k = \left\{ \uB_{t_y}(y)
		: y \in \mathcal{A}, t_y \in (2^{-k-1}, 2^{-k}]\right\}.
	\end{equation}	
	Set $H_0=F_0$ and let $G_0$ be a maximal disjoint subcollection of $H_0.$ Now, suppose we have defined $G_0, \dots G_{k-1}$. Then let 
	\begin{equation}
		H_{k} =\left\{ B\in F_{k} : B\cap B' = \emptyset ,\ \forall B' \in G_0\cup \dots \cup G_{k-1}\right\},
	\end{equation}
	and take $G_{k}$ to be a maximal disjoint subcollection of $H_{k}$. Note that $G_k$ contains finitely many balls. We define the countable set $\mathcal{C} \subseteq \mathcal{A}$ by
		\begin{equation}
				\mathcal{C} = \bigcup_{k=1}^\infty \left\{ y \in \mathcal{A} : \uB_{t_y}(y)\in G_k\right\}.
	\end{equation}
	
	Let us verify that the three properties claimed in the lemma are valid.  Lemma~\ref{lem: vitali}\eqref{item: Vitali disjoint} holds by construction, and both parts of Lemma~\ref{lem: vitali}\eqref{item: Vitali containment}  follow directly from \eqref{eqn: basic containment} in 
	Lemma~\ref{lem: bad bound} thanks to the assumption that $t_y \leq 1/200$.

	To establish Lemma~\ref{lem: vitali}\eqref{item: Vitali cover}, fix any $x \in \mathcal{A}$ and let $k\in \mathbb{N}$ be chosen so that $t_x \in (2^{-k-1}, 2^{-k}]$. Then either $ \uB_{t_x}(x) \in H_k$ or not. In the first case, since $G_k$ is a maximal set, we know that $\uB_{t_x}(x)$ intersects some $\uB_{t_y}(y)\in G_k$ (where possibly $x=y$). In this case $t_y/2 \leq t_x \leq 2t_y$. So, if we take $z \in  \uB_{t_x}(x)\cap \uB_{t_y}(y)$, the triangle inequality and  \eqref{eqn: bad bound general} imply that 
\begin{equation}
\begin{split}
		d_{g(t_y)}(x, y)& \leq d_{g(t_y)}(y, z) +d_{g(t_y)}(x, z) \\
		&\leq 4t_y^{1/2} + 4^{\lambda+1} t_x^{1/2}\\
		&  \leq 10 t_y^{1/2}.
\end{split}\end{equation}
	The final inequality holds provided we have taken $\ETA$ sufficiently small. 
	In the second case, when $\uB_{t_x}(x) \not \in H_k$, we see that $\uB_{t_x}(x)$ must intersect some $\uB_{t_y}(y) \in H_\ell$ with  $\ell\in\{1,\dots, k-1\}$. Then, since $t_x \leq t_y,$ we find by \eqref{eqn: bad containment of balls} that $\uB_{t_x}(x)\subseteq B_{g(t_y)}(x,5t_y^{1/2})$. In particular, $B_{g(t_y)}(x,5t_y^{1/2})$ and $\uB_{t_y}(y)$ intersect nontrivially, and thus by the triangle inequality $x \in B_{g(t_y)}(y, 10t_y^{1/2})$.

	So, in both the first and second cases, we have $x \in B_{g(t_y)}(y, 10t_y^{1/2})$ for some $y \in \mathcal{C}$. In order to complete the proof of (2), 
we apply \eqref{eqn: bad bound general} to find that
	\begin{equation}\begin{split}
			B_{g(t_y)}(y, 10t_y^{1/2} )
		 &\subseteq B_{g(36t_y)}(y, 2(36t_y)^{1/2} ),
	\end{split}
	\end{equation}
where the final containment holds provided we choose $\ETA$ small enough.
This completes the proof of the lemma.
			\end{proof}
\medskip

\subsection{Good and bad sets on an arbitrary ball}\label{subsec: good bad}
Throughout this section, we fix $\e\in(0,1)$ and $\BETA \in (0,1/2)$ we assume that $(M,g(t))_{t\in[0,1]}$ is a Ricci flow satisfying \eqref{eqn: scalar lower bound flow5} and \eqref{eqn: entropy lower bound flow5} with $\delta$ chosen according to Theorem~\ref{thm: integral bounds for Ricci curvature}.

For a ball $\uB$,  we define the  {\it stopping time} $t(x)$ for each $x \in  B_{g(1)}(x_0, 2)$ by
\begin{equation}\label{eqn: stopping time definition}
\begin{split}
t(x) = \inf \Big\{ t'\leq 1/200 \, :\, \fint_{t/4}^t \fint_{{\uB}_t(x)}& |\Ric_{g(s)}| \, d\vol_{g(s)} \, ds 
	 \leq t^{\BETA/2-1}\e \quad \forall t \in [t' , 1/200]\Big\}.	
\end{split}
\end{equation}
Observe that, provided we take $\e < 200^{-2}$, applying \eqref{eqn: new conclusion} with $t=200^{-1}$ ensures that
\begin{align}\label{eqn: condition at 200-2}
\fint_{1/800}^{1/200} &\fint_{{\uB}_{1/200}(x)}|\Ric_{{g}(t)}| \, d\vol_{g(t)}\,dt  \leq 200^{\BETA/2-1} \e,
\end{align}
so the stopping condition holds at $t'=1/200$. 

For a  ball $\uB_{t_0}(x_0)$ with $t_0<1$, we define the  stopping time $t(x)$  by
\begin{equation}
	t(x) = t_0\, \tilde{t}(x)
\end{equation}
for each $x \in  B_{g(t_0)}(x_0, 2t_0^{1/2})$, where $\tilde{t}(x)$ is the stopping time defined in \eqref{eqn: stopping time definition} applied to the rescaled flow $	\tilde{g}(t) = t_0^{-1} g(t_0\,t).$

The good and bad sets on $\uB_{t_0}(x_0)$, respectively, are defined by
\begin{equation}
\label{eqn: good and bad set def}
\begin{split}
\mathcal{G}(\uB_{t_0}(x_0)) & = \left\{ x \in B_{g(t_0)}(x_0, 2t_0^{1/2}) \ :\ t(x) = 0\right\},\\
\mathcal{A}(\uB_{t_0}(x_0)) &  = \left 	\{ x \in B_{g(t_0)}(x_0, 2t_0^{1/2})\ :\ t(x) > 0\right\}.	
\end{split}	
\end{equation}

In the following proposition, we establish estimates on the good and bad sets that will be iteratively applied to establish Theorem~\ref{thm: decomposition theorem}. When convenient, we adopt the shorthand $t_x = t(x)$.
\begin{proposition}\label{prop: good and bad} Fix $n\geq 2$, $\e\in (0,1)$, and $\theta\in(0,1/2)$. There exists $\delta =\delta(n,\e,\theta)>0$ such that the following holds. Let $(M,g(t))_{t\in [0,1]}$ be a Ricci flow satisfying \eqref{eqn: scalar lower bound flow5} and \eqref{eqn: entropy lower bound flow5}. Fix  $\uB_{t_0}(x_0).$
\begin{enumerate}
	\item\label{item: goodbad good} For any $x \in \mathcal{G}(\uB_{t_0}(x_0))$
	and for any $s,s' \in [0,t_0]$, the metrics $g(s)$ and $g(s')$ at $x$ satisfy
	\begin{equation}\label{eqn: 0 to 1}
		(1-\e) g(s) \leq g(s') \leq (1+\e) g(s).
	\end{equation}
\item\label{item: goodbad bad}
 Suppose $x \in \mathcal{A}(\uB_{t_0}(x_0))$ and fix $s_0\in [ t_x, t_0]$. 
Then  for all $s, s' \in [s_0, t_0]$ and $y \in B_{g(s_0)}(x,2s_0^{1/2})$, the metrics $g(s)$ and $g(s')$ at $y$ satisfy 
	\begin{equation}\label{eqn: bound at stopping time}
		(1-\e)g(s) \leq g(s') \leq (1+\e)g(s).
	\end{equation}

\item\label{item: goodbad content}  There is a countable collection $\mathcal{C} =\mathcal{C} (\uB_{t_0}(x_0))\subseteq \mathcal{A}(\uB_{t_0}(x_0))$ such that  $(\mathcal{C} , t_y)$ is a covering pair for $\mathcal{A}(\uB_{t_0}(x_0))$ in the sense of Definition~\ref{def: covering pair}, where $t_y=t(y)$ is the stopping time, and 
	\begin{equation}\label{eqn: content bound}
		\sum_{y \in \mathcal{C}} {t}_y^{(n-\BETA)/2} \leq  \e \,t_0^{(n-\BETA)/2} \,.
	\end{equation}
\end{enumerate}
	\end{proposition}
\begin{proof} 
Up to rescaling the flow, we may assume without loss of generality that $t_0=1.$ We choose $\delta$ sufficiently small that \eqref{eqn: new conclusion} holds.

Observe that \eqref{item: goodbad good} follows immediately from \eqref{item: goodbad bad}, since the estimate \eqref{eqn: 0 to 1} is a particular case of \eqref{eqn: bound at stopping time} with $s_0=0$. Let us prove \eqref{item: goodbad bad}. 
	 Thanks to \eqref{eqn: bad bound general}, it suffices to establish \eqref{eqn: bound at stopping time} for $s,s' \in [s_0, 1/200]$.  
Together the pointwise estimates of Proposition~\ref{prop: pointwise estimates} and the definition of the stopping time imply that 
\begin{align}
	|\Ric_{g(t)}|& \leq C\e t^{\BETA/2-1}\quad \text{ on }  B_{g(t)}(x, 3t^{1/2}) \times \{t\}\\
\shortintertext{
for all $t\in[s_0,1/200]$. So, calling upon \eqref{eqn: bad containment of balls}, we find that }
\label{eqn: bound intermediate}
	|\Ric_{g(t)}| &\leq C\e t^{\BETA/2-1}\quad \text{ on }  B_{g(s_0)}(x, 2s_0^{1/2}) \times [s_0,1/200].
\end{align}
Now, for any $y \in  B_{g(s_0)}(x, s_0^{1/2})$, we integrate \eqref{eqn: bound intermediate} from $s$ to $s'$ precisely as in the proof of Lemma~\ref{lem: bad bound} to find
\begin{equation}\label{eqn: X}
	(1-C\e)g(s) \leq g(s') \leq (1+C\e)g(s')
\end{equation}
where $C=C(n)$. Up to further decreasing $\delta$ so that \eqref{eqn: X} holds with $\e'=\e/C$ in place of $\e$, this establishes \eqref{eqn: bound at stopping time} and hence \eqref{item: goodbad good} and \eqref{item: goodbad bad}.

We now prove \eqref{item: goodbad content}. We apply the Vitali-type Lemma~\ref{lem: vitali}, taking $\mathcal{A}=\mathcal{A}({\uB})$ with the mapping $y\mapsto t_y$ given by the stopping time $t_y = t(y).$ Lemma~\ref{lem: vitali} ensures the existence of a covering pair $(\mathcal{C},t_y)$ for $\mathcal{A}(\uB)$ in $\uB$ with $\mathcal{C}\subseteq \mathcal{A}({\uB})$. 
 We prove the content bound \eqref{eqn: content bound} in the following way. For any $y\in \mathcal{C}$, the definition of the stopping time $t_y$ guarantees that
	\begin{equation}\label{a}
 \int_{t_y/4}^{t_y} \fint_{{\uB}_{t_y}(y)}  |\Ric_{{g}(s)} | \, d\vol_{{g}(s)} \, ds= \e t_y^{\BETA/2}.
	\end{equation}
By rearranging terms in \eqref{a} and calling upon the volume lower bound in \eqref{eqn: under reg scale euclidean volumes}, we find that
\begin{equation}\label{eqn: bound from stopping time}
\begin{split}
	t_y^{(n-\BETA)/2 } &\leq \frac{Ct_y^{-\BETA}}{\e}\int_{t_y/4}^{t_y} \int_{{\uB}_{t_y}(y)}  |\Ric_{{g}(s)} | \, d\vol_{{g}(s)} \, ds \\
	& \leq \frac{C}{\e}
	\int_{t_y/4}^{t_y}s^{-\BETA}
	 \int_{{\uB}_{t_y}(y)} |\Ric_{{g}(s)} | \, d\vol_{{g}(s)} \, ds.
\end{split}	
\end{equation}
Now, by Lemma~\ref{lem: vitali}\eqref{item: Vitali disjoint}  and \eqref{item: Vitali containment} respectively, the balls $\{{\uB}_{t_y}(y)\}_{y \in \mathcal{C}}$ are pairwise disjoint and contained in $\uB$. So, we sum \eqref{eqn: bound from stopping time}  over all $y \in \mathcal{C}$ and apply \eqref{eqn: new conclusion} to discover that
\begin{align}\label{eqn: Y}
\sum_{y \in \mathcal{C}} t_y^{(n -\BETA)/2} 
  & \le \frac{C}{\e} \int_{0}^{1} s^{-\BETA}\int_{\uB}  |\Ric_{g(s)} | \, d\vol_{g(s)} \, ds \leq C \e.
\end{align}
Again, up to further decreasing $\delta$ so that \eqref{eqn: Y} holds for $\e'=\e/C$, this concludes the proof of \eqref{item: goodbad content} and thus the proposition.
\end{proof}

\medskip

\subsection{The $k$th good and bad sets and the proof of Theorem~\ref{thm: decomposition theorem} with $\eta =\e$}
In this section, we apply Proposition~\ref{prop: good and bad} inductively in order to define $k$th good and bad sets and establish Theorem~\ref{thm: decomposition theorem} in the case when $\eta =\e$, and thus $\hat{t}=1$. Separating the proofs when $\eta <\e$ and $\eta =\e$ is convenient as it allows us to apply Corollary~\ref{corollary: volume bounds} below to establish estimates needed for the case $\eta <\e$. 
\begin{proof}[Proof of Theorem~\ref{thm: decomposition theorem} when $\eta =\e$]
Let $\delta$ be chosen according to Proposition~\ref{prop: good and bad}.
\\

{\it Step 1:} We inductively define sets
$\mathcal{G}^{k}\subseteq \tilde{\mathcal{A}}^{k-1}$, $\tilde{\mathcal{A}}^{k} \subseteq \tilde{\mathcal{A}}^{k-1}$ and $\mathcal{C}^{k}\subseteq \tilde{\mathcal{A}}^{k}$ for each $k \in \mathbb{N}$ satisfying the following properties:
\begin{enumerate}
	\item\label{item: ind claim1}  For all $x\in \mathcal{G}^{k}$ and $s,s' \in [0,1]$, we have 
\begin{align}\label{eqn: induction good bound}
(1-\e)^{k} g(s')\leq g(s) \leq (1+\e)^{k}g(s').
\end{align}
\item\label{item: ind claim2} For each $x \in \mathcal{C}^k$ we have a mapping $y\mapsto  t_y \in (0,200^{-k})$ such that $(\mathcal{C}^k, t_y)$ is a covering pair for $\tilde{\mathcal{A}}^k$ in $\uB$ as in Definition~\ref{def: covering pair} and such that
 \begin{equation}\label{eqn: induction content bound}
	\sum_{y \in \mathcal{C}^k} t_y^{n/2} \leq \e^k.
\end{equation} 
\item\label{item: ind claim3}
Furthermore, if $y \in \mathcal{C}^k,$ then 
\begin{equation}\label{eqn: bbb}
	(1-\e)^k g(s') \leq g(s) \leq (1+\e)^k g(s')
\end{equation} 
for all $x \in B_{g(t_y)} (y, 2t_y^{1/2})$ and for all $s, s' \in [t_y, 1]$
\end{enumerate}

In the claim above and in its proof, we suppress in the notation the dependence of $t_y$ on $k$ for $y\in\mathcal{C}^k.$ Let $\tilde{\mathcal{A}}^0=\uB$, and for $k=1,$ we set
\begin{equation}
\begin{split}
\mathcal{G}^{1} & = \mathcal{G}(\uB),\\
\tilde{\mathcal{A}}^{1} &  =\mathcal{A}(\uB),	
\end{split}	
\end{equation}
as defined in \eqref{eqn: good and bad set def}. 
Let $(\mathcal{C}^1,t_y)$ be the covering pair provided by Proposition~\ref{prop: good and bad}. Then 
properties \eqref{item: ind claim1}--\eqref{item: ind claim3} above for $k=1$ follow directly from Proposition~\ref{prop: good and bad}.

Now, suppose that we  have defined the sets $\mathcal{G}^{k}\subseteq \tilde{\mathcal{A}}^{k-1}, \tilde{\mathcal{A}}^{k}\subseteq \tilde{\mathcal{A}}^{k-1},$ and $\mathcal{C}^{k}\subseteq \tilde{\mathcal{A}}^{k}$ satisfying properties \eqref{item: ind claim1}--\eqref{item: ind claim3}.
We define  $\mathcal{G}^{k+1}$ by
\begin{align}
\mathcal{G}^{k+1}& = \tilde{\mathcal{A}}^k \cap \bigcup_{y \in \mathcal{C}^k} \mathcal{G}(\uB_{t_y}(y)).
\end{align}
If $x \in \mathcal{G}^{k+1},$ then $x\in \mathcal{G}(\uB_{t_y}(y))$ for some $y \in \mathcal{C}^k$. So, Proposition~\ref{prop: good and bad}\eqref{item: goodbad good} applied to $\mathcal{G}(\uB_{t_y}(y))$ implies that for all $s,s'\in[0,t_y]$, we have
\begin{equation}\label{eqn: aaa}
	(1-\e) g(s') \leq g(s) \leq (1+\e)g(s')
\end{equation} 
at $x$.
The inductive hypothesis ensures that  
\eqref{eqn: bbb} holds at $x$. 
Together, \eqref{eqn: aaa} and \eqref{eqn: bbb} imply that 
\begin{align}
(1-\e)^{k} g(s')\leq g(s) \leq (1+\e)^{k}g(s')
\end{align}
for all $s,s'\in [0,1].$
Therefore, \eqref{eqn: induction good bound} holds with $k+1$ replacing $k$ for each $x \in \mathcal{G}^{k+1}$. 

Similarly, define
\begin{align}
\tilde{\mathcal{A}}^{k+1} &=\tilde{\mathcal{A}}^k\cap \bigcup_{y \in \mathcal{C}^k} \mathcal{A}({\uB}_{t_y}(y)).
\end{align}
For each $y\in \mathcal{C}^k$, Proposition~\ref{prop: good and bad}\eqref{item: goodbad content} ensures the existence of a covering pair $(\mathcal{C}_y, z\mapsto t_z)$ for $\mathcal{A}({\uB}_{t_y}(y))$ in $\uB_{t_y}(y)$ with $t_z \in (0,t_y/200) \subseteq  (0, 200^{-(k+1)})$. We set 
\begin{equation}
\mathcal{C}^{k+1} = \bigcup_{y \in \mathcal{C}^k} \mathcal{C}_y\,.
\end{equation}
Together Proposition~\ref{prop: good and bad}\eqref{item: goodbad content} and the inductive hypothesis \eqref{item: ind claim2} ensure that $(\mathcal{C}^{k+1}, z\mapsto t_z)$ is a covering pair for $\tilde{\mathcal{A}}^{k+1}$ in $\uB$. Furthermore, Proposition~\ref{prop: good and bad}\eqref{item: goodbad content}  and \eqref{eqn: induction content bound} show that
 \begin{equation}
	\sum_{z \in \mathcal{C}^{k+1}} t_z^{n/2} \leq \e^{k+1},
\end{equation}
so \eqref{eqn: induction content bound} holds for $\mathcal{C}^{k+1}$ with $k+1$ replacing $k$. Thus property \eqref{item: ind claim2} holds for $k+1$. Finally, Proposition~\ref{prop: good and bad}\eqref{item: goodbad bad} along with \eqref{eqn: bbb} ensures that
\begin{equation}
	(1-\e)^{k+1} g(s') \leq g(s) \leq (1+\e)^{k+1} g(s')
\end{equation} 
for all $x \in B_{g(t_z)} (z, 2t_z^{1/2})$ and for all $s,s' \in [t_z, 1]$. Thus property \eqref{item: ind claim3} holds for $k+1.$
This concludes the proof of the claim.\\

{\it Step 2:} We now finish the proof of Theorem~\ref{thm: decomposition theorem}. Theorem~\ref{thm: decomposition theorem}\eqref{item: decomp G bounds} follows directly from \eqref{eqn: induction good bound} and the definition of the sets $\mathcal{G}^k$. Noting that 
\begin{equation}
	\mathcal{A}^k:=	B_{g(1)}(x_0, 2)\setminus \bigcup_{\ell=1}^k \mathcal{G}^\ell \subset \tilde{\mathcal{A}}^k,
\end{equation}
Theorem~\ref{thm: decomposition theorem}\eqref{item: decomp Ak bounds} follows from property \eqref{item: ind claim2} above.
Next, since $\mathcal{G}^k \subseteq \tilde{\mathcal{A}}^{k-1}$ by construction, we have
\begin{equation}\label{eqn: vol bound Gk a}
\begin{split}
	\vol_{g(1)}(\mathcal{G}^k) & \leq \vol_{g(1)}(\tilde{\mathcal{A}}^{k-1}) 	\leq \sum_{y \in \mathcal{C}^{k-1} }\vol_{g(1)} (B_{g(36t_y)}(y, 2(36t_y)^{1/2}).	
\end{split}	
\end{equation}
We apply \eqref{eqn: volume forms} followed by \eqref{eqn: under reg scale euclidean volumes} to find that 
\begin{equation}
	\label{eqn: vol bound Gk b}
	\begin{split}
		\vol_{g(1)} (B_{g(36t_y)}&(x, 2(36t_y)^{1/2})  \leq (1+2\delta) \vol_{g(36t_y)} (B_{g(36t_y)}(x, 2(36t_y)^{1/2}) \leq Ct_y^{n/2},	
	\end{split}
\end{equation}
where $C$ is a dimensional constant. Therefore, by \eqref{eqn: vol bound Gk a}, \eqref{eqn: vol bound Gk b}, and \eqref{eqn: induction content bound}, we see that 
\begin{equation}
	\vol_{g(1)}(\mathcal{G}^k)  \leq C\sum_{y \in \mathcal{C}^{k-1}} t_y^{1/2}  \leq C\e^{k-1}.
	\end{equation}
Pairing this estimate with Theorem~\ref{thm: decomposition theorem}\eqref{item: decomp G bounds}, this establishes Theorem~\ref{thm: decomposition theorem}\eqref{item: decomp volume bounds G}. Finally, observe that the same argument shows that 
$
	\vol_{g(200^{-k})}(\tilde{\mathcal{A}}^k) \leq C\e^k.
	$
So, defining
$	\mathcal{A} = \cap_{k=1}^\infty \tilde{\mathcal{A}}^k,
$
we see that 
\begin{equation}
	\vol_{g(0)}(\mathcal{A}) = \lim_{k \to \infty} \vol_{g(200^{-k})}(\mathcal{A}) \leq \lim_{k \to \infty} \vol_{g(200^{-k})}(\tilde{\mathcal{A}}^k)=0.
\end{equation}
This proves Theorem~\ref{thm: decomposition theorem}\eqref{item: decomp A vol} and thus concludes the proof of the theorem.
\end{proof}

An immediate consequence of Theorem~\ref{thm: decomposition theorem} with $\hat{t}=1$ is the following comparison of volumes of balls.
\begin{corollary}\label{corollary: volume bounds}
	Fix $n\geq 2$ and $\e>0$. 
 There exists $\delta= \delta(n,\e)>0$, such that the following holds. 
 Let $(M,g(t))_{t \in (0,1]}$ be a Ricci flow satisfying  \eqref{eqn: scalar lower bound flow5} and \eqref{eqn: entropy lower bound flow5}.
Fix $x_0 \in M$, $t_0\in(0,1]$. For all $s,t \in (0,t_0]$, we have 
\begin{equation}
	(1-\e) \leq \frac{\vol_{g(t)}(B_{g(t_0)}(x_0,{2t_0^{1/2}}))}{\vol_{g(s)}(B_{g(t_0)}(x_0,{2t_0^{1/2}}))} \leq (1+\e)
\end{equation}
\end{corollary}
\begin{proof}
	Up to parabolic rescaling, we may assume that $t_0=1$. Thanks to \eqref{eqn: volume forms}, it suffices to show that $\vol_{g(0)}(B_{g(1)}(x_0, 2)) \leq (1+\e) \vol_{g(1)}(B_{g(1)}(x_0, 2))$ provided $\delta$ is taken to be sufficiently small depending on $\e$. Let $\delta$ also be taken sufficiently small according to Theorem~\ref{thm: decomposition theorem}.
	Applying Theorem~\ref{thm: decomposition theorem} with $\eta =\e$ and thus $\hat{t}=1,$ we find that 
	\begin{equation}\begin{split}
	\vol_{g(0)}(B_{g(1)}(x_0, 2)) = \sum_{k=1}^\infty \vol_{g(0)}(\mathcal{G}^k) 
	& \leq  \vol_{g(0)}(\mathcal{G}^k) +  \sum_{k=2}^\infty (1+\e)^k \e^{k-1}\\
	& \leq (1+\e)\vol_{g(1)}(x_0, 2) + C\e\\
	& \leq (1+C\e) \vol_{g(1)}(x_0, 2).
	\end{split}\end{equation}
	The final inequality follows from Lemma~\ref{rmk: good charts under regularity scale}. 
	By further decreasing $\delta$, we may replace $\e$ by $\e/C$ to complete the proof.
\end{proof}

\subsection{Improved Ricci estimate and good and bad sets on the initial ball}
When $\eta<\e$, we will need to apply a refined form of Proposition~\ref{prop: good and bad}. To this end, we first show that an improved integral estimate for the Ricci curvature holds outside a set of small content. As usual, we use the notation $\uB_t(x) = B_{g(t)}(x, 4 t^{1/2})$ as defined in \eqref{eqn: scale invariant balls}.
\begin{lemma}[Improved Ricci integral estimate]\label{lem: improved ricci integral} 
	Fix $n\geq 2$, $\e>0$ and $\BETA \in (0,1/2)$. 
 There exists $\delta= \delta(n,\e, \BETA)>0$ such that the following holds. 
 Let $(M,g(t))_{t\in[0,1]}$ be a Ricci flow satisfying  \eqref{eqn: scalar lower bound flow5} and \eqref{eqn: entropy lower bound flow5}.
 Then for any $x_0 \in M$ and $\eta \leq \e$, there exist $\hat{t} = \hat{t}(n, \e, \theta, \eta)\in (0,1]$ and  an exceptional set $E \subset B_{g(1)}(x_0,2)$  such that the following holds.
\begin{enumerate}
	\item For all $x\in B_{g(1)}(x_0,2)\setminus E$, we have the improved  integral Ricci curvature estimate
	\begin{equation}\label{eqn: eta integral ricci a}
	\int_0^{\hat{t}} \left(\frac{s}{\hat{t}}\right)^{-\BETA/2} \fint_{\uB_{\hat{t}}(x)} |\Ric_{g(s)}(y)| \,d\vol_{g(s)}(y)\,ds \leq  \eta^2 .
\end{equation}
\item There is a finite set $\{y_\ell\}_{\ell =1}^N \subset B_{g(1)}(x_0,2)$ such that
\begin{equation}\label{eqn: content bound EetaR}
E\subset \bigcup_{\ell=1}^N B_{g(\hat t)}(y_\ell ,2\hat t^{1/2}), \qquad N\leq \eta \hat t^{-n/2}.
\end{equation}

\end{enumerate}	
\end{lemma}

\begin{proof} Fix any $\e \in (0,1/2)$.
Let $\delta =\delta(n,\e, \BETA)$ be chosen sufficiently small so that we may apply Theorem~\ref{thm: integral bounds for Ricci curvature} with $t=\e=1$ and Lemma~\ref{lem: bad bound} and Lemma~\ref{rmk: good charts under regularity scale} with $\lambda$ to be determined in the course of the proof.
Throughout the proof, we use the shorthand  $B= B_{g(1)}(x_0,2)$.
\\

{\it Step 1:}
We first claim that 
	\begin{equation}\label{eqn: integral estimate improved a}
		\int_{B} \int_0^{t} \left(\frac{s}{t}\right)^{-\BETA/2} \fint_{\uB_t(x)} |\Ric_{g(s)}(y)| \,d\vol_{g(s)}(y)\,ds\,d\vol_{g(0)}(x) \leq 2 t^{\theta}
	\end{equation}
for any $t\leq \hat{t}$, provided $\hat{t}$ is taken small enough. 
\\

{\it Step 1a:} Recall from Corollary~\ref{corollary: volume bounds} and Lemma~\ref{rmk: good charts under regularity scale} that
\begin{equation}
\vol_{g(s)}(B_{g(1)}(x_0,4)) \le (1+\e)\vol_{g(1)}(B_{g(1)}(x_0,4)) \leq (1+\e)4^n \om_n.
\end{equation}
This together with Theorem~\ref{thm: integral bounds for Ricci curvature} (taking $t=\e=1$) implies that for any $t\leq 1,$ 
 we have
\begin{equation}\label{eqn: basic improved}
\begin{split}
	& \int_0^{t} s^{-\BETA/2} \int_{\uB_1(x_0)} |\Ric_{g(s)}(y)|\,d\vol_{g(s)}(y)\\
& \leq t^{\BETA/2} \int_0^{t} s^{-\BETA}\int_{\uB_1(x_0)} |\Ric_{g(s)}(y)|\,d\vol_{g(s)}(y)\leq (1+\e)4^n\om_n t^{\BETA/2}.
\end{split}	
\end{equation}

{\it Step 1b:} A basic maximal function argument shows that  
 \begin{equation}\label{eqn: average bound}
 \begin{split}
 	\int_B\fint_{\uB_t(x)} |\Ric_{g(s)}|& \,d\vol_{g(s)}(y)\,d\vol_{g(0)}(x) \leq  (1+2\e) \int_{B_{g(1)}(x,4)}  |\Ric_{g(s)}(y)|\,d\vol_{g(t)}(y).
 \end{split}
 \end{equation}
To see this,  take $\hat{t}$ small enough such that $4\hat{t}^{1/2}\leq 1$. By the proof of Corollary~\ref{corollary: volume bounds} and Lemma~\ref{rmk: good charts under regularity scale}, we have
\begin{equation}
(1-\e)\om_n (4t)^{n/2} \leq \vol_{g(s)}(\uB_t(x))\le (1+\e)\om_n (4t)^{n/2}\end{equation}
 for any $s\in[0,t]$.  
Furthermore, by Lemma~\ref{lem: bad bound}, for any $x \in B$, we have  $\uB_t(x) \subset \uB_1(x_0)$. With these observations in hand, we see that for fixed $s\in(0,t)$,
\begin{equation}\begin{split}
	\int_B\fint_{\uB_t(x)} &|\Ric_{g(s)}| \,d\vol_{g(s)}(y)\,d\vol_{g(0)}(x)\\
	& \leq \frac{1+\e}{\om_n t^{n/2}} \int_B \int_{\uB_1(x_0)}\chi_{\uB_t(x)}(y) |\Ric_{g(s)}(y)| \,d\vol_{g(s)}(y) \,d\vol_{g(0)}(x)\\
	& \leq\frac{1+\e}{\om_n t^{n/2}} \int_{B_{g(1)}(x,2)} \left( \int_B \chi_{\uB_t(y)}(x)\,d\vol_{g(0)}(x)\right) |\Ric_{g(s)}|(y) \,d\vol_{g(s)}(y)\\
	& \leq (1+2\e) \int_{B_{g(1)}(x,2)}  |\Ric_{g(s)}(y)|\,d\vol_{g(t)}(y).
\end{split}\end{equation}
Here $\chi_{\uB_t(x)}$ is the indicator function of $\uB_t(x)$. This establishes \eqref{eqn: average bound}.

{\it Step 1c:}
Multiplying \eqref{eqn: average bound} by $s^{-\theta/2}$ and integrating with respect to $s$ and then combining the subsequent estimate with \eqref{eqn: basic improved}, we establish that 
	\begin{equation}\label{eqn: eta integral ricci proof}
	\int_0^{t} s^{-\BETA/2} \fint_{ \uB(t_\eta)} |\Ric_{g(s)}(y)| \,d\vol_{g(s)}(y)\,ds \leq  t^{\BETA/2}
\end{equation}
 holds for $t\leq \hat{t}.$ Multiplying both sides of \eqref{eqn: eta integral ricci proof} by $t^{\theta/2}$, we obtain \eqref{eqn: integral estimate improved a}.
\medskip

{\it Step 2:}
 Now, let 
\begin{equation}
	E = \left\{ x\in B : \int_0^{\hat t} \left(\frac{s}{\hat t}\right)^{-\BETA/2} \fint_{\uB_t(x)} |\Ric_{g(s)}(y)| \, d\vol_{g(s)}(y)\,ds \geq \hat t^{\theta/2}\right\}.
\end{equation}
By definition, \eqref{eqn: eta integral ricci a} holds on $B\setminus E$, so it remains to show the content bound \eqref{eqn: content bound EetaR}. Let 
\begin{equation}
	B_{g(\hat t)}(E, \hat t^{1/2}) = \bigcup_{x\in E} B_{g(\hat t)}(x,\hat  t^{1/2}).
\end{equation}
We see that every $x \in B_{g(\hat t)}(E, t^{1/2}) $ satisfies
\begin{equation}\label{eqn: improved integral bound b}
	\int_0^{\hat t} \left(\frac{s}{\hat t}\right)^{-\BETA/2} \fint_{B_{g(\hat t)}(x,8 \hat t^{1/2})} |\Ric_{g(s)}(y)| \,d\vol_{g(s)}(y)\,ds \geq c_n\hat t^{\BETA/2}
\end{equation}
Thus, by applying Chebyshev's inequality to \eqref{eqn: integral estimate improved a}, we find that 
\begin{equation}
\vol_{g(0)}(B_{g(\hat t)}(E,\hat  t^{1/2})) \leq C_n \hat t^{\theta/2}.
\end{equation}
Take $\{y_\ell\}_{\ell=1}^N$ to be a maximal $\hat t^{1/2}/4$-dense subset of $E$ with respect to $g(\hat t)$. We conclude that \eqref{eqn: content bound EetaR} holds by taking $\hat{t}$ sufficiently small depending on $\eta$.
\end{proof}

We now proceed as in Section~\ref{subsec: good bad} to decompose a ball $\uB_t$ satisfying the improved Ricci integral estimate \eqref{eqn: eta integral ricci a}.
For a ball $\uB$,  we define the  refined stopping time $t(x)$ for each $x \in  B_{g(1)}(x_0, 2)$ by
\begin{equation}\label{eqn: stopping time definition-2}
\begin{split}
t(x) = \inf \Big\{ t'\leq 1/200 \, :\, \fint_{t/4}^t \fint_{{\uB}_t(x)}& |\Ric_{g(s)}| \, d\vol_{g(s)} \, ds  \leq t^{\BETA/4-1}\, \hat{t}^{\theta/4} \quad \forall t \in [t' , 1/200]\Big\}.	
\end{split}
\end{equation}
As in Section~\ref{subsec: good bad}, provided we take $t_0 < 200^{-2}$, if a ball $\uB$ satisfies \eqref{eqn: eta integral ricci a}, then
\begin{align}\label{eqn: condition at 200}
\fint_{1/800}^{1/200} &\fint_{{\uB}_{1/200}(x)}|\Ric_{{g}(t)}| \, d\vol_{g(t)}\,dt  \leq 200^{\BETA/4-1} \, \hat{t}^{\BETA/4},
\end{align}
so the stopping condition holds at $t'=1/200$. 

For a  ball $\uB_{t}(x_0)$ with $t<1$, we define the  stopping time $t(x)$  by
\begin{equation}
	t(x) = t\, \tilde{t}(x)
\end{equation}
for each $x \in  B_{g(t)}(x_0, 2t^{1/2})$, where $\tilde{t}(x)$ is the stopping time defined in \eqref{eqn: stopping time definition-2} applied to the rescaled flow $	\tilde{g}(s) = t^{-1} g(t\,s).$
The refined good and bad sets on $\uB_{t}(x_0)$, respectively, are defined by
\begin{equation}
\label{eqn: good and bad set def ref}
\begin{split}
\hat{\mathcal{G}}(\uB_{t}(x_0)) & = \left\{ x \in B_{g(t)}(x_0, 2t^{1/2}) \ :\ t(x) = 0\right\},\\
\hat{\mathcal{A}}(\uB_{t}(x_0)) &  = \left 	\{ x \in B_{g(t)}(x_0, 2t^{1/2})\ :\ t(x) > 0\right\}.	
\end{split}	
\end{equation}
In the following proposition, we establish estimates on the good and bad sets. This is exactly Proposition~\ref{prop: good and bad} from Section~\ref{sec: decomposition} with $t_0^{\BETA/4}$ in place of $\e$ and $\theta/2$ in place of $\theta$; the proof is identical and thus omitted. Again, when convenient, we adopt the shorthand $t_x = t(x)$.

\begin{proposition}\label{prop: good and bad refined} Fix $n\geq 2$, $\theta\in (0,1/2)$, $\eta >0$, $\hat{t} \in (0,1]$, and $x_0 \in M$. Suppose that the improved Ricci estimate \eqref{eqn: eta integral ricci a} holds on $\uB_{\hat{t}}(x_0)$. Then the following holds.
\begin{enumerate}
	\item\label{item: goodbad good refined} For any $x \in \mathcal{G}(\uB_{\hat{t}}(x_0))$
	and for any $s,s' \in [0,\hat{t}]$, the metrics $g(s)$ and $g(s')$ at $x$ satisfy
	\begin{equation}\label{eqn: 0 to 1 refined}
		(1-\eta ) g(s) \leq g(s') \leq (1+\eta ) g(s).
	\end{equation}
\item\label{item: goodbad bad refined}
 Suppose $x \in \mathcal{A}(\uB_{\hat{t}}(x_0))$ and fix $s_0\in [ t_x,\hat{t}]$. 
Then  for all $s, s' \in [s_0, \hat{t}]$ and $y \in B_{g(s_0)}(x,2s_0^{1/2})$, the metrics $g(s)$ and $g(s')$ at $y$ satisfy 
	\begin{equation}\label{eqn: bound at stopping time refined}
		(1-\eta )g(s) \leq g(s') \leq (1+\eta)g(s).
	\end{equation}

\item\label{item: goodbad content}  There is a countable collection $\mathcal{C} =\mathcal{C} (\uB_{\hat{t}}(x_0))\subseteq \mathcal{A}(\uB_{\hat{t}}(x_0))$ such that  $(\mathcal{C} , t_y)$ is a covering pair for $\mathcal{A}(\uB_{\hat{t}}(x_0))$ in the sense of Definition~\ref{def: covering pair}, where $t_y=t(y)$ is the stopping time and 
	\begin{equation}\label{eqn: content bound eta}
		\sum_{y \in \mathcal{C}} {t}_y^{(n/2-\BETA/4)} \leq  \hat{t}^{n/2} \,.
	\end{equation}
	In particular,
		\begin{equation}\label{eqn: content bound b}
		\sum_{y \in \mathcal{C}} {t}_y^{n/2} \leq \eta  \hat{t}^{n/2} \,.
	\end{equation}
\end{enumerate}
	\end{proposition}

Now, we prove Theorem~\ref{thm: decomposition theorem} in the case when $\eta <\e$. The proof by induction is completely analogous to the proof when $\eta<\e ,$ with the only modification coming from the first step of the iteration.
\begin{proof}[Proof of Theorem~\ref{thm: decomposition theorem} when $\eta <\e$]
Let $\delta$ be chosen according to Lemma~\ref{lem: improved ricci integral}.  
 We inductively define sets
$\mathcal{G}^{k}\subseteq \tilde{\mathcal{A}}^{k-1}$, $\tilde{\mathcal{A}}^{k} \subseteq \tilde{\mathcal{A}}^{k-1}$ and $\mathcal{C}^{k}\subseteq \tilde{\mathcal{A}}^{k}$ for each $k \in \mathbb{N}$ satisfying the following properties:
\begin{enumerate}
	\item\label{item: ind claim1-new}  For all $x\in \mathcal{G}^{k}$, we have 
\begin{align}\label{eqn: induction good bound-new}
(1-\eta)(1-\e)^{k-1} g(s')\leq g(s) \leq (1+\eta)(1+\e)^{k-1}g(s')
\end{align}
for all $s, s' \in [0,\hat{t}]$.
\item\label{item: ind claim2-new} For each $x \in \mathcal{C}^k$ we have a mapping $y\mapsto  t_y \in (0,200^{-k})$ such that $(\mathcal{C}^k, t_y)$ is a covering pair for $\tilde{\mathcal{A}}^k$ in $\uB$ as in Definition~\ref{def: covering pair} and such that
 \begin{equation}\label{eqn: induction content bound-new}
	\sum_{y \in \mathcal{C}^k} t_y^{n/2} \leq \hat{t}^{n/2}\eta  \e^{k-2}.
\end{equation} 
\item\label{item: ind claim3-new}
Furthermore, if $y \in \mathcal{C}^k,$ then 
\begin{equation}\label{eqn: bbb-new}
	(1-\eta)(1-\e)^{k-1} g(s') \leq g(s) \leq (1+\eta)(1+\e)^{k-1} g(s')
\end{equation} 
for all $x \in B_{g(t_y)} (y, 2t_y^{1/2})$ and for all $s,s' \in [t_y, \hat{t}]$.
\end{enumerate}
As before, we suppress in the notation the dependence of $t_y$ on $k$ for $y\in\mathcal{C}^k.$ 
The proof of the inductive step is identical to the that in the case when $\hat{t}=1$, so we need only to establish the base case. 
Let $\tilde{\mathcal{A}}^0=\uB$. 

Let $E$ be chosen according to Lemma~\ref{lem: improved ricci integral}; by Lemma~\ref{lem: improved ricci integral}(2) we see that
\begin{equation}\label{eqn: E}
E\subset \bigcup_{\ell=1}^N \uB_{\hat{t}}(x_\ell), \qquad N\leq \eta \hat{t}^{-n/2}.
\end{equation}

Next, consider a maximal $\hat{t}^{1/2}/4$-dense set $\{x_i\}_{i=1}^{N'}$ in $\uB\setminus E$ with respect to $g(\hat{t})$. In this way,
 \begin{equation}
 	\uB\setminus E \subseteq \bigcup_{i=1}^{N'}B_{g(\hat{t})}(x_i,\hat{t}^{1/2}),
 \end{equation}
 and thanks to Corollary~\ref{corollary: volume bounds}, we have that 
  \begin{equation}\label{eqn: N' bound}
  	N'\leq C_n \hat{t}^{-n/2}\,.
  \end{equation}
Since \eqref{eqn: eta integral ricci a} holds for each $i=1,\dots, N'$,we apply Proposition~\ref{prop: good and bad refined} to decompose $B_{g(\hat{t})}(x_i,4\hat{t}^{1/2})$ for each $i=1,\dots, N'$. Now, we set
\begin{align}
	\mathcal{G}^1=  \bigcup_{i=1}^{N'} \hat{\mathcal{G}}(\uB_{\hat{t}(x_i)}), \qquad \tilde{\mathcal{A}}^1 = E \cup \bigcup_{\ell=1}^{N'} \hat{\mathcal{A}}(\uB_{\hat{t}(x_\ell)})
\end{align}
as defined in \eqref{eqn: good and bad set def ref}. 
For each $x_\ell$ for $\ell = 1, \dots, N$, we define $t_{x_\ell} = \hat{t}^{1/2}$, and  set 
\begin{equation}
	\mathcal{C}^1 = \bigcup_{\ell=1}^{N} \{ x_\ell\} \cup \bigcup_{\ell=1}^{N'} \hat{\mathcal{C}}(\uB_{\hat{t}(x_\ell)})\,.
\end{equation}
Then 
properties \eqref{item: ind claim1-new}--\eqref{item: ind claim3-new}  follow directly from Proposition~\ref{prop: good and bad} along with \eqref{eqn: E} and \eqref{eqn: N' bound}, since
\begin{equation}
\sum_{\mathcal{C}^1} t_y^{n/2} = N \hat{t}^{n/2} + \eta N'  \hat{t}^{n/2 } \leq C_n\eta\, .
\end{equation}
We have thus inductively defined the sets $\mathcal{G}^k,\tilde{\mathcal{A}}^k$, and $\mathcal{C}^k$. The remainder of the proof is identical to the proof in the case when $\eta =\e$ shown in the previous subsection. This completes the proof of Theorem~\ref{thm: decomposition theorem}.
\end{proof}

\vspace{.4cm}



\section{$L^{\PP}$ bounds for the metric coefficients}\label{sec: metric estimates}
In this section, we prove Theorem~\ref{claim1}, which we restate below as Theorem~\ref{claim7} for the convenience of the reader. 

 \begin{theorem}[Theorem~\ref{claim1} restated]\label{claim7} 
Fix $n\geq 2$, ${\PP}\in [1, \infty)$, and $\e>0$. 
There exists $\delta=\delta(n,{\PP},\e)>0$ such that the following holds. 
Let $(M,g)$ be a complete Riemannian $n$-manifold with bounded curvature satisfying 
	\begin{align}
		R\geq -\delta, &\qquad \qquad \nu(g, 2) \geq -\delta.
	\end{align}
	 Then for any $x_0 \in M$, 
	 there is an open set $\Omega\subset M$ containing $x_0$ and a smooth diffeomorphism 
	 $\psi: \Omega \to B(0,1) \subset \R^n$ with $\psi(x_0)=0$ 	  satisfying 
\begin{align}\label{eqn: w1p 1 7}
  \fint_{B(0,1)} |(\psi^{-1})^*g -g_{euc} |^{\PP} \, dy \leq \e, \qquad
 \fint_{\Omega} |\psi^* g_{euc} - g|^{\PP} \, d\vol_g \leq \e
 \end{align}	
{Furthermore, for any $\kappa>1$ and $q_0 \in [\kappa,\infty)$,
we may choose $\delta$ additionally small depending on $q_0 ,\kappa$ such that for any $f \in W^{1,q}(B(0,1))$ with $q \in [\kappa, q_0]$, we have}
	 \begin{align}
\label{eqn: Lp norms small error pf}	 	
(1-\e)\| \psi^* f\|_{L^{q/\kappa}(\Om)} &\leq \| f\|_{L^q(B(0,1))} \leq (1+\e) \| \psi^* f\|_{L^{\kappa q}(\Om)},\\
\label{eqn: gradient small error pf}
(1-\e)\| \na \psi^* f\|_{L^{q/\kappa}(\Om)} &\leq \|\nabla f\|_{L^q(B(0,1))} \leq (1+\e) \|\na \psi^* f\|_{L^{\kappa q}(\Om)} 	. 	
	 \end{align}
\end{theorem}
\medskip

	

\begin{remark}
	{\rm
	The estimates in \eqref{eqn: w1p 1 7} are equivalent to 
			 \begin{align}\label{eqn: w1p 1 intro}
	\fint_{\Omega} \|d \psi\|_\infty^{\PP}\,d\vol_g \leq 1+\e,
\quad	 \quad		
	 	\fint_{B(0,1)}  \|d\psi^{-1}\|_\infty^{\PP}\,dy \leq 1+\e.
 \end{align}	

Here, given a map $u:(M,g)\to (N,h)$, the notation $\|d u\|_\infty(x)$ indicates 
	the operator norm of the linear map
	 $d u_x: (T_x M,g)\to (T_{u(x)}M, h)$.  Given a bilinear form $B: TM\times T M\to \R$, we let $\|B\|_\infty = \sup \|B(v,\cdot)\|_\infty/\|v\|_g$.
	}
\end{remark}	 

\begin{remark}\label{rmk: R instead of 1}
	{\rm
	{
	It will be apparent in the proof that, given any number $R\geq 1,$ by  choosing $\delta>0$ to additionally small depending on $R$, we may obtain the conclusion of Theorem~\ref{claim7} with $B(0,R)$ replacing $B(0,1)$.
	}}
\end{remark}
	 
\begin{proof}

 Let us begin with some initial observations. Without loss of generality we may assume that $\e\leq \e(n, {\PP})$, where $\e(n, {\PP})$ will be determined in the proof.
	By Theorem~\ref{prop: uniform existence time and small curvature estimates}, the Ricci flow $g(t)$ with $g(0)=g$ exists for $t\in [0,1]$ and $|\Rm_{g(t)}|\leq \ETA/t$ for all $x\in M$ and $t\in (0,1]$.  Here $\ETA$ may be taken as small as needed by choosing $\delta$ sufficiently small. We let 
	\begin{equation}
		\Omega = B_{g(1)}(x_0, 1)\, .
	\end{equation}
By Lemma~\ref{rmk: good charts under regularity scale}, we obtain a smooth diffeomorphism $\psi:B_{g(1)}(x_0,2)\to U\subset \R^n$, with inverse $\phi = \psi^{-1}$ such that $\psi(x_0)=0$, $\psi(\Omega) = B(0,1)$,  and
\begin{equation}
			\label{eqn: metric bounds 2}
		(1-\e/2) g_{euc}  \leq   \phi^*g(1)\leq (1+\e/2)g_{euc} \, ,
\end{equation}
for all $x\in U$. This holds as long as $\ETA$ (and hence $\delta)$ has been chosen to be sufficiently small depending on $\e$ and $n$.\\

Now, let 
\begin{equation}
	B_{g(1)}(x_0,2)=\bigcup_{k=1}^\infty \mathcal{G}^k \cup\mathcal{A}\, ,
\end{equation}
 be the decomposition provided by the Decomposition Theorem~\ref{thm: decomposition theorem} with $\e = \eta $. From \eqref{eqn: metric bounds 2} and Theorem~\ref{thm: decomposition theorem}\eqref{item: decomp G bounds},  for every $x\in \phi^*\mathcal{G}^k$, we have
 \begin{align}
(1-\e)^{k+1} g_{euc} \leq  \phi^*g \leq (1+\e)^{k+1} g_{euc}	\, .
\end{align}
Therefore, we have 
\begin{equation}\label{eqn: bb}
	|\phi^* g -g_{euc}|_{g_{euc}} \leq (1+\e)^{k+1}-1 \leq \e(k+1)(1+\e)^k\, ,
\end{equation}
for all  $x\in \phi^*\mathcal{G}^k$, and likewise
\begin{equation}
|g -\psi^* g_{euc}|_{g}  \leq \e(k+1)(1+\e)^k\, ,
\end{equation}
for all $x \in \mathcal{G}^k$.
%
%
%
Furthermore, $\vol_g(\mathcal{A})=\vol_{euc}(\phi^*\mathcal{A})=0$ and for all $k\geq 2,$ we have
\begin{align}\label{eqn: G bound}
\vol_{g}(\mathcal{G}^k)&\leq (1+\e)^k \e^{k-1},\\
\label{eqn: pullback G bound}\vol_{{euc}}(\phi^*\mathcal{G}^k)&\leq (1+\e)^k \e^{k-1}.	
\end{align}



We begin by proving the first estimate in \eqref{eqn: w1p 1 7}. 
Set 
\begin{equation}
	\mu(r) =\vol_{euc}(\{y \in B(0,1) \ : \ |\phi^*g -g_{euc}|_{g_{euc}} \geq r\})
\end{equation}
and $r_k = \e(k+1)(1+\e)^k$. 
Note that  $\mu(0) \leq \om_n$, while  \eqref{eqn: bb} and 
 \eqref{eqn: pullback G bound} ensure that 
 \begin{equation}\label{eqn: mu bound phi}
 \begin{split}
\mu(r_k) \leq \vol_{euc}(B(0,1)\setminus\cup_{\ell=1}^k \phi^*\mathcal{G}^\ell) &\leq \sum_{\ell=k+1}^\infty \vol_{euc}(\phi^*\mathcal{G}^\ell) \\
&\leq \sum_{\ell=k+1}^\infty (1+\e)^\ell\e^{\ell-1} \leq C\e^k,
 \end{split}	
 \end{equation}
where $C$ is a universal constant. We apply the layer cake formula 
to find that
\begin{equation}
\begin{split}
\int_{B(0,1)}   |\phi^*g -g_{euc}|^{\PP}\,dy & = {\PP} \int_0^\infty r^{{\PP}-1}\mu(r)\,dr \\
 &\leq\mu(0) {\PP} \int_0^{r_0} r^{{\PP}-1}\,d r+ {\PP} \sum_{k=0}^\infty \mu(r_k) \int_{r_k}^{r_{k+1}} r^{{\PP}-1}\,dr  \\
& \leq \e  \mu(0)+  \sum_{k=0}^\infty \mu(r_k)r_{k+1}^{\PP} \\
&\leq \e \omega_n+C\e^{\PP}\sum_{k=0}^\infty \e^{k}(k+2)^{\PP}(1+\e)^{\PP(k+1)} \\
&\leq C\e \omega_n\, ,
\end{split}
\end{equation}
for some $C= C(n, P)$ provided that $\e$ is small enough depending on $P$. Further decreasing $\delta$, we may replace $\e$ by $\e / C$. Dividing through by $\om_n,$ we establish the first estimate in \eqref{eqn: w1p 1 7}. The second estimate is entirely analogous, with the only additional point to note being that $\vol_g( \Omega) \leq (1+\e) \om_n$ thanks to \eqref{eqn: bb} and \eqref{eqn: G bound}.
%



{

We now show \eqref{eqn: Lp norms small error pf}. Let $\sigma = \sigma(\e, \kappa,q)$ be chosen later in the proof and let $\delta = \delta(n,\sigma)$ be sufficiently small to apply the Decomposition Theorem~\ref{thm: decomposition theorem} with $\sigma
$ in place of $\e$.  Fix any $f \in W^{1,q}(B(0,1))$. From \eqref{eqn: bb} and $\vol_{euc}(\phi^*\mathcal{A}) =0$, we find that 
\begin{equation}
\begin{split}
	\| f\|_{L^q(B(0,1))}^q
	&\leq  \sum_{k=1}^\infty \int_{\phi^*\mathcal{G}^k\cap B(0,1)} |f|^{ q} \,dx \\
	& \leq   \sum_{k=1}^\infty (1+\sigma)^{nk/2+1}\int_{\mathcal{G}^k \cap \Omega} |\psi^*f|^{ q} \,d\vol_{g}.
\end{split}
\end{equation}
	Next, for each $k$, 
	we apply H\"{o}lder's inequality with $\kappa$ and $\kappa'=\kappa/(\kappa-1)$ 
	and apply \eqref{eqn: G bound} to find that 
	\begin{equation}\begin{split}
	\| f\|_{L^q(B(0,1))}^q	
	& \leq \sum_{k=1}^\infty (1+\sigma)^{nk/2+1}\vol_{g}(\mathcal{G}^k)^{1/\kappa'} \left(\int_{\mathcal{G}^k} |\psi^*f|^{{\kappa q}} \,d\vol_g\right)^{1/\kappa}\\
	& \leq \| \psi^* f\|_{L^{{\kappa q}}(\Omega)}^q \sum_{k=1}^\infty (1+\sigma)^{(n+1/\kappa')k}\sigma^{k/\kappa'}.
\end{split}\end{equation}
 Now, by choosing $\sigma = \sigma(\e, \kappa)$ sufficiently small, the sum on the right-hand side is bounded above by $1+\e$, and thus by $(1+\e)^q$ for any $q\geq 1$; after taking the $1/q$ power of both sides, we have shown that 
\begin{equation}
	\| f\|_{L^q(B(0,1))} \leq (1+\e) \| \psi^* f\|_{L^{{\kappa q}}(\Omega)}.
\end{equation}
The proof of the other inequality in \eqref{eqn: Lp norms small error pf} is completely analogous.

The proof of \eqref{eqn: gradient small error pf} is similar. 
From \eqref{eqn: bb}, we have
\begin{equation}
\begin{split}
\int_{B(0,1)}|\na_{euc} f|_{euc}^{q}\,dx 
&\leq  \sum_{k=1}^\infty  \int_{\phi^*\mathcal{G}^k\cap B(0,1)}|\na_{euc} f|_{euc}^{q}\,dx\\
	&\leq \sum_{k=1}^\infty(1+\sigma)^{(q+n/2)k+1 } \int_{\mathcal{G}^k} |\nabla_{g}\psi^*f|^{q} d\vol_{g}\\
	&\leq\left(\int_{\Omega}|\nabla_{g}\psi^* f|^{{\kappa q}} d\vol_{g}\right)^{1/\kappa}  \sum_{k=1}^\infty (1+\sigma)^{(q+n/2)k+1}\vol_{g}(\mathcal{G}^k)^{1/\kappa'} \\
	&\leq \left(\int_{\Omega}|\nabla_{g}\psi^* f|^{{\kappa q}} d\vol_{g}\right)^{1/\kappa} \sum_{k=1}^\infty (1+\sigma)^{(q+n)k}\sigma^{k/\kappa'}  \\
	&  \leq (1+\e)  \|\nabla_{g}\psi^* f\|_{L^{{\kappa q}}(\Omega)}^q.
\end{split}\end{equation}
where again the final inequality holds provided that $\sigma$ is sufficiently small depending on $q_0,\e$ and $\kappa$. The other inequality in \eqref{eqn: gradient small error pf} is analogous. This completes the proof.}
\end{proof}

As a consequence of the proof of Theorem~\ref{claim7}, we obtain the following estimate for the Ricci flow.  In short, it tells us that at in least in $L^P$ the metrics $g(t)$ converge as $t\to 0$.  

\begin{corollary}\label{cor: useful for limit}
	Fix $n\geq 2$, ${\PP}\in [1, \infty)$, and $\e>0$. 
There exists $\delta=\delta(n,{\PP},\e)>0$ such that the following holds. 
Let $(M,g)$ be a complete Riemannian $n$-manifold with bounded curvature satisfying 
	\begin{align}
		R\geq -\delta, &\qquad \qquad \nu(g, 2) \geq -\delta.
	\end{align}
For all $\eta\in (0,\e)$, there exists $\hat{t}_\eta = \hat{t}_\eta(n, P, \e, \eta)$ such that for any $x_0 \in M$, $s,t \in [0,\hat{t}_\eta]$. 
	 \begin{align}\label{eqn: lp eta}
  \fint_{B_{g(1)}(x_0,1)} |g(s) -g(t) |_{g(t_\eta)}^{\PP} \, d\vol_{g(t_\eta)} \leq \eta
 \end{align}	
Moreover, for any $\kappa>1$, we may choose $\delta$ additionally depending on $\kappa$ such that 
for all $s,t\in [0,\hat{t}_\eta]$
\begin{equation}\label{eqn: lp est wrt g1}
	  \fint_{B_{g(1)}(x_0,1)} |g(s) -g(t) |_{g(1)}^{\PP/\kappa} \, d\vol_{g(1)} \leq \eta.
\end{equation}
\end{corollary}
\begin{remark}{\rm
Upgrading this $L^P$ coefficient convergence to $d_p$ convergence is morally, though not explicitly stated, in Section \ref{s:global_convergence}.  This is in fact significantly harder, as one needs to understand how the analysis on $g(s)$ varies, not just the metric coefficients.
}
\end{remark}

\begin{proof}
	The proof of \eqref{eqn: lp eta} is identical to the proof of \eqref{eqn: w1p 1 7}, but now we apply the Decomposition Theorem~\ref{thm: decomposition theorem} with $\eta $ in place of $\e$. Then, the proof of \eqref{eqn: lp est wrt g1} follows precisely be repeating the proof of \eqref{eqn: gradient small error pf}.
\end{proof}

\vspace{.4cm}

\section{Epsilon regularity}\label{sec: proof of main theorem}

In this section, we prove the epsilon regularity theorem, Theorem~\ref{thm: main thm, Lp def new}, and  the uniform $L^\infty$ Sobolev embedding, Theorem~\ref{Sob-embedding}. The main tool is Theorem~\ref{claim7}, established in the previous section.

\subsection{Preliminary lemmas} Let us first establish two preliminary lemmas that will be needed in the proof of Theorem~\ref{thm: main thm, Lp def new}.
The first lemma allows us to localize the $d_p$ distance in Euclidean space. We use the notation $d_{q,B(0,R)}$ to denote $d_{q, g_{euc}, B(0,R)}$, that is,
\begin{equation}
d_{q, B(0,R)} (x, y) = \sup \Big\{ |f(x) - f(y) |  \ : \ \int_{B(0,R)}|\na f|^q \, dx \leq 1\Big\}.
\end{equation}

Our first lemma tells us that on Euclidean space we may localize the $d_p$ distance:

\begin{lemma}\label{lem: localizing dp}
For all $\e>0$ and $p\geq n+1$ there exists $R=R(n,\e,p)$ such that for all 
	{$q \geq p-1/2$} and for all $x \in \B_{q,g_{euc}}(0,1)$, we have
\begin{equation}\label{eqn: localizing dp}
	|d_{q, B(0,R)}(x,0) - d_{q, g_{euc}}(x,0)| \leq \e.
\end{equation}
\end{lemma}
\begin{proof}
One inequality in \eqref{eqn: localizing dp} follows immediately from the definition: for any  $x \in B(0,R)$,  we have
	\begin{equation}\label{eqn: local dp one direction}
	d_{q,g_{euc}}(x,0) \leq d_{q, B(0,R)}(x,0).
	\end{equation}

Next, we show that for any $x$ such that $d_{q, B(0,R)}(x,0) \leq 1+\e$, we have 
\begin{equation}\label{eqn: local dp other direction}
	d_{q, B(0,R)}(x,0) \leq d_{q,g_{euc}}(x,0) + \e.
\end{equation}
Note that this immediately implies that \eqref{eqn: local dp other direction} holds for any $x \in \B_{q,g_{euc}}(0,1)$ and thus together with \eqref{eqn: local dp one direction} will complete the proof.
  Let $f$ be a function such that 
	\begin{equation}\label{eqn: 3 properties}
		\int_{B(0,R)} |\na f|^q \,dx \leq 1, \qquad f(0)=0, \qquad d_{q, B(0,R)}(x,0) <f(x) +\e.
	\end{equation}
We will modify $f$  to produce an admissible test function for $d_{p,g_{euc}}(x,0)$.
 Without loss of generality, we may replace $f$ with $\min\{|f|,2\}$, which also satisfies \eqref{eqn: 3 properties}. Now, consider a cutoff function $\psi$ such that $0 \leq \psi \leq 1$ and 
 \begin{equation}
 	\psi =\begin{cases}
 				1 &\text{ in } B(0,2)\\
 				 0 &\text{ in } \R^n \setminus B(0,R)
 	\end{cases} ,   \qquad |\na \psi | \leq CR^{-1}
 \end{equation}
 for a dimensional constant $C$. By choosing $R$ sufficiently large depending on $n, \e$ and $p$, we have
 	$\| \na \psi \|_{L^q(\R^n) } \leq CR^{n/q-1} \leq \e/2$ for any $q\geq p-1/2.$
So, letting $\hat{f} = (1-\e)f\psi$, we have 
 \begin{equation}
 \hat f(0) = 0, \qquad \hat f(x) \geq  (1-2\e) d_{q, B(0,R)}(x,0),
 \end{equation}
 and 
\begin{equation}\begin{split}
	\| \na \hat{f}\|_{L^q(\R^n)} &\leq (1-\e)\| \psi \na f\|_{L^q(\R^n)} + (1-\e)\| f \na \psi\|_{L^q(\R^n)}\\
	& \leq (1-\e)\|\na f\|_{L^q(B(0,R))} +2 (1-\e) \| \na \psi\|_{L^q(\R^n)}\\
	& \leq  1- \e + \e = 1.
	\end{split}
\end{equation}
In particular, $\hat{f}$ is an admissible test function for $d_{q,g_{euc}}(x,0)$ and so 
\begin{equation}
	d_{q,g_{euc}}(x,0) \geq (1-2\e) d_{q, B(0,R)}(x,0).
\end{equation}
After replacing $\e$ with $\e/2$, this completes the proof of \eqref{eqn: local dp other direction} and thus of the lemma.
\end{proof}
\begin{remark}\label{rmk: localize f}
	{\rm In the course of proving of Lemma~\ref{lem: localizing dp}, we have established the following statement under the hypotheses of Lemma~\ref{lem: localizing dp}. 
For any $x\in \R^n$ such that $d_{q,B(0,R)}(x,0)\leq 2,$ we may find a function $\hat{f} \in W^{1,q}_0(B(0,R))$ such that $\| \na f\|_{L^q(\R^n)}\leq 1$, $
 \hat f(0) = 0$ and $\hat f(x) \geq  (1-2\e) d_{q, B(0,R)}(x,0).$
	}
\end{remark}

The next lemma also deals with the $d_p$ distance on Euclidean space, and establishes uniform continuity of $d_{q,U}$ with respect to $q$. Again, here the notation $d_{q, U}$ is used to denote $d_{q, g_{euc}, U}$ for a set $U \subset \R^n$.
{ 
\begin{lemma}\label{p-continuity} Fix $p>n$ and let $U\subset \mathbb{R}^n$ be a bounded open set with smooth boundary. Then $d_{q,U}(x,y)$ converges to $d_{p,U}(x,y)$ uniformly on precompact set $V$ with $\overline{V} \Subset U$ as $q\rightarrow p$.
\end{lemma}
\begin{proof}
We first consider the case $q\rightarrow p^-$.
By definition, for $q<p$, we have
\begin{equation}\label{eqn: one direction free}
d_{p,U}(x,y)\leq |U|^{\frac{1}{q}-\frac{1}{p}}d_{q,U}(x,y).
\end{equation}
In particular, for any $\e>0$ we have $ d_{p,U}(x,y) \leq (1+\e)  d_{q,U}(x,y)$ for $q<p$ sufficiently close to $p$. 
 Suppose by way of contradiction that there are  $\e_0>0$, $x_i,y_i\in V$ and $p_i\rightarrow p^-$ such that for all $i>>1$,
\begin{equation}\label{eqn: contradiction assumption}
d_{p,U}(x_i,y_i)+2\e_0\leq d_{p_i,U}(x_i,y_i).
\end{equation}
For each $i$, let $f_i \in W^{1,p}(U)\cap C^0_{loc}(U)$ be a function such that $f_i(y_i)=0$, $\int_U |\nabla f_i|^{p_i}\leq 1$, and 
\begin{equation}\label{eqn: almost sup}
 d_{p_i,U}(x_i,y_i)<|f_i(x_i)-f_i(y_i)|+\e_0.
 \end{equation}

After passing to a subsequence, $x_i,y_i\rightarrow x_\infty, y_\infty\in \bar V$ by compactness. It is easily verified that $d_{p,U}(x_i,y_i)\to d_{p, U }(x_\infty, y_\infty)$ as $ i\to \infty$. For any fixed $q<p$, then for $i$ sufficiently large we may apply H\"older's inequality to find
{
\begin{equation}\int_U |\nabla f_i|^q \, dx \leq |U|^{1- q/p_i}.\end{equation}
}
In particular, we see that $\| \na f_i \|_{L^q(U)}$ is bounded uniformly in $i$. Moreover, by the Morrey-Sobolev inequality, $\{f_i\}$ is uniformly H\"{o}lder continuous.  We therefore see that  $f_i\rightharpoonup f_\infty$  in $W^{1,q}(U)$ with $\int_U |\na f_\infty|^q \,dx \leq 1 $ for all $q<p$, and $f_i \to f_\infty $ uniformly. Letting $q$ tend to $p$, we see that $\int_U |\nabla f_\infty|^p\leq 1$, and so $f_\infty$ is an admissible test function for $d_{p,U}(x_\infty, y_\infty)$. Thus, passing to the limit and using \eqref{eqn: contradiction assumption} and \eqref{eqn: almost sup}, we reach a contradiction, since
\begin{equation}
\begin{split}
d_{p,U}(x_\infty, y_\infty) + \e_0 \leq 
 \lim_{i\rightarrow +\infty}|f_i(x_i)-f_i(y_i)|
&= |f_\infty(x_\infty)-f_\infty(y_\infty)|\\
&\leq d_{p,U}(x_\infty,y_\infty).
\end{split}
\end{equation}

Now we consider the case when $q\rightarrow {p}^+$. Again by \eqref{eqn: one direction free}, we see that for any $\e>0$, we have $d_{p,U}(x,y) \geq (1-\e) d_{q,U}(x,y)$ for $q>p$ sufficiently close to $p$.
Suppose by way of contradiction that there exist $p_i>p,\e_0>0$ and $p_i\rightarrow p^+$, $x_i,y_i\in V$ such that for all $i$,
\begin{equation}\label{eqn: contra 2}
d_{p_i,U}(x_i,y_i)+\e_0<d_{p,U}(x_i,y_i).
\end{equation}
We may assume $x_i,y_i\rightarrow x_\infty,y_\infty\in \bar V$. Let $\sigma$ be a small constant to be determined later.  By smooth approximation, let $f \in C^1(\overline{U})$  function such that $\int_U |\nabla f|^p\leq 1+\sigma\e_0$ and 
\begin{equation}\label{eqn: dpbound}
 d_{p,U}(x_\infty,y_\infty)\leq |f(x_\infty)-f(y_\infty)|+\sigma\e_0.
\end{equation}
We let $\Lambda = \sup_{\bar{U}} |\nabla f|.$
 For $i$ sufficiently large, we have 
\begin{equation}
\begin{split}
\left( \int_U |\nabla f|^{p_i}\, dx\right)^{1/p_i} &\leq  \Lambda^{(p_i-p)/p_i}\left(\int_U |\nabla f|^p \right)^{1/p_i} \leq (1+\sigma\e_0).
\end{split}
\end{equation}
So,  $f/(1+\sigma \e_0)$ is an admissible test function for $d_{p_i, U }(x_i, y_i)$.  We thus see that 
\begin{equation}\label{eqn: fxi bound}
 |f(x_i) - f(y_i)| \leq (1+\sigma \e_0)d_{p_i, U}(x_i, y_i) .
\end{equation}
So, by \eqref{eqn: contra 2}, \eqref{eqn: dpbound},  and \eqref{eqn: fxi bound} respectively, we find
\begin{equation}\begin{split}
	\limsup_{i\to \infty } d_{p_i, U}(x_i, y_i) + \e_0 &  \leq d_{p,U}(x_\infty, y_\infty ) \\
	& \leq \lim_{i \to \infty} |f(x_i) - f(y_i)| + \sigma\e_0 \\
	& \leq(1+\sigma \e_0)  \liminf_{i \to \infty} d_{p_i, U}(x_i, y_i) + \sigma \e_0.
	\end{split}
\end{equation}
We reach a contradiction by choosing $\sigma$ small enough.
\end{proof}

\subsection{Proof of Theorem~\ref{thm: main thm, Lp def new}}
In this section, we prove Theorem~\ref{thm: main thm, Lp def new}, which we restate below as Theorem~\ref{thm: main thm, Lp def proof} for the convenience of the reader.

\begin{theorem}[Theorem~\ref{thm: main thm, Lp def new} Restated] \label{thm: main thm, Lp def proof}
{Fix $n\geq 2$. For any $\e>0$ and $p\geq n+1$, there exists $\delta = \delta(n,\e,p)$ such that the following holds.}
 Let $(M^n,g)$ be a complete Riemannian manifold with bounded curvature satisfying 
	\begin{align}
		R\geq -\delta, &\qquad \qquad \nu(g, 2) \geq -\delta.
	\end{align}
 Then for all $x\in M$, 
 \begin{equation}\label{eqn: GHclose}
	d_{GH}\left( (\B_{p,g}(x, 1), d_{p,g}) ,\,  (\B_{p,g_{euc}}(0, 1), d_{p,g_{euc}})\right) < \e. 
\end{equation}
Moreover,  
\begin{equation}\label{eqn: volumes close p}
(1-\e) |\B_{p,g_{euc}}(0, r) | \leq \vol_{g} (\B_{p,g}(x_0, r)) \leq (1+\e) |\B_{p,g_{euc}}(0, r) | 
\end{equation}
for all $0<r<1$ where $|\cdot|$ denotes the Euclidean volume.  In particular, the measure $d\vol_g$ on the metric measure space $(M,d_{p,g},d\vol_g)$ is a doubling measure for all scales $r\leq 1$.
\end{theorem}

\begin{proof}[Proof of Theorem~\ref{thm: main thm, Lp def proof}]

Let $\e'>0$  be a fixed number to be specified later in the proof.
To begin, notice that we may choose $R=R(n,p/2,\e')\geq 2$ according to Lemma~\ref{lem: localizing dp}  and $\kappa =\kappa(n,p,\e', R)>1$ sufficiently close to $1$ according to Lemma \ref{p-continuity} so that for any $z,y\in \B_{p,g_{euc}}(0,2)$, we have
\begin{equation}\label{dp-euc-close}
\begin{split}
|d_{{p/\kappa},B(0,R)}(z,y)-d_{p,g_{euc}}(z,y)|<\e';\\
|d_{\kappa p,B(0,R)}(z,y)-d_{p,g_{euc}}(z,y)|<\e'.
\end{split}
\end{equation}
Up to possibly increasing $R$ depending on $p$ and $n$, we have $\B_{p,g_{euc}}(0,2)\subset B(0,R/2)$ by \eqref{eqn: uniform ball containment}.

Now, choose $\delta=\delta(n,p,\e',\kappa,R) = \delta(n,p,\e')$ sufficiently small such that we may apply Theorem~\ref{claim7} with $\e=\e'$ and with $R$ as above (see Remark~\ref{rmk: R instead of 1}), obtaining a diffeomorphism $\psi: \Omega' \to B(0,R)$ satisfying the properties of Theorem~\ref{claim7}. 
We claim that for any $z,y\in \B_{p,g_{euc}}(0,2)$, we have
 \begin{equation}\label{d_p-comp-diff}
 |d_{p,g}(\psi^{-1}(z),\psi^{-1}(y))- d_{{p},g_{euc}}(z,y)|<\e,
 \end{equation}
and so in particular, the diffeomorphism $\psi^{-1}$ is an $\e$-Gromov-Hausdorff approximation between $(\B_{p,g_{euc}}(0,2), d_{p,euc})$ and $(\Omega , d_{p,g})$, where we define $ \Omega = \psi^{-1}( \B_{p,g_{euc}}(0,2))\subset M$. 

Fix $z,y\in \B_{p,g_{euc}}(0,2)$ and set $y_0=\psi^{-1}(y)$, $z_0 = \psi^{-1}(z) \in \Omega $ for brevity of notation. Thanks to \eqref{dp-euc-close},  in order to prove \eqref{d_p-comp-diff} it suffices to show that
\begin{align}
\label{eqn: compare pk b}	 d_{p,g}(z_0,y_0) & \leq (1+3\e')\,d_{{\kappa p},B(0,R)}(z,y), \\
\label{eqn: compare pk a}	 d_{p,g}(z_0,y_0)& \geq 
(1-3\e')
d_{p/\kappa,B(0,R)}(z,y) .
\end{align}
We begin by showing \eqref{eqn: compare pk b}. 
Let $f\in W^{1,p}(M)\cap C^0(M)$ be a function such that $\int_M |\nabla f|^pd\vol_g\leq 1$ and 
\begin{equation}
 d_{p,g}(z_0,y_0) \leq  (1+\e')|f(z_0)-f(y_0)| .
\end{equation}
 Thanks to Theorem~\ref{claim7}, we find that 
\begin{equation}\begin{split}
\|\na \psi^*f \|_{L^{\kappa p}(B(0,R))}& \leq (1+\e')\|\na f \|_{L^p(\Omega)}\\
&\leq (1+\e')\|\na f \|_{L^p(M)}\leq 1+\e'.\end{split}
\end{equation}
So, $h= \psi^*f/(1+\e')$ is an admissible test function for $d_{\kappa p,B(0,R)}(y,z)$ and thus $d_{\kappa p,B(0,R)}(y,z) \geq |h(z) - h(y)|.$ Furthermore, 
\begin{equation}
|h(z)-h(y)| \geq (1+\e ')^{-1}|f(z_0) - f(y_0)| \geq (1+\e')^{-2}d_{p,g}(z_0,y_0) .
\end{equation} This establishes \eqref{eqn: compare pk b}.

The proof of \eqref{eqn: compare pk a} is similar.  Fix a function $h\in W^{1,p/\kappa}(B(0,R))\cap C^0(B(0,R))$  such that $\int_{B(0,R)}|\na h|^{p/\kappa}dx\leq 1$ and 
\begin{equation}
d_{p/\kappa,B(0,R)}(z_0,y_0)< (1+\e')|h(z_0)-h(y_0)|.
\end{equation}
By Remark~\ref{rmk: localize f}, we may assume $h$ to be vanishing outside $B(0,R)$ and hence $f=\psi^* h$ can be extended to a function on $M$.
By Theorem~\ref{claim7}, we have
\begin{equation}
\|\na f\|_{L^p(M)} \leq (1+\e') \|\na h\|_{L^{p/\kappa}(B(0,R))}\leq 1+\e'.
\end{equation}
Hence, $f/(1+\e')$ is an admissible test function for $d_{p,g}(y,z)$, and so
\begin{equation}
(1+\e')d_{p,g}(y,z)\geq |f(y)-f(z)| =|h(y_0)-h(z_0)| .
\end{equation}
This establishes \eqref{eqn: compare pk a}, and hence \eqref{d_p-comp-diff}.
From \eqref{d_p-comp-diff}, it follows immediately that 
 \begin{equation}\label{eqn: containment p balls}
\psi^{-1}(\B_{p,g_{euc}}(0, 1-8\e'))\subseteq \B_{p,g}(x, 1)\subseteq\psi^{-1}(\B_{p,g_{euc}}(0, 1+8\e'))
 \end{equation}
 and hence \eqref{eqn: GHclose} holds taking $\e' = \e/8$.
Finally, \eqref{eqn: volumes close p} is now an immediate consequence of \eqref{eqn: containment p balls}, Theorem~\ref{thm: decomposition theorem} and rescaling argument. This completes the proof of Theorem~\ref{thm: main thm, Lp def proof}.
 \end{proof}

 %
 %
 %
 %
 %
 %
 %
 

\subsection{Proof of Theorem~\ref{Sob-embedding}}
Now we prove the $L^\infty$ Sobolev inequality on manifolds with small entropy and $R_-$.  We restate Theorem~\ref{Sob-embedding} below as Theorem~\ref{Sob-embedding S7}

\begin{theorem}[$L^\infty$ Sobolev Embedding]\label{Sob-embedding S7}
	
Fix $p,q_0\geq n+1$, then there exists $\delta= \delta(n,p,q_0)>0$   such that the following holds.
  Let  $(M^n,g)$ be a complete Riemannian manifold with bounded curvature satisfying 
		\begin{align}
		R\geq -\delta, 	\qquad \qquad 	\nu(g, 2) \geq -\delta \,.
	\end{align}
Then for all  $x_0 \in M$ and $q \in (n,q_0)$, there exists $C(n,q)>0$ such that for all $f \in W_0^{1,q}(\B_{p,g}(x_0, 1))$, we have 
	\begin{equation}\label{So-em}
		\| f\|_{L^\infty(\B_{p,g}(x_0,1))} \leq  C_{n,q} \| \na f\|_{L^q(\B_{p,g}(x_0,1))}.
	\end{equation}
	{
	For all $f\in W^{1,q}(\B_{p,g}(x_0,1))$ and {$x,y\in \B_{p,g}(x_0,1)$}, we have 
	\begin{equation}\label{So-em-2}
	|f(x)-f(y)|\leq C_{n,q}\|\nabla f\|_{L^q(B_{p,g}(x_0,1))}.
	\end{equation}
	 }
	{For all $f \in W^{1,q}(M)$, we have 
	\begin{equation}\label{So-em global}
\|f\|_{L^\infty(M)}\leq  C_{n,q} \left(\|\nabla f\|_{L^{q}(M)}+ \|f\|_{L^{q}(M)}\right).
\end{equation}}
In terms of the $d_p$ distance we can upgrade \eqref{So-em-2} to a H\"older embedding: 
there exists $\alpha= \alpha(n,q) \in (0,1)$
	\begin{equation}\label{So-em-2-section 7}
	|f(x)-f(y)|\leq C_{n,q, p} d_p(x,y)^{\alpha}\|\nabla f\|_{L^q(\B_{p,g}(x_0,1))}\, 
	\end{equation}
for all $x,y\in \B_{p,g}(x_0,1)$.
\end{theorem}

\begin{proof}[Proof of Theorem~\ref{Sob-embedding S7}]
Fix any $\e \leq 1/4$,
 and fix $\delta =\delta (n,p, \e)$ according to Theorem~\ref{thm: main thm, Lp def proof}. As in the proof of Theorem~\ref{thm: main thm, Lp def proof}, 
 we obtain a diffeomorphism 
   {$\psi:\B_{p,g}(x_0,5+\e)\rightarrow \tilde\Omega$ where $\B_{p,euc}(0,5)\subset \tilde\Omega\subset \B_{p,euc}(0,5+2\e)$}
 satisfying the properties of Theorem~\ref{claim7}.
  Recall from \eqref{eqn: uniform ball containment} that there exists $R= R(n)$ such that
  {$\B_{p,euc}(0,6)\subset B(0,R)$} for all $p\geq n+1$. 

Let $f\in W^{1,q}_0(\mathcal{B}_{p,g}(x_0,1))$ and extend $f$ by zero to be defined in all of $\tilde \Omega.$ 
Let $h=\psi_*f$ and then naturally extend $h$ by zero to be defined on $B(0,R)$.
  By \eqref{eqn: gradient small error pf} of Theorem~\ref{claim7} and Remark~\ref{rmk: R instead of 1}, for any $q \in (\kappa n, q_0)$, we have 
\begin{equation}\label{eqn: Linfty Sob grad error}
	 {\| \nabla h\|_{L^{q/\kappa}(B(0,R))}}\leq (1+\e) \|\nabla f\|_{L^q(\mathcal{B}_{p,g}(x_0,1))}.
\end{equation}	
So, applying the Morrey-Sobolev embedding on $B(0,R)$ followed by \eqref{eqn: Linfty Sob grad error}, we have 
\begin{equation}
\begin{split}
\|f\|_{L^\infty(\mathcal{B}_{p,g}(x_0,1))}=\|h\|_{L^\infty(B(0,R))}
&\leq \tilde C_{n,q} \|\na h\|_{L^{q/\kappa}(B(0,R))}\\
& \leq   C_{n,q} \|\nabla f\|_{L^q(\mathcal{B}_{p,g}(x_0,1))}.
\end{split}
\end{equation}
This completes the proof of \eqref{So-em}. 

{
To prove \eqref{So-em-2} consider any function $f \in W^{1,q}(\B_{p,g}(x_0,1))$ and  let $h=\psi_*f$ be the function defined on $\tilde\Omega' = \psi(\B_{p,g}(x_0,1))$. By  the Morrey-Sobolev inequality on Euclidean space, for all $x,y\in \B_{p,g}(x_0,1)$,
\begin{equation}\label{e1}
\begin{split}
|f(x)-f(y)| =| h(\psi(x)) - h(\psi(y))|&\leq C_{n,q} \|\nabla h \|_{L^{q/\kappa}(\tilde\Omega' )}\\
&\leq C_{n,q}\|\nabla f\|_{L^{q}(\B_{p,g}(x_0,1))}.
\end{split}
\end{equation}
}
In fact, we may apply the H\"{o}lder Morrey-Sobolev embedding on $\tilde{\Omega}'$ in \eqref{e1} above to see that 
\begin{equation}
\begin{split}
|f(x)-f(y)| =| h(\psi(x)) - h(\psi(y))|&\leq C_{n,q} |\psi(x) - \psi(y)|^{1- n\kappa/q} \|\nabla h \|_{L^{q/\kappa}(\tilde\Omega' )}\\
&\leq C_{n,q}  |\psi(x) - \psi(y)|^{1- n\kappa/q}\|\nabla f\|_{L^{q}(\B_{p,g}(x_0,1))}.
\end{split}
\end{equation}
Hence it suffices to show that  $|\psi(x)-\psi(y)|\leq Cd_{p,g}(x,y)$ for some $C(n,p)>0$ for $x,y\in \B_{p,g}(x_0,1)$. Since $\tilde \Omega\supset \B_{p,euc}(5)$ and by the proof of Theorem~\ref{thm: main thm, Lp def proof}, $\psi(x),\psi(y)\in \B_{p,euc}(2)=B(0,(2S^{-1})^\frac{p}{p-n})$ where $S=S(n,p)$ is given by \eqref{eqn: Euclidean ball and p ball}. Consider the test function $\varphi(z)=\min\{|\psi(x)-\psi(z)|, 2(2S^{-1})^\frac{p}{p-n} \}$ which is compactly supported on $\tilde\Omega$, by   By \eqref{eqn: gradient small error pf} of Theorem~\ref{claim7} and Remark~\ref{rmk: R instead of 1}, we see that $C\varphi$ is admissible function for $d_{p,g}(x,y)$ for some $C(n,p)>0$. The claim follows.

To prove \eqref{So-em global}, let $f\in W^{1,q}(M)$. Fix $x_0\in M$ and let $\psi$ be the diffeomorphism obtained above such that $\psi(x_0)=0$. Let $h=\psi_* f$ be the push-forward of $f$, which is defined on $\tilde\Omega$. Let $\varphi$ be a cutoff function so that $\varphi\equiv1$ on $\B_{p,euc}(0,\frac12)$, $\varphi$ vanishes outside $\B_{p,euc}(0,1)$, and $|\partial\varphi|\leq C_{n}$ (recall  \eqref{eqn: Euclidean ball and p ball}).
 Then by the Morrey-Sobolev inequality on Euclidean space  followed by Theorem~\ref{claim7}, we have 
\begin{equation}
\begin{split}
|f(x_0)|\leq \|h\varphi \|_{L^\infty(\B_{p,euc}(0,1))} 
&\leq  C_{n,q} \|\nabla (h\varphi)\|_{L^{q/\kappa}(\B_{p,euc}(0,1))}\\
&\leq C_{n,q} \left(\|\nabla h\|_{L^{q/\kappa}(\B_{p,euc}(0,1))}+ \|h\|_{L^{q/\kappa}(\B_{p,euc}(0,1))}\right)\\
&\leq C_{n,q} \left(\|\nabla f\|_{L^{q}(M)}+ \|f\|_{L^{q}(M)}\right).
\end{split}
\end{equation}
Since $x_0$ is arbitrarily chosen, this completes the proof.
\end{proof}




\section{Global convergence theorem} \label{s:global_convergence}
In this section we prove Theorem~\ref{global-thm}, which we restate below as Theorem~\ref{global-thm proof}.

{
\begin{theorem}[Theorem \ref{global-thm} restated]\label{global-thm proof} 
For $p\geq n+1$, there exists $\delta=\delta(n,p)>0$ such that if $\{(M_i,g_i,x_i)\}$ is a sequence of complete pointed Riemannian manifolds with bounded curvature such that 
\begin{equation}
R_{g_i}\geq -\delta \quad\text{and}\quad \nu(g_i,2)\geq -\delta,
\end{equation}
then there exists a pointed rectifiable Riemannian space $(M_\infty, g_\infty, x_\infty)$, with  $M_\infty$ topologically a smooth manifold, such that the following holds.
\begin{enumerate}
	\item We have $d_p((M_i, g_i, x_i),  (M_\infty, g_\infty, x_\infty))\to 0$ in the sense of Definition~\ref{def: dp convergence}.
	\item  The space $(X_\infty, g_\infty, x_\infty)$ is  $W^{1,p}$-rectifiably complete and $d_p$-rectifiably complete in the sense of Definitions~\ref{def: rect completeness} and \ref{def: dp RC}  respectively.
\end{enumerate}
\end{theorem}
}
This section is organized in the following way. In Section~\ref{sec: construct limit for limit theorem}, we construct the rectifiable Riemannian space that will ultimately be shown to be the pointed $d_p$ limit in Theorem~\ref{global-thm proof}. Then, in Section~\ref{sec: RC}, we show that this limit is $W^{1,p}$-rectifiably complete. In Section~\ref{sec: dp global}, we establish the pointed $d_p$ convergence of the sequence, and finally in Section~\ref{sec: proof global} we put these pieces together to conclude the proof of Theorem~\ref{global-thm proof}.

\subsection{Constructing the limit space}\label{sec: construct limit for limit theorem}

We first obtain the  rectifiable Riemannian space $(M_\infty, g_\infty, x_\infty)$ that will ultimately be the pointed $d_p$ limit and establish the integral convergence of the metric tensors to this limit.
\begin{proposition}\label{constructon-limit-metric}
Fix $P \geq n+1$. There exists  $\delta=\delta(n, P)>0$ such that the following holds. Suppose $\{(M^n_i,g_i,x_i)\}$ is a sequence of complete pointed Riemannian manifolds with bounded curvature satisfying
\begin{equation}
R_{g_i}\geq -\delta,\qquad \quad \nu(g_i,2)\geq -\delta.
\end{equation}
Then there exists a pointed rectifiable Riemannian space $(M_\infty, g_\infty, x_\infty)$, with  $M_\infty$ topologically a smooth manifold, 
such that up to a subsequence, the following holds.
For any compact subset $\Omega\Subset M_\infty$, 
 we can find subsets $\Omega_i\Subset M_i$  and   diffeomorphisms $\psi_{i, \Omega} : \Omega \to \Omega_i$ such that $\psi_{i,\Omega}(x_\infty)=x_i$ and 
	\begin{equation}\label{eqn: LP conv limit}
	\| \psi_{i,\Omega}^* g_i - g\|_{L^q(\Omega)} \to 0
	\end{equation}
	for any $q \in [1,P]$.
\end{proposition}

\begin{proof}
We proceed in several steps. The first three steps involve constructing the pointed rectifiable Riemannian space $(M_\infty,g_\infty, x_\infty)$ that will later be shown to be the $d_p$ limit of $(M_i, g_i, x_i)$, while in the fourth step we establish the  convergence of the metric tensors \eqref{eqn: LP conv limit}.
\\

{\it Step 1: Constructing the smooth pointed topological manifold $(M_\infty, x_\infty)$.}
Fix $\lambda>0$ to be specified later in the proof. By Theorem~\ref{prop: uniform existence time and small curvature estimates} and Perelman's no local collapsing theorem \eqref{eqn: no local collapsing}, if $\delta =\delta(n, \lambda)$ is taken sufficiently small, then there is a sequence of complete Ricci flow solutions $\{(M_i, g_i(t))\}_{t\in [0,1]}$ such that $g_i(0)=g_i$ and 
\begin{equation}\label{eqn: curv and IR bounds}
\left\{ 
\begin{array}{ll}
|\Rm(g_i(t))|\leq \lambda t^{-1};\\
\mathrm{inj}(g_i(t))\geq c_n\sqrt{t}.
\end{array}
\right.
\end{equation}
By Hamilton's compactness \cite{HamiltonCptness}, after passing to subsequence, $(M_i,g_i(t),x_i)\rightarrow (M_\infty,g_\infty(t),x_\infty)$ in the pointed $C^\infty$-Cheeger-Gromov sense so that $g_\infty(t)$ is a solution to the Ricci flow on $M_\infty\times (0,1]$ also satisfying \eqref{eqn: curv and IR bounds}.
\\

{\it Step 2: Constructing the rectifiable Riemannian metric $g_\infty$.}
{We now construct a rectifiable Riemannian metric as an $L^P$ limit of $g_\infty(t)$ as $t$ tends to zero.

From the previous step, for any pre-compact set $\Omega\Subset M_\infty$ containing $x_0$,  we can find a sequence of maps $\psi_{i,\Omega}:\Omega\rightarrow M_i$, each a  diffeomorphism onto its image, such that $\psi_{i,\Omega}(x_\infty)=x_i$ and $\psi_i^*g_i(t)\rightarrow g_\infty(t)$  smoothly on any compact subsets of $M_\infty \times (0,1]$.
For notational convenience, we will use $g_i$ and $g_i(t)$ to denote $\psi_{i,\Omega}^*g_i$ and $\psi_{i,\Omega}^*g_i(t)$ respectively. By way of a covering argument, it suffices to consider the case $\Omega=B_{g_\infty(1)}(x_\infty,1)$.

 Fix any $\eta>0$. By \eqref{eqn: lp est wrt g1} of Corollary~\ref{cor: useful for limit} (taking $\e =1/10$, $\kappa =2 $ and  $2\PP$ in place of $\PP$), we may find $t_\eta = t_\eta(n, P, \eta) \in (0,1)$ such that for all $s,t \in (0,t_\eta)$, we have 
\begin{equation}\label{eqn: limit metric main estimate existence}
	\int_\Omega |g_i(s) - g_i(t)|_{g_i(1)}^P \, d\vol_{g_i(1)} \leq \eta\, .
\end{equation}
So, passing to the smooth limit as $i \to \infty$, we find that 
\begin{equation}
\int_\Omega |g_\infty(s) - g_\infty(t)|_{g_\infty(1)}^P \, d\vol_{g_\infty(1)} \leq \eta
\end{equation}
for all $s,t \in (0,t_\eta)$. Then, we see that $g_\infty(t)$ is a Cauchy sequence and hence $g_\infty=\lim_{t\rightarrow 0^+}g_\infty(t)$ exists in $L^P(\Omega,g_\infty(1))$. Next, by letting $\Omega_j\Subset M_\infty$ be an exhaustion of $M_\infty,$ we define $g_\infty$ on all of $M_\infty$.
}
\\

{\it Step 3: Verifying that $(M_\infty, g_\infty, x_\infty)$ is a rectifiable Riemannian space.}
Now we claim that $(M_\infty,g_\infty, x_\infty)$ is a rectifiable Riemannian space. 
{Let $m = d\vol_g$. To construct a rectifiable atlas for $(M_\infty, g_\infty, x_\infty)$, let $\{x_j\}\subset M_\infty$ be a collection of points such that $\{ B_{g_\infty(1)}(x_j, 1/2)\}$ covers $M_\infty$ and $\{B_{g_{\infty}(1)}(x_j, 1/4)\}$ are pairwise disjoint. By taking $\lambda>0$ sufficiently small in Step 1, we may apply Lemma~\ref{rmk: good charts under regularity scale} and \eqref{eqn: limit metric main estimate existence} to obtain charts $\phi_j  : U_j \to B_{g_\infty(1)}(x_j ,2)$ where $U_j \subset \R^n$ such that 
\begin{equation}\label{eqn: close to euc limit}
\int_{U_j} | \phi^*_j g_\infty - g_{euc}|_{g_{euc}}^P \, dx \leq C.
\end{equation}
Let $U_{a,j} = \{ x \in U_j : | \phi^*_j g_\infty - g_{euc}|_{g_{euc}} \leq a\}$ and let $\phi_{a,i}:U_{a,j} \to M_\infty$ be defined by $\phi_j|_{U_{a,j}}.$  We easily check from \eqref{eqn: close to euc limit} that $\{ (U_{a,j}, \phi_{a,j})\}_{a,i \in \mathbb{N}}$ is a rectifiable atlas for $(M_\infty, x_\infty, m)$ and that $g_\infty$ is a rectifiable Riemannian metric with respect to this rectifiable atlas.
}
\\

{\it Step 4: Convergence.}
Finally, the $L^P$ convergence \eqref{eqn: LP conv limit}  of $g_i\rightarrow g_\infty$ follows from the $L^P$ convergence in \eqref{eqn: limit metric main estimate existence}
and a diagonal subsequence argument. 
This completes the proof of the proposition.
\end{proof}

\smallskip

\begin{remark}\label{remark-on L^p}{\rm
From the proofs of Proposition~\ref{constructon-limit-metric} and Theorem~\ref{claim7}, we may immediately deduce the following properties of the limit space $(M_\infty, g_\infty, x_\infty)$ constructed in Proposition~\ref{constructon-limit-metric}. 
Given any $x \in M_\infty$, we may apply  Lemma~\ref{rmk: good charts under regularity scale} to obtain a diffeomorphism $\phi : B(0,2) \to \Omega'\subset M_\infty$ such that $\phi(0)=x$ and $\Omega: = \phi(B(0,1))$ satisfies  $B_{g_\infty(1)}(x,1-\e )\subset \Omega \subset B_{g_\infty(1)}(x,1-\e ) $.  In a slight abuse of notation, we identify $\phi^* g$ and $\phi^*g_k$ with $g$ and $g_k$ respectively.  We have  the following estimates:
 \begin{equation}\label{eqn: basic LP estimates RC}\begin{split}
 	\int_B |g_\infty -g_k |_{euc}^P \to 0, 
  & 	\qquad \int_B |g_\infty^{-1} -g_k^{-1}|_{euc}^P  \to 0,\\
 \int_B |g_\infty g^{-1}_k-\text{Id}|_{euc}^P \to 0, & 
 	  \qquad \int_B |g_k g^{-1} -\text{Id} |_{euc}^P \to 0.  
 	   \end{split}
 \end{equation}
 Here the measure of integration can be taken to be $dx, d\vol_g,$ or $d \vol_{g_k}$. Furthermore, for any fixed $p\geq n+1$ and for $\kappa^2 =(p+n)/2n>1$, we may choose $\delta$ in Proposition~\ref{constructon-limit-metric} additionally depending on $p$ and $\kappa$ such that 
\begin{equation}
(1-\e)\| f\|_{L^{p/\kappa}(B(0,1),g_{euc})} \leq  \| f\|_{L^p(B(0,1),g_\infty)}\leq (1+\e)\| f\|_{L^{\kappa p}(B(0,1),g_{euc})}.
\end{equation}
Similarly, we may replace $g_{euc}$ with $g_k$ above.
}
\end{remark}

\subsection{$W^{1,p}$-rectifiable completeness of the limit space}\label{sec: RC}
In this section, we prove that the limiting rectifiable Riemannian space $(M_\infty, g_\infty, x_\infty)$ obtained in Proposition~\ref{constructon-limit-metric} is $W^{1,p}$-rectifiably complete as in Definition~\ref{def: rect completeness}. Given $1< p<\infty$ fixed, we let $W^{1,p}(M_\infty,g_\infty)$  denote the Sobolev space as defined in Section~\ref{sec: w1p structure}. For any function $u \in W^{1,p}(M_\infty,g_\infty)$,  we let $G_{M_\infty,g_\infty,u}$ denote the least $p$-weak upper gradient of $u$, whose existence is guaranteed by Proposition~\ref{prop: basic prop w1p}. We will show that the function $u_a = \phi^*_a u$ is  differentiable a.e. in $U_a$, and thus we may let $|\na_{g_\infty} u| :X\to \R$ be the function defined in charts by 
 \begin{equation}\label{eqn: gradient coords}
|\na_{g_{\infty}} u| (\phi_a(x)) = \left( g_{\infty}^{ij} \pa_i u_a(x) \pa_j u_a(x)\right)^{1/2}
 \end{equation}
 for a.e. $x \in U_a$. 
\begin{proposition}\label{prop: rect completeness} Fix $p \geq n+1$. By choosing $P=P(p,n)$ sufficiently large and thus $\delta = \delta(n,p)$ sufficiently small in Proposition~\ref{constructon-limit-metric},  the limiting rectifiable Riemannian space $(M_\infty, g_\infty, x_\infty)$  constructed in Proposition~\ref{constructon-limit-metric} is $W^{1,p}$-rectifiably complete in the sense of Definition~\ref{def: rect completeness}. That is,
\begin{enumerate}
	\item[(a)] $W^{1,p}(M_\infty, g_\infty)$ is dense in $L^p(M_\infty, g_\infty)$;
\item[(b)]  For any $u \in W^{1,p}(M_\infty, g_\infty)$, the function $u_a = \phi^*_a u $ is weakly differentiable in $U_a$, and thus the function  $|\na_{g_\infty} u|$ in \eqref{eqn: gradient coords} is defined for   $m$-a.e. $x \in M_\infty$. Moreover,  the least weak upper gradient satisfies $G_{M_\infty,g_\infty,u}=|\na_{g_\infty} u|$ for   $m$-a.e. $x \in M_\infty$.
\end{enumerate}
\end{proposition}

In order to prove Proposition~\ref{prop: rect completeness}, we will first establish a localized version of Proposition~\ref{prop: rect completeness}(b) in  Section~\ref{sec: local w1p RC} below (Proposition~\ref{prop: rc}). Using this proposition, we will then prove Proposition~\ref{prop: rect completeness} in Section~\ref{sec: w1p RC pf}.

 \subsubsection{The local estimate}\label{sec: local w1p RC}
 In this section, we will prove the following local version of Proposition~\ref{prop: rect completeness}(b). In order to alleviate notation, we let $g = g_\infty.$ 
 Fix $x \in M_\infty $ and let $\phi: B(0,2) \to \Omega' $ with $\phi(0)=x$ be as in Remark~\ref{remark-on L^p}. As in Remark~\ref{remark-on L^p}, we slightly abuse notation by identifying $\phi^* g$ and $\phi^*g_k$ with $g$ and $g_k$ respectively.
We set $B= B(0,1)$ and denote by $W^{1,p}(B,g)$  the  $W^{1,p}$  space as defined in Section~\ref{sec: w1p structure} with respect to the metric $g$ on the space $B$. We let $G_{B, g,u}$ denote the least $p$-weak upper gradient of $u$ in $W^{1,p}(B,g)$ and $|\dot \g|_g = g(\dot \g , \dot \g)^{1/2}.$  We will use the analogous notation for $g_k$ and $g_{euc}$.
 \begin{proposition}\label{prop: rc} Fix $p \geq n+1$. By choosing  $P=P(n,p)$ sufficiently large and thus $\delta = \delta(n, p)$ sufficiently small in Proposition~\ref{constructon-limit-metric},   any $u \in W^{1,p}(B,g)$ is differentiable $m$-a.e. and we have  $ G_{B, g,u}= |\na_g u| $ for $m$-a.e. $x \in B$.
\end{proposition}

In preparation for the proof of Proposition~\ref{prop: rc}, we first prove some preliminary lemmas. 
  Let $\mathfrak{M}_{g}$ and $\mathfrak{M}_{euc}$ denote  the collection of all absolutely continuous curves with respect to $g$ and $g_{euc}$ respectively in $B$. 
 Note that $\mathfrak{M}_{euc}$ is also the collection of $g_k$ absolutely continuous curves in $B$ for every $k$. 
  \begin{lemma}\label{lem: same rect} Fix $P$ as in Proposition~\ref{constructon-limit-metric}. For any $1\leq p \leq P/2$, we have  $\text{Mod}_{g,p}(\mathfrak{M}_g \setminus \mathfrak{M}_{euc}) = 0$ and $\text{Mod}_{g_{euc}, p }(\mathfrak{M}_{euc}\setminus \mathfrak{M}_g) = 0.$
 \end{lemma}
\begin{proof} Consider the family $\Gamma \subset \mathfrak{M}_g$ of curves $\g:[a,b] \to B(0,1)$ such that 
$
\int_a^b \left| g^{-1}(\gamma)\right|_{euc}^{1/2} |\dot \gamma|_g \, dt = +\infty.
$
Because  $|g^{-1}|_{euc}^{1/2} \in L^{P/2}(B,g)$ by \eqref{eqn: basic LP estimates RC}, it follows from the definition of $\text{Mod}_p$ that   $\text{Mod}_{g,P/2}(\Gamma)= 0$.
Now, notice that $\mathfrak{M}_g \setminus \mathfrak{M}_{euc}\subset \Gamma$. Indeed, for any $\g \in \mathfrak{M}_g \setminus \mathfrak{M}_{euc}$, we apply the Cauchy-Schwarz inequality to find that 
$
	  \int_a^b |g^{-1}(\gamma)|_{euc}^{1/2} |\dot \g|_g \,dt  \geq  \int_a^b |\dot \gamma(t)|_{euc}\,dt  = +\infty,
$
and so $\g \in \Gamma$.
It follows that $\text{Mod}_{g,p}(\mathfrak{M}_g \setminus \mathfrak{M}_{euc}) = 0$. The other claim is proven analogously.
\end{proof}
  In view of Lemma~\ref{lem: same rect},  we will henceforth let $\mathfrak{M} = \mathfrak{M}_g \cap \mathfrak{M}_{euc}$ and will restrict our attention to curves in $\mathfrak{M}$.
 
\begin{lemma} Fix $p\geq n+1$ and let $\kappa^2 =(p+n)/2n>1$. By choosing  $P=P(n,p)$ sufficiently large and thus $\delta = \delta(n, p)$ sufficiently small in Proposition~\ref{construction}, for any family of curves
 $\Gamma \subset \mathfrak{M}$ with $\text{Mod}_{g, p}(\Gamma) = 0$, we have $\text{Mod}_{g_k, p/\kappa}(\Gamma) = \text{Mod}_{g_{euc}, p/\kappa}(\Gamma) =0$.
\end{lemma}
\begin{proof}
Let $\Gamma\subset\mathfrak{M}$ be a family of curves with $\text{Mod}_{g,p}(\Gamma)=0$. By definition, there exists $F\geq0 $ such that $F\in L^p(B,g)$ and $\int_a^b F(\g )|\dot \g|_g \,dt = + \infty$
for all $g \in \Gamma.$ For each $k$, consider the function $F_k = F |g_k g^{-1}|_{euc}^{1/2}$. By Remark~\ref{remark-on L^p}, we have $F_k \in L^{p/\kappa}(B,g_k)$ provided $P$ is chosen sufficiently large depending on $p$ and $n$. Then 
$
\int_a^b F_k(\g )|\dot \g|_{g_k} \,dt \geq \int_a^b F(\g )|\dot \g|_g \,dt = + \infty,
$
and so $\text{Mod}_{g_k, p/\kappa}(\Gamma)=0$. The same proof holds for $g_{euc}$ letting $F_{euc} = F |g^{-1}|_{euc}^{1/2}.$
\end{proof}

\begin{lemma}\label{lem: sobolev sobolev} Fix $p\geq n+1$ and let $\kappa^2 =(p+n)/2n>1$. By choosing  $P=P(n,p)$ sufficiently large and thus $\delta = \delta(n, p)$ sufficiently small in Proposition~\ref{construction},  we have $W^{1,\kappa p}(B,g_{euc}) \subset  W^{1,p}(B,g)\subset W^{1,p/\kappa}(B, g_{euc}).$
\end{lemma}
\begin{proof} We prove the second inclusion; the proof of the first inclusion is analogous. Fix $u \in W^{1,p}(B, g)$, so by definition we have 
	\begin{align}
		| u(\gamma(a)) - u(\gamma(b))| & \leq \int_a^b G_{B,g,u}(\g) |\dot \gamma|_{g} \,dt  \leq \int_a^b G_{B,g,u}(\g) |g|_{euc}^{1/2} |\dot \gamma|_{g_{euc}} \,dt 
	\end{align}
for all $\g \in \mathfrak{M}\setminus \Gamma$, where $\Gamma \subset \mathfrak{M}$ is a family of curves such that $\text{Mod}_{g,p}(\Gamma) = 0$.	 
Letting $H= G_{B,g,u}|g|_{euc}^{1/2}$, we directly see that  $H$ satisfies the weak upper gradient condition for $u$ with respect to $g_{euc}$. Moreover,  by Lemma~\ref{lem: same rect}, we see that $\text{Mod}_{g_{euc}, p/\kappa} (\Gamma) = 0$. This implies that $H$ is a $p/\kappa$-weak upper gradient for $u$ with respect to $g_{euc}$. 
Furthermore, we deduce from Remark~\ref{remark-on L^p} that $H \in L^{p/\kappa}(B, g_{euc})$, and so $u \in W^{1,p/\kappa}(B, g_{euc})$. 
\end{proof}
Note that in Lemma~\ref{lem: sobolev sobolev}, we have used the Newtonian space definition of $W^{1,p/\kappa}(B,g_{euc})$. It is known (see \cite[Theorem 7.13]{Hiwash03}) that this space agrees with the typical definitions of Sobolev spaces on Euclidean space.

\begin{lemma}\label{lem: pq upper gradient} Fix $1 \leq q \leq p <\infty$.
	Let $u \in W^{1,p}(B,g)$ and let $G$ be a $q$-weak upper gradient of $u$. Then $G$ is a $p$-weak upper gradient of $u$. 
\end{lemma}
\begin{proof}
As usual, let $G_{B,g,u}$ denote the least $p$-weak upper gradient of $u$. We will show that $G \geq G_{B,g,u}$ for $m$-a.e. $x \in B$, which implies directly that $G$ is a $p$-weak upper gradient.
To this end, we employ a trick from \cite[Lemma 6.3]{Hiwash03}. 
Since $G$ is a $q$-weak upper gradient for $u$, we know that the weak upper gradient condition is satisfied for $G$ for all $\g \in \mathfrak{M}\setminus \Gamma$, where $\text{Mod}_{g, q}(\Gamma)=0$.
	 By definition, there exists a function $F \in L^q(B,g)$ such that $\int_a^b F(\g) |\dot\g|_g \,dt =+\infty$ for all $g \in \Gamma$. Therefore, the function $G_k = G + F/ k$ is an upper gradient of $u$, and in particular, a $p$-weak upper gradient for $u$. Therefore, $G_k \geq G_{B,g,u}$ for $m$-a.e. $x \in B.$ Now, since $F$ is finite $m$-a.e., we have $G_k \to G$ and thus $G \geq G_{B,g,u}$ for $m$-a.e. $x \in B$.
\end{proof}

We are now ready to prove Proposition~\ref{prop: rc}.
\begin{proof}[Proof of Proposition~\ref{prop: rc}]
We consider two cases: first, the case when $u$ is a smooth function in $B$, and then the general case when $u \in W^{1,p}(B,g)$. In each case, we proceed in two steps: first showing that $|\na_g u|$ is a $p$-weak upper gradient of $u$ with respect to $g$, and then showing that it is the least $p$-weak upper gradient of $u$ with respect to $g$.
\\

{\bf Case 1: $u \in C^\infty(B, g_{euc})$.}

{\it Step 1: $|\na_g u|$ is a $p$-weak upper gradient for $u$ with respect to $g$.}
The main estimate toward showing Step 1 is the following.
Up to passing to a subsequence, we have 
		\begin{equation}\label{eqn: conv on curves}
			\int_a^b |\na_{g_k} u| (\g) |\dot \g|_{g_k} \,dt \to  \int_a^b |\na_g u|(\g) |\dot \g|_{g} \,dt 
		\end{equation}
		for every curve $\g: [a,b] \to B$ in $\mathfrak{M}\setminus \Gamma$, where $\Gamma$ is a family of curves with $\text{Mod}_{g,p}(\Gamma)=0$. Once we establish this fact, it will easily follow that $|\na_g u|$ is a $p$-weak upper gradient for $u$ with respect to $g$.
 	Indeed, since $g_k$ is a smooth metric for each $k$, we know that $|\na_{g_k} u |$ is the least upper gradient for $u$ with respect to $g_k.$ So, for any $\g \in \mathfrak{M}\setminus \Gamma$, we have 
\begin{align}
	| u (\gamma(a)) - u(\gamma(b))| & \leq \int_a^b |\na_{g_k} u | (\gamma) |\dot \gamma|_{g_k} \,dt \to  \int_a^b |\na_{g} u | (\gamma) |\dot \gamma|_{g} \,dt
\end{align}
as $k \to \infty$. Since $\text{Mod}_{g,p}(\Gamma)=0$, it follows that  $|\na_{g} u |$ is a $p$-weak upper gradient for $u$ with respect to $g$.

We employ Lemma~\ref{lem: 5.7 Hiwash} in order to prove \eqref{eqn: conv on curves}. More specifically, for any curve $\g \in \mathfrak{M}$, we have $| \int_a^b |\na_{g_k} u|(\gamma) |\dot \gamma|_{g_k} \, dt - \int_a^b |\na_{g} u|(\gamma) |\dot \gamma|_{g} \, dt | \leq I + II$, where 
	\begin{equation}
	I = \left| \int_a^b (|\na_{g_k} u|-|\na_{g} u|)(\g) |\dot \gamma|_{g_k } \, dt\right|, \qquad II = \left| \int_a^b |\na_{g} u|(\g) \left( |\dot \g|_{g_k}  - |\dot \gamma|_{g}\right) \,dt \right|. 
	\end{equation}
	We claim that terms $I$ and $II$ tend to zero as $k \to \infty$ for all $\g \in \mathfrak{M}\setminus \Gamma$, where $\text{Mod}_{g, p}(\Gamma) =0$. 
	Indeed, consider the sequence of functions $F_k= \left||\na_{g_k} u| - |\na_{g} u|\right||g_{k} \, g^{-1}|_{euc}^{1/2}$, and notice that $ I \leq \int_a^b F_k (\gamma) |\dot \g|_g \,dt$.
Thanks to \eqref{eqn: basic LP estimates RC}, we see that $F_k \to 0$ in $L^p(B, g)$, provided $P$ is large enough depending on $p$.
Then, Lemma~\ref{lem: 5.7 Hiwash} implies that, after passing to a subsequence, $I \to 0$ for all $\g \in \mathfrak{M}\setminus\Gamma_{I}$ where $\Gamma_{I}\subset \mathfrak{M}$ is a family of curves with $\text{Mod}_{g,p}(\Gamma_{I})=0$.

We argue in a similar fashion for term $II$. Consider the sequences $\tilde{F_k} = |\na_{g} u| \big( |g_k g^{-1}|_{euc}^{1/2} - 1\big) $ and $\hat{F}_k =  |\na_{g} u| \big(1 - |g g_k^{-1}|_{euc}^{-1/2}\big)$. Since $|\na_{g} u| \geq0$, we have $	\int_a^b |\na_{g} u|(\gamma) \left( |\dot \g|_{g_k} - |\dot \g |_{g} \right) \,dt 
\leq 	 \int_a^b \tilde{F}_k (\gamma) |\dot \g |_{g}  \, dt$
 and $
	\int_a^b |\na_{g} u|(\gamma) \left(  |\dot \g |_{g} -|\dot \g|_{g_k}  \right) \,dt 
	 \leq \int_a^b \hat{F}_k (\gamma) |\dot \g |_{g}  \, dt .
$
Furthermore, thanks to \eqref{eqn: basic LP estimates RC}, we see that $\tilde{F}_k, \hat{F}_k\to 0$ in $L^p(B,g)$ provided that $P$ is chosen sufficiently large depending on $p$. So, by Lemma~\ref{lem: 5.7 Hiwash}, we see that, after passing to a further subsequence, we have $II \to 0$ for all $\g \in \mathfrak{M}\setminus \Gamma_{II}$ where $\text{Mod}_{g,p}(\Gamma_{II})=0$. Letting $\Gamma = \Gamma_{I} \cup \Gamma_{II}$, we conclude that \eqref{eqn: conv on curves} holds. This completes Step 1.\\
\medskip

{\it Step 2: $|\na_g u|$ is the least $p$-weak upper gradient of $u$ with respect to $W^{1,p}(B,g)$.} Let $H \in L^p(B, g)$ be any $p$-weak upper gradient of $u$ with respect to $g$. By definition,  we have
\begin{align}
	|u(\g(a)) - u(\g(b))| & \leq \int_a^b H(\gamma) |\dot \g |_g \,dt  \leq \int_a^b H |g \,g_k^{-1}|_{euc}^{1/2} |\dot \g |_{g_k } \, dt
\end{align}
for all $\g \in \mathfrak{M}\setminus \Gamma$ where $\Gamma$ is a family of curves such that $\text{Mod}_{g,p}(\Gamma) =0$,
In particular, letting  $H_k = H |g \,g_k^{-1}|_{euc}^{1/2}$, we see immediately that $H_k$ satisfies the upper gradient condition for every $\g \in \mathfrak{M}\setminus \Gamma$. By Lemma~\ref{lem: same rect}, we see that $\text{Mod}_{g_k, p/\kappa}(\Gamma ) = 0$, and thus $H_k$ is a $p/\kappa$-weak upper gradient for $u$ with respect to $g_k$. 
Moreover, it follows from Remark~\ref{remark-on L^p} that $H_k \in L^{p/\kappa}(B, g)\cap L^{p/\kappa}(B, g_k)$ and that $H_k \to H $ in $L^{p/\kappa}(B, g)$. In particular,  	$H_k \to H $ for $m$-a.e. $x \in B$.

 Similarly, since  $u \in C^{\infty}(B,g_{euc})$, we deduce from \eqref{eqn: basic LP estimates RC} that $|\na_{g_k} u| \to |\na_g u|$ in $L^p(B,g)$, and therefore $m$-a.e.
 Since $|\na_{g_k} u|$ is the least $q$-weak upper gradient for $u$ with respect to $g_k$ for any $q$, we have that $|\na_{g_k} u| \leq H_k$ $m$-a.e. Passing $k\to \infty$, it follows that $|\na_{g} u|\leq H$, and so $|\na_{g} u|$ is the least $p$-weak upper gradient for $u$ with respect to $g$.\\
 \medskip
%
%
%
%

{\bf Case 2: $u \in W^{1,p}(B,g)$.}

{\it Step 1: $|\na_g u |$ is a $p$-weak upper gradient of $u$ in $W^{1,p}(B,g)$.}
 By Lemma~\ref{lem: sobolev sobolev}, we know that $u \in W^{1,p/\kappa}(B,g_{euc})$, and thus we may find a sequence $u_i \in C^\infty(B,g_{euc})$ such that $u_i \to u$ in $W^{1,p/\kappa}(B,g_{euc})$. 
So,  using Remark~\ref{remark-on L^p}, we see that $|\na_g u_i| \to |\na_g u|$ in $L^q(B,g)$ for $q =p/\kappa^2$. Therefore, applying Lemma~\ref{lem: 5.7 Hiwash} once again, after passing to a subsequence, we have 
\begin{equation}
\int_a^b |\na_g u_i | (\g) |\dot \g|_g\,dt \to \int_a^b |\na_g u|(\g) |\dot \g|_g \,dt 
\end{equation}
as $i \to \infty$ for all $\g \in \mathfrak{M} \setminus \Gamma_0$, where $\text{Mod}_{g, q}(\Gamma_0)=0.$ Moreover, applying Case 1 to each smooth function $u_i$, we  know that $|\na_g u_i|$ is a $q$-weak upper gradient for $u_i$  and thus 
\begin{equation}\label{eqn: i pweak}
|	u_i(\g (a)) - u_i (\g(b))|  \leq \int_a^b |\na_{g} u_i| (\g)  |\dot \g|_g\,dt
\end{equation}
for all $\g \in \mathfrak{M} \setminus \Gamma_i$, where $\Gamma_i $ is a family of curves such that $\text{Mod}_{g, q}(\Gamma_i)=0.$
Now, since $p/\kappa>n$, we have that $u_i \to u$ uniformly by the Morrey-Sobolev inequality on Euclidean space. 
 So, letting $\Gamma = \Gamma_0 \cup \bigcup \Gamma_i$, we see that $\text{Mod}_{g,q}(\Gamma)=0$, and letting $i \to \infty$ in \eqref{eqn: i pweak}, we find that 
$
|	u(\g (a)) - u(\g(b))|  \leq \int_a^b |\na_g u|(\g)  |\dot \g|_g\,dt
$
for all $\g \in \mathfrak{M}\setminus \Gamma$.
This proves that $|\na_g u|$ is a $q$-weak upper gradient of $u$ with respect to $g$. Applying Lemma~\ref{lem: pq upper gradient}, we see that $|\na_g u|$ is also $p$-weak upper gradient with respect to $g$, completing Step 1.\\

{\it Step 2: $|\na_g u |$ is the least $p$-weak upper gradient of $u$ in $W^{1,p}(B,g)$.} 
 Let $v_i \equiv u_i - u \in W^{1,p}(B,g)$.  As noted in the previous step, $v_i\to 0$ in $W^{1,p/\kappa}(B, g_{euc})$, so in particular $|\na_g v_i| \to 0$ $m$-a.e. Moreover, applying the previous step, $|\na_g v_i|$ is a $p$-weak upper gradient of $v_i$ in $W^{1,p}(B,g)$. 
 
 Now, consider any $p$-weak upper gradient $H \in L^p(B,g)$ for $u$ with respect to $g$ and let $H_i = H+ |\na_g v_i|$, which converges pointwise $m$-a.e. to $H$.  Applying the triangle inequality, we see that $H_i$ is a $p$-weak upper gradient for $u_i$. Moreover, from Case 1, we know that $|\na_g u_i|$ is the least $p$-weak upper gradient for $u_i$ with respect to $g$, and hence 
 \begin{equation}\label{eqn: 22}
 H_i \geq |\na_g u_i| 
 \end{equation} for $m$-a.e. $x \in B(0,1)$. Since $|\na_g u_i| \to |\na_g u|$ in $L^q(B,g_{euc})$, after passing to further subsequences the sequence also converges pointwise a.e, and thus $m$-a.e. Thus sending $i\to \infty$ in \eqref{eqn: 22}, we see that $|\na_g u|$ is the least $p$-weak upper gradient for $u$ with respect to $g$. This concludes the proof of Step 2 and thus of the proposition.
\end{proof}
\subsubsection{Proof of Proposition~\ref{prop: rect completeness}}\label{sec: w1p RC pf}

We are now ready to prove Proposition~\ref{prop: rect completeness}.
\begin{proof}[Proof of Proposition~\ref{prop: rect completeness}]
We first prove part (b). Fix $u \in W^{1,p}(M_\infty, g_\infty)$. By Proposition~\ref{prop: rc}, we know that the pull-back of $u$ is weakly differentiable in charts and that $|\na_{g_\infty}u|$ is defined for a.e $x \in M_\infty$. To show that $|\na_{g_\infty}| = G_{M_\infty, g_\infty, u}$, we proceed in two steps, first showing that $|\na_{g_\infty} u|$ is a $p$-weak upper gradient of $u$  and then showing that it is the least $p$-weak upper gradient.

  We claim that $|\na_{g_\infty} u|$ is a $p$-weak upper gradient of $u$ with respect to $W^{1,p}(M_\infty, g_\infty)$. To this end, let $\{x_j\}\subset M_\infty$ be a collection of points such that $\{ B_{g_\infty(1)}(x_j, 1/2)\}$ covers $M_\infty$ and $\{B_{g_{\infty}(1)}(x_j, 1/4)\}$ are pairwise disjoint. For each $j$, let $\phi_j: B(0,2)\to \Omega_j \subset B_{g_\infty(1)}(x_j, 1/2)$ be as in Remark~\ref{remark-on L^p}.

 Consider any absolutely continuous curve $\g: [a,b] \to M_\infty$. The continuous image under $\g$ of the compact set $[a,b]$ is compact, and thus the image of $\g$ in $M_\infty$ intersects finitely many of the $\Omega_j$. Take a finite partition $a=a_0 < a_1 < \dots < a_N = b$ of the interval $[a,b]$  such that the image of $[a_i, a_{i+1}]$ is contained entirely in $\Omega_j$ for one $j$. 
   Then, applying the triangle inequality and Proposition~\ref{prop: rc}, we see that 
\begin{equation}\begin{split}
|u(\g(a)) - u(\g(b))|&  \leq \sum_{i=1}^N |u(\g(a_{i-1}))- u(\g (a_i)) |  \\
& \leq \sum_{i=1}^N \int_{a_{i-1}}^{a_i} |\na_{g_\infty} u| |\dot \g |_{g_\infty} \,dt = \int_a^b  |\na_{g_\infty} u| |\dot \g |_{g_\infty} \,dt .
\end{split}\end{equation}
It follows that $|\na_{g_\infty} u|$ is a $p$-weak upper gradient for $u$ with respect to $W^{1,p}(M_\infty,g_\infty)$. 

Now we show that $|\na_{g_\infty} u|$ is the least $p$-weak upper gradient of $u$ with respect to $W^{1,p}(M_\infty, g_\infty)$, thus proving (b). Notice that, for each $j$, the restriction of $G_{M_\infty,g_\infty,u}$ to $\Omega_j$ is a $p$-weak upper gradient for $u$ with respect to $W^{1,p}(\Omega_j,g_\infty)$. So, since we know from Proposition~\ref{prop: rc} that $|\na_{g_\infty}u|$ is the least $p$-weak upper gradient of $u$ with respect to $W^{1,p}(\Omega_j, g_\infty)$, it follows that $|\na_{g_\infty} u| \leq G_{M_\infty,g_\infty, u}$ for $m$-a.e. $x \in \Omega_j$. Since the collection $\{\Omega_j\}$ covers $M_\infty,$  we thus see that $|\na_{g_\infty} u| \leq G_{M_\infty,g_\infty, u}$ for $m$-a.e. $x \in M_\infty.$ This completes the proof of (b).

Finally, we prove (a). Fix any $v \in L^p(M_\infty, g_\infty)$ and $\e>0$. We wish to show that there exists $u \in W^{1,p}(M_\infty, g_\infty)$ such that $\| v-u\|_{L^p(M_\infty, g_\infty)}\leq \e.$ It is apparent that bounded and compactly supported functions are dense in $L^p(M_\infty, g_\infty)$, and thus we may assume without loss of generality that $v$ is bounded and compactly supported. 
Let $\{x_j\}_{j=1}^N\subset M_\infty$ be a finite collection such that the support of $v$ is contained in $\cup_{j=1}^N \Omega_j$, where the $\Omega_j$ are defined as in the previous step. Let $\{\psi_j\}$ be a partition of unity subordinate to $\{\Omega_j\}_{j=1}^N$ such that $\phi_j^*\psi_j$ is a smooth function in $B(0,1)$. Since $u$ is bounded, for each $j$ we  have $\phi_j^* u  \in L^{\kappa p}(B(0,1), g_{euc})$. So, we may find a smooth function $\tilde{v}_j $ such that $\| \phi_j^*u - \tilde{v}_j\|_{L^{\kappa p }(B, g_{euc} )} \leq \e/2N.$ Thus, by Remark~\ref{remark-on L^p}, we have $\| u -v_j \|_{L^p(\Omega_j, g_\infty)} = \| \phi_j^*u - \tilde{v}_j\|_{L^p(B, g)} \leq \e/N.$ Here we let $v_j= \phi_* \tilde{v}_j.$ Finally, we let $v = \sum_{j=1}^N \psi_j v_j$. Thanks to Lemma~\ref{lem: sobolev sobolev} and part (b) above, we see that $v \in W^{1,p}(M_\infty, g_\infty)$.  Finally,
\begin{equation}
\| u-v\|_{L^p(M_\infty, g_\infty) } \leq \sum_{j=1}^N \| u -v_j \|_{L^p(\Omega_j, g_\infty)} \leq \e.
\end{equation}
This completes the proof of (a) and thus of the proposition. 
\end{proof}

\subsection{Convergence in $d_p$}\label{sec: dp global}

In this section, we establish the following proposition, which is Proposition~\ref{global-thm proof}(1).
\begin{proposition}\label{prop: dp convergence global} Fix $p\geq  n+1$. We may choose $P= P(n,p)$ sufficiently large and thus $\delta =\delta(n,p)$ sufficiently small in Proposition~\ref{constructon-limit-metric} such that if $(M_i, g_i, x_i)$ and $(M_\infty, g_\infty , x_\infty)$ are as in Proposition~\ref{constructon-limit-metric}, then, after passing to subsequence, 
\begin{equation}\label{eqn: dp convergence limit space}
	d_p ((M_i, g_i, x_i), (M_\infty, g_\infty , x_\infty)) \to 0.
\end{equation} 
\end{proposition}
\vspace{.1cm}

In order to prove Proposition~\ref{prop: dp convergence global}, we will   first  show the following local version.
\begin{proposition}\label{prop: dp convergence global but local} Fix $p\geq  n+1$. We may choose $P= P(n,p)$ sufficiently large and thus $\delta =\delta(n,p)$ sufficiently small in Proposition~\ref{constructon-limit-metric} such that if $(M_i, g_i, x_i)$ and $(M_\infty, g_\infty , x_\infty)$ are as in Proposition~\ref{constructon-limit-metric}, the following holds. For any compact set $\Omega\Subset M_\infty$, after passing to subsequence,  we can find $\Omega_i\Subset M_i$ such that 
\begin{equation}
d_{p}\left((\Omega_i,g_i),(\Omega,g_\infty) \right)\rightarrow 0	
\end{equation}
 as $i\rightarrow +\infty$  in the sense of Definition~\ref{def: dp convergence}. 
\end{proposition}

Before proving the propositions, we establish two lemmas that will be needed in the proofs. 
First, we show that the supremum in the definition of $d_{p,g,\Omega}(x,y)$ is achieved.
\begin{lemma}\label{existence-minimizer}
For $p\geq n+1$, there exists $\delta(n,p)>0$ such that the following holds. Let $(M^n,g)$ be a complete Riemannian manifold with bounded curvature such that
$$R_{g}\geq -\delta, \quad \quad \nu(g,2)\geq -\delta.$$
Then for any bounded subset $\Omega\Subset M$ and $x,y\in \Omega$, there exists a function $f \in W^{1,p}(\Omega)$ such that $d_{p,g,\Omega}(x,y) = |f(x)-f(y)|$ and $\int_{\Omega} |\na f|^p \, d\vol_{g} =1$.
\end{lemma}
\begin{proof}
Consider a maximizing sequence $f_i \in W^{1,p}(\Omega)$ with $\int_{\Omega}|\na f_i|^p \,d \vol_g \leq 1$, $f_i(x) =0$, and $f_i(y) \to d_{p,g,\Omega}(x,y).$ 
Choosing $\delta = \delta(n,p)$ according to Theorem~\ref{Sob-embedding S7}, we may apply the Theorem~\ref{Sob-embedding S7} to see that $|f_i | \leq C =C(n,p, g, \Omega)$ on $\Omega$. In particular, the sequence $f_i$ is uniformly bounded in $W^{1,p}(\Omega)$. Hence after passing to a subsequence, $f_i$ converges weakly in $W^{1,p}(\Omega)$ to a function $f \in W^{1,p}(\Omega)$ with $\int_{\Omega} |\na f|^p\,d\vol_g \leq 1.$ Moreover, by applying Theorem~\ref{Sob-embedding S7} to rescalings of the metric, we see that each $f_i$ is continuous with a modulus of continuity that is uniform in $i$. So, by the Arzel\`{a}-Ascoli theorem, $f_i$ converges uniformly to $f$. In particular, $f(x) =0$ and $f(y) = d_{p,g,\Omega}(x,y)$. Finally, note that  $\int_{\Omega} |\na f|^p d\vol_g=1$, otherwise a multiple $\kappa f$ for $\kappa>1$ would be an admissible test function for $d_{p,g,\Omega}(x,y)$. This completes the proof.
\end{proof}

Next, we use Gehring's lemma and the doubling property of the $d_p$ metrics to show that a function $f\in W^{1,p}(\Omega)$ achieving the supremum in  $d_{p,g,\Omega}(x,y)$ enjoys higher integrability properties.
\begin{lemma}\label{improved-estimate-GehringLemma}
Fix $p\geq n+1$. There exist constants $\delta(n,p)>0$, $\kappa(n,p)>1$, and $C_0(n,p)>0$  such that the following holds. Let $(M,g)$ be a complete Riemannian manifold with bounded curvature and 
$$R_{g}\geq -\delta,\quad \qquad \nu(g,2)\geq -\delta.$$
Fix $\Omega \Subset M$ and  $x,y\in \Omega$, and let $f\in W^{1,p}(\Omega)$ be a function achieving the supremum in $d_{p,g,\Omega}(x,y)$. Then for all $\B_{p,g}(z,4r)\subset \Omega$ such that  $r<\frac{1}{10}\min\{d_{p,g}(x,y),1\}$, we have
\begin{equation}
\bigg(\fint_{\B_{p,g}(z,r)}|\nabla f|^{ p\kappa } d\vol_g \bigg)^{\frac{1}{ p\kappa }}\leq C_0 \bigg(\fint_{\B_{p,g}(z,4r)}|\nabla f|^{ p} d\vol_g \bigg)^{\frac{1}{ p}}.
\end{equation}
\end{lemma}
\begin{proof}

Let $x,y\in \Omega\Subset M$ and $ f\in W^{1,p}(\Omega)$ be a function such that $\int_\Omega|\nabla f|^p d\vol_g = 1$ and $ d_{p,g,\Omega}(x,y)=|f(x)-f(y)|.$ We may assume without loss of generality that $f(x)=0$. Fix $r<\frac{1}{10}\min\{d_{p,g}(x,y),1\}$ and $z\in M$  such that $\B_{p,g}(z,4r)\subset \Omega$. 

{\it Step 1:} Fix any $q \in (n,p)$. We claim that there exists  $C_{n,p,q}>0$ such that 
\begin{equation}\label{reverse-Holder}
\bigg(\fint_{B_{p,g}(z,r)} |\nabla f|^{p} d\vol_g \bigg)^{1/p}\leq C_{n,p,q}\bigg(\fint_{B_{p,g}(z,4r)} |\na f|^qd\vol_g \bigg)^{1/q}.
\end{equation}

To see this, notice that by the definition of $d_{p,g,\Omega}$,  the function $f$ satisfies 
\begin{equation}\label{eqn: min pb}
\int_\Omega |\nabla f|^pd\vol_g =\inf\left\{ \int_\Omega |\nabla h|^p d\vol_g \ : \ h\in W^{1,p}(\Omega),\ |h(x)-h(y)|=d_{p,g,\Omega}(x,y)\right\}.
\end{equation}
So, computing the Euler-Lagrange equation associated to \eqref{eqn: min pb}, we see that  for  any $h\in W^{1,p}(\Omega)$ satisfying $h(x)=h(y)$, we have
\begin{equation}\label{p-harmonic}
\int_\Omega |\nabla f|^{p-2} \langle \nabla f,\nabla h \rangle d\vol_g=0.
\end{equation}
 
We first consider the case when $x\in \B_{p,g}(z,2r)$ and $y\notin \B_{p,g}(z,2r) $.
 Let $\phi$ be a cutoff function on $M$ such that $\phi\equiv 1$ on $\B_{p,g}(z,r)$ and vanishes outside $\B_{p,g}(z,2r)$. We will make the construction more precise below. 
Since  $f(x)=0$, the function $h=f\phi^p$ satisfies $h(x)=h(y)=0$. Therefore, choosing $h=f\phi^p$ as a test function in \eqref{p-harmonic} and applying Young's inequality, we find that
\begin{equation}\label{eqn: energy est a}
\begin{split}
\int_\Omega |\na f|^p \phi^p\, d\vol_g &\leq p\int_\Omega |\na f|^{p-1}\phi^{p-1}\, |\na \phi| f \, d\vol_g \\
&\leq \e \int_\Omega |\na f|^p \phi^p\,d\vol_g + \frac{p-1}{\e} \|f \|_{L^\infty(\B_{p,g}(z,2r))}^p\int_\Omega |\na \phi|^p \,d\vol_g.
\end{split}
\end{equation}
Consequently, absorbing the first term on the right-hand side of \eqref{eqn: energy est a} and using the volume estimate \eqref{eqn: volumes close p}, we find that 
\begin{equation}\label{eqn: energy}
\bigg(\fint_{\B_{p,g}(z,r)} |\nabla f|^{p} d\vol_g\bigg)^{1/p}\leq C \|f\|_{L^\infty(\B_{p,g}(z,2r))} \bigg(\fint_{\B_{p,g}(z,2r)} |\nabla \phi|^p d\vol_g\bigg)^{1/p}
\end{equation}
where $C=C(n,p)$.
By applying  the Sobolev inequality of Lemma~\ref{Sob-embedding S7} and the volume estimate \eqref{eqn: volumes close p} to the rescaled metric  $\tilde g=\lambda^2 g$ where $4r\lambda^{1-\frac{n}{p}}=1$ so that $\B_{p,\tilde g}(z,1)=\B_{p,g}(z,4r)$, we find that
\begin{equation}\label{Sobolev-Gehring}
\begin{split}
\|f\|_{L^\infty(\B_{p,g}(z,2r))}\leq C_{n,p,q} r^\frac{p}{p-n} \bigg(\fint_{\B_{p,g}(z,4r)} |\na f|^qd\vol_g\bigg)^{1/q}.
\end{split}
\end{equation}

On the other hand, let us now construct a good cutoff function $\phi$.  Begin by constructing a cutoff function $\Phi$ on Euclidean space such that $\Phi=1$ on $\B_{p,euc}(0,\frac23)$, vanishes outside $\B_{p,euc}(0,\frac45)$ and $|\partial \Phi|^{10p}\leq C_{n,p}\Phi^{10p-1}$. Let $\phi$ be its pull-back along the diffeomorphism obtained from $\tilde g$.  A similar argument as in Theorem~\ref{Sob-embedding S7} using Theorem~\ref{thm: decomposition theorem} shows that
 \begin{equation}\label{eqn: cutoff bd}
 \bigg(\fint_{\B_{p,g}(z,2r)} |\nabla \phi|^p d\vol_g\bigg)^{1/p}\leq C_{n,p}r^\frac{-p}{p-n}.
 \end{equation}
By combining \eqref{eqn: energy}, \eqref{Sobolev-Gehring}, and \eqref{eqn: cutoff bd}, we conclude that  \eqref{reverse-Holder} holds in this case.

Next, consider the case when $y\in \B_{p,g}(z,2r)$, and so $x\notin \B_{p,g}(z,2r)$. Applying the same argument to the function $\tilde f=d_{p,g,\Omega}(x,y)-f$, we deduce the same inequality \eqref{reverse-Holder} because $|\nabla f|=|\na\tilde f|$. Finally, if $x,y\notin \B_{p,g}(z,2r)$, we consider $\tilde f=f-\bar f$ where $\bar f\in \mathbb{R}$ so that $\int_{\B_{p,g}(z,2r)} \tilde f d\vol_g=0$. Hence, we still have \eqref{Sobolev-Gehring} and thus the proof above can be carried over without any change. We thus have \eqref{reverse-Holder} in this case as well.

{\it Step 2:} 
Since $(M,d_{p,g},d\vol_g)$ is a metric measure space and the measure $d\vol_g$ is a doubling measure with respect to $d_{p,g}$ for scales $r \leq 1$, by choosing $q=(n+p)/2$ we may apply the form of Gehring's lemma in \cite[Theorem 3.1]{Outi2007} to see that there is $\tilde p>p$ and $C_0>1$ depending only on $n,p$ such that for all $\B_{p,g}(z,4r)\subset \Omega$ where $r<\frac{1}{10}\min\{d_{p,g}(x,y),1\}$, we have 
\begin{equation}\label{reverse-Holder-2}
\left(\fint_{\B_{p,g}(z,r)}|\nabla f|^{\tilde p} d\vol \right)^{\frac{1}{\tilde p}}\leq C_0 \left(\fint_{\B_{p,g}(z,4r)}|\nabla f|^{ p} d\vol \right)^{\frac{1}{ p}}.
\end{equation}

Note that the reverse H\"older inequality assumption on \cite[Theorem 3.1]{Outi2007} is stated on balls of same radius. It is clear from the proof that 
{\eqref{reverse-Holder}} suffices, see also the classical Gehring's Lemma \cite{gehring1973lp} on Euclidean space. Letting $\kappa = \tilde{p}/p$ completes the proof.
\end{proof}

\begin{remark}\label{rmk: limit space}
{\rm 
Assume that $P= P(n,p)$ is taken sufficiently large and thus $\delta = \delta(n,p)$ is taken sufficiently small in Proposition~\ref{constructon-limit-metric} so that we may apply Proposition~\ref{prop: rect completeness} to the limit space  $(M_\infty, g_\infty, x_\infty)$ constructed in Proposition~\ref{constructon-limit-metric}. Since $(M_\infty, g_\infty, x_\infty)$ is $W^{1,p}$-rectifiably complete, we see from the proof of Proposition~\ref{constructon-limit-metric} that the $\e$-regularity theorem, Theorem~\ref{thm: main thm, Lp def proof},  and  Sobolev inequalities of Theorem~\ref{Sob-embedding S7} pass to the limit  $(M_\infty, g_\infty, x_\infty)$.
In particular, the proofs and conclusions of Lemmas~\ref{existence-minimizer} and \ref{improved-estimate-GehringLemma} also hold for $(M_\infty, g_\infty, x_\infty)$.}
\end{remark}


We are now ready to prove Proposition~\ref{prop: dp convergence global but local}.
\begin{proof}[Proof of Proposition~\ref{prop: dp convergence global but local}]Assume that $P =P(n,p)$ in Proposition~\ref{constructon-limit-metric} is taken large enough to apply Proposition~\ref{prop: rect completeness} and thus Remark~\ref{rmk: limit space}.
Fix $\Omega\Subset M_\infty$ and choose $\bar\Omega\subsetneq \Omega\subsetneq \tilde \Omega\Subset M_\infty$. 
From the Cheeger-Gromov convergence established in as in Step 1 of the proof of Proposition~\ref{constructon-limit-metric}, we may assume that $g_i$ are all defined on $\tilde\Omega$ via  the diffeomorphisms $\psi_i : \tilde{\Omega} \to \tilde{\Omega}_i \subset M_i$. 

{\it Step 1:} We claim that $d_{p,g_i,\Omega}(x,y)\rightarrow d_{p,g_\infty,\Omega}(x,y)$ for every $x,y\in \Omega$. 

Fix $x,y\in \Omega$.
By Lemma~\ref{existence-minimizer} and Remark~\ref{rmk: limit space}, we may find $f \in W^{1,p}(\tilde{\Omega}, g_\infty)$ such that $d_{p,g_\infty,\tilde\Omega}(x,y)=|f(x)-f(y)|$ and $\int_{\tilde\Omega}|\nabla^{g_\infty} f|^pd\vol_{g_\infty}=1$. Furthermore, applying Lemma~\ref{improved-estimate-GehringLemma} (see Remark~\ref{rmk: limit space}) to a covering of $\tilde{\Omega}$ by $g_\infty,p$-balls, we find that 
\begin{equation}\label{eqn: Omega tilde est}
	\left( \int_{\Omega} |\na^{g_\infty} f|^{\kappa p} \, d\vol_{g_\infty} \right)^{1/\kappa p} \leq C_1,
\end{equation}
where $C_1$ depends on $n,p,\Omega,\tilde\Omega$, and $d_{p,g_\infty}(x,y)$,  and   where $\kappa =\kappa(n, p) >1$ is the constant obtained in Lemma~\ref{improved-estimate-GehringLemma}.
 Here we have used the fact that the topologies induced by $d_{g_\infty(1)}$ and $d_{p,g_\infty}$ are equivalent. This can be seen by taking a multiple of $ d_{g_\infty(1)}$ as a test function for $d_{p,g_\infty}$ together with  \eqref{eqn: close to euc limit} and Morrey-Sobolev inequality.
 
Fix $\e>0$. We aim to show that  $f/(1+\e)$ is an admissible test function for  $d_{p,g_i,\Omega}(x,y)$ for $i$ sufficiently large. 
By Proposition~\ref{constructon-limit-metric} and the $W^{1,p}$-rectifiable completeness of the limit space, we see that $|\na^{g_i} f |^p = (1+\mathcal{E}_{1,i}) |\na^{g_\infty} f|^p $ and  $d \vol_{g_i} = (1+\mathcal{E}_{2,i})d\vol_{g_\infty}$, where $\mathcal{E}_{1,i}, \mathcal{E}_{2,i}  : \Omega \to \R$ are errors such that $\mathcal{E}_{1,i}, \mathcal{E}_{2,i} \to 0$ in $L^q(\Omega, g_\infty)$ for any $q \leq P/(n+p)$. Therefore, provided we choose $P =P(n,p)$ large enough so that the  H\"{o}lder conjugate $\kappa'$ of $\kappa$ is less than $P/2(n+p)$, we use H\"{o}lder's inequality and \eqref{eqn: Omega tilde est} to see that  
 \begin{equation}
 	\begin{split}
 		\left( \int_\Omega |\na^{g_i} f|^p \,d\vol_{g_i} \right)^{1/p}& = \left( \int_\Omega |\na^{g_\infty} f|^p (1 + \mathcal{E}_{1,i} + \mathcal{E}_{2,i} + \mathcal{E}_{1,i}\mathcal{E}_{2,i}) \,d\vol_{g_\infty} \right)^{1/p}\\
 		& \leq 1 +  \left( \int_\Omega |\na^{g_\infty} f|^p (\mathcal{E}_{1,i} + \mathcal{E}_{2,i} + \mathcal{E}_{1,i}\mathcal{E}_{2,i}) \,d\vol_{g_\infty} \right)^{1/p}\\
 		& \leq 1 + \| \na^{g_\infty} f\|_{L^{\kappa p}(\Omega)} \| \mathcal{E}_{1,i} + \mathcal{E}_{2,i} + \mathcal{E}_{1,i}\mathcal{E}_{2,i}\|_{L^{\kappa'}(\Omega)} \leq 1 + \e,
 	\end{split}
 \end{equation}
where the final inequality holds for $i$ sufficiently large and makes use of \eqref{eqn: Omega tilde est}. 
Therefore, $f/(1+\e)$ is an admissible test function for $d_{p,g_i, \Omega}(x,y)$ and consequently
\begin{equation}
d_{p,g_\infty,\tilde\Omega}(x,y)\leq (1+\e)d_{p,g_i,\Omega}(x,y)
\end{equation}
for $i$ sufficiently large. Letting $i\rightarrow +\infty$ and then $\e \to 0$, we find that 
\begin{equation}
	d_{p,g_\infty,\tilde\Omega}(x,y)\leq \liminf_{i\rightarrow +\infty}d_{p,g_i,\Omega}(x,y).
\end{equation}
 Then, an  argument analogous to  the proof of Lemma~\ref{p-continuity} shows that $d_{p,g,\Omega}$ is continuous with respect to domain and hence we may let $\tilde\Omega $ tend to $ \Omega$ to conclude that 
\begin{equation}
d_{p,g_\infty,\Omega}(x,y)\leq \liminf_{i\rightarrow +\infty}d_{p,g_i,\Omega}(x,y).
\end{equation}

We now apply the analogous argument to $x,y \in \bar \Omega$ with the roles of $g_\infty$ and $g_i$ swapped, making use of the crucial fact that the upper bound in Lemma~\ref{improved-estimate-GehringLemma} and thus \eqref{eqn: Omega tilde est} are uniform in $i$.  We find that for any $\e>0$,
$d_{p,g_i,\Omega}(x,y)\leq (1+\e)d_{p,g_\infty,\bar \Omega}(x,y) $
for $i$ sufficiently large and hence
\begin{equation}
\limsup_{i\rightarrow +\infty}d_{p,g_i,\Omega}(x,y)\leq d_{p,g_\infty,\Omega}(x,y).
\end{equation}
This completes the proof of Step 1.
\\

{\it Step 2:} We claim that
\begin{equation}
	d_{GH}((\Omega, d_{p, g_i, \Omega}), (\Omega, d_{p,g_\infty, \Omega})) \to 0
\end{equation}
and that the volumes  of $p$-balls converge, thereby establishing the proposition. Indeed, fix any $\e>0$. Letting $g_i(t)$ and $g_\infty(t)$ be the Ricci flows as in the proof of Proposition~\ref{constructon-limit-metric}. From the smooth convergence of $g_i(1)$ to $g_\infty(1)$, we see that there exists $N\in \mathbb{N}$ such that $\Omega$ can be covered by $\{B_{g_i(1)}(z_j,\e)\}_{j=1}^N$ for $i$ sufficiently large. 
For each $i$ and $z,w\in \Omega$, let $f\in W^{1,p}(\Omega)$ be a maximizer of $d_{p,g_i,\Omega}(z,w)$, whose existence is guaranteed  by Lemma~\ref{existence-minimizer}. Since $g_i(1)$ has uniformly bounded geometry, we apply Morrey-Sobolev inequality and estimate \eqref{eqn: gradient small error} of Theorem~\ref{claim7} to $f$ to see that 
\begin{equation}
d_{p,g_i,\Omega}(z,w)\leq C_0(n,p,\Omega)d_{g_i(1)}(z,w)^{1-\frac{n}p}
\end{equation}
 for all $z,w\in \Omega$ and $i$ sufficiently large. 

In particular, $\Omega$ can be covered by $\{\B_{i,\Omega}(z_j,C_0\e^{1-\frac{n}{p}})\}_{j=1}^N$ 
where $\B_{i,\Omega}$ is the ball with respect to $d_{p,g_i,\Omega}$ and $N$ is independent of $i\rightarrow +\infty$. Together with $d_{p,g_i,\Omega}(z_j,z_k)\rightarrow d_{p,g_\infty,\Omega}(z_j,z_k)$ for each pair of $z_j,z_k$, this proves the Gromov-Hausdorff convergence. The volume convergence in Definiton~\ref{def: dp convergence}  follows from this Gromov-Hausdorff convergence together with the $L^P$ convergence of the metric coefficients on Proposition~\ref{constructon-limit-metric}.
 This completes the proof of the proposition.
\end{proof}

Finally, we use Proposition~\ref{prop: dp convergence global but local} to establish Proposition~\ref{prop: dp convergence global}.
\begin{proof}[Proof of Proposition~\ref{prop: dp convergence global}]
	Let $P(n,p)$ and thus $\delta(n,p)$ be as in Proposition~\ref{prop: dp convergence global but local}. Note that the largest radii less than or equal to $1$ such that $\B_{p,g_\infty}(y, 4r) \Subset M_\infty $ for $y\in M_\infty$ and $\B_{p,g_i}(y, 4r) \Subset M_i$ for $y\in M_i$ respectively, are both equal to $1$ thanks to the $\e$-regularity Theorem~\ref{thm: main thm, Lp def new} and Remark~\ref{rmk: limit space}. Moreover, again by Theorem~\ref{thm: main thm, Lp def new} and Remark~\ref{rmk: limit space}  and \eqref{eqn: Euclidean ball and p ball}, there exist $r < R$ depending on $n,p$ such that for any $y \in M_\infty$ and $y_i \in M_i$, we have 
	\begin{equation}\label{eqn:  ball containment  limit}
		\begin{split}
	B_{g_\infty(1)}(y ,r ) &\subset \B_{p,g_\infty} (y, 1) \subset B_{g_\infty(1)}(y ,R )\\
		B_{g_i(1)}(y_i  ,r ) &\subset 	\B_{p,g_i} (y_i, 1) \subset B_{g_i(1)}(y_i  ,R ),
		\end{split}
	\end{equation}
	 where $g_\infty(t)$ and $g_i(t)$ are the Ricci flows as in the proof of Proposition~\ref{constructon-limit-metric}. Consequently, from the second containments in \eqref{eqn:  ball containment  limit}, we see that for any $N \in \mathbb{N}$
	\begin{equation}\label{eqn: cov1}
	\begin{split}
			Cov_{g_\infty}(x_\infty ,  N ) &\subset B_{g_\infty(1)}(x_\infty, 2NR), \\
			 Cov_{g_i}(x_i ,  N ) &\subset B_{g_i(1)}(x_i, 2NR).
	\end{split}
	\end{equation}
	
So, for any $N \in \mathbb{N}$, choose $\Omega\subset M_\infty$ to be a compact set such that 
\begin{equation}\label{eqn: cov2}
	B_{g_\infty(1)}(x_\infty, 2NR+1) \subset \Omega \subset B_{g_\infty(1)}(x_\infty, 2NR+2).
\end{equation}
 By the Cheeger-Gromov convergence used in Proposition~\ref{constructon-limit-metric} to obtain the map $\psi : \Omega \to M_i$ and the set $\Omega_i := \psi (\Omega)\subset M_i$, we see that 
 \begin{equation}\label{eqn: cov3}
 B_{g_i(1)}(x_i, 2NR) \subset	\Omega_i\subset B_{g_i(1)}(x_i, 2NR+3)
 \end{equation}
  for $i$ sufficiently large. Therefore, combining \eqref{eqn: cov1}, \eqref{eqn: cov2}, and \eqref{eqn: cov3},  we find that 
\begin{equation}\begin{split}
	Cov_{g_\infty}(x_\infty ,  N ) \subset 	\Omega ,\\ 
	Cov_{g_i}(x_i ,  N ) \subset \Omega_i.
\end{split}
\end{equation}
Now, recall that the metrics $g_\infty(1)$ and $g_i(1)$ satisfy a uniform curvature bound.  Therefore, by volume comparison, there exists $N'  \in \mathbb{N}$ depending only on $N, n, $ and $p$ such that 
\begin{equation}\begin{split}
 B_{g_\infty(1)}(x_\infty, 2NR+2) & \subset \bigcup_{a = 1}^{N'} B_{g_\infty(1)}(y_a, r),\\
 B_{g_i(1)}(x_i, 2NR+3) & \subset \bigcup_{a = 1}^{N'} B_{g_i(1)}(y_{a,i}, r)
 \end{split}
\end{equation} 
for some $\{ y_a\}_{a=1}^{N'} \subset M_\infty$ and $\{y_{a,i}\}_{a=1}^{N'} \subset M_i$, where $r, R$ are as in \eqref{eqn:  ball containment  limit} and depend only on $n,p$. So, applying the first containment of \eqref{eqn:  ball containment  limit} to each ball above, we find that 
\begin{equation}\begin{split}
 B_{g_\infty(1)}(x_\infty, 2NR+2) & \subset Cov_{g_\infty}(x_\infty, N'),\\
 B_{g_i(1)}(x_i, 2NR+3) & \subset Cov_{g_i}(x_i, N').
 \end{split}
\end{equation} 
Together with \eqref{eqn: cov2} and \eqref{eqn: cov3}, this shows that 
\begin{equation}\begin{split}
 \Omega & \subset Cov_{g_\infty}(x_\infty, N'),\\
\Omega_i & \subset Cov_{g_i}(x_i, N').
 \end{split}
\end{equation} 
 
Now, having obtained the appropriate sets $\Omega$, $\Omega_i$ for each $N \in \mathbb{N}$, we may applying Proposition~\ref{prop: dp convergence global but local} to  complete the proof.
\end{proof}

\subsection{Proof of Theorem~\ref{global-thm proof}}\label{sec: proof global}
We finally prove Theorem~\ref{global-thm proof}.
\begin{proof}[Proof of Theorem~\ref{global-thm proof}]
Fix $p\geq n+1.$ 
Choose $P= P(n,p)$ sufficiently large according to Propositon~\ref{prop: rc} and Proposition~\ref{prop: dp convergence global}. 
Now let $\delta = \delta(n,P) =\delta(n,p)$ sufficiently small to apply Propostion~\ref{constructon-limit-metric}.
 
By Proposition~\ref{constructon-limit-metric}, we establish we obtain the space $(M_\infty, g_\infty, x_\infty)$, and  applying Proposition~\ref{prop: dp convergence global} implies the pointed $d_p$ convergence of Theorem~\ref{global-thm}(1). Proposition~\ref{prop: rc} yields the first claim in (2), that is, that the limit space $(M_\infty, g_\infty, x_\infty)$ is $W^{1,p}$-rectifiably complete. Finally, we show that $(M_\infty, g_\infty, x_\infty)$ is $d_p$-rectifiably complete, that is, that the topology generated by $d_{p,g_\infty}$ agrees with the topology of $M_\infty.$ Indeed, this follows from the observation that Propositions~\ref{constructon-limit-metric} and \ref{prop: rc} imply that the topology generated by $d_{p, M_\infty, g_\infty}$ agrees with the topology generated by $d_{p, M_\infty, g_{\infty}(1)}$, which in turn agrees with the topology of a smooth manifold $M_\infty.$
\end{proof}

\begin{remark}
	\rm{
Let $(M_\infty ,g_\infty, x_\infty)$ be the limit rectifiable Riemannian space obtained in Theorem~\ref{global-thm proof}. We have shown that for any suitable compact set $\Omega\subset M_\infty$, we have the convergence along the sequence of the relative $d_p$ distances on $M_i$ to $d_{p, g_\infty, \Omega}$. 
	It is worth noting that, for any $x,y \in M_\infty$ and for any exhasution $\{\Omega_a\}$ of $M_\infty$ by compact sets,  we have 
	\begin{equation}\label{eqn: lim rel}
		\lim_{a \to \infty } d_{p,g_\infty,\Omega_a}(x,y) = d_{p, g_\infty, M}(x,y).
	\end{equation} 
	To see this, we first see directly from the definition that the relative $d_p$ distance is monotone decreasing with respect to set inclusion, so the limit on the left-hand side of \eqref{eqn: lim rel} exists and 
	\begin{equation}
		\lim_{a \to \infty } d_{p,g_\infty,\Omega_a}(x,y) \geq  d_{p, g_\infty, M}(x,y).
	\end{equation} 
	On the other hand, for each $a \in \mathbb{N}$ consider a function $f_a \in W^{1,p}(\Omega_a)$ achieving the supremum in $d_{p,g_\infty,\Omega_a}(x,y)$ (recall Lemma~\ref{existence-minimizer} and Remark~\ref{rmk: limit space}). Then there exists $f \in W^{1,p}(M_\infty)$ such that on every compact set $\Omega \subset M_\infty,$ $f_a  \to f$ weakly in $W^{1,p}(\Omega)$ and uniformly. Thus, $\int_{M_\infty} |\na f|^p \,d\vol_g \leq 1$ and so $f$ is an admissible test function for $d_{p, g_\infty, M}(x,y)$. Moreover,
	\begin{equation}
		|f(x) - f(y) | = \lim_{a \to \infty} |f_a(x) -f_a(y)| = \lim_{a\to \infty} d_{p,g_\infty,\Omega_a}(x,y).
	\end{equation}
	We therefore establish the opposite inequality in \eqref{eqn: lim rel} and conclude.
	}
\end{remark}

\subsection{Proof of Theorem~\ref{thm: flat torus}}
We now prove Theorem~\ref{thm: flat torus}, which provides a form of stability for rigidity of the flat metric as the only metric on a torus with nonnegative scalar curvature.
\begin{proof}[Proof of Theorem~\ref{thm: flat torus}]
	Fix $\delta = \delta(n,p)$ according to Theorem~\ref{global-thm}. Any compact Riemannian manifold with $\nu(g, 2) \geq -\delta$ has volume bounded below a constant $C=C(\delta)$, so choose $V_0 > C$ so that the hypotheses of the theorem are not vacuous. 
	
	 Consider a sequence of tori $(M_i ,g_i)$ with $\nu(g_i, 2) \geq -\delta, \vol_{g_i}(M_i) \leq V_0$, and $R_{g_i} \geq -1/i$. 
	 By Theorem~\ref{global-thm}, up to a subsequence, $(M_i, g_i)$ converges in the $d_p$ sense to a rectifiable Riemannian space $(M_\infty, g_\infty)$, which is constructed in Proposition~\ref{constructon-limit-metric}. Moreover, $M_\infty$ is diffeomorphic to tori.
	 
	 By the proof of Proposition~\ref{constructon-limit-metric}, the Ricci flows $(M_i, g_i(t))_{t \in [0,1]}$ exist and, at  each time slice has uniformly bounded geometry, $\vol_{g_i(t)} ( M_i)  \leq 2V_0$, and $R_{g_i(t) } \geq -1/i$. By Hamilton's compactness \cite{HamiltonCptness}, after passing to subsequence, $(M_i,g_i(t))\rightarrow (M_\infty,g_\infty(t))$ in the pointed $C^\infty$-Cheeger-Gromov sense so that $g_\infty(t)$ is a solution to the Ricci flow on $M_\infty\times (0,1]$  with $R_{g_\infty(t)} \geq 0$ for all $t \in (0,1]$. Moreover, $M_\infty$ is diffeomorphic to a torus. So, by Schoen-Yau \cite{Schoen_1979} and Gromov-Lawson \cite{GromovLawson}, we see that $g_\infty(t)$ is a flat metric on the torus for each $t$. Since $(M_\infty, g_\infty(t))_{t \in (0,1]}$ is a Ricci flow, it follows that each $g_\infty(t)$ is the same flat metric. 
	 
	 Furthermore, from the proof of Proposition~\ref{constructon-limit-metric},  we know that the metric coefficients of $g_\infty(t)$ converge in $L^p$ to $g_\infty$, the limiting rectifiable Riemannian metric. It follows that $g_\infty$ is the flat metric on $M_\infty$.
\end{proof}

\medskip

\section{Examples}	\label{examples-entropy}

In this section, we will construct various examples of sequences of complete Riemannian manifolds $(M_i,g_i)$ with bounded curvature that satisfy the almost non-negative entropy and scalar curvature assumptions of our main theorems. In each example, the $d_p$  limits of our spaces will be either Euclidean space or a flat torus, and these limits do not agree with their Gromov-Hausdorff  and Intrinsic Flat limits.
\subsection{The basic building block: a two-parameter family of  metrics}
We begin by constructing a two-parameter family of metrics on $\mathbb{R}^{n+1}$ for $n\geq 3$ that serve as the basic building block for constructing all of our examples. Let $h$ denote the standard metric on $\mathbb{S}^{n-1}$. We  
define the two-parameter family of metrics $g_{\delta,\e}$ on $M=\mathbb{R}_+\times \mathbb{S}^{n-1}\times \mathbb{R}$ by 
\begin{equation}\label{eqn: metric 2}
g_{\delta, \e}=dr^2+f_{\delta,\e}^2(r) h+\varphi_{\delta,\e}^2(r)dx^2\, .
\end{equation}
The warping factor $f$ will be used to identify $\R^+\times \mathbb{S}^{n-1}$ topologically with $\mathbb{R}^n$, however geometrically this will be done in a way to add a large amount of positive curvature to the space.  The warping factor $\varphi$ will be constructed so that it will {\it slowly} degenerate as $r\to 0$.  If this degeneration is sufficiently slow we will see that we can preserve the positive scalar curvature, and, much more challenging, the lower entropy as well.  If $\varphi(0)=0$, then this would imply that the line $\{0^n\}\times \R$ has a fully degenerate metric $g$ along it, in particular $d_g((0^n,s),(0^n,t))=0$ for any two points along the line $\{0^n\}\times \R$.  The parameters $\epsilon,\delta>0$ are built so that we may approach such a degenerate limit smoothly and in different ways, depending on our end goal.\\

The functions $f_{\delta,\e}$ and $\varphi_{\delta,\e}$ are now precisely defined in the following way.  
\\ 

\noindent {\it Definition of $\varphi_{\delta,\e}$:}
For $\e>0$, let $\phi_\e:\R_+\to \R_+$ be the function such that 
\begin{equation}
\phi_\e(r)=
\left\{
\begin{array}{ll}
\e  \qquad &\text{ for }  r\leq \frac{1}{2}\e;\\
\psi_1  \qquad &\text{ for }  \frac{1}{2}\e\leq r\leq 2\e;\\
r  \qquad &\text{ for }  2\e\leq r\leq \frac{1}{2};\\
\psi_2 \qquad &\text{ for }  \frac12\leq r\leq 2;\\
1   \qquad &\text{ for }  r\geq  2;
\end{array}
\right.
\end{equation}
where $\psi_i(r)$ are smooth non-decreasing functions with $\psi_2''\leq 0$, 
\begin{align}\label{construction}
|\psi_1^{(k)}|\leq 8\e^{-k+1}\quad\text{and}\quad  |\psi_2^{(k)}|\leq 4^k.
\end{align}
We then let $\varphi_{\delta, \e} :\R_+\to \R_+$ be defined by
\begin{equation}
	\varphi_{\delta,\e}(r) = \phi_{\e}(r)^\delta.
\end{equation}
Observe that $\varphi_{\delta, \e}$ satisfies the properties
\begin{equation}\label{eqn: key derivative bounds for phi}
	\left| \frac{\varphi_{\delta,\e}'(r)}{\varphi_{\delta,\e}(r)} \right| \leq \frac{50 \delta}{r}, \qquad \left| \frac{\vphi_{\delta,\e}''(r)}{\vphi_{\delta,\e}(r)}\right| \leq \frac{50\delta}{r^2}, \qquad \left|\frac{\varphi^{(k)}_{\delta,\e}(r)}{\varphi_{\delta,\e}(r)}\right|  \leq  \frac{ C_k \delta}{r^k}\,.
\end{equation}
\begin{remark}\label{general-singularity}
	{\rm As we have defined the two parameter family of metrics, when $\e=0$, then the corresponding metric $g_{\delta,\e}$ vanishes at $r=0$. In fact, we can modify the construction so that $g_{\delta,\e}|_{r=0}$ agrees with any prescribed singular metric $l(x)dx^2$ along $r=0$. More precisely, given fixed $\delta,\e>0$ and a smooth function $l:\mathbb{R}\rightarrow [0,+\infty)$, we could replace $\varphi_{\delta,\e}$ by the function $\varphi_{\delta,\e}(r,x)=\left[1-\phi_\e(r)^\delta\right]l(x)+\phi_\e(r)^\delta$
 so that $\varphi_{\delta,\e}(r_0,x_0)\rightarrow 1$ when $r_0>0$ and $\varphi_{\delta,\e}(0,x_0)\rightarrow l(x_0)$ point-wise as $\delta\rightarrow 0$. 
	}
\end{remark}
 
\noindent {\it Definition of $f_{\delta,\e}$.}
Let $\zeta:\R_+ \to \R_+$ be a smooth non-increasing cutoff function such that $\zeta(x)\equiv 1$ on $[0,1/2]$,  vanishes on $[1, \infty)$ and satisfies $|\zeta'|^2+|\zeta''|\leq 100$.

 Define $\tilde f_{\delta,\e}$ to be the solution of the following ODE: 
\begin{equation}\label{eqn: ODE}
\begin{cases}
\tilde f_{\delta,\e}'&=\left[1-10^4n\delta(1-\zeta(\frac{r}{100\e}))\right];\\
\tilde{f}_{\delta,\e}(0)&=0.
\end{cases}
\end{equation}
In this way, the corresponding metric $dr^2+f^2_{\delta,\e}h$ coincides with the Euclidean metric on $\mathbb{R}^n$ for $r$ sufficiently small.
Finally, we define 
\begin{equation}\label{def-f}
f_{\delta,\e}(r)  = \zeta(r/4) \tilde f_{\delta,\e}(r) + (1-\zeta( r/4)) r,
\end{equation}
so that $f_{\delta,\e}$ is equal to the  solution to the ODE for $r \leq 2$,  the function $r$ for $r \geq 4$, and interpolates smoothly in between.
\medskip

Crucially, this two-parameter family of metrics satisfies a lower bound on entropy and scalar curvature that is uniform for all $\e$ and $\delta$ sufficiently small.  Geometrically, what is happening is that the warping factor is changing so slowly that even though the actual metric geometry may be behaving very poorly, in some weaker sense ($d_p$ sense! though we will not directly appeal to this) the geometry looks very Euclidean at all points and scales.  This sense of closeness to Euclidean space will be good enough to force the small lower entropy bound on the example.

\begin{theorem}\label{main example singular}
	Fix $n\geq 3$, $\eta>0$ and $L>0$. There exist $\e_0>0$ and $\delta_0>0$ depending on $n,\eta$ and $L$ such that the following holds. For all $\e \leq \e_0$ and $\delta \leq \delta_0,$ the metric $g_{\delta, \e}$ defined in \eqref{eqn: metric 2} satisfies
	\begin{align}
\label{eqn: scalar LB example}		R_{g_{\delta,\e}}& \geq -\eta,\\
		\nu(g_{\delta,\e}, L) & \geq -\eta.
	\end{align}
\end{theorem}
The scalar curvature lower bound and entropy lower bound of Theorem~\ref{main example singular} will be established in Sections~\ref{subsec: scalar lower bound} and \ref{subsec: entropy lower bound} respectively.

\begin{remark}
	{\rm 
	Note that the metrics $g_{\delta, \e}$ are defined on an $n+1$ dimensional space, so fixing $n\geq 3$ means that our examples are of dimension $4$ or higher.
	}
\end{remark}

Let us again discuss the examples geometrically, this time with more of a focus on how each parameter behaves in the construction. One can think of the metric $g_{\delta, \e}$ defined in \eqref{eqn: metric 2} in the following way. The portion $dr^2+f_{\delta,\e}^2(r) h$ of the metric $g_{\delta, \e}$ agrees with the Euclidean metric on $\R^n$ far from $0 \in \R^{n}$, while in a neighborhood of $0 \in \R^n,$ it is a smoothed-out cone metric on $\R^n$ with cone angle proportional to $\delta.$ The parameter $\e$ governs the scale at which this cone metric is smoothed out. This component can roughly be thought of as Euclidean $\R^n$, although taking the smoothed cone in place of $\R^n$ provides a crucial positive scalar curvature contribution in order to guarantee that \eqref{eqn: scalar LB example} holds as long as $n\geq 3$. The component $\varphi_{\delta,\e}^2(r)dx^2$ adds a fiber at each point on $(\R^n, dr^2+f_{\delta,\e}^2(r) h)$.  Away from $0 \in \R^n$, these fibers are Euclidean, but for $r$ small, the fibers become increasingly degenerate.

 If we choose $\e=\delta$ and let $\e\rightarrow 0$, then the metric tensors converge smoothly to the Euclidean  metric $g_\infty=\lim_{\e\rightarrow 0}g_{\e,\e}=\sum_{i=1}^{n+1} (dx^i)^2$.
However, if $\e$ is relatively small compared to $\delta$, then the limiting metric will be Euclidean away from a ray $\ell=\{x:x_i=0 \;\text{for}\;i=1,\dots,n\}$ in $\mathbb{R}^{n+1}$ and $\ell$ will be collapsed to a point along the sequence. For instance, if we choose $\delta=(-\log \e)^{-1/2}$ and let $\e\rightarrow 0$, then in a pointwise sense,  $\lim_{\e\rightarrow 0}g_{\delta(\e),\e}=\sum_{i=1}^n (dx^i)^2+(1-\chi_{r=0})(dx^{n+1})^2$. 
In both of these two examples, the constructed sequence converges to the Euclidean metric in $L^p_{loc}$ for all $p>1$, while in the latter case the Gromov-Hausdorff limit is very different; see Example~\ref{example: one line on rn} below. This will correspond to our general $d_p$ $\e$-regularity theorem when the entropy and scalar curvature have lower bound converging to $0$.

\subsection{Examples constructed from the main building block}

In the following, we will make use of the metrics $g_{\delta, \e}$ with $\delta=(-\log \e)^{-1/2}$  to produce sequences of metrics whose $d_p$ limits and Gromov-Hausdorff limits are entirely different. First, we go into greater depth concerning the basic metric $g_{\delta, \e}$ with $\delta=(-\log \e)^{-1/2}$. We will take $n=3$ since this is the borderline case. The case of $n\geq 3$ can be constructed similarly.

\vspace{2mm}

\begin{example}[Collapsing along a line in Euclidean space]\label{example: one line on rn}
{\rm Let $n\geq 3$.
By choosing  $\delta=(-\log \e)^{-1/2}$ in \eqref{eqn: metric 2}, we obtain a sequence of metrics which degenerate along a ray in $\mathbb{R}^{n+1}$ and remain the flat metric away from it. In the Gromov-Hausdorff limit, the ray collapses to a point. On the other hand, from construction, it is easy to see that $g_\e$ converges to the Euclidean metric in $L^p_{loc}(\mathbb{R}^{n+1})$ for all $p>1$. Theorem~\ref{global-thm} (and the proof of Proposition~\ref{constructon-limit-metric}), the pointed $d_p$ limit is the flat Euclidean space.  In particular, notice that Theorem~\ref{thm: main thm, Lp def new} implies that $\vol_{g_\e}\left(\mathcal{B}_{p,g_\e}(0,1) \right)\rightarrow \vol_{g_{euc}}(\mathcal{B}_{p,euc}(0,1))$ as $\e\rightarrow 0$, while  the volumes of metric balls are tending to infinity: 
\begin{equation}\label{eqn: vol blowup}
\vol_{g_\e}\left( B_{g_\e}(0,1)\right)\rightarrow +\infty\,.
\end{equation}
 Indeed,  for $\e$ sufficiently small, $B_{g_\e}(0,1)$ contains the Euclidean strip $\{(x,y)\in \mathbb{R}^{n}\times \mathbb{R}: \frac{1}{4}\leq |x|\leq \frac{1}{2},|y|\leq \frac{1}{2\e}\}$. Since $g_\e$ converges smoothly uniformly to the Euclidean metric away from $|x|=0$, we see that \eqref{eqn: vol blowup} holds. 
 }
\end{example}

We see from the above example that the metric degeneration which causes the metric collapse occurs along a line in $\R^4$.  More generally, we conjecture that this metric collapsing can occur only along co-dimension $3$ subsets along converging sequences

\vspace{2mm}

\begin{example}[$d_p$ convergence does not hold for all $p$]\label{example: delta fixed}
	{\rm In contrast to Example~\ref{example: one line on rn}, in this example we only pass $\e\rightarrow 0$ but fix $\delta>0$  small in the construction of \eqref{eqn: metric 2}. The corresponding sequence of metrics converges pointwise, and in $L^p_{loc}$ for $p$ less than some $p_1(\delta)$,   to $g_\infty=g_{cone}+r^\delta dx^2$, which degenerates at $r=0$. By Theorem~\ref{global-thm} (and the proof of Proposition~\ref{constructon-limit-metric}), the sequence converges in the pointed $d_p$ sense to $(\R^{n+1}, g_\infty, 0^n)$ for $p\in [n+2, p_0(\delta)]$. However, this $d_p$ convergence to  $(\R^{n+1}, g_\infty, 0^n)$ does not hold  for all $p\in [n+2, \infty)$. Indeed, for $p$ sufficiently large, the metric space $(\R^{n+1},d_{p,g_\infty}, 0^n)$ is topologically distinct from the underlying topology on $\R^{n+1}$, and in particular is not $d_p$-complete. 
	This illustrates that $\delta$ must be taken to depend on $p$ in our $\e$-regularity theorems, and if we only assume a lower bound on the entropy lower bound and scalar curvature along the sequence, then the limiting rectifiable Riemannian metric $g_\infty$ may have an inverse that is only bounded in $L^{p_0}_{loc}$ for some $p_0(\delta)>1$ but not all $p>1$. 
	}
\end{example}

\vspace{2mm}

\begin{example}[Collapsing lines in Euclidean space]\label{example: lines on rn}
{\rm
In this example, we use the building block of Example~\ref{example: one line on rn} to construct a sequence of metrics on $\mathbb{R}^{n+1}$ for $n\geq 3$ whose Gromov-Hausdorff limit is the taxicab metric, while the $d_p$ limit is the flat metric on $\mathbb{R}^{n+1}$. The basic idea of the construction is to cut off the building block of Example~\ref{example: one line on rn} to obtain a degenerating metric on a tubular neighborhood of a line in Euclidean space, and to glue this metric into tubular neighborhoods of an increasing dense collection
 of lines in $\R^{n+1}.$

Let us now go into the details of this construction. Let $n\geq 3$. First, we obtain a collection of disjoint strips (i.e.  tubular neighborhoods around lines) $\{\mathcal{S}_{i,j}(r_0)\}$ for $i \in \mathbb{N}$ and $j\in \{1,\dots ,n+1\}$ in the following way. Define the projection $\pi_1:\mathbb{R}^{n+1}\rightarrow \mathbb{R}^{n}$ by $\pi_1(x)=(x_2,\dots,x_{n+1})$, and let  $\pi_j$ be defined analogously for each $j=1,\dots ,n+1$.  Next, for each $j=1,\dots,  n+1$ and $(k_1, \dots, k_{n})\in \mathbb{Z}^n$, we consider the collection of  points $\{z_{k_1,\dots , k_{n},j}\}\subset \mathbb{R}^{n} = \pi_j(\R^{n+1})$ with coordinates given by 
\begin{equation}\label{z}
z_{k_1,\dots , k_{n},j} = ( (100n+ 10j) r_0 k_1, \dots , (100n+ 10j) r_0 k_{n})\, . 
\end{equation}
Up to re-indexing the countable set of  $(k_1,\dots ,k_{n})\in \mathbb{Z}^{n}$ by $i \in \mathbb{N}$, we let $\{z_{i,j}\}=\{z_{k_1,\dots, k_{n},j}\}$. Now, let $B_{\mathbb{R}^{n}}(z,r)$ denote the Euclidean ball in $\mathbb{R}^{n}$  and define the strip $S_{i,j}(r_0)$ of radius $r_0$ around the line  $\{\pi_j(x) = z_{i,j}\}$ by
\begin{equation}
\mathcal{S}_{i,j}(r_0)=\pi_j^{-1}\left(B_{\mathbb{R}^{n}}(z_{i,j},r_0) \right).
\end{equation}

It is easy to check that the collection of  strips  $\{ \mathcal{S}_{i,j}(r_0)\}$ are $200n r_0$ dense, in the sense that for any $x \in \R^{n+1},$ there exists $\mathcal{S}_{i,j}$ such that $dist_{g_{euc}}(x, S_{i,j}) \leq 200n r_0$, and that   these strips are pairwise disjoint.  

Now, with the collection $\{\mathcal{S}_{i,j}(r_0)\}$ of disjoint strips in hand, and for any $r_0>0$ fixed, we use \eqref{eqn: metric 2} to define a metric $g_{r_0}$ on each strip in the following way. Up to a rigid motion of $\R^{n+1}$, it suffices to define $g_{r_0}$ on the strip $\mathcal{S}=\pi_{n+1}^{-1}( B_{\R^{n}}(0, r_0))$. Let $\delta = (-\log \e)^{-1/2}$, and take $\e$ depending on $r_0$ to be sufficiently small so that $R_{g_{\delta, \e}} \geq - r_0^3$ and $\nu(g_{\delta, \e}, 2 r_0^{-2}) \geq -r_0$ by Theorem~\ref{main example singular}.   Then, consider the rescaled metric 
	\begin{equation}\label{gr0}
	g_{r_0} = dr^2 + (r_0/5)^2f_{\e}^2(5r/r_0) h + \vphi_{\e}^2(5r/r_0 ) dx_{n+1}^2,
	\end{equation}
where $\e = \e(r_0)$, which satisfies $R \geq -r_0.$ and $\nu(g_{r_0}, 2)\geq -r_0$. 
	Note that after this rescaling, the metric $g_{r_0}$ agrees with the Euclidean metric outside of the strip $B_{g_{euc}}(\ell, r_0)$.   Finally, by restricting this metric to the set $|z|<r_0,$ we define the metric $g_{r_0}$ on the strip $\mathcal{S}$ .

Finally, let 
 $\tilde{g}_{r_0}$ be the metric on $\R^{n+1}$ be defined by 
\begin{equation}
\tilde{g}_{r_0} = \begin{cases}
	g_{euc} & \text{ if } x \in \R^{n+1} \setminus \bigcup \mathcal{S}_{i,j}(r_0),\\
	g_{r_0} & \text{ if } x \in  \mathcal{S}_{i,j}(r_0).
\end{cases}
\end{equation}

Direct computation shows that $\tilde g_{r_0}$ converges to $g_{\mathbb{R}^{n+1}}$ in $L^p_{loc}(\mathbb{R}^{n+1})$ for all $p \geq 1$. In particular, this implies that $\vol_{\tilde{g}_{r_0}}(\Omega) \to \vol_{euc} (\Omega)$ for any compact set $\Omega\subset \mathbb{R}^{n+1}$. By Theorem~\ref{global-thm} (and the proof of Proposition~\ref{constructon-limit-metric}), the sequence converges to $(\R^n, g_{euc}, 0^n)$ in the pointed $d_p$ sense for all $p \in [n+1, \infty)$. 
 However, in the pointed Gromov-Hausdorff topology, this sequence converges to the taxicab metric.
 
 To roughly explain this, consider the metrics $(\mathbb{R}^{n+1}, \hat g_{\epsilon})\equiv (\mathbb{R}^{n+1}, \epsilon^{-2}\tilde g_{r_0})$. Here, $\e =\e(r)$ is the parameter chosen above. Clearly the metrics $(\mathbb{R}^{n+1}, \hat g_{\epsilon})$ are  isometric to $(\mathbb{R}^{n+1}, \tilde g_{r_0})$ by a Euclidean dilation.  Let $\ell_{ij}\equiv \pi^{-1}_j(z_{ij})$ denote the lines we have glued around.  Then, roughly, we have on each such line $\ell_{ij}$ that $\hat g_\epsilon=1$ in the direction of this line, and that $\hat g_\epsilon\approx \epsilon^{-2}$ in all other directions and at all other points.  We also have, in coordinates, that these lines are $o(\epsilon)$ dense.  Clearly, a path of minimal length from $x$ to $y$ is now one which stays on these lines as long as possible, and moving from one line to another now causes an error which is approximately $\epsilon^{-1}o(\epsilon)=o(1)$.  In particular, we see that $d_{\hat{g}_\epsilon}(x,y)=\sum |x_i-y_i| + o(1)$.  Further, as $\epsilon\to 0$ we see a minimal path is any path which is always moving in coordinate directions (specifically, along our increasing dense collection of lines $\ell_{ij}$).  Hence, $(\mathbb{R}^n, \hat g_{\epsilon})$ is limiting to the taxi-cab metric.

}\end{example}

\vspace{2mm}

Next, we construct some examples in the compact setting.

\begin{example}[Collapsing circle in torus]\label{example-torus-1}
	{\rm In this example, we construct a sequence of metrics $\{g_i\}_{i\in \mathbb{N}}$ on the torus $\mathbb{T}^{n+1}$ for $n\geq 3$ such that each $g_i$ coincides with the flat metric away from a shrinking tubular neighborhood of a fixed ${S}^1\subset \mathbb{T}^{n+1}$. The sequence $g_i$ becomes degenerate along this ${S}^1$, and in the Gromov-Hausdorff limit, the $S^1$ collapses to a point. In particular,  the metric space arising as the Gromov-Hausdorff limit is not topologically a torus. The $d_p$ limit will be the flat torus for any $p\geq n+2$.

We begin the construction. Fix $n\geq 3$ and $r_0>0$. Let ${g}_{r_0}$ be the metric on Euclidean space $\R^{n+1}$ defined in \eqref{gr0} in the previous example, which we recall agrees with the Euclidean metric outside the strip $B_{g_{euc}}(\ell ,r_0)$ and is translation invariant in the $x_{n+1}$ direction. Now, consider the torus $\mathbb{T}^{n+1} = \R^{n+1}/ \mathbb{Z}^{n+1}$, equipped with the metric $\tilde{g}_{r_0}$ given by descending $g_{r_0}$ to the quotient.

	We now let  $r_0 \to 0$. 
	For every $r_0$, the smooth Riemannian manifold $(\mathbb{T}^{n+1}, \tilde{g}_{r_0})$ satisfies $R_{\tilde{g}_{r_0}}  \geq - r_0 $.
Moreover, for any $\delta>0$, there exists $\tau_0 = \tau_0(\delta)$ such that $\nu (g_{flat}, \tau_0) \geq -\delta/2$. So, arguing as in the proof of Theorem~\ref{thm: entropy goes to zero} below, we find that for $r_0 \leq  \bar{r}_0(\delta)$, we have  $\nu(\tilde{g}_{r_0}, \tau_0)\geq -\delta$. We directly see that the metrics $\tilde{g}_{r_0}$ converge in $L^p$ for every $p$ to the flat metric on $\mathbb{T}^{n+1}$. Applying the proof of  Theorem~\ref{thm: main thm, Lp def new} we see that $(\mathbb{T}^{n+1}, \tilde{g}_{r_0})$ converges to $(\mathbb{T}^{n+1}, g_{flat})$ in the $d_p$ sense for all $p \in [n+2, \infty)$. 	 On the other hand, in the Gromov-Hausdorff topology, the $S^1$ factor corresponding to the projection of the degenerating line $\ell$ collapses to a point in the limit. In particular, the metric space arising in the Gromov-Hausdorff limit is not topologically a torus.
	}
\end{example}
\vspace{2mm}

By replicating the construction of the degeneracy, we can  construct examples so that the sequence of metrics on the torus $\mathbb{T}^{n+1}$ converges to a metric  space $Y^k$ with $k<n+1$ or even $k=0$ (a point) in the Gromov-Hausdorff topology.
\begin{example}[Collapsing of $\mathbb{T}^{n+1}$ to $\mathbb{T}^n$]\label{example-torus-2}
	{\rm 
	In the previous example,  Example~\ref{example-torus-1}, we constructed sequence of metrics on $\mathbb{T}^{n+1}$ which degenerate along a single $S^1$ inside. In this example, we will modify the construction to obtain a sequence of metrics $g_i$ on the torus that  degenerate along  increasing dense sequences of parallel copies of $S^1$ and remain flat away from them. In the Gromov-Hausdorff topology, the $(\mathbb{T}^{n+1},g_i)$ collapses to the $n$-dimensional flat torus $\mathbb{T}^{n}$. In particular the Gromov-Hausdorff limit is one dimension lower than the dimension of each manifold in the sequence. On the other hand, the limit respect to $d_p$  is the flat torus $\mathbb{T}^{n+1}$ for $p\geq n+2$.

More precisely, we again fix an $n+1$-dimensional flat torus for $n\geq 3$, identified with $[0,1]^{n+1} / \sim$ and with coordinates $(x_1,\dots , x_{n+1})$. Now, consider a maximal $100 r_0$ dense set $\{z_i\}$ in $[0,1]^{n+1} /\sim$. 
	Let $S^1_i = \{ (x_1,\dots, x_n) = z_i \} \subset \mathbb{T}^{n+1}$, 
	 and let $\mathcal{S}_i$ be the strip around $S_i^1 $ of radius $r_0$ as in the previous example. We let 
	 \begin{equation}
	  \tilde{g}_{r_0} = \begin{cases}
	 	g_{euc} & \text{ for } x \in \mathbb{T}^{n+1} \setminus \bigcup \mathcal{S}_{i}(r_0),\\
	 	g_{r_0} & \text{ for } x \in \mathcal{S}_i.
	 \end{cases}
	 \end{equation}
Here, $g_{r_0}$ is the same metric on a strip around $S^1$ as defined in Example~\ref{example-torus-1} above.
	For every $r_0,$ the smooth Riemannian manifold $(\mathbb{T}^{n+1},\tilde{g}_{r_0})$ satisfies $R_{\tilde{g}_{r_0}}  \geq - r_0 $, and for every $\delta>0$, there exist $\tau_0= \tau_0(\delta)$ and $\bar{r}_0 = \bar{r}_0(\delta)$ such that for $r_0 \leq \bar{r}_0$, we have $\nu(\tilde{g}_{r_0}, \tau_0)\geq -\delta$. In the Gromov-Hausdorff topology, $(\mathbb{T}^{n+1},\tilde{g}_{r_0})$ converges to the $n$-dimensional flat torus $\mathbb{T}^n$ with the usual distance as $r_0 \to 0.$ On the other hand, $(\mathbb{T}^{n+1},\tilde{g}_{r_0})$ converges to the $n+1$-dimensional flat torus with respect to $d_p$ convergence for each $p\geq n+2$ by (the proof of) Theorem~\ref{thm: main thm, Lp def new}.
	}
\end{example}
\vspace{2mm}

\begin{example}[Collapsing $\mathbb{T}^{n+1}$ to a point]\label{example-torus-3}
	{\rm In this example, we further modify the construction in Examples~\ref{example-torus-1} and \ref{example-torus-2} to produce a sequence of metrics on $\mathbb{T}^{n+1}$ such that the sequence collapses to a point in the Gromov-Hausdorff and Intrinsic Flat topologies. Once again, the $d_p$ limit will still be the flat torus $\mathbb{T}^{n+1}$ for $p\geq n+2$.
	The basic idea of the construction is to choose an increasingly dense collection of strips around copies of $S^1 \subset \mathbb{T}^{n+1}$ with all different orientations, in a similar fashion to Example~\ref{example: lines on rn}, and then to paste the degenerating metrics of  Examples~\ref{example-torus-1} and \ref{example-torus-2} into each of these strips.
	
More specifically, we begin with the sequence of metrics $\tilde{g}_{r_0}$ on $\mathbb{R}^{n+1}$ constructed in Example~\ref{example: lines on rn}. Without loss of generality, we may assume that $r_0$ is always chosen so that for each $j = 1,\dots , n+1,$ we have that $1/(100n +10j) r_0$ is an integer. With this assumption, the metrics $\tilde{g}_{r_0}$ are invariant under the $\mathbb{Z}^{n+1}$ action on $\mathbb{R}^{n+1}$, so we may consider the quotient $\mathbb{T}^{n+1} = \R^{n+1} / \mathbb{Z}^{n+1}$ equipped with the metric descending from $\tilde{g}_{r_0}$ under the quotient, which we again denote by $\tilde{g}_{r_0}$.

The smooth Riemannian manifolds $(\mathbb{T}^{n+1}, \tilde{g}_{r_0})$ satisfy $R_{\tilde{r_0}} \geq -r_0$, and for any $\delta>0$, there exist $\tau_0 = \tau_0(\delta)$ and $\bar{r}_0 = \bar{r}_0(\delta)$ such that for $r_0 \leq \bar{r}_0$, we have  $\nu(\tilde{g}_{r_0} ,\tau_0 ) \geq -\delta$, provided $\e(r_0)$ is taken to be sufficiently small depending on $r_0$. Sending $r_0 \to 0$, the metrics converge in the Gromov-Hausdorff topology to a point. To see this, we claim that for any $x,y \in (\mathbb{T}^{n+1}, \tilde{g}_{r_0}),$
\begin{align}
	dist_{\tilde{g}_{r_0}}(x,y) \leq 10 \times 200 n^2 r_0 + n\e(r_0). 
\end{align}
Indeed, let $S^1_{ij}$ denote the $S^1$ factor in $\mathbb{T}^{n+1}$ that is the projection of the line $\pi_j^{-1}(z_{ij}) \in \mathbb{R}^{n+1}$ in the construction of Example~\ref{example: lines on rn}. Then, we have
$dist_{\tilde{g}_{r_0}}(y , \cup S^1_{ij}) \leq 200 n r_0$ and $dist_{\tilde{g}_{r_0}}(x , \cup S^1_{ij}) \leq 200 nr_0$ because the collection $\{ S^1_{ij}\}$ is $200n r_0$ dense. Furthermore, for any two points $\tilde{x}, \tilde{y} \in \cup S^1_{ij}$, we have $dist_{\tilde{g}_{r_0}}(\tilde{x} , \tilde{y}) \leq n\e(r_0) + \times 200n^2 r_0.$ Furthermore by \cite[Corollary 2.1]{sormani2011intrinsic}, we have $(\mathbb{T}^{n+1},\tilde g_{r_0})$ converging to the zero current in the Intrinsic Flat sense. However, we directly see that the metric tensors converge to the the flat metric on $\mathbb{T}^{n+1}$ in $L^p$ for all $p < \infty.$ Using this fact and  appealing to the proof of Theorem~\ref{thm: main thm, Lp def new}, we see that  $(\mathbb{T}^{n+1}, \tilde{g}_{r_0})$  converges to the flat $n+1$-dimensional torus in the $d_p$ sense for each $p \in [ n+2,\infty)$.

	}
\end{example}

\subsection{Scalar curvature of the metrics $g_{\delta, \e}$}\label{subsec: scalar lower bound} In this subsection, we show that negative part of the scalar curvature of $g_{\delta,\e}$ can be made arbitrarily small if $\delta,\e$ are small enough. More precisely, we have the following.
 \begin{proposition}\label{prop: scalar nonneg} 
 For any $n\geq 3$ and $\eta>0$, there exist  $\e_0>0$ and $\delta_0>0$ depending on $\eta$   such that for all $\e \leq \e_0$ and $\delta \leq \delta_0$, the metric $g_{\delta,\e}$ defined in \eqref{eqn: metric 2} satisfies $R_{g_{\delta,\e}}\geq -\eta$.
 \end{proposition}

To begin with, we need the following expression for  the scalar curvature of a metric  $g$ taking the general form of \eqref{eqn: metric 2}.
\begin{lemma}\label{Scalar}Let $M=\mathbb{R}_+\times \mathbb{S}^{n-1}\times \mathbb{R}$ and let $h$ denote the standard metric on $\mathbb{S}^{n-1}$. For any metric $g$ on $M$ taking the form
\begin{equation}
\label{eqn: metric: general form}	
g=dr^2+f^2(r)h+\varphi^2(x,r) dx^2,
\end{equation}
the scalar curvature $R_g$ of $g$ is given by
\begin{equation}\label{eqn: expression for R}
R_g=\frac{(n -1)}{f^2}[2-(f^2)'']+\frac{(n-4)(n-1)}{f^2}[1-(f')^2]-\frac{2\varphi''}{\varphi}-\frac{2(n-1)\varphi' f'}{\varphi f}.
\end{equation}
Here, the prime denotes a derivative  with respect to $r$.
\end{lemma}
\begin{proof}
Let $i,j,k,...$ be the coordinates on $\mathbb{S}^{n-1}$. As noted in the statement of the lemma, for a function $F$, we will use $F'$ to denote $\partial_r F$. We let $F_x$ denote $\pa_x F.$  We first compute the Christoffel symbols that will be needed in our computation. First, we have 
\begin{equation}
\begin{split}
\Gamma_{rA}^B=\frac{1}{2} g^{BC}\left(\partial_r g_{AC}+\partial_A g_{rC}-\partial_Cg_{rA} \right)&=\frac{1}{2} g^{BC} \partial_r g_{AC}
=
\begin{cases}
\displaystyle\frac{f'}{f} \delta^i_j  &  \text{if}\;\; A=j, B=i;\\
	\displaystyle\frac{\varphi'}{\varphi} & \text{if}\;\; A=x, B=x;\\
	0 & \text{otherwise}.
\end{cases}
\end{split}
\end{equation}
Next, note that 
\begin{equation}
\begin{split}
\Gamma_{ij}^k&=\frac{1}{2} g^{kl}\left(\partial_i g_{jl}+\partial_jg_{il}-\partial_lg_{ij} \right)=\frac{1}{2} h^{kl}\left(\partial_i h_{jl}+\partial_jh_{il}-\partial_lh_{ij} \right)=\tilde\Gamma_{ij}^k,
\end{split}
\end{equation}
where $\tilde\Gamma_{ij}^k$ denotes the Christoffel symbol for the standard metric $h$. Next, we compute 
\begin{equation}
\begin{split}
\Gamma_{ij}^r&=-\frac{1}{2}g^{rr}\partial_r g_{ij} =-ff' h_{ij};\\
\Gamma_{xx}^r&
=-\frac12 \partial_r \varphi^2=-\varphi \varphi'; \\
\Gamma_{xx}^x&=\frac{1}{2}g^{xx}\left( \partial_x g_{xx}+\partial_xg_{xx}-\partial_x g_{xx}\right)=\frac{\varphi_x}{\varphi}.
\end{split}
\end{equation}
%
The remaining Christoffel symbols vanish: $\Gamma_{ij}^x=\Gamma_{ix}^\bullet=\Gamma_{xx}^i=0$. 

With these Christoffel symbols in hand, we compute the Ricci curvatures $R_{rr}, R_{ij}$ and $R_{xx}$. 
We have 
\begin{equation}
\begin{split}
R_{rr}&=\partial_A \Gamma^A_{rr}-\partial_r\Gamma^A_{rA}+\Gamma^A_{AB}\Gamma^B_{rr}-\Gamma_{rA}^B\Gamma_{Br}^A\\
&=0-(\partial_r\Gamma_{ri}^i+\partial_r \Gamma_{rx}^x)+0-(\Gamma_{ri}^j\Gamma^j_{ri}+\Gamma_{rx}^x \Gamma_{rx}^x)\\
&=-\sum_{i=1}^{n-1}\partial_r \left(\frac{f'}{f} \right)-\partial_r \left(\frac{\varphi'}{\varphi} \right)-\sum_{i=1}^{n-1}  \left(\frac{f'}{f}\right)^2-\frac{(\varphi')^2}{\varphi^2}=-(n-1)\frac{f''}{f}-\frac{\varphi''}{\varphi}.
\end{split}
\end{equation}
For $R_{ij}$, we have 
\begin{equation}
\begin{split}
R_{ij}&=\partial_A \Gamma^A_{ij}-\partial_j \Gamma^A_{iA}+\Gamma^A_{AB}\Gamma^B_{ij}-\Gamma_{jA}^B\Gamma^A_{Bi}\\
&=\tilde R_{ij}+(\partial_r \Gamma_{ij}^r+\partial_x \Gamma_{ij}^x)-(\partial_j\Gamma^r_{ir}+\partial_j\Gamma^x_{ix})  +(\Gamma^r_{rk}\Gamma_{ij}^k+\Gamma^x_{xk}\Gamma_{ij}^k+\Gamma^A_{Ar}\Gamma_{ij}^r+\Gamma_{Ax}^A\Gamma_{ij}^x)\\
&\quad -(\Gamma_{jk}^r\Gamma_{ri}^k+\Gamma_{jk}^x\Gamma_{xi}^k+\Gamma_{jr}^B\Gamma_{iB}^r+\Gamma_{jx}^B\Gamma_{iB}^x)\\
&=\tilde R_{ij}+\partial_r \Gamma_{ij}^r+\Gamma_{ij}^k(\Gamma_{rk}^r+\Gamma_{xk}^x)+\Gamma_{ij}^r (\Gamma_{kr}^k+\Gamma_{xr}^x) -(\Gamma_{jk}^r\Gamma_{ri}^k+\Gamma_{jr}^k\Gamma_{ik}^r)\\
&=\tilde R_{ij} -(f')^2 h_{i j}-ff'' h_{i\bar j}+(-ff' h_{ij})\left[(n-1)\frac{f'}{f}+\frac{\varphi'}{\varphi} \right]+2(f')^2h_{ij}\\
&=\tilde R_{ij}- h_{ij}\left(ff''+(n-2)(f')^2+\frac{\varphi' ff'}{\varphi}\right).
\end{split}
\end{equation}
Here, $\tilde{R}_{ij}$ denotes the Ricci curvature of $h$.
Finally, for $R_{xx}$, we have
\begin{equation}
\begin{split}
R_{xx}&=\partial_A \Gamma^A_{xx}-\partial_x\Gamma^A_{xA}+\Gamma^A_{AB}\Gamma^B_{xx}-\Gamma_{xA}^B\Gamma_{Bx}^A\\
&=(\partial_r \Gamma^r_{xx}+\partial_x \Gamma^x_{xx})-(\partial_x\Gamma_{xx}^x)+(\Gamma_{xx}^r\Gamma_{Ar}^A+\Gamma_{xx}^x\Gamma_{Ax}^A)-(\Gamma_{xr}^B\Gamma_{Bx}^r+\Gamma_{xx}^B\Gamma_{Bx}^x)\\
&=\left(-\varphi \varphi''-(\varphi')^2\right)+\Gamma_{xx}^r(\Gamma^i_{ir}+\Gamma^x_{rx})+\Gamma_{xx}^x\Gamma_{xx}^x-(\Gamma_{xr}^x\Gamma_{xx}^r+\Gamma_{xx}^r\Gamma_{rx}^x+\Gamma_{xx}^x\Gamma_{xx}^x)\\
&=\left(-\varphi \varphi''-(\varphi')^2\right)+\Gamma_{xx}^r\Gamma^i_{ir}-\Gamma_{xr}^x\Gamma_{xx}^r\\
&=\left(-\varphi \varphi''-(\varphi')^2\right)+(n-1)(-\varphi \varphi')(\frac{f'}{f})+(\varphi')^2\\
&=-\varphi\varphi''-(n-1)\frac{\varphi \varphi' f'}{f}.
\end{split}
\end{equation}

Therefore, using the fact that $\tilde R_{ij}=(n-2)h_{ij}$, we have
\begin{equation}
\begin{split}
R&=g^{rr}R_{rr}+g^{ij}R_{ij}+g^{xx}R_{xx}\\
&=-(n-1)\frac{f''}{f}-\frac{\varphi''}{\varphi}+\frac{1}{f^2}h^{ij}\left( \tilde R_{ij}- h_{ij}\left(ff''+(n-2)(f')^2+\frac{\varphi' ff'}{\varphi}\right)\right)\\
&\quad -\frac1{\varphi^2}\left( \varphi\varphi''+(n-1)\frac{\varphi \varphi' f'}{f}\right)\\
&=-(n-1)\frac{f''}{f}-\frac{2\varphi''}{\varphi}-\frac{(n-1)\varphi' f'}{\varphi f}+(n-1) \left[ \frac{n-2}{f^2}-\frac{f''}{f} -(n-2)\frac{(f')^2}{f^2}-\frac{\varphi' f'}{\varphi f}\right]\\
&=-\frac{2(n-1)f''}{f}-\frac{2\varphi''}{\varphi}+\frac{(n-2)(n-1)}{f^2}-\frac{(n-2)(n-1)(f')^2}{f^2}-\frac{2(n-1)\varphi' f'}{\varphi f}\\
&=-\frac{2(n-1)f''}{f}+\frac{(n-2)(n-1)}{f^2}[1-(f')^2]-\frac{2(n-1)\varphi' f'}{\varphi f} -\frac{2\varphi''}{\varphi}
\end{split}
\end{equation}
Finally, to arrive at \eqref{eqn: expression for R}, we note that $2 f''/f = (f^2)'' /f^2 - 2(f')^2/f^2$, and so 
\begin{equation}
-\frac{2(n-1)f''}{f} = - \frac{(n-1) (f^2)''}{f^2} + \frac{2(n-1) (f')^2}{f^2}= \frac{n-1}{f^2}[2-(f^2)''] - \frac{2(n-1)}{f^2}[1 - (f')^2].  
\end{equation}
Rearranging this expression, we arrive at \eqref{eqn: expression for R}, completing the proof.
\end{proof}

With Lemma~\ref{Scalar} in hand, we are now ready to prove Proposition~\ref{prop: scalar nonneg}.
\begin{proof}[Proof of Proposition~\ref{prop: scalar nonneg}]
For notational convenience, we will omit the indices $\delta$ and $\e$, letting $g, f,$ and $\varphi$ denote $g_{\delta,\e}$, $f_{\delta,\e}$, and $\varphi_{\delta,\e}$ respectively. We will assume that $\e \leq \e_0$ and $\delta \leq \delta_0$, where  $\e_0,\delta_0<\frac14$ will be fixed within the proof.  By Lemma~\ref{Scalar},  the scalar curvature of $g$  takes the form  
\begin{equation}\label{eqn: scalar I II III}
\begin{split}
R&=\frac{(n-1)}{f^2}\left(2-(f^2)''\right)+\frac{(n-4)(n-1)}{f^2}\left[ 1-(f')^2\right]-\frac{2\varphi''}{\varphi}-\frac{2(n-1)\varphi'f'}{\varphi f}\\
&=\mathbf{I}+\mathbf{II}+\mathbf{III}+\mathbf{IV}.
\end{split}
\end{equation}
We estimate the scalar curvature from below in three different intervals of $r$ in the cases below.
First, recall from the definition  \eqref{def-f} of $f$  that 
 \begin{equation}\label{eqn: f tilde f}
\begin{cases}
	f(r)= \tilde{f}(r) & \text{ for } r \leq 2,\\
	f(r)=r & \text{ for }r \geq 4,
\end{cases}
\end{equation}
where $\tilde{f}$ is the solution of the ODE \eqref{eqn: ODE}.
We first collect some useful estimates for $\tilde f$. For notational convenience, we denote $\sigma=10^4n\delta (1-\zeta)$ and $\sigma_0=10^4n\delta$ in the definition of \eqref{eqn: ODE} so that $\tilde f'=1-\sigma$ where $\sigma$ increases from $0$ to $\sigma_0=10^4n\delta$. We will assume $\delta_0$ is small enough so that $\sigma_0<\frac14$. Since $\tilde f(0)=0$, by integrating, we therefore find that
\begin{equation}\label{tildef-estimate}
\left\{
\begin{array}{ll}
\frac34 r<(1-\sigma_0)r\leq \tilde f(r)\leq r;\\
(1-\sigma_0)\leq \tilde f'(r)\leq 1.
\end{array}
\right.
\end{equation}

{\bf Case 1. $r\leq \frac{1}{2}\e$.} 
In this case, $\varphi(r) \equiv \e$, so that terms $\mathbf{III}$  and $\mathbf{IV}$ 
in \eqref{eqn: scalar I II III} vanish, and 
$f=\tilde f$  thanks to \eqref{eqn: f tilde f}. We first consider the case when $n\geq 4$ since $\mathbf{II}$ is clearly nonnegative by \eqref{eqn: ODE}. By \eqref{eqn: ODE} and \eqref{tildef-estimate}, if $\delta_0$ is sufficiently small so that $\sigma< 1$ then
\begin{equation}
\begin{split}
\mathbf{I}+\mathbf{II}\geq \frac{n-1}{ f^2}\left(2-( f^2)'' \right)&=\frac{2(n-1)}{ f^2}\left(1- f  f''-( f')^2 \right)\\
&\geq \frac{2(n-1)  (2-\sigma)\sigma}{ f^2}\\
&>\frac{2(n-1)\sigma}{f^2}.
\end{split}
\end{equation}
Here we have used $\sigma'\geq 0$ and hence $R>0$ when $n\geq 4$. It remains to consider $n=3$. In this case, we have 
\begin{equation}
\begin{split}
\mathbf{I}+\mathbf{II}&=\frac{2}{ f^2}\left(1-( f^2)''+( f')^2 \right)\\
&= \frac2{ f^2} \left( 1-2 f  f''-(f')^2\right)\\
&\geq \frac2{ f^2} \left( 1-(1-\sigma)^2\right)\geq  \frac{2\sigma}{ f^2}.
\end{split}
\end{equation}
And hence we also have $R>0$  when $n=3$ as long as $\delta_0$ is small enough.
\smallskip

{\bf Case 2. $r\in [\frac{1}{2}\e,2]$.} In this case, we still have $f=\tilde f$ by \eqref{eqn: f tilde f} and $\sigma\equiv 10^4n\delta$ in this range. Therefore by \eqref{tildef-estimate} and the computation in {\bf Case 1}, for $n\geq 3$ and sufficiently small $\delta_0$,
\begin{equation}\label{eqn: I lb case 2}
\mathbf{I}+\mathbf{II}\geq \frac{\sigma}{f^2}\geq \frac{10^4n\delta}{r^2}.
\end{equation}
On the other hand, by \eqref{eqn: key derivative bounds for phi}, we see that 
\begin{equation}
\begin{split}
\mathbf{III}&=-\frac{2\varphi''}{\varphi}
\geq -\frac{100\delta}{r^2}
\end{split}
\end{equation}
Similarly, using \eqref{eqn: key derivative bounds for phi} and  \eqref{tildef-estimate} we find that
\begin{equation}
\begin{split}
\mathbf{IV}&=-\frac{2(n-1)f'\varphi'}{f\varphi}\geq  -\frac{1200n\delta}{r^2}.
\end{split}
\end{equation}
Hence, $R=\mathbf{I}+\mathbf{II}+\mathbf{III}+\mathbf{IV}>0$.
\hfill\\



{\bf Case 3. $r> 2$.} In this case, $\varphi\equiv 1$, and so terms $\mathbf{III}$ and $\mathbf{IV}$
 in \eqref{eqn: scalar I II III} vanish. Furthermore, we directly see that $\mathbf{I}=\mathbf{II}=0$ when $r \geq 4$ since $f(r) = r$ there. 
So, it remains to show that $\mathbf{I}+\mathbf{II} \geq -\eta$ in the case when $r \in (2,4)$. Note that for $r$ in this interval, we have from \eqref{tildef-estimate} that 
\begin{equation}
f(r)=\zeta(r/4) \tilde f(r)+(1-\zeta(r/4))r\geq \frac{1}{2}r\geq 1.
\end{equation}
Therefore, it remains to estimate $|2-(f^2)''|$ and $|1-(f')^2|$ for $r \in (2,4)$. By rewriting $2-(f^2)''=(r^2-f^2)''$ and $1-(f')^2=(r')^2-(f')^2$, it suffices to estimate $(r-f)''$. By the construction of $\zeta$ and \eqref{tildef-estimate} for $r\in (2,4)$,
\begin{equation}
\begin{split}
|(f-r)''|&\leq |\zeta''(r/4)(\tilde f-r)|+|\zeta'(r/4) (\tilde f-r)'|+|\zeta(r/4) (\tilde f-r)''|\\
&\leq 10|\tilde f-r|+10|\tilde f'-1|\\
&\leq 100\sigma_0.
\end{split}
\end{equation}
By combinning with \eqref{tildef-estimate}, we conclude that $||f-r||_{C^2((2,4))}\leq C_n\delta_0$. And hence for $r\in (2,4)$,
\begin{equation}
\begin{split}
R&=\mathbf{I}+\mathbf{II}+\mathbf{III}+\mathbf{IV}\\
&\geq -(n-1)|(r^2-f^2)''|-(n-4)(n-1)|1-(f')^2|\\
&\geq -C_n\delta_0.
\end{split}
\end{equation}
Therefore, if $\delta_0$ is sufficiently small, then the right hand side will be larger than $-\eta$. This completes the proof.
\end{proof}


\subsection{Entropy lower bound for the metrics $g_{\delta,\e}$}\label{subsec: entropy lower bound}
In this section, we will show that the entropy of the metric $g_{\delta,\e}$ on $\mathbb{R}^{n+1}$ can be made arbitrarily small by taking $\e,\delta$ to be sufficiently small for $n\geq 3$.

\begin{proposition}\label{thm: entropy goes to zero}
	For any $n\geq 3,\eta,L>0$, there exists $\e_0>0$ such that for any $\e,\delta \leq \e_0,$ the metric $g_{\delta,\e}$ defined in \eqref{eqn: metric 2} satisfies
	\begin{align}\label{eqn: entropy bound}
		\nu (g_{\delta,\e},L) \geq -\eta.
	\end{align}
\end{proposition}



 
Before we give the proof of Proposition~\ref{thm: entropy goes to zero}, we start with some basic notation and preliminaries. Throughout this section, it will be convenient to rescale the metric by $\tau^{-1}$ so that we only need to estimate $\mu(g, \tau)= \mu(\tau^{-1}g, 1)$. Given a minimizer $f_\tau$ of $\mu(\tau^{-1}g,1)$, we define the associated function
\[
u_\tau = (4\pi)^{-n/4} \exp\{-f_\tau/2\}
\]
which satisfies $\int_M u_{\tau}^2 \, d\vol_{\tau^{-1}g}=1.$  In a slight abuse of terminology, we will refer to this function $u_{\tau}$ as a minimizer of $\mu(\tau^{-1}g,1)$. A minimizer $u_\tau$ of $\mu(\tau^{-1}g,1)$  is a positive smooth function satisfying the Euler-Lagrange equation
\begin{equation}\label{eqn: EL lemma tau=1}
 	- 4 \Delta u_\tau + Ru_\tau - 2u_\tau \ln u_\tau -\left(\frac{n+1}{2}\log(4\pi) + (n+1) +\mu(\tau^{-1}g,1 )\right)u_\tau =0,
 \end{equation}
where the Laplacian and scalar curvature are with respect to $\tau^{-1}g$. On Euclidean space $\R^{n+1}$, it is well-known that the minimizers $u$ of $\mu(g_{euc}, 1)$ are uniquely given by the Gaussian functions 
 \begin{equation}
	\label{eqn: euclidean gaussian}
	u^2 = (4\pi)^{-(n+1)/2} \exp\{ -|x-y|^2/4\},
\end{equation}
where $|\cdot|$ is the Euclidean metric and $y$ is a fixed point on $\mathbb{R}^{n+1}$. Indeed, this is precisely the log-Sobolev inequality on Euclidean space. The following lemma shows that these functions are in fact  the only bounded $W^{1,2}$ sub-solutions of \eqref{eqn: EL lemma tau=1} on Euclidean space.
\begin{lemma}[Characterization of solutions on Euclidean space]\label{lem: euclidean gaussian} Fix $\mu \leq 0$ and 
	let $u \in W^{1,2}(\R^{n+1})$ be a bounded solution to 
		\begin{equation}\label{eqn: limit equation lem}
		4\Delta u +  u\log u^2 +\left( \frac{n+1}{2}\log(4\pi) +(n+1) + \mu\right) u \geq 0 
	\end{equation}
	in $\R^{n+1}$
	with 
	\begin{equation}\label{eqn: norm leq 1}
	\int_{\R^n}u^2 \,dx \leq 1.
	\end{equation}
	 Then $\mu = 0$ and $u$ takes the form \eqref{eqn: euclidean gaussian}
	  for some $y \in \R^{n+1}$.
\end{lemma}
 \begin{proof}
By \cite[Lemma 2.3]{ZhangLS}, a  bounded solution $u$ of \eqref{eqn: limit equation lem} with  $\|u\|_{L^2(\R^{n+1})}\leq 1$  has Gaussian decay in the sense that for any $y\in \mathbb{R}^{n+1}$, there exist positive numbers $r_0, a$ and $A$ depending on $y$ such that 
\[
u(y) \leq A \exp\{-a|x-y|^2\} \qquad \text{ when }|x-y| \geq r_0.
\]
In particular, multiplying the equation \eqref{eqn: limit equation lem} by $u$ and integrating over $\R^{n+1}$, we are justified in integrating by parts to find
 	\begin{equation}\label{eqn: ibp}
 	\begin{split}
 		0 & \geq \int_{\R^{n+1}} -4u\Delta u  - u^2 \log u^2 -((n+1)+\frac{n+1}{2}\log(4\pi)+ \mu) u^2 \,dx \\
		& = \int_{\R^{n+1}} 4|\na u|^2 - u^2 \log u^2 -((n+1)+\frac{n+1}{2}\log(4\pi)+ \mu) u^2\,dx .	
 	\end{split}	
 	\end{equation}
	Now, set $w = u/\| u\|_{L^2(\R^{n+1})}$ so that $\| w\|_{L^2(\R^{n+1})}=1.$ Dividing \eqref{eqn: ibp} by $\| u\|_{L^2(\R^{n+1})}^2$, we see that
\begin{equation}
\begin{split}
		0 & \geq  \int_{\R^{n+1}} 4|\na w|^2 -w^2 \log w^2 -w^2 \log \|u\|_{L^2(\R^{n+1})}^2 - ((n+1)+\frac{n+1}{2}\log(4\pi)+\mu) w^2\, dx\\
		& =  \int_{\R^{n+1}} 4|\na w|^2 -w^2 \log w^2  - ((n+1)+ \frac{n+1}{2}\log(4\pi)) w^2\, dx- \log \|u\|_{L^2(\R^{n+1})}^2- \mu\\
		& \geq- \log \|u\|_{L^2(\R^{n+1})}^2- \mu,
\end{split}
\end{equation}
	where the final inequality is obtained by applying the  Euclidean log-Sobolev inequality to $w$. By \eqref{eqn: norm leq 1} and $\mu\leq 0$, we 
conclude that $\mu = 0$ and $\|u\|_{L^2(\R^{n+1})}=1$. Then from \eqref{eqn: ibp}, we see that $u$ is a minimizer of the Euclidean log-Sobolev inequality and so $u^2$ is a Gaussian as in \eqref{eqn: norm leq 1}. This concludes the proof.
\end{proof}

 The following lemma ensures the existence of the minimizer of the entropy for $(M,g_{\delta,\e})$ with exponential decay at infinity.
 \begin{lemma}[Existence and estimates for extremals]\label{lemma: existence of extremal}
	Fix $\e,\delta,L>0$ and $\tau \in (0,L)$ and let $g_{\delta,\e}$ be the metric defined in \eqref{eqn: metric 2}. A  minimizer $u_\tau$ of the entropy $\mu(\tau^{-1}g_{\delta,\e},1)$ exists.  Furthermore there exist constants $a, A>0$ and a point $y \in M$ depending on $\delta,\e$ and $\tau$ such that  $u_\tau $ satisfies
	\begin{equation}\label{eqn: gaussian decay lem}
			u_{\tau}(x) \leq A \exp\{-a d^2(x, y)\}.
	\end{equation}
	Here $d(\cdot, \cdot)$ denotes the geodesic distance with respect to $\tau^{-1} g_{\delta,\e}$.
\end{lemma}

\begin{proof}
The existence of minimizers  follows from  minor modifications of the existence proof of  \cite[Theorem 1.1(a)]{ZhangLS}; we therefore only outline these modifications. As in the proof of  \cite[Theorem 1.1(a)]{ZhangLS}, we let $v_k$ denote a minimizer of the entropy restricted to the ball $B_{g}(0, k)$ (here we let $g=\tau^{-1}g_{\delta,\e}$)  and let $x_k$ be a point where $v_k$ achieves its maximum. Viewing $g$ as a metric on $\R^n \times \R$ with coordinates $(z,y) \in \R^n \times \R$, the metric is translation invariant with respect to the $y$ component. Hence we may assume without loss of generality that $x_k = (z_k , 0) $ for all $k$. There are then two cases to consider: either the sequence $\{z_k\}$ is bounded in $\R^n$, or it is not. In the case that the sequence is bounded, the existence of a minimizer follows just as in  \cite[Theorem 1.1(a)]{ZhangLS}. The case that $\{z_k\}$ is unbounded is even simpler than the corresponding case in \cite{ZhangLS}: since $\mu(\tau^{-1} g_{\delta,\e} , 1) < 0,$ then arguing as in  \cite[Theorem 1.1(a)]{ZhangLS} we obtain a bounded $W^{1,2}$  solution to \eqref{eqn: limit equation lem} with $\mu <0$ on $\R^{n+1}$, a contradiction to Lemma~\ref{lem: euclidean gaussian}.
Finally, the Gaussian decay is established in \cite[Lemma 2.3]{ZhangLS}.
\end{proof}

\bigskip 
As usual, let $g_{\delta,\e}$ be the metric defined in \eqref{eqn: metric 2}. 
In the next lemma, we show that the rescaled metric $\tau^{-1}g_{\delta,\e}$ is uniformly close to the Euclidean metric  in any compact subset centered at $\bar x$ after appropriate change of coordinate. To do this, we define an explicit diffeomorphism centered at $\bar x$ that takes into  account how and in which direction the metric is degenerating. 

In the application, we will take $\bar x$ be a point such that $u_\tau$, the minimizer of $\mu(\tau^{-1}g_{\delta,\e},1)$, achieves its maximum at $\bar x$.  Thanks to the symmetry of $g_{\delta,\e}$, we can and will assume without loss of generality that $\bar y =0$ and that  $\bar z = \bar{r} \bar{e}$ for a fixed unit vector $\bar{e}\in \mathbb{S}^{n-1}\subset \mathbb{R}^n$. For this reason, in what follows we will always assume $\bar x$ as above. Before giving the precise statement, we would like to introduce some notation. Consider $\delta,\e>0$ and $\tau>0$ fixed. Let $\ell \in \R^n \times \R $ be the line defined by 
\begin{equation}
\ell = \{ (0^n,y) : y \in \R\}.
\end{equation}

For $\bar x = (\bar r\bar e, 0) \in M =\R^n \times \R$ where $\bar{e}\in \mathbb{S}^{n-1}\subset \mathbb{R}^n$,  we define an associated diffeomorphism $\Phi_{\tau,\bar{r},\delta, \e} : \R^n\times \R \to M$ by 
\begin{equation}\label{eqn: good diffeomorphism}
	\Phi_{\tau,\bar{r}, \delta,\e}(x) = \bigg(\tau^{1/2}z +\bar r \bar e, \ \frac{\tau^{1/2}}{\varphi_{\delta,\e}(\bar r + \tau^{1/2})}\, y\bigg).
\end{equation}
Notice that $\Phi_{\tau,\bar{r}, \delta,\e}(0)= \bar x$.  In principle, there are two rescalings performed by $\Phi_{\tau,\bar{r}, \delta,\e}$. Firstly, we rescale the coordinates centered at $\bar x$ by $\tau$ to compensate the scaling of metric by $\tau^{-1}$. Secondly, we additionally rescale the $y$ coordinate to account for the degeneracy of $\varphi_{\delta,\e}$. Under the diffeomorphism, the image of $\ell$ where the metric degenerates  will play a key role. Let $\bar \rho=\tau^{-1/2} \bar r$ and denote the pullback of $\ell$ as 
\begin{equation}
	\tilde{\ell} : = \Phi_{\tau, \bar{r}, \delta,\e}^*(\ell) = \{ ( - \bar \rho \bar {e}, y) : y \in \R\}\, .
\end{equation}
In particular, if $\bar \rho$ tends to infinity along the sequence, then the degeneracy $\tilde \ell$ is pushed off to infinity and becomes invisible in the limit. On the other hand, if $\bar \rho$ stays bounded, we will show that the singularity of the sequence of rescaled pull-back metric is still mild and close to Euclidean metric away from $\tilde {\ell}$.


The next lemma shows that the pull-back of $\tau^{-1}g_{\delta,\e}$ under $\Phi_{\tau,\bar{r}, \delta,\e}$ converges to Euclidean metric away from $\tilde\ell$ as $\delta,\e\rightarrow 0$.

\begin{lemma}[Good charts]\label{lemma: good charts}
Fix $n\geq 3, L>0$ and consider sequences $\e_i ,\delta_i\to 0,$ $\tau_i \in (0,L]$ and $\bar{r}_i \in (0,\infty)$. Let $\bar{\rho}_i = \tau_{i}^{-1/2}\bar r_i$ and assume $\bar{\rho}_i \to \bar{\rho}_{\infty} \in [0, \infty]$. Define $	\tilde{\ell}_\infty =  \{(-\bar \rho_\infty  \bar e, y) : y \in \R\}   $ if $\bar \rho_\infty <\infty$ and $\tilde{\ell}_\infty =  \emptyset$ if $\bar \rho = \infty$.

Then the metrics 
$
g_i := \Phi_{\tau_i, \bar{r}_i, \e_i}^*(\tau_i^{-1}g_{\delta_i,\e_i}) 
 $
 converge to the Euclidean metric $g_{\mathbb{R}^{n+1}}$ in $C^\infty_{loc}(\R^{n+1}\setminus \tilde{\ell}_{\infty})$. Furthermore, in the case when $\bar\rho_\infty <\infty,$ we have
\begin{equation}\label{eqn: lower bound near line}
\min\bigg\{\frac{1}{2},\bigg(\frac{|z+\bar\rho_i \bar e_i|}{ 1 + \bar \rho_i}\bigg)^{\delta_i}\bigg\}g_{euc} \leq 	g_i (z,y) \leq g_{euc}\, .
\end{equation}
Here $|\cdot|$ denotes the Euclidean norm on $\R^n$.
\end{lemma}

\begin{proof}
For notational convenience, we omit the index $i$ and let $\vphi = \vphi_{\delta, \e}$. Let $g = \Phi_{\tau, \bar{r},\delta, \e}^*(\tau^{-1}g_{\delta,\e})$. Then at a point $(z,y)\in \mathbb{R}^n \times \mathbb{R}$, we have 
\begin{equation}
\begin{split}
		g|_{(z,y)} 
		& = g_{cone, f, \tau} \left(z +\bar\rho \bar e \right) + \tilde{\vphi}^2 (|z+\bar \rho \bar{e}|)dy^2,
\end{split}
\end{equation}
 where $g_{cone,f,\tau}$ denotes the pull-back of the cone metric on the $\mathbb{R}^n$ component and we set
\begin{equation}\label{eqn: def phi tilde}
	\tilde{\vphi}(s) = \frac{\vphi(\tau^{1/2}s)}{\vphi(\tau^{1/2}(1+\bar \rho))}.
\end{equation}
It is clear that the cone metrics $g_{cone, f, \tau}$ converge smoothly to the standard Euclidean metric on $\R^n$ on any compact set as $i\rightarrow +\infty$. Therefore, to show the desired smooth convergence of the metrics away from $\tilde{\ell}_{\infty}$, it remains to show that
\begin{equation}\label{eqn: vphi to 1}
 \tilde{\vphi} (|\cdot + \bar \rho \bar e|) \to 1
\end{equation}
  smoothly on compact subsets of $\R^n \setminus\{  -\bar \rho_\infty \bar e\},$ 
  To this end, note that  $\tilde{\vphi}$ satisfies 
$	|{\tilde{\vphi}'(s)}/{\tilde{\vphi}(s)}| \leq 50\delta/s$ by \eqref{eqn: key derivative bounds for phi}
and that $\tilde{\varphi}( 1+\bar \rho)=1$ by definition.
In this way, we have 
\begin{equation}\label{eqn: integral of log phi}
\begin{split}
|\log \tilde \varphi(|z + \bar \rho \bar{e}|)| 
= \left| \int^{|z + \bar \rho \bar{e}|}_{1 +\bar \rho}\frac{\tilde{\varphi}'(s)}{\tilde{\varphi}(s)}\, ds \right|
&\leq  \left| \int^{|z + \bar \rho \bar{e}|}_{1+ \bar \rho}\frac{50\delta}{s}\, ds \right|=50\delta \left| \log \left(\frac{ |z + \bar \rho \bar e| }{ 1 + \bar \rho}\right)  \right|.
\end{split}
\end{equation}
To establish \eqref{eqn: vphi to 1}, and hence the desired convergence, we consider two cases.

{\bf Case 1. $\bar{\rho}_\infty = \infty$:}  When $\bar \rho = \infty$, it is clear that $\frac{ |\cdot + \bar \rho_i \bar e| }{1 + \bar \rho_i} \to 1$   uniformly on compact subsets as $i\to \infty$, and hence $\tilde{\vphi}(| \cdot + \bar \rho_i \bar e|) \to 1$ uniformly on compact subsets by \eqref{eqn: integral of log phi}. The higher order convergence follows analogously by using \eqref{eqn: key derivative bounds for phi}, and thus we establish \eqref{eqn: vphi to 1} in this case.

\hfill

{\bf Case 2. $\bar{\rho}_\infty <\infty$:} 
Fix $\gamma \in (0,1]$. For any $z\in \R^n$ in the annular region defined by
\begin{equation}
\g ( 1 +\bar \rho_\infty) \leq |z+ \bar \rho_\infty \bar e|  \leq \g^{-1}( 1 +\bar \rho_\infty),
\end{equation}
 we see from \eqref{eqn: integral of log phi} that 
$|\log \tilde \varphi(|z + \bar \rho \bar{e}|)|  \leq 2\delta |\log \g|$
for $i$ sufficiently large.
 Exponentiating both sides, we discover that 
\begin{equation}\label{decay-metric}
\begin{split}
\gamma^{2\delta}  \leq \tilde \varphi ( | z + \bar \rho\bar e|) \leq \gamma^{-2\delta}.
\end{split}
\end{equation}
So, as $\e,\delta $ tend to zero, we see that $\tilde\varphi( |z + \bar \rho\bar e|)$ converges uniformly to $1$ for all $z$ in this set. 
The convergence of the higher derivatives of $\vphi$ follows in the same way thanks to \eqref{eqn: key derivative bounds for phi}  and we see that \eqref{eqn: vphi to 1} holds in this case as well.


\medskip

Finally, we show \eqref{eqn: lower bound near line}. The upper bound is immediate from the construction of the metrics. To establish the lower bound  in \eqref{eqn: lower bound near line}, we note that 
\begin{equation}\label{eqn: lower bd cone}
g_{cone, f, \tau} \geq \frac{1}{2} g_{\mathbb{R}^n} 
\end{equation}
by construction. 
Next, notice that for any $z$ such that $|z+ \bar \rho \bar e|  \geq 1+ \bar \rho$, we have $\tilde{\vphi}(|z+ \bar \rho \bar e|) \geq 1$ because $\vphi$ is a monotone increasing function. On the other hand, for any $z $ with $0<|z + \rho \bar \e| \leq 1+ \bar \rho$, 
we exponentiate the left- and right-hand sides of \eqref{eqn: integral of log phi} to find that 
\begin{equation}\label{decay metric 2}
\tilde \varphi ( | z + \bar \rho\bar e|)\geq  \bigg( \frac{|z + \bar \rho \bar e|}{1+\bar \rho}\bigg)^{\delta} .
\end{equation}
Together \eqref{eqn: lower bd cone} and \eqref{decay metric 2}, we  establish \eqref{eqn: lower bound near line}. This concludes the proof of the lemma.
\end{proof}

In the proof of Proposition~\ref{thm: entropy goes to zero}, we would like to show that the entropy $\mu(\tau^{-1}g_{\delta,\e},1)$ converges to $0$ for any $\tau \in (0,L]$ as $\e, \delta \to 0$ by analyzing the limit of the corresponding sequence of minimizers. A key point will be to show that the limit is non-trivial, for which we need the following  uniform mean value inequality.

\begin{lemma}[Mean value inequality]\label{lem: MVI}
Fix $n\geq 3, L >0$ and consider sequences $\delta_i,\e_i \to 0,$ $\tau_i \in (0,L]$, and $\bar{r}_i \in (0,\infty)$. Let $u_i$ be a minimizer of $\mu(\tau_i^{-1}g_{\delta_i,\e_i},1)$, let $\Phi_{\tau_i,\bar r_i,\delta_i,\e_i}$ be the diffeomorphism defined in \eqref{eqn: good diffeomorphism}, and let $v_i=\Phi_{\tau_i,\bar r_i,\delta_i,\e_i}^*u_i$. 

Suppose that  $\bar{\rho}_i = \bar{r}_i/\tau_{i}^{1/2}\rightarrow \bar\rho_{\infty}$ for some $\bar\rho_\infty\in [0,+\infty)$. Then there is $N(n,\bar\rho_\infty)\in \mathbb{N}$ and $C(n)>0$ such that if $i>N$, then for all $x
\in \mathbb{R}^{n+1}$, we have  
\begin{equation}\label{mean-value}
\|v_i\|_{L^\infty(B_{euc}(x,{1/4}))}\leq C(n) \left(\int_{B_{euc}(x,1)}v_i^{2}\,d\vol_{g_i} \right)^{1/2}
\end{equation}
for $i$ sufficiently large.
Here, the balls are taken with respect to the Euclidean metric on  $\R^{n+1}$.
In particular, $\|v_i\|_{L^\infty(\mathbb{R}^{n+1})}\leq C(n)$.

If $\bar \rho_i\rightarrow +\infty$, then the $L^\infty$-estimate holds in the following sense: for all $\Omega\Subset \mathbb{R}^{n+1}$, there is $N(\Omega)\in \mathbb{N}$ such that if $i>N$, $\|v_i\|_{L^\infty(\Omega)}\leq C(n)$.
\end{lemma}

\begin{proof}
We let $g_i = \Phi_{\tau_i,\bar r_i,\delta_i,\e_i}^* \tau_i^{-1}g_{\delta_i,\e_i}$. For notational convenience, we will omit the index $i$ when no confusion can arise. Moreover, each ball is taken with respect to the Euclidean metric.

If $\bar \rho_i \rightarrow  +\infty$, then by Lemma \ref{lemma: good charts}, $g_i\rightarrow g_{euc}$ in $C^k_{loc}(\mathbb{R}^{n+1})$ for all $k\in \mathbb{N}$,  and the result follows from standard Moser iteration;  see for example \cite{li2018} or \cite{ZhangLS}.
It therefore suffices to consider the case where $\bar \rho_\infty<\infty$. In this case, we modify the Moser iteration argument to account for the mild singularity of $g$ near $\ell$.

 Keeping in mind the lower bound for the metric \eqref{eqn: lower bound near line}  established in Lemma \ref{lemma: good charts}, we define the function $\Lambda(x)=\max\big\{ 2,\big(\frac{|z+\bar\rho \bar e|}{ 1 + \bar \rho}\big)^{-\delta}\big\}$ (where $x = (z,y) \in \R^n \times \R$) so that $g_{euc} \leq \Lambda(x) g$. We will make repeated use of the upper bound in \eqref{eqn: lower bound near line}, which implies that $d\vol_g \leq dx$ and $|\pa u| \leq |\na_g u|$ for any function $u$, where $|\pa u|$ denotes the Euclidean norm of the Euclidean gradient.

As a first step, fix any $p>1$ and $\Omega\Subset \mathbb{R}^{n+1}$. Provided $\delta<\delta_0(n,p)$, we note that  $\|v\|_{L^{2/p}(\Omega, g_{euc})}$ is uniformly bounded for any $i$. Indeed, apply H\"{o}lder's inequality and use the fact that  $\int_{\R^n} v^2 \,d\vol_g =1$ to see that
\begin{equation}\label{initial-Lp-bound}
\begin{split}
\int_{\Omega}v^{2/p}dx \leq \int_{\Omega}\Lambda^{(n+1)/2}v^{2/p} d\vol_g
&\leq  \left(\int_{\Omega}v^2 d\vol_g\right)^{1/p}\left( \int_{\Omega}\Lambda^{(n+1)p'/2}dx\right)^{1/p'}\leq C,
\end{split}
\end{equation}
where $C= C(diam(\Omega),p,n)$,
provided that $\delta$ is small enough so that $\Lambda^{(n+1)p'/2}$ is integrable. Here $p'$ denotes the H\"{o}lder conjugate of $p$.

Next, let $\bar x=(-\bar \rho\bar e,\bar y)$ for some $\bar y\in \mathbb{R}$. Let $1\geq r_1>r_0>1/4$ and $\phi$ be a Lipschitz cutoff on $\mathbb{R}^{n+1}$ defined by 
\begin{equation}
\phi(x)=
\left\{
\begin{array}{ll}
1 &\quad\text{on} \quad B(\bar x, r_0);\\
\displaystyle \frac{r_1-|x-\bar x|}{r_1-r_0}&\quad\text{on} \quad B(\bar x , r_1);\\
0 &\quad\text{outside} \quad B(\bar x,r_1).
\end{array}
\right.
\end{equation}

We will henceforth use $B_r$ to denote $B(\bar x,r )$ for simplicity. 
We now establish an energy estimate. On one hand,
since $v$ satisfies \eqref{eqn: EL lemma tau=1}, we multiply the equation by $v^{p-1}\phi^2$ and integrate to find that for $p>1$,
\begin{equation}\label{equ--1}
\begin{split}
 \quad \int_{B_{r_1}}\phi^2 v^{p-1}  (-4\Delta v) d\vol_g 
&= \int_{B_{r_1}} \phi^2 v^{p}  (-R + 2\log v +((n+1)+\mu+\log (4\pi)^{(n+1)/2}) d\vol_g\\
&\leq \int_{B_{r_1}} C_n \phi^2 v^{p}  +2 \phi^2 v^p \log v \; d\vol_g.
\end{split}
\end{equation}
In the second line, we have  used the fact that $R\geq -1$ thanks to Proposition~\ref{prop: scalar nonneg}, provided we choose $\e$ and $\delta$ sufficiently small (thus $i$ sufficiently large), and that $\mu = \mu(g_i,1) \leq 0$. Here the connection is with respect to $g$.
On the other hand, we integrate by parts and apply the Cauchy-Schwarz inequality to find
\begin{equation}\label{equ--2}
\begin{split}
 \int_{B_{r_1}} \phi^2 v^{p-1}  (-4\Delta v) d\vol_g
&=\int_{B_{r_1}}4 \nabla v \cdot \nabla (\phi^2 v^{p-1}) d\vol_g\\
&\geq  \int_{B_{r_1}} 4(p-1)\phi^2 v^{p-2}|\nabla v|^2-8v^{p-1} \phi|\nabla v||\nabla \phi| d\vol_g\\
&\geq  \int_{B_{r_1}}2(p-1)\phi^2 v^{p-2}|\nabla v|^2-\frac{8}{p-1}v^p  |\nabla \phi|^2d\vol_g.
\end{split}
\end{equation}
Combining \eqref{equ--1} and \eqref{equ--2}, we see that 
\begin{equation}\label{eqn: energy a}
2(p-1) \int_{B_{r_1}}\phi^2 v^{p-2}|\nabla v|^2 \, d\vol_g \leq 
\int_{B_{r_1}} C_n \phi^2 v^{p}  +2 \phi^2 v^p \log v +\frac{8}{p-1} v^p |\nabla \phi|^2\; d\vol_g
\end{equation}

Now, let us use the shorthand $\bar \mu=(n+2)/(n+1)>1$ and $\nu=(n+2)/(n+3)<1$. Noticing that $(2\nu)^* = 2(n+2) \nu /(n+2-  2\nu) = 2\bar\mu$, we apply the Euclidean  Sobolev inequality,
 followed by \eqref{eqn: lower bound near line} and then H\"{o}lder's inequality to see that
\begin{equation}\label{equ--3}
\begin{split}
 \bigg(\int_{B_{r_1}} |\phi v^{p/2}|^{2\bar\mu} dx \bigg)^{1/2\bar\mu}&\leq C_n\bigg(\int_{B_{r_1}} |\partial(\phi v^{p/2})|^{2\nu} dx\bigg)^{1/2\nu} \\
&\leq C_n\bigg(\int_{B_{r_1}}\Lambda^{(n+1)/2}|\nabla(\phi v^{p/2})|^{2\nu} d\vol_g\bigg)^{1/2\nu} \\
&\leq C_n\bigg(\int_{B_{r_1}} \Lambda^{(n+2)(n+1)/2} dx \bigg)^{1/2(n+1)}\left( \int_{B_r}|\nabla (\phi v^{p/2})|^2d\vol_g\right)^{1/2}\\
&\leq C_n(1+\bar \rho^2)^{c_n\delta}\left( \int_{B_r}|\nabla (\phi v^{p/2})|^2d\vol_g\right)^{1/2}.
\end{split}
\end{equation}

Using the same trick of interchanging $g$ and $g_{euc}$, \eqref{eqn: energy a} and H\"{o}lder's inequality imply that 
\begin{equation}
\begin{split}
\int_{B_{r_1}} |\nabla(\phi v^{p/2})|^{2} d\vol_g
&\leq  \int_{B_{r_1}}p^2\phi^2 v^{p-2} |\nabla v|^2+ 2v^{p}|\nabla \phi|^2 d\vol_g\\
&\leq   \int_{B_{r_1}} C_n p\phi^2 v^p+C_n p v^{p}|\nabla \phi|^2 +C_n p \phi^2 v^p \log v  \,  d\vol_g\\
&\leq  \int_{B_{r_1}}C_np\phi^2 v^p+C_n\Lambda v^{p}|\partial \phi|^2+C_n p  \phi^2 v^p \log v \, dx\\
&\leq \frac{C_n p(1+\bar \rho^2)^{c_n\delta}}{(r_1-r_0)^2} \bigg(\int_{B_{r_1}} v^{p\lambda} dx\bigg)^{1/\lambda}.
\end{split}
\end{equation}
Here we let  $\lambda=2(n+2)/(2n+3) \in (1,\bar\mu)$, and we have assumed that $p$ is bounded uniformly away from $1$ so that $p/(p-1)$ is bounded above by a universal constant. We have used the fact that $\int_{\mathbb{R}^{n+1}}v^2d\vol_g = 1$ to control the term arising from $\log v$. By combining this with \eqref{equ--3} and replacing $p$ by $p\lambda$, we conclude that for all $p>\frac{2n+4}{2n+3}$, 
\begin{equation}\label{moser}
\|v\|_{L^{p\sigma}(B_{r_0}, g_{euc})}\leq \left[\frac{C_np (1+\bar \rho^2)^{c_n\delta}}{(r_1-r_0)^2} \right]^{1/p}\|v\|_{L^{p}(B_{r_1}, g_{euc})}
\end{equation}
where $\sigma=\frac{2n+3}{2n+2}>1$. Let $R_0=1$ and $R_k=1-\sum_{i=1}^k2^{-i-1}$ so that $\lim_{k\rightarrow +\infty}R_k=1/2$. Applying \eqref{moser} inductively with $r_1=R_k$, $r_0=R_{k+1}$ and $p=p_0\sigma^k$ for $p_0\in (\frac{2n+4}{2n+3},2)$, we have if $\delta$ is small depending only on $n$, $p_0$ and $\bar \rho_\infty$, then 
\begin{equation}
\begin{split}
\log \|v\|_{L^{p_0\sigma^{k+1}}(B_{R_{k+1}}, g_{euc})}&\leq \log \|v\|_{L^{p_0}(B_{1}, g_{euc})}+\sum_{i=1}^k \frac{C_n}{\sigma^i}\log  \left(C_n(1+\bar \rho^2)^{c_n\delta}\sigma^{i} 2^i\right)\\
&\leq C(n)+ \log \|v\|_{L^{p_0}(B_{1}, g_{euc})}.
\end{split}
\end{equation}
Here we have used the fact that $\bar\rho_i\rightarrow \bar\rho_\infty$. By letting $k\rightarrow +\infty$, we conclude that 
\begin{equation}\label{p-mean-value}
\|v_i\|_{L^{\infty}(B_{1/2}(\bar x))}\leq C(n)\left(\int_{B_1(\bar x)}v_i^{p_0}  \, dx\right)^{1/p_0}
\end{equation}
for some $p_0\in (\frac{2n+4}{2n+3},2)$. Together with \eqref{initial-Lp-bound}, we have the upper bound of $v_i$ on $B_{1/2}(\bar x)$.  The result follows by choosing $p_0<2$ and apply H\"older's inequality once more using the integrability of $\Lambda$ when we replace the volume form from $dx$ to $d\vol_g$. For $\bar x=(\bar z, \bar y)$ where $|\bar z+\bar \rho_i\bar e|\geq 1/2$, we can repeat the argument but apply the iteration on a small ball $B_{1/4}(\bar x)$ so that $g$ is uniformly equivalent to $g_{euc}$. 
\end{proof}

Now we are ready to prove Proposition~\ref{thm: entropy goes to zero} 
\begin{proof}[Proof of Proposition~\ref{thm: entropy goes to zero}]

Let $\eta>0$ and for a complete Riemannian metric $g$, we define 
\begin{equation}
\tau_0(g)=\sup\{ s>0: \mu(g,\tau)>-\eta\;\; \text{for all}\; \tau\in (0,s)\}.
\end{equation}

Note that since $g_{\delta,\e}$ has bounded curvature and is not isometric to Euclidean space, we have $\mu(g_{\delta,\e},\tau)<0$;
 see Lemma~\ref{lem: rigidity} or the proof of \cite[Lemma 17.19]{Chow3}, replacing Perelman's differential Harnack estimate with its noncompact generalization 
 established in \cite{CTY11}.

{\bf Claim 1:}
For $\delta,\e\in (0,1)$ fixed, we have $\tau_0(g_{\delta,\e})>0$.
\begin{proof}
[Proof of Claim 1]
Since it is standard and similar to the proof in the compact case, see \cite[Proposition 17.20]{Chow3}, we sketch the proof only.
It suffices to show that $\lim_{\tau\rightarrow 0^+}\mu(g_{\delta,\e},\tau)=0$. Since the constants $\e$ and $\delta$ are fixed, we will omit the subscripts $\delta,\e$ for notational convenience.
 Suppose the claim is not true, we can find a sequence of $\tau_i\rightarrow 0^+$ such that $\lim_{i\rightarrow +\infty}\mu(g,\tau_i)<-\eta$ for some $\eta>0$.  Consider the rescaled metric $g_i=\tau_i^{-1}g$ so that $\mu(g_i,1)<-\eta$ for $i$ sufficiently large. By Lemma \ref{lemma: existence of extremal}, we can find a sequence of minimizers $u_i$ of $\mu(g_i,1)$ such that 
\begin{equation}\label{eqn: EL final pf}
 	- 4 \Delta u_i + R_iu_i - 2u_i \log u_i -\left(\frac{n+1}{2}\log(4\pi) + (n+1)+\mu(\tau_i^{-1}g,1 )\right)u_i =0.
 \end{equation}
 Since $u_i$ has exponential decay at infinity by Lemma \ref{lemma: existence of extremal}, we can find $p_i\in \mathbb{R}^{n+1}$ such that $u_i(p_i)=\max u_i$. Moreover, $u_i(p_i)\geq C(n,\eta)$ by applying the maximum principle to \eqref{eqn: EL final pf}. On the other hand, since $g$ is a smooth metric, it is easy to show that $(\mathbb{R}^{n+1},g_i,p_i)$ converges to $(\mathbb{R}^{n+1},g_{euc},0)$ as $i\rightarrow +\infty$ in the $C^\infty$-Cheeger-Gromov sense. Hence $u_i$ will converge to $u_\infty>0$ (modulo diffeomorphism) in $C^\infty_{loc}(\mathbb{R}^{n+1})$ which has $\int_{\mathbb{R}^{n+1}}u_\infty^2dx\leq 1$ and solves the equation
 \begin{equation}
 	4 \Delta u_\infty + 2u_\infty \log u_\infty +\left(\frac{n+1}{2}\log(4\pi) + (n+1)+\mu_\infty\right)u_\infty =0
 \end{equation}
 where $\mu_\infty = \lim_{i \to \infty} \mu(\tau_i^{-1}g, 1) \leq -\eta.$
 Moreover, a standard Moser iteration argument shows that $u_i$ are uniformly bounded and hence $u_\infty$ is bounded as well. This can be proved using the argument in the work of \cite{li2018} or \cite{ZhangLS}, or  by modifying the proof of Lemma \ref{lem: MVI}. Therefore, using the equation we see that $u_\infty\in W^{1,2}(\mathbb{R}^{n+1})$ and hence $\mu_\infty=0$ by Lemma~\ref{lem: euclidean gaussian},  which contradicts $\mu(\tau_i^{-1}g,1)<-\eta$ for $i$ sufficiently large.
\end{proof}

We now prove that there is $\e_0$ small enough such that for all $\e,\delta<\e_0$, we have $\tau_0(g_{\delta,\e})\geq L$. The proof is similar to that of Claim 1 above, but additional care must be taken with respect to the convergence of the metrics and their corresponding minimizers. Suppose the conclusion is not true, and thus we can find a sequence of $\delta_i,\e_i\rightarrow 0^+$ such that $\tau_i=\tau_0(g_{\delta_i,\e_i})<L$ for all $i$. By definition, we have $\mu_i=\mu(g_{\delta_i,\e_i},\tau_i)=\mu(\tau_i^{-1}g_{\delta_i,\e_i},1)=-\eta$. 
By Lemma \ref{lemma: existence of extremal}, we can find a minimizer $u_i$ of $\mu(\tau_i^{-1}g_{\delta_i,\e_i},1)$ which attains its maximum at some $\bar x_i=(\bar z_i,\bar y_i)\in \mathbb{R}^n\times \mathbb{R}$. We may assume without loss of generality that $\bar y_i=0$ and $\bar z_i=\bar r_i \bar e$ for some fixed $\bar e\in \mathbb{S}^{n-1}\subset \mathbb{R}^n$ by the symmetry of $g_{\delta_i,\e_i}$. We again let $\bar \rho_i = \tau_i^{1/2}/\bar r_i$.  Up to a subsequence, which we will not relabel, we have $\bar \rho_i \to \bar \rho_\infty \in [0,\infty].$

Let $\Phi_{\tau_i,\bar r_i,\delta_i,\e_i}$ be the diffeomorphism defined in \eqref{eqn: good diffeomorphism} using $\bar x_i$ above and let $g_i = \Phi_{\tau_i,\bar r_i,\delta_i,\e_i} (\tau_i^{-1} g_{\delta_i,\e_i})$. If $\bar \rho_\infty = +\infty,$ then  by Lemma~\ref{lemma: good charts}, we have $(\mathbb{R}^{n+1},\tau_i^{-1}g_{\delta_i,\e_i},\bar x_i)\rightarrow (\mathbb{R}^{n+1},g_{euc},0)$ in the $C^\infty$-Cheeger-Gromov sense. Then the proof of Claim 1 carries over and we reach a contradiction to the assumption that $\mu_i = -\eta$.
 It therefore suffices to consider the case when $\bar \rho_\infty<+\infty$.
 
 As in Lemma~\ref{lemma: good charts}, we let $\tilde{\ell}_\infty = \{ (-\bar \rho_\infty \bar e , y): y \in \R\}$.
By Lemma~\ref{lemma: good charts}, $g_i$ converges to $g_{euc}$ in $C^k_{loc}(\R^{n+1} \setminus \tilde \ell_\infty)$. Together with the $L^\infty$ estimate of $v_i=\Phi_{\tau_i,\bar r_i,\delta_i,\e_i}^*u_i$ from Lemma~\ref{lem: MVI}, we have $v_i\rightarrow v_\infty$ in $C^\infty_{loc}(\mathbb{R}^{n+1}\setminus \tilde \ell_\infty)$ and $L^p_{loc}(\mathbb{R}^{n+1})$ for $p>0$ by the Dominated Convergence Theorem. Therefore, $\int_{\mathbb{R}^{n+1}}v_\infty^2 dx\leq 1$ by Fatou's lemma and $v_\infty$ solves 
 \begin{equation}
 	4 \Delta v_\infty + 2v_\infty \log v_\infty +\left(\frac{n+1}{2}\log(4\pi) + (n+1) -\eta \right)v_\infty =0
 \end{equation}
on  $\R^{n+1} \setminus \tilde \ell_\infty$. It remains to establish the following claim.

{\bf Claim 2:}
The limit $v_\infty$ is non-trivial and $v_\infty\in W^{1,2}(\mathbb{R}^{n+1})$. 
\begin{proof}
[Proof of claim] 
We first show that $v_\infty$ is non-trivial. 
By Lemma~\ref{lem: MVI}, we have 
\begin{equation}
v_i(0)\leq C\bigg(\int_{B_1(0)}v_i^{2}d\vol_{g_i}\bigg)^{1/2}
\end{equation}
for some universal constant $C$ independent of $i$. On the other hand, by the decay rate of $v_i$ from Lemma~\ref{lemma: existence of extremal}, we may apply maximum principle to the Euler-Lagrange equation at its maximum point,  which by our selection of $\bar x_i$ is the origin to show that at $x=0$,
\begin{equation}
\begin{split}
2v_i \log v_i&\geq R_{g_i} v_i -\left(\frac{n+1}{2}\log(4\pi)+(n+1)-\eta \right)v_i
\end{split}
\end{equation}
and hence 
\begin{equation}
\int_{B_1(0)}v_i^{2}d\vol_{g_i} \geq c(n,\eta).
\end{equation}
This shows that $v_\infty \neq 0$ as $v_i\rightarrow v_\infty$ in $L^{p}_{loc}(\mathbb{R}^{n+1})$ and $g_i\rightarrow g_{\mathbb{R}^{n+1}}$ in $L^p_{loc}(\mathbb{R}^{n+1})$ by construction for all $p>0$.

It remains to show that $v_\infty\in W^{1,2}(\mathbb{R}^{n+1})$. Since $\int_{\mathbb{R}^{n+1}}v_\infty^2 dx\leq 1$, it remains to consider $||\partial v_\infty||_{L^2(\mathbb{R}^{n+1})}$. We first point out that for each $g_i$, $||\nabla v_i||_{L^2(\mathbb{R}^{n+1},g_i)}$ is uniformly bounded. This can be seen by integrating the Euler-Lagrange equation with a cutoff function $\phi$ with respect to $g_i$, 
\begin{equation}
\begin{split}
0&=\int_{\mathbb{R}^{n+1}} \phi v_i \left(-4\Delta_{g_i} v_i+R_{g_i}v_i-\left(\frac{n+1}{2}\log(4\pi)+(n+1)-\eta \right)v_i -2v_i\log v_i\right) d\vol_{g_i}
\end{split}
\end{equation}
 Since the scalar curvature of $g_i$ is uniformly bounded from below and $v_i$ is uniformly bounded by Lemma~\ref{lem: MVI} for $i$ sufficiently large, we have for suitable cutoff function $\phi$,
\begin{equation}
\begin{split}
C(n,\eta)\int_{\mathbb{R}^{n+1}}v_i^2d\vol_{g_i}&\geq \int_{\mathbb{R}^{n+1}} \phi v_i \left(-\Delta_{g_i} v_i\right) d\vol_{g_i}+8\int_{\mathbb{R}^{n+1}}v_i^2 \frac{|\nabla \phi|^2}{\phi} d\vol_{g_i}\\
&=\int_{\mathbb{R}^{n+1}} \nabla v_i \cdot \nabla (\phi v_i)d\vol_{g_i}+8\int_{\mathbb{R}^{n+1}}v_i^2 \frac{|\nabla \phi|^2}{\phi} d\vol_{g_i}\\
&\geq \frac12\int_{\mathbb{R}^{n+1}} \phi |\nabla v_i|^2 d\vol_{g_i} \, .
\end{split}
\end{equation}

 By letting $\phi\rightarrow 1$, this gives the uniform boundedness of $||\nabla v_i||_{L^2(\mathbb{R}^{n+1},g_i)}$. Now let $B_{\mathbb{R}^n}(-\bar\rho_\infty \bar e,r)$ be the ball of radius $r$ in $\mathbb{R}^n$ and $p>1$. By the metric equivalence from Lemma~\ref{lemma: good charts}, $\Lambda^{-1} g_{euc}\leq g_i\leq g_{euc}$ for some $\Lambda(z,y)\leq \max\{|z+\bar \rho_\infty \bar e|^{-\delta},2\}$,
\begin{equation}
\begin{split}
 \int_{B_{\mathbb{R}^n}(-\bar\rho_\infty \bar e,r)\times [-r,r]}|\partial v_i|^{2/p} dx
&\leq \int_{B_{\mathbb{R}^n}(-\bar\rho_\infty \bar e,r)\times [-r,r]}\Lambda^{n/2}|\nabla v_i|^{2/p} d\vol_{g_i}\\
&\leq \left(\int_{\mathbb{R}^{n+1}}|\nabla v_i|^2 d\vol_{g_i}\right)^{1/p}\left(\int_{B_{\mathbb{R}^n}(-\bar\rho_\infty \bar e,r)\times [-r,r]} \Lambda^{np^*/2}dx\right)^{1/p^*}\\ 
&\leq C^{1/p} \left[\frac{Cr^{n+1}}{1-c_n\delta_i p^*}\right]^{1/p^*}\, .
\end{split}
\end{equation}
Hence if we let $i\rightarrow +\infty$ followed by $p\rightarrow 1$ and $r\rightarrow +\infty$, we have $v_\infty \in W^{1,2}( \mathbb{R}^{n+1})$. 
\end{proof}

By the claim and the proof of Lemma~\ref{lem: euclidean gaussian}, we can deduce that $\lim_{i\rightarrow +\infty}\mu(\tau_i^{-1}g_{\delta_i,\e_i},1)=\mu_\infty=0$ which contradicts with  the fact that $\mu(\tau_i^{-1}g_{\delta_i,\e_i},1)=-\eta$. This completes the proof.
\end{proof}

\bibliographystyle{alpha}
\bibliography{epsReg}

\end{document}